\documentclass{article}
\usepackage[%
journal=FoCM,
lang=british,
]{ems-journal}

\usepackage{mathrsfs}
\usepackage{longtable}

\numberwithin{equation}{section}
\theoremstyle{plain}
\newtheorem{theorem}{Theorem}[section]
\newtheorem{proposition}[theorem]{Proposition}
\newtheorem{lemma}[theorem]{Lemma}
\newtheorem{corollary}[theorem]{Corollary}
\theoremstyle{definition}
\newtheorem{definition}[theorem]{Definition}
\newtheorem{assumption}[theorem]{Assumption}
\newtheorem{remark}[theorem]{Remark}

\DeclareMathOperator{\dist}{dist}
\newcommand{\norm}[1]{\lVert#1\rVert}
\newcommand{\abs}[1]{\lvert#1\rvert}
\newcommand{\N}{\mathbb{N}}
\newcommand{\mat}[1]{\mathbf{#1}}
\newcommand{\vect}[1]{\mathbf{#1}}
\newcommand{\Lop}{\mathcal{L}}
\newcommand{\Iop}{\mathcal{I}}
\newcommand{\Pop}{\mathcal{P}}
\newcommand{\eps}{\varepsilon}

\begin{document}

\title{Numerical Stability of Extended-Domain Spectral Collocation: Floating-Point Regularization and an Equilibrium Law for the Frame Constants}
\titlemark{Numerical Stability of Extended-Domain Spectral Collocation}

\emsauthor{1}{
	\givenname{Po-Yi}
	\surname{Wu}
	\mrid{}
	\zblid{}
	\orcid{}}{P.-Y.~Wu}

\Emsaffil{1}{
	\department{}
	\organisation{Independent researcher}
	\rorid{}
	\address{}
	\zip{}
	\city{}
	\country{Taiwan}
	\affemail{r04527030@gmail.com}}

\classification[65N35, 65T40, 42C15, 31A15]{65N12}


\begin{abstract}
Extended-domain spectral collocation expands the solution of a boundary value problem in the Laplacian eigenbasis of an enclosing box and collocates at as many nodes as modes. The square system is exponentially ill-conditioned, and the impossibility theorem for equispaced data rules out stable geometric convergence; yet such solvers converge geometrically to a few orders of magnitude above the unit roundoff. We explain why in one dimension. The collocation solution is the solution operator applied to the interpolant of the data in an operator-dependent test space whose Lebesgue constant must grow exponentially for geometric convergence, so exact solutions can amplify data errors exponentially. In floating-point arithmetic the error of any coefficient vector is bounded by its nodal residual, its distance to a competitor and two frame constants, a leakage and a restricted Marcinkiewicz--Zygmund constant, bounded uniformly in the number of modes at a fixed positive threshold; for a backward-stable solve, under a condition on the perturbed matrix, the bound is uniform in the conditioning for suitably analytic data and, with computed constants, within a factor of about ten of the error floor of a direct solver. For the Poisson operator the rate of the two constants along the tail of the spectrum is given by a constrained equilibrium problem of logarithmic potential theory. For an interval of length $L$ extended by $\delta$ at each end, the problems for $\delta$ and $L^2/(4\delta)$ are dual, the rate is linear for $\delta = L/2$, and it starts with slope minus the fraction of the extended domain occupied by the extension, so for $\delta < L/2$ the leakage constant obeys no power law with the exponent of the end of the spectrum. In experiments, convection--diffusion operators behave like the Poisson operator at moderate P\'eclet numbers.
\end{abstract}

\maketitle

\tableofcontents

\section{Introduction}

Spectral methods converge geometrically for analytic solutions on simple domains \cite{boyd2001chebyshev, canuto2007spectral, trefethen2000spectral}, and a long line of work extends them to complex geometries by embedding the physical domain in a larger, simple one: Fourier embedded domain methods, which extend the data smoothly and periodically to a rectangle \cite{boyd2005fourier, bueno2006fourier}; Fourier embedded and immersed boundary methods \cite{sabetghadam2009fourier, fang2011towards, stein2016immersed}; spectral collocation on an enclosing box \cite{lui2009spectral, wu2018development}; Fourier continuation along grid lines \cite{bruno2010high, lyon2010high}; and circular embedding \cite{gu2021efficient}. When the solution is expanded in the Laplacian eigenbasis of an enclosing box, the basis functions restricted to the physical domain form a redundant \emph{frame} rather than a basis. The simplest such method, the extended-domain collocation of \cite{wu2018development}, enforces the differential equation and the boundary conditions at exactly as many equispaced nodes as there are modes and solves the resulting square system directly, without oversampling and without regularization. Its matrix is exponentially ill-conditioned (Corollary~\ref{thm:poisson_illcond}, for the bulk collocation matrix of the Poisson operator, and Section~\ref{sec:num_greens}), and the impossibility theorem of Platte, Trefethen and Kuijlaars \cite{platte2011impossibility} shows that no method acting on equispaced samples can be both stable and geometrically convergent. Yet the method converges geometrically, to errors between about $10^{-13}$ and $10^{-10}$ on a square for the Poisson equation and for convection--diffusion equations with P\'eclet numbers, called Reynolds numbers there, up to $10^{5}$, and to about $10^{-12}$ for the Poisson equation on a disk with relocated boundary nodes \cite[Secs.~4 and 5]{wu2018development}. This paper explains why, and determines what limits the accuracy.

We analyse the one-dimensional prototype, collocation with $N$ modes in the Laplacian eigenbasis of an interval $\tilde I = (-\delta, L+\delta)$ that contains the physical domain $I = (0,L)$, for the Poisson operator and for constant-coefficient convection--diffusion. The explanation proceeds in four steps, each resting on results stated in Section~\ref{sec:results}.

\paragraph{1. The exact method is unstable.} The interpolatory collocation solution is $u_N = \Lop^{-1}\Pop_N f$, the continuous solution operator applied to the interpolant of the data in the operator-dependent test space $\Psi_N = \Lop(\Phi_N)$, where $\Phi_N$ is the trial space with the boundary conditions built in (Theorem~\ref{thm:representation}). If interpolation in the test space converges geometrically for all data analytic in a fixed neighbourhood of $\bar I$, the Lebesgue constant of this interpolation grows exponentially (Proposition~\ref{prop:impossibility}), and the exact solution amplifies errors in the data by up to that constant: computed exactly, the collocation solution of the one-dimensional test problem converges for exactly evaluated data but diverges beyond a threshold for data rounded to double precision (Remark~\ref{rem:divergence} and Section~\ref{sec:num_1d}). The exponential rates of the trial frame are explicit. In the Chebyshev variable the frame is a polynomial space on a subinterval of $[-1,1]$ sampled at nodes with the restricted arcsine density, and the growth rate $\kappa_\Lambda$ of its Lebesgue constant and the decay rate $\kappa_\sigma$ of its smallest singular value are differences of the logarithmic potential of the node measure, given in closed form by the Clausen function of the angle $\omega = \pi L/(L+2\delta)$ (Theorem~\ref{thm:rates}). For the Poisson operator the Lebesgue constant of the test space grows at the same rate $\kappa_\Lambda$, without any hypothesis on convergence (Corollary~\ref{cor:lebesgue_rate}).

\paragraph{2. Floating-point arithmetic regularizes.} The accuracy of the method is therefore a property of floating-point arithmetic. The $L^2$ error of an arbitrary coefficient vector, in particular of the output of an LU or QR factorization, is bounded by its nodal residual, its distance to any competitor, and two constants of the sampled frame at a threshold $t$: the leakage $\gamma_t$ of the singular subspace below $t$ into the solution, and the restricted Marcinkiewicz--Zygmund constant $\Gamma_t$ of the complementary subspace (Theorem~\ref{thm:two_constant}). A backward-stable solver with backward error $\eta$ thus attains the accuracy of a frame approximation regularized at the level of the backward error, provided that the perturbed matrix has no singular value far below $\eta$ (Corollary~\ref{cor:regimes_two_constant}), and for data analytic in a sufficiently large complex neighbourhood of the extended domain its error is bounded by a polynomial in $N$, uniformly in the conditioning of the system (Theorem~\ref{thm:competitor} and Corollary~\ref{cor:unconditional}). Evaluated with the exactly computed constants, the bound reproduces the error floor of an LU solve to within a factor of about ten (Section~\ref{sec:num_bound}).

\paragraph{3. An equilibrium law determines the two constants.} It remains to determine $\gamma_t$ and $\Gamma_t$. At a fixed threshold they are bounded uniformly in $N$ (Proposition~\ref{prop:apriori}). For the Poisson operator, at the end of the spectrum, $t = e^{-N\kappa_\sigma}$, they are at most $t^{a(\omega)}$ and $t^{a(\omega)-1}$ up to powers of $N$, with the exponent $a(\omega) = 1 - \kappa_\Lambda/\kappa_\sigma$ formed from the rates of Step~1 (Corollary~\ref{cor:exponent} and Proposition~\ref{prop:classical}), and we determine them along the whole tail of the spectrum: at the threshold $t = e^{-N\tau}$, $0 < \tau < \kappa_\sigma$,
\[
\max\{\gamma_t,\ t\Gamma_t\} = e^{Nr_\omega(\tau) + o(N)} \qquad (N \to \infty),
\]
where the \emph{tail law} $r_\omega(\tau)$ is given by a constrained equilibrium problem of logarithmic potential theory: an energy problem for probability measures bounded by the node density, extended to the whole extended interval, in the external field $-\tau$ on the physical interval (Theorem~\ref{thm:tail_law}). Its minimizer equals the node density on a layer at each end of the physical interval and leaves a gap next to it in the extension (Proposition~\ref{prop:equilibrium} and Corollary~\ref{cor:layer}). It describes both the extremal polynomials, which vanish at every node of the layer, and the optimal Lagrange interpolation sets, on subsets of the nodes and on virtual points of the extension, and the proof matches lower bounds from the former against upper bounds from the latter (Theorems~\ref{thm:tail_lower} and \ref{thm:dual} and Corollary~\ref{cor:sharpness}); the tail law is also the value of a two-level variant of Rakhmanov's extremal problem for polynomials bounded on a discrete set \cite{rakhmanov2007bounds} (Proposition~\ref{prop:extremal}).

\paragraph{4. The structure of the law decides the power law.} The equilibrium problems of the extensions $\delta$ and $L^2/(4\delta)$ are exactly dual (Theorem~\ref{thm:tail_duality}), so that the tail law is linear for the self-dual extension $\delta = L/2$, $r_{\pi/2}(\tau) = -\tau/2$, and $\max\{\gamma_t,\ t\Gamma_t\}$ behaves like $t^{1/2}$ at the exponential scale (Corollary~\ref{cor:self_dual}). For every extension the law starts with the slope $-2\delta/(L+2\delta)$, minus the fraction of the extended domain occupied by the extension; the proof builds near-optimal zero distributions from the explicit solution of a boundary-layer problem at the ends of the physical interval (Theorem~\ref{thm:initial_slope}). This fraction is smaller than the endpoint exponent $a(\omega)$ exactly when $\delta < L/2$ (Proposition~\ref{prop:critical}). Consequently, for every $\delta < L/2$ the leakage obeys no power law $\gamma_t \le Ct^{a(\omega)}$ with $C$ independent of $N$ and $t$, whereas for $\delta > L/2$ the tail law starts below the rate of that power law (Corollary~\ref{cor:critical}); for $\delta = L$ and $\delta = 2L$, certified interpolation configurations prove its bounds up to a power of $N$ along most of the tail (Corollary~\ref{cor:certified_instances}). The power law is the counterpart, for the sampled frame and in the solution metric, of a bound conjectured for the Fourier extension in \cite[Eqs.~(5.20) and (5.21)]{adcock2014numerical}.

Two further results complete the picture. One might expect convection, whose Green's function decays exponentially upstream, to localize the errors introduced by the frame. It does not appear to: the operator enters only through the test space, convection with coefficient $k$ shifts the phase of the $j$-th mode of the test space by at most $k(L+2\delta)/(j\pi)$, which is small for the high modes where the ill-conditioning resides (Lemma~\ref{lem:phase}), and in controlled experiments in one and two dimensions and on a disk the convergence rate does not depend on the operator and, at moderate P\'eclet numbers, the floor depends on it only weakly (Section~\ref{sec:numerics}). For the oversampled formulation with truncated singular value decomposition, an error bound of Adcock--Huybrechs type holds with a Marcinkiewicz--Zygmund constant of the test space as stability constant (Theorem~\ref{thm:regularized}); computed exactly, this constant grows exponentially for the square system and far more slowly under fourfold oversampling, which attains errors of a few tens of units of roundoff (Section~\ref{sec:num_solvers}).

\paragraph{Related work.} The approximation theory of frames explains the stability of the \emph{regularized least-squares} formulation. The Fourier extension problem is exponentially ill-conditioned \cite{huybrechs2010fourier, adcock2014numerical}, but with oversampling and truncated singular value decomposition the computed approximation satisfies an $\eps$-regularized best-approximation bound, its coefficients remain bounded, and the error converges to the level of the truncation parameter $\eps$ \cite{adcock2014numerical, adcock2019frames, adcock2020frames2}: the frame is ill-conditioned, but the numerical method is stable. Extensions to arbitrary domains and fast solvers followed \cite{matthysen2018function, coppe2020az}, and Adcock and Shadrin \cite{adcock2023fast} proved that polynomial frames on an extended interval give a well-conditioned approximation from equispaced samples whose error decays exponentially down to a prescribed tolerance, under linear oversampling. The two constants of Step~2 are the counterparts of those of \cite[Thm.~3.7]{adcock2020frames2} (Section~\ref{sec:res_fp}), here for an arbitrary solver, for collocation and in the solution metric; the square system lies outside the regularized theory, which requires truncation, and oversampling for accuracy at the level of the truncation parameter.

That rounding can stabilize an ill-conditioned computation was observed before this theory. For the collocation form of the Fourier extension, Huybrechs \cite{huybrechs2010fourier} found the computed solution accurate to machine precision on the interval although it does not resemble the exact least-squares solution, both with an explicit truncated singular value decomposition and with MATLAB's backslash. For the method of fundamental solutions, whose collocation matrices are also exponentially ill-conditioned, the ill-conditioning was found not to affect the quality of the solution \cite{kitagawa1988numerical, fairweather1998method, chen2006comments}, and Barnett and Betcke \cite{barnett2008stability} argued that the achievable accuracy is set by the unit roundoff times the norm of the coefficient vector. The effective condition number \cite{chan1988effectively, christiansen1994effective, li2008effective} accounts for the direction of the right-hand side, which is insignificant for Gaussian elimination and QR factorization, whose backward errors are general perturbations of the matrix \cite[Sec.~3]{li2008effective}. For a basis rather than a frame the ill-conditioning does limit the accuracy of a direct solve, as for flat radial basis functions \cite{fornberg2011stable} and Vandermonde matrices \cite{brubeck2021vandermonde}: a basis admits only one representation, and when its coefficients are large there is no bounded alternative for rounding to find. Closest to the present paper, Herremans and Huybrechs \cite{herremans2025function} prove that approximation in the numerical span of a rank-deficient basis behaves like approximation in exact arithmetic with a penalty proportional to the size of the coefficients times the unit roundoff, and that regularized solutions of backward-stable solvers attain this accuracy \cite[Thm.~20]{herremans2025function}; for an unregularized solve they bound the residual a posteriori \cite[Thm.~18]{herremans2025function}, and the error remains open because a rounding perturbation can make the smallest singular value of the perturbed matrix arbitrarily small \cite[Sec.~4.2.1 and Rem.~19]{herremans2025function}. They also show that uniform sampling is near-optimal in finite precision for the Fourier extension frame and random samples \cite[Sec.~7]{herremans2025function}; equispaced interpolation, the square case, is not covered. Corollary~\ref{cor:regimes_two_constant}(iii) supplies the square case under a hypothesis on the smallest singular value of the perturbed matrix.

On the potential-theoretic side, equilibrium problems with an upper constraint were introduced by Rakhmanov \cite{rakhmanov1996equilibrium} for polynomials of a discrete variable and extended to external fields by Dragnev and Saff \cite{dragnev1997constrained}, see also \cite{kuijlaars1999extremal}; Rakhmanov \cite{rakhmanov2007bounds} characterized by such a problem the size of polynomials bounded on a discrete set, and Geronimo and Liechty \cite{geronimo2020fourier} used the constrained equilibrium measure in the Riemann--Hilbert asymptotics of the least-squares Fourier extension from equispaced samples. Here a constrained equilibrium problem with an external field determines the stability constants of the sampled frame along the whole tail of its spectrum.

\paragraph{Scope.} The analysis of the two constants is one-dimensional; Section~\ref{sec:multi} extends the representation and the ill-conditioning to tensor-product domains, and two-dimensional boxes and a disk with relocated boundary nodes are treated numerically. The tail law and its consequences are proved for the Poisson operator and at the exponential scale: they determine the two constants at the thresholds $e^{-N\tau}$ up to factors $e^{o(N)}$, not at a fixed threshold $t > 0$ as $N \to \infty$, where the constants set the accuracy floor and are computed exactly (Section~\ref{sec:num_bound}). Whether, for $\delta > L/2$, the tail law lies below the rate of the power law along the whole tail, and not only at its start, is open (Section~\ref{sec:conclusion}).

Section~\ref{sec:setting} introduces the notation and states the main results in the order of the four steps, with the dependence between them. Sections~\ref{sec:structure}--\ref{sec:tail} prove them: Section~\ref{sec:structure} the exact structure, Section~\ref{sec:potential} the rates, Section~\ref{sec:fp} the floating-point bounds and Section~\ref{sec:tail} the tail law and its structure. Section~\ref{sec:regularized} treats the regularized oversampled formulation, Section~\ref{sec:multi} tensor-product domains and curved boundaries, and Section~\ref{sec:numerics} the controlled experiments; Section~\ref{sec:conclusion} concludes. The standard facts that the proofs use, about the Dirichlet problem, the Clausen function and regularized least squares, are collected with proofs in Appendix~\ref{app:standard}; the technical estimates behind Theorem~\ref{thm:rates}, Proposition~\ref{prop:apriori} and Theorems~\ref{thm:competitor}--\ref{thm:dual} are in Appendices~\ref{app:rates}--\ref{app:tail}, and Appendix~\ref{app:start} contains the boundary layer behind Theorem~\ref{thm:initial_slope}.

\section{Setting and Main Results}
\label{sec:setting}

We consider a linear second-order operator $\Lop$ on a physical domain $I = (0, L)$, with closure $\bar I = [0,L]$, discretized with basis functions defined on the larger domain $\tilde{I} = (-\delta, L+\delta)$, $\delta > 0$.
\subsection{Basis and Discretization}

\begin{definition}[Laplacian Eigenbasis on the Extended Domain]
The basis functions $\{w_j\}_{j\ge 1}$ are the normalized eigenfunctions of $-d^2/dx^2$ on $\tilde{I}$ with homogeneous Dirichlet conditions,
\begin{equation} \label{eq:basis_func}
w_j(x) = \sqrt{\frac{2}{\tilde L}} \sin\bigl(\mu_j (x+\delta)\bigr), \qquad \mu_j = \frac{j\pi}{\tilde L},\quad \lambda_j = \mu_j^2, \quad \tilde L := L+2\delta .
\end{equation}
\end{definition}

\begin{figure}[ht!]
    \centering
    \includegraphics[width=0.9\textwidth]{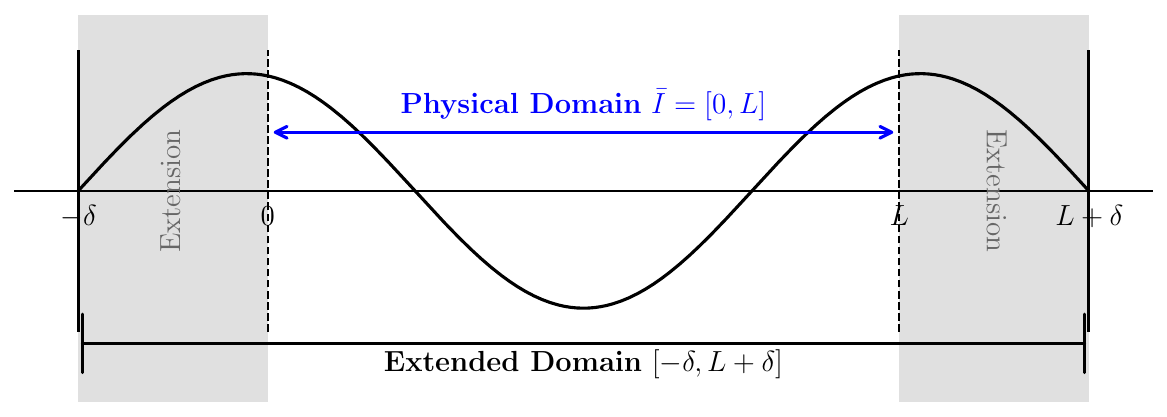}
    \caption{The extended-domain method in one dimension. The physical problem on $\bar I=[0,L]$ is embedded in $[-\delta, L+\delta]$; the basis functions $w_j$ vanish at $-\delta$ and $L+\delta$ but in general not at the physical endpoints.}
    \label{fig:domain_geometry}
\end{figure}

The geometry is shown in Figure~\ref{fig:domain_geometry}. The functions $\{w_j\}$ form an orthonormal basis of $L^{2}(\tilde{I})$, so their restrictions to $I$ form a tight frame of $L^2(I)$ with frame bound $1$, as do the restricted Fourier bases of \cite[Example~1]{adcock2019frames}: $\sum_j |\langle f, w_j\rangle_{L^2(I)}|^2 = \norm{f}_{L^2(I)}^2$ for $f \in L^2(I)$, by Parseval's identity for the extension of $f$ by zero. Since $\norm{w_j}_{L^2(I)}^2 = L/\tilde L + \sin(2\pi j\delta/\tilde L)/(\pi j)$ tends to $L/\tilde L$, the rescaled functions $(\tilde L/L)^{1/2} w_j$ are asymptotically normalized and form a tight frame with frame bound $\tilde L/L$; this ratio measures the redundancy of the frame \cite[Sec.~2.2]{huybrechs2010fourier} and grows with $\delta$.

\begin{definition}[Interpolatory Collocation]
\label{def:collocation}
The solution $u_N = \sum_{j=1}^N c_j w_j$ is sought in the space $V_N = \operatorname{span}\{w_1, \dots, w_N\}$. With the equispaced nodes $x_m = (m-1)h$, $h = L/(N-1)$, $m = 1,\dots,N$, and the equation $\Lop u = f$ with Dirichlet data $u(0)=g_0$, $u(L)=g_L$, the coefficients solve the square system $\mat{A}\vect{c}=\vect{f}$,
\begin{equation}\label{eq:square_system}
(\mat{A})_{mj} = \begin{cases} w_j(x_m), & m \in \{1, N\},\\ (\Lop w_j)(x_m), & 2 \le m \le N-1,\end{cases}\qquad
(\vect f)_m = \begin{cases} g_0,\ g_L, & m = 1,\ N,\\ f(x_m), & 2 \le m \le N-1. \end{cases}
\end{equation}
\end{definition}

For the two operators considered throughout,
\begin{align}
\text{Poisson } (\Lop_P = -d^2/dx^2):\qquad & \Lop w_j = \lambda_j w_j, \label{eq:phi_poisson}\\
\text{convection--diffusion } (\Lop_{CD} = -d^2/dx^2 + k\,d/dx):\qquad & \Lop w_j = \lambda_j w_j + k\,w_j'. \label{eq:phi_cd}
\end{align}
The coefficient $k$ is constant unless stated otherwise; it is evaluated pointwise at the nodes, so no approximation of $k$ is involved. Replacing two rows of the system by the boundary values is algebraically equivalent to basis recombination, i.e.\ to restricting the trial space to functions satisfying the boundary conditions; Definition~\ref{def:spaces} makes this explicit.

\subsection{Trial Space, Test Space, and the Constants of the Sampled Frame}
\label{sec:constants}

\begin{definition}[Trial Space with Boundary Conditions, Test Space]
\label{def:spaces}
$\Phi_N := \{v \in V_N : v(0) = v(L) = 0\}$ and $\Psi_N := \Lop(\Phi_N) = \{\Lop v : v \in \Phi_N\}$.
\end{definition}

\begin{assumption}[Unisolvence]
\label{ass:unisolvence}
Interpolation at the interior nodes $\{x_m\}_{m=2}^{N-1}$ is unisolvent in $\Psi_N$; equivalently (Lemma~\ref{lem:dimensions}), $\mat A$ in \eqref{eq:square_system} is nonsingular.
\end{assumption}

Every collocation method assumes that its square matrix is nonsingular. For the convection--diffusion operator this is a genuine hypothesis; for the Poisson operator Section~\ref{sec:structure} reduces it to two explicit nonvanishing conditions, which hold for all sufficiently large $N$ (Proposition~\ref{prop:unisolvence_poisson}).

Under Assumption~\ref{ass:unisolvence}, $\Pop_N: C(\bar I) \to \Psi_N$ denotes interpolation at the interior nodes: for $g \in C(\bar I)$, $\Pop_N g$ is the unique $\psi \in \Psi_N$ with $\psi(x_m) = g(x_m)$ for $2 \le m \le N-1$. Its Lagrange functions are $\ell_m^\Psi \in \Psi_N$, $2 \le m \le N-1$, with $\ell_m^\Psi(x_{m'}) = \delta_{mm'}$, so that $\Pop_N g = \sum_{m=2}^{N-1} g(x_m)\,\ell_m^\Psi$. The continuous solution operator of the homogeneous Dirichlet problem for $\Lop u = -u'' + k(x)u'$, $k \in C(\bar I)$, is $\Lop^{-1}$ of Lemma~\ref{lem:dirichlet_1d}(ii) in Appendix~\ref{app:standard}, with the Green's function $G \ge 0$ of \eqref{eq:green_general}; by Lemma~\ref{lem:dirichlet_1d}(iii),
\begin{equation}\label{eq:Linv}
\begin{aligned}
(\Lop^{-1} g)(x) &= \int_0^L G(x,y)\, g(y)\, dy, \\
\norm{\Lop^{-1}}_\infty &:= \sup_{g\neq 0}\frac{\norm{\Lop^{-1}g}_{L^\infty(I)}}{\norm{g}_{L^\infty(I)}} = \max_{x\in\bar I}\int_0^L G(x,y)\,dy.
\end{aligned}
\end{equation}

\begin{definition}[Lebesgue Constant]
\label{def:lebesgue}
Let $X$ be an $n$-dimensional space of continuous functions on $\bar I$ and $\{y_m\}_{m=1}^n \subset \bar I$ nodes at which interpolation in $X$ is unisolvent. The interpolation operator $\Iop_{X}: C(\bar I) \to X$ maps $g$ to the element of $X$ agreeing with $g$ at the nodes, and its Lebesgue constant is
\begin{equation}
\Lambda(X) = \sup_{g \ne 0} \frac{\norm{\Iop_{X} g}_{L^\infty(I)}}{\norm{g}_{L^\infty(I)}} = \max_{x \in \bar I} \sum_{m=1}^n \abs{\ell_m(x)},
\end{equation}
with the Lagrange functions $\ell_m \in X$, $\ell_m(y_{m'}) = \delta_{mm'}$. We write $\Lambda_N := \Lambda(V_N)$ for the trial space with all $N$ nodes, and $\Lambda_N^\Psi := \Lambda(\Psi_N)$ for the test space of Definition~\ref{def:spaces} with the interior nodes.
\end{definition}

For the floating-point analysis the trial space with boundary conditions is written in a basis $\{\varphi_j\}_{j=1}^{n}$ of $\Phi_N$, $n = N-2$. The choice used throughout is $\varphi_j := w_{j+2} - a_j w_1 - b_j w_2$, $1 \le j \le N-2$: each of the modes $w_3, \dots, w_N$ is corrected by the two lowest modes so that $\varphi_j(0) = \varphi_j(L) = 0$, which fixes $a_j$ and $b_j$ uniquely because the boundary functionals are linearly independent on $\operatorname{span}\{w_1,w_2\}$ (see the proof of Lemma~\ref{lem:dimensions}); explicitly, with $s_i := \sin(i\theta_0)$ and $\theta_0 := \pi\delta/\tilde L$, $a_j = s_{j+2}/s_1$, $b_j = 0$ for odd $j$ and $a_j = 0$, $b_j = s_{j+2}/s_2$ for even $j$. The $\varphi_j$ lie in $\Phi_N$, are linearly independent because $w_{j+2}$ occurs in $\varphi_j$ only, and are $N-2$ in number, hence a basis.

The functions $\psi_j := \Lop\varphi_j$ span $\Psi_N$, $\mat B_{mj} := \psi_j(x_m)$ for sampling nodes $x_1, \dots, x_M \in \bar I$ ($M \ge n$; for the interpolatory method $M = n$ and these are the interior nodes of the $N$-grid, relabelled), $\mat D = \operatorname{diag}(1/\norm{\mat B\vect e_j}_2)$ normalizes the columns, where $\vect e_j$ denotes the $j$-th unit vector, here of $\mathbb R^n$, so that $\mat B\vect e_j$ is the $j$-th column of $\mat B$, which we assume to be nonzero (for $M = n$ this holds under Assumption~\ref{ass:unisolvence}, which makes $\mat B$ nonsingular), and $\mat B_D := \mat B\mat D = \mat U\mathbf\Sigma\mat V^T$ is its thin singular value decomposition: $\mathbf\Sigma = \operatorname{diag}(\sigma_1, \dots, \sigma_n)$ with $\sigma_1 \ge \dots \ge \sigma_n \ge 0$, $\mat U \in \mathbb R^{M\times n}$ and $\mat V \in \mathbb R^{n\times n}$ have orthonormal columns, and the $i$-th columns $\vect u_i$ of $\mat U$ and $\vect v_i$ of $\mat V$, the left and right singular vectors, satisfy $\mat B_D\vect v_i = \sigma_i\vect u_i$. A coefficient vector $\vect y$ in the normalized basis defines the function $\Phi\mat D\vect y := \sum_j (\mat D\vect y)_j\varphi_j \in \Phi_N$.

The obstacle to bounding the error of the computed solution is the passage from the nodes to the interval, which for the interpolatory node set costs an exponentially large factor (Section~\ref{sec:num_solvers}). The following definition splits the coefficient space into the part that the nodes resolve and the part that they do not, and attaches a constant to each, as the regularized-frame theory does \cite{adcock2014numerical, adcock2020frames2}; the two-constant bound of Theorem~\ref{thm:two_constant} rests on it.

\begin{definition}[Leakage and Restricted Marcinkiewicz--Zygmund Constants]
\label{def:two_constants}
Denote by $\sigma_1 \ge \dots \ge \sigma_n$ the singular values of $\mat B_D$ and by $\vect v_1, \dots, \vect v_n$ its right singular vectors. For a threshold $t \ge 0$ let $\mathcal V_t := \operatorname{span}\{\vect v_i : \sigma_i \le t\}$ and $\mathcal V_t^{\perp} := \operatorname{span}\{\vect v_i : \sigma_i > t\}$, the spans of the singular vectors that the nodes see to size at most $t$ and to size above $t$, and define
\begin{equation}\label{eq:two_constants}
\gamma_t := \sup_{\vect w \in \mathcal V_t\setminus\{0\}} \frac{\norm{\Phi\mat D\vect w}_{L^2(I)}}{\norm{\vect w}_2}, \qquad
\Gamma_t := \sup_{\vect w \in \mathcal V_t^{\perp}\setminus\{0\}} \frac{\norm{\Phi\mat D\vect w}_{L^2(I)}}{\norm{\mat B_D\vect w}_2},
\end{equation}
with $\gamma_t := 0$ if $\mathcal V_t = \{0\}$ and $\Gamma_t := 0$ if $\mathcal V_t^\perp = \{0\}$.
\end{definition}

The leakage $\gamma_t$ measures how large a function of $\Phi_N$ can be on $I$ when its coefficient vector is a unit vector that the nodes see only to size $t$; $\Gamma_t$ is the discrete-to-continuous constant of the well-resolved subspace. Both are computable as square roots of extremal generalized eigenvalues of the Gram matrix $\mat D\bigl(\int_I \varphi_j\varphi_k\bigr)\mat D$ in the singular basis (Section~\ref{sec:num_bound}). For $t = 0$ and $\mat B$ of full column rank, $(M/L)^{1/2}\Gamma_0$ is the Marcinkiewicz--Zygmund constant $C_M$ of Definition~\ref{def:mz} in Section~\ref{sec:regularized} with the trial function $\Phi\mat D\vect w$ in place of the test function $\Lop\Phi\mat D\vect w$ in the numerator, and $\Gamma_0$ grows exponentially whenever the leakage of the last singular vector decays at a smaller exponential rate than $\sigma_n$ (Corollary~\ref{cor:regimes_two_constant}(i)), as it does in the computations of Section~\ref{sec:num_bound}; for the Poisson operator its rate is $\kappa_\Lambda$ (Corollary~\ref{cor:lebesgue_rate}).

For $M \ge n$ nodes, the data $\vect f = (f(x_m))_{m=1}^M$ and $\eps > 0$ define the $\eps$-regularized solution
\begin{equation}\label{eq:tsvd}
\vect c_\eps := \mat D\, \mat V \mathbf\Sigma_\eps^{\dagger}\mat U^T \vect f, \qquad (\mathbf\Sigma_\eps^\dagger)_{ii} = \begin{cases}1/\sigma_i, & \sigma_i > \eps,\\ 0, & \sigma_i \le \eps,\end{cases}
\qquad u_N^\eps := \sum_{j=1}^n (\vect c_\eps)_j \varphi_j \in \Phi_N .
\end{equation}
For $M = n$, the interior nodes of the $N$-grid and $\eps \to 0$ this is the interpolatory method.

\subsection{The Chebyshev Variable and the Node Potential}
\label{sec:variable}

The rates of Section~\ref{sec:results} are expressed in the variable in which the sine frame becomes a polynomial space.

\begin{lemma}[Polynomial Form of the Frame]
\label{lem:chebyshev_variable}
Let $\theta = \pi(x+\delta)/\tilde L$ and $z = \cos\theta$. Then $w_j(x) = \sqrt{2/\tilde L}\,\sin\theta\, U_{j-1}(z)$ with $U_k$ the Chebyshev polynomial of the second kind, so that $V_N = \sqrt{2/\tilde L}\,\sin\theta\cdot\mathbb P_{N-1}(z)$. The extended domain $\tilde I$ corresponds to $z \in (-1,1)$ and the physical domain $I$ to $z \in (-c_0, c_0)$, $c_0 = \cos\theta_0$, $\theta_0 = \pi\delta/\tilde L$; we write $\omega := \pi - 2\theta_0 = \pi L/\tilde L \in (0,\pi)$. For $v = \sum_j c_j w_j = \sqrt{2/\tilde L}\,\sin\theta\, q(z)$, $q = \sum_j c_j U_{j-1}$,
\begin{equation}\label{eq:chebyshev_norms}
\norm{v}_{L^2(\tilde I)}^2 = \sum_j c_j^2 = \frac{2}{\pi}\int_{-1}^{1} q(z)^2\sqrt{1-z^2}\,dz, \qquad
\norm{v}_{L^2(I)}^2 = \frac{2}{\pi}\int_{-c_0}^{c_0} q(z)^2\sqrt{1-z^2}\,dz .
\end{equation}
The equispaced nodes $x_m$ correspond to $\theta_m = \theta_0 + (m-1)\omega/(N-1)$ and $z_m = \cos\theta_m$; their normalized counting measures converge weakly to
\begin{equation}\label{eq:node_measure}
d\mu_\omega(z) = \frac{dz}{\omega\sqrt{1-z^2}}\quad\text{on } [-c_0, c_0],
\end{equation}
the arcsine density of $[-1,1]$ restricted to $[-c_0,c_0]$ and renormalized by the factor $\pi/\omega = \tilde L/L$.
\end{lemma}
\begin{proof}
$\sin(j\theta) = \sin\theta\,U_{j-1}(\cos\theta)$; the first norm identity is the orthonormality of the $w_j$, and the integrals follow from $dx = (\tilde L/\pi)d\theta$ and $\sin\theta\,d\theta = -dz$, which also give $\int_{-1}^1 U_{j-1}U_{k-1}\sqrt{1-z^2}\,dz = (\pi/2)\delta_{jk}$; the nodes are equispaced in $\theta$, and the weak convergence of their counting measures is the convergence of Riemann sums for continuous functions of $\theta$.
\end{proof}

Thus interpolation in $V_N$ is weighted polynomial interpolation of degree $N-1$ on $[-c_0,c_0]$, with the coefficient norm of the larger interval $[-1,1]$, and the relevant object is the logarithmic potential of the node measure,
\begin{equation}\label{eq:potential}
U_\omega(z) := -\int \log|z-s|\,d\mu_\omega(s) = -\frac{1}{\omega}\int_{\theta_0}^{\pi-\theta_0}\log|z - \cos\varphi|\,d\varphi .
\end{equation}

The rates below are expressed through the Clausen function \cite[Eq.~(4.9)]{lewin1981polylogarithms}
\[
\mathrm{Cl}_2(x) := -\int_0^x\log|2\sin(t/2)|\,dt ;
\]
its elementary properties are collected in Lemma~\ref{lem:cl2} of Appendix~\ref{app:standard}, and its value $\mathrm{Cl}_2(\pi/2) = \sum_{k\ge0}(-1)^k(2k+1)^{-2} = 0.91596\ldots$ is Catalan's constant \cite[Eq.~(4.14)]{lewin1981polylogarithms}.

\paragraph{Analyticity regions.} For $\rho > 1$ let
\begin{equation}\label{eq:ellipse}
E_\rho := \Bigl\{\tfrac{c_0}{2}\bigl(w + w^{-1}\bigr) : \rho^{-1} < |w| < \rho\Bigr\}
\end{equation}
be the open Bernstein ellipse of $[-c_0,c_0]$ with parameter $\rho$: the open region bounded by the ellipse with foci $\pm c_0$ and semi-axes $\tfrac{c_0}{2}(\rho \pm \rho^{-1})$, which contains $[-c_0,c_0]$. Put
\begin{equation}\label{eq:rho_star}
\rho^\ast := \frac{1 + \sin\theta_0}{\cos\theta_0} = \tan\Bigl(\frac\pi4 + \frac{\theta_0}{2}\Bigr) > 1 .
\end{equation}
Since $\tfrac{c_0}{2}(\rho^\ast + 1/\rho^\ast) = 1$, the ellipse $E_{\rho^\ast}$ has real semi-axis one, so that $E_\rho \cap \mathbb R \subset (-1,1)$ for $\rho \le \rho^\ast$ with equality for $\rho = \rho^\ast$, and $\pm1 \notin E_\rho$. The principal branch of $\arccos$ is analytic on $\mathbb C\setminus\bigl((-\infty,-1]\cup[1,\infty)\bigr) \supset E_{\rho^\ast}$, and for $1 < \rho \le \rho^\ast$
\begin{equation}\label{eq:omega_rho}
\Omega_\rho := \bigl\{(\tilde L/\pi)\arccos z - \delta : z \in E_\rho\bigr\}
\end{equation}
is an open simply connected subset of the $x$-plane that contains $\bar I$ and is mapped conformally onto $E_\rho$ by $x \mapsto \cos\bigl(\pi(x+\delta)/\tilde L\bigr)$; its intersection with $\mathbb R$ is an open interval contained in $\tilde I$, equal to $\tilde I$ for $\rho = \rho^\ast$. Analyticity of the data on $\Omega_\rho$ is the hypothesis under which the test space contains approximations of the data with coefficient vectors of moderate norm (Theorem~\ref{thm:competitor}).

\subsection{Main Results}
\label{sec:results}

The results follow the four steps of the introduction. Step~1, the instability of the exact method, is Section~\ref{sec:res_exact}, with the rates of the frame in Section~\ref{sec:res_rates}; Step~2, the floating-point bound and its uniform form for analytic data, is Sections~\ref{sec:res_fp} and \ref{sec:res_analytic}; Steps~3 and~4, the equilibrium law of the two constants, its duality and its start, are Section~\ref{sec:res_tail}, and the explicit bounds that attain the law are Section~\ref{sec:res_bounds}. Section~\ref{sec:res_dependence} lists the ingredients of each proof. The results of Section~\ref{sec:res_exact} hold for $\Lop u = -u'' + k(x)u'$ with $k \in C(\bar I)$, and those of Section~\ref{sec:res_rates} concern the trial frame alone. From Section~\ref{sec:res_fp} on, $\Lop = \Lop_P$ or $\Lop = \Lop_{CD}$ with constant $k \ge 0$, and the collocation nodes are the interior nodes of the $N$-grid ($M = n$); Theorem~\ref{thm:two_constant} and Corollary~\ref{cor:regimes_two_constant}(ii) allow $M \ge n$ sampling nodes, and the results that need the Poisson operator or Assumption~\ref{ass:unisolvence} say so.

\subsubsection{Exact Arithmetic}
\label{sec:res_exact}

\begin{theorem}[Exact Representation of the Interpolatory Solution]
\label{thm:representation}
Let Assumption~\ref{ass:unisolvence} hold, let $f \in C(\bar I)$, and let $u_N \in V_N$ solve \eqref{eq:square_system} with $g_0 = g_L = 0$. Then $u_N \in \Phi_N$ and
\begin{equation}\label{eq:representation}
u_N = \Lop^{-1} \Pop_N f, \qquad\text{i.e.}\qquad u_N(x) = \int_0^L G(x,y)\,(\Pop_N f)(y)\,dy \quad (x \in \bar I).
\end{equation}
\end{theorem}

\begin{corollary}[Exact Error Identity and Lebesgue Bound]
\label{cor:error}
With $u = \Lop^{-1}f$,
\begin{equation}\label{eq:error_bound}
u - u_N = \Lop^{-1}(f - \Pop_N f), \qquad
\norm{u - u_N}_{L^\infty(I)} \le \norm{\Lop^{-1}}_\infty\,\bigl(1 + \Lambda_N^\Psi\bigr)\, \dist_\infty(f, \Psi_N),
\end{equation}
where $\dist_\infty(f,\Psi_N) := \inf_{\psi\in\Psi_N}\norm{f - \psi}_{L^\infty(I)}$.
\end{corollary}

\begin{proposition}[Geometric Convergence Forces Exponential Instability]
\label{prop:impossibility}
Let $N_1 \ge 3$, let Assumption~\ref{ass:unisolvence} hold for every $N \ge N_1$, let $E \subset \mathbb C$ be a compact set containing $\bar I$ in its interior, and suppose that there are $C > 0$ and $\sigma > 1$ with $\norm{f - \Pop_N f}_{L^\infty(I)} \le C\sigma^{-N}\norm{f}_{L^\infty(E)}$ for all $N \ge N_1$ and all $f$ analytic in a neighbourhood of $E$, $\Pop_N$ acting on complex-valued functions by complex linearity. Then there are $c > 0$ and $\beta > 1$ with $\Lambda_N^\Psi \ge c\,\beta^N$ for all $N \ge N_1$.
\end{proposition}

The operator and the frame interact only through Theorem~\ref{thm:representation}: the operator supplies the bounded map $\Lop^{-1}$ and the test space $\Psi_N$, the frame the interpolation operator $\Pop_N$. The proposition, whose convergence hypothesis is that of the impossibility theorem \cite{platte2011impossibility}, makes the exact method unstable: geometric convergence forces $\Lambda_N^\Psi$ to grow exponentially, and by Corollary~\ref{cor:error} the exact solution amplifies errors in the data by up to $\Lambda_N^\Psi$; computed exactly, the solution of rounded data is observed to diverge beyond a threshold in $N$ (Remark~\ref{rem:divergence}). These results are proved in Section~\ref{sec:structure}, which also represents the discrete Green's function of the method (Corollary~\ref{cor:greens} and Remark~\ref{rem:mass}). For the Poisson operator Assumption~\ref{ass:unisolvence} holds for all sufficiently large $N$, and $\Lambda_N^\Psi$ grows exponentially without any hypothesis on convergence, at the rate of the trial frame (Corollary~\ref{cor:lebesgue_rate}).

\begin{proposition}[Unisolvence for the Poisson Operator]
\label{prop:unisolvence_poisson}
Let $\Lop = \Lop_P$. There is $N_0$, depending only on $\delta$ and $L$, such that Assumption~\ref{ass:unisolvence} holds for every $N \ge N_0$.
\end{proposition}

The proposition is proved in Section~\ref{sec:tail}, from the two functions of $V_N$ that vanish at the interior nodes (Lemma~\ref{lem:correction}).

\subsubsection{Rates of the Trial Frame}
\label{sec:res_rates}

For the Poisson operator the exponential growth required by Proposition~\ref{prop:impossibility} has the explicit rate of the trial frame (Corollary~\ref{cor:lebesgue_rate}). The next results concern the trial frame $V_N$ sampled at all $N$ nodes, that is, its synthesis matrix $(\mat W_N)_{mj} := w_j(x_m)$, $1 \le m, j \le N$, and not the collocation matrix $\mat B_D$ of the test space; they give the exponential rates of the frame in closed form, and Proposition~\ref{prop:classical} and Section~\ref{sec:res_tail} carry them over to the test space. For the Poisson operator, $\mat A_P := \mat W_N\mathbf\Lambda$, $\mathbf\Lambda = \operatorname{diag}(\lambda_j)$, is the bulk collocation matrix, the equation enforced at every node without boundary rows. For a nonsingular matrix $\mat M$, $\sigma_{\max}(\mat M) = \norm{\mat M}_2$ and $\sigma_{\min}(\mat M) = 1/\norm{\mat M^{-1}}_2$ denote its largest and smallest singular values and $\kappa_2(\mat M) := \sigma_{\max}(\mat M)/\sigma_{\min}(\mat M)$ its spectral condition number; $\mat W_N$ is nonsingular (Lemma~\ref{lem:lebesgue_condition_link}).

\begin{theorem}[Growth Rates of the Lebesgue Constant and of the Inverse Smallest Singular Value]
\label{thm:rates}
Let $\Lambda_N = \Lambda(V_N)$ be the trial-space Lebesgue constant at the $N$ equispaced nodes. Then
\begin{equation}\label{eq:rates}
\begin{aligned}
\lim_{N\to\infty}\frac{\log\Lambda_N}{N} &= \kappa_\Lambda(\omega) := U_\omega(0) - U_\omega(c_0) = \frac{\mathrm{Cl}_2(\omega) - \tfrac12\mathrm{Cl}_2(2\omega)}{\omega}, \\
\lim_{N\to\infty}\frac{\log\bigl(1/\sigma_{\min}(\mat W_N)\bigr)}{N} &= \kappa_\sigma(\omega) := U_\omega(0) - U_\omega(1) = \frac{\mathrm{Cl}_2(\omega) + \mathrm{Cl}_2(\pi-\omega)}{\omega}.
\end{aligned}
\end{equation}
\end{theorem}

\begin{corollary}[Leakage Exponent at the End of the Spectrum]
\label{cor:exponent}
Let $x_1, \dots, x_N$ be the equispaced nodes of Definition~\ref{def:collocation}. For $v \in V_N\setminus\{0\}$ let $t(v) := \bigl(\sum_{m=1}^N v(x_m)^2\bigr)^{1/2}/\norm{v}_{L^2(\tilde I)}$ be its nodal size at these nodes and $g(v) := \norm{v}_{L^2(I)}/\norm{v}_{L^2(\tilde I)}$ its leakage, and let $\ell^\ast \in V_N$ be the Lagrange function of $V_N$ at these nodes for a node nearest to the centre of $I$: $\ell^\ast(x_m) = \delta_{mm^\ast}$ for $1 \le m \le N$, where $m^\ast$ minimizes $|x_m - L/2|$ (either of the two such indices when $N$ is even); it exists by Lemma~\ref{lem:lebesgue_condition_link}. Then $N^{-1}\log t(\ell^\ast) = -\kappa_\sigma(\omega) + O(\log N/N)$, the rate of $\sigma_{\min}(\mat W_N)$ in Theorem~\ref{thm:rates}, and
\begin{equation}\label{eq:exponent}
\lim_{N\to\infty}\frac{\log g(\ell^\ast)}{\log t(\ell^\ast)} = a(\omega) := 1 - \frac{\kappa_\Lambda(\omega)}{\kappa_\sigma(\omega)} = \frac{\tfrac12\mathrm{Cl}_2(2\omega) + \mathrm{Cl}_2(\pi-\omega)}{\mathrm{Cl}_2(\omega) + \mathrm{Cl}_2(\pi-\omega)} .
\end{equation}
\end{corollary}

\begin{proposition}[Duality of the Exponent]
\label{prop:duality}
$a(\omega) + a(\pi - \omega) = 1$ for all $\omega \in (0,\pi)$; equivalently $\kappa_\Lambda(\omega)/\kappa_\sigma(\omega) = a(\pi-\omega)$.
\end{proposition}

\begin{corollary}[Exponential Ill-Conditioning]
\label{thm:poisson_illcond}
For every $\alpha \in (0,\kappa_\sigma(\omega))$ there is a constant $C > 0$, independent of $N$, such that $\kappa_2(\mat{A}_P) \ge C\, N^{-3/2} e^{\alpha N}$ for all $N \ge 2$.
\end{corollary}

These results are proved in Section~\ref{sec:potential}, from a uniform estimate of the discrete potential of the nodes (Appendix~\ref{app:rates}). The closed forms give $a(\pi/2) = 1/2$ ($\delta = L/2$), $a(\pi/3) = 3/5$ ($\delta = L$) and $a(2\pi/3) = 2/5$ ($\delta = L/4$), with $a(\omega) \to 1$ as $\omega \to 0$ ($\delta \to \infty$) and $a(\omega) \to 0$ as $\omega \to \pi$ ($\delta \to 0$); for $\delta = L/2$ the two rates are $\kappa_\Lambda = 2\,\mathrm{Cl}_2(\pi/2)/\pi$ and $\kappa_\sigma = 4\,\mathrm{Cl}_2(\pi/2)/\pi$, multiples of Catalan's constant. In the variable $\theta$ the physical domain and the extension are arcs of lengths $\omega$ and $\pi-\omega$, and Proposition~\ref{prop:duality} exchanges their roles; $\delta = L/2$ is the self-dual case. The complement $1 - a(\omega) = \kappa_\Lambda/\kappa_\sigma$ is the counterpart of the exponent $a(\gamma;T)$ of \cite[Eq.~(5.23)]{adcock2014numerical} at $\gamma = 1$, the ratio of the same two rates for the equispaced Fourier extension. Section~\ref{sec:num_potential} compares the rates and the exponent with exact computations.

\subsubsection{Floating-Point Arithmetic}
\label{sec:res_fp}

By Section~\ref{sec:res_exact} the exact solution is unstable, and the accuracy of a computed solution is a property of floating-point arithmetic. A computed coefficient vector is not the exact interpolant, and the next theorem bounds the error of an arbitrary coefficient vector through the two constants of Definition~\ref{def:two_constants}.

\begin{theorem}[Two-Constant Bound]
\label{thm:two_constant}
Let $u = \Lop^{-1}f$, let $\hat{\vect y} \in \mathbb R^{n}$ be arbitrary, $\hat u_N := \Phi\mat D\hat{\vect y}$, and let $\hat\rho := \norm{\vect f - \mat B_D\hat{\vect y}}_2$ be its nodal residual, $\vect f = (f(x_m))_m$. Then for every $t \ge 0$ and every $\vect y \in \mathbb R^n$ with $u_{\vect y} := \Phi\mat D\vect y$ and $\vect d := \hat{\vect y} - \vect y$,
\begin{equation}\label{eq:two_constant_bound}
\norm{u - \hat u_N}_{L^2(I)} \le \norm{u - u_{\vect y}}_{L^2(I)} + \Gamma_t\bigl(\norm{\vect f - \mat B_D\vect y}_2 + \hat\rho + t\norm{\vect d}_2\bigr) + \gamma_t\norm{\vect d}_2 .
\end{equation}
\end{theorem}

Theorem~\ref{thm:two_constant} makes no assumption on how $\hat{\vect y}$ was obtained and none on the frame; its content is in the sizes of $\gamma_t$, $\Gamma_t$, $\hat\rho$ and $\norm{\vect d}_2$, which the following corollary evaluates in three regimes.

\begin{corollary}[Exact Arithmetic, Regularized Solutions, and Backward-Stable Solvers]
\label{cor:regimes_two_constant}\leavevmode
\begin{enumerate}
    \item[(i)] Let $M = n$ and let Assumption~\ref{ass:unisolvence} hold. For the interpolatory solution $u_N$ of Theorem~\ref{thm:representation}, $\hat\rho = 0$, and with $t = 0$ the bound reads $\norm{u - u_N}_{L^2(I)} \le \norm{u - u_{\vect y}}_{L^2(I)} + \Gamma_0\norm{\vect f - \mat B_D\vect y}_2$, where $\Gamma_0 \ge \norm{\Phi\mat D\vect v_n}_{L^2(I)}/\sigma_n$.
    \item[(ii)] For the regularized solution \eqref{eq:tsvd} with threshold $\eps$, $\hat\rho \le \norm{\vect f - \mat B_D\vect y}_2 + \eps\norm{\vect y}_2$ and $\norm{\vect d}_2 \le 2\norm{\vect y}_2 + \eps^{-1}\norm{\vect f - \mat B_D\vect y}_2$, and with $t = \eps$
    \[
    \norm{u - u_N^\eps}_{L^2(I)} \le \norm{u - u_{\vect y}}_{L^2(I)} + (3\Gamma_\eps + 2\gamma_\eps\eps^{-1})\bigl(\norm{\vect f - \mat B_D\vect y}_2 + \eps\norm{\vect y}_2\bigr).
    \]
    \item[(iii)] Let $M = n$, and let $\hat{\vect y}$ be computed by a solver with backward error $\eta > 0$, i.e.\ suppose that there is a matrix $\mat E \in \mathbb R^{n\times n}$ with $(\mat B_D + \mat E)\hat{\vect y} = \vect f$ and $\norm{\mat E}_2 \le \eta$. Then $\hat\rho \le \eta\norm{\hat{\vect y}}_2$; if in addition this $\mat E$ satisfies $\sigma_{\min}(\mat B_D + \mat E) \ge \eta/K$ for some $K > 0$, then $\norm{\vect d}_2 \le K\bigl(\norm{\vect y}_2 + \eta^{-1}\norm{\vect f - \mat B_D\vect y}_2\bigr)$, and with $t = \eta$
    \begin{equation}\label{eq:lu_bound}
    \norm{u - \hat u_N}_{L^2(I)} \le \norm{u - u_{\vect y}}_{L^2(I)} + \bigl[(1 + 2K)\Gamma_\eta + K\gamma_\eta\eta^{-1}\bigr]\bigl(\norm{\vect f - \mat B_D\vect y}_2 + \eta\norm{\vect y}_2\bigr).
    \end{equation}
\end{enumerate}
\end{corollary}

Theorem~\ref{thm:two_constant} and Corollary~\ref{cor:regimes_two_constant} are proved in Section~\ref{sec:fp}. The bound of (i) is an estimate of the Lebesgue type. For an exact solve of perturbed data, $\hat\rho$ is the size of the perturbation, and for $t > 0$ the bound contains the term $\gamma_t\norm{\vect d}_2$ with $\norm{\vect d}_2 \ge \norm{\hat{\vect y}}_2 - \norm{\vect y}_2$, which carries the growth of the coefficient vector beyond the threshold of Remark~\ref{rem:divergence}. Item (ii) is the counterpart of \cite[Thm.~3.7]{adcock2020frames2}, whose constants $\kappa^\eps_{M,N}$ and $\eps\lambda^\eps_{M,N}$, the norm of the regularized reconstruction map on the retained singular subspace and the norm of the part of a function that the truncation discards, play the roles of $\Gamma_\eps$ and $\gamma_\eps$ here; the constants $C_1$ and $C_2$ of \cite[Thm.~5.5]{adcock2014numerical} are the same two ratios. Item (iii) has the form of (ii) with $\eps$ replaced by $\eta$: a backward-stable direct solver, without any regularization, attains the accuracy of an $\eta$-regularized frame approximation, provided that $\sigma_{\min}(\mat B_D + \mat E) \ge \eta/K$ with a moderate $K$. What decides the accuracy is therefore the size of $\gamma_t$ and $\Gamma_t$ at $t = \eta$, and the next two propositions bound them.

\begin{proposition}[A Priori Bounds on the Two Constants]
\label{prop:apriori}
There are $N_0$ and $C_\Phi > 0$, depending only on $\delta$, $L$ and $k$, such that for all $N \ge N_0$ and all $\vect w \in \mathbb R^n$, $\norm{\Phi\mat D\vect w}_{L^2(\tilde I)} \le C_\Phi\norm{\vect w}_2$. Consequently $\gamma_t \le C_\Phi$ for every $t \ge 0$, $\Gamma_t \le C_\Phi/t$ for every $t > 0$, and under the hypotheses of Corollary~\ref{cor:regimes_two_constant}(iii)
\begin{equation}\label{eq:apriori_lu}
\norm{u - \hat u_N}_{L^2(I)} \le \norm{u - u_{\vect y}}_{L^2(I)} + (1 + 3K)\,C_\Phi\bigl(\norm{\vect y}_2 + \eta^{-1}\norm{\vect f - \mat B_D\vect y}_2\bigr).
\end{equation}
\end{proposition}

\begin{proposition}[Lebesgue-Type Bounds on the Two Constants]
\label{prop:classical}
Let Assumption~\ref{ass:unisolvence} hold, and let $\Lambda^\Psi_{N,2} := \max_m\norm{\ell_m^\Psi}_{L^2(I)} \le \sqrt L\,\Lambda_N^\Psi$.
\begin{enumerate}
    \item[(i)] For every $t \ge 0$,
    \begin{equation}\label{eq:classical}
    \gamma_t \le \frac{L^2}{\pi^2}\sqrt n\,\Lambda^\Psi_{N,2}\,t, \qquad \Gamma_t \le \frac{L^2}{\pi^2}\sqrt n\,\Lambda^\Psi_{N,2},
    \end{equation}
    and under the hypotheses of Corollary~\ref{cor:regimes_two_constant}(iii)
    \begin{equation}\label{eq:classical_lu}
    \norm{u - \hat u_N}_{L^2(I)} \le \norm{u - u_{\vect y}}_{L^2(I)} + (1 + 3K)\,\frac{L^2}{\pi^2}\sqrt n\,\Lambda^\Psi_{N,2}\bigl(\norm{\vect f - \mat B_D\vect y}_2 + \eta\norm{\vect y}_2\bigr).
    \end{equation}
\end{enumerate}
If $\Lop = \Lop_P$, there are $N_0$ and $C$, depending only on $\delta$ and $L$, such that for all $N \ge N_0$ moreover:
\begin{enumerate}
    \item[(ii)] $\norm{\ell_m^\Psi}_{L^\infty(I)} \le N^C\exp\bigl(N[U_\omega(z_m) - U_\omega(c_0)]\bigr)$ for $2 \le m \le N-1$, with $z_m = \cos\theta_m$ as in Lemma~\ref{lem:chebyshev_variable};
    \item[(iii)] $\Lambda_N^\Psi \le N^Ce^{N\kappa_\Lambda(\omega)}$ and $\Lambda^\Psi_{N,2} \le N^Ce^{N\kappa_\Lambda(\omega)}$;
    \item[(iv)] $\gamma_t \le N^Ce^{N\kappa_\Lambda(\omega)}\,t$ and $\Gamma_t \le N^Ce^{N\kappa_\Lambda(\omega)}$ for all $t \ge 0$; in particular, at $t = e^{-N\kappa_\sigma(\omega)}$,
    \begin{equation}\label{eq:endpoint_law}
    \gamma_t \le N^C\,t^{\,a(\omega)}, \qquad \Gamma_t \le N^C\,t^{\,a(\omega)-1}.
    \end{equation}
\end{enumerate}
\end{proposition}

Proposition~\ref{prop:apriori} is proved in Appendix~\ref{app:apriori}, part (i) of Proposition~\ref{prop:classical} in Section~\ref{sec:fp}, and parts (ii)--(iv) in Section~\ref{sec:tail}, from the boundary correction that turns a Lagrange function of the trial frame into one of the test space. The two propositions are complementary. Estimate \eqref{eq:apriori_lu} is uniform in $N$: if some coefficient vector $\vect y$ of bounded norm reproduces the data to the accuracy of the backward error, the computed solution has bounded error, however ill-conditioned the system. At a fixed threshold $t$ the bounds $\gamma_t \le C_\Phi$ and $\Gamma_t \le C_\Phi/t$ are the only ones independent of $N$, whereas the Lebesgue-type bounds grow at the rate $\kappa_\Lambda$; at the end of the spectrum the latter carry the exponent of Corollary~\ref{cor:exponent}, with a power of $N$ in place of a constant (Section~\ref{sec:res_tail}).

\subsubsection{Analytic Data}
\label{sec:res_analytic}

The condition shared by Corollary~\ref{cor:regimes_two_constant}(iii) and Proposition~\ref{prop:apriori} for a small bound, a competitor $\vect y$ of moderate norm whose nodal residual is at the level of the backward error, is satisfied for sufficiently analytic data. Recall $\rho^\ast$ and $\Omega_\rho$ from \eqref{eq:rho_star}--\eqref{eq:omega_rho}.

\begin{theorem}[Competitors of Moderate Coefficient Norm in the Test Space]
\label{thm:competitor}
Let $1 < \rho \le \rho^\ast$, and let $f$ extend analytically to an open simply connected set $\mathcal D \supset \bar\Omega_\rho$. Then $u = \Lop^{-1}f$ extends analytically to $\mathcal D$, and there is a constant $C$, depending on $f$, $\rho$, $\delta$, $L$ and $k$, such that for every $N \ge 3$ and every integer ${2 \le p \le N-1}$ there is $v_p \in \Phi_N$ with
\begin{enumerate}
    \item[(i)] $\norm{f - \Lop v_p}_{L^\infty(I)} \le C p^{9/2}\rho^{-p}$ and $\norm{u - v_p}_{L^\infty(I)} \le C p^{1/2}\rho^{-p}$;
    \item[(ii)] $\norm{v_p''}_{L^2(\tilde I)} + k\norm{v_p'}_{L^2(\tilde I)} \le C p^{11/2}(\rho^\ast/\rho)^p$;
    \item[(iii)] for $N \ge N_0$, with $N_0$ and $C_B$ of Lemma~\ref{lem:column_norms}, the coefficient vector $\vect y_p$ with $v_p = \Phi\mat D\vect y_p$ satisfies $\norm{\vect y_p}_2 \le C_BC\sqrt N\,p^{11/2}(\rho^\ast/\rho)^p$.
\end{enumerate}
\end{theorem}

\begin{corollary}[Bound Uniform in the Conditioning for Backward-Stable Solvers]
\label{cor:unconditional}
Assume the hypotheses of Theorem~\ref{thm:competitor} and of Corollary~\ref{cor:regimes_two_constant}(iii), and let $N \ge N_0$ of Lemma~\ref{lem:column_norms}; $C$ denotes constants depending on $f$, $\rho$, $\delta$, $L$ and $k$.
\begin{enumerate}
    \item[(i)] For every integer ${2 \le p \le N-1}$,
    \begin{equation}\label{eq:unconditional}
\norm{u - \hat u_N}_{L^2(I)} \le C\sqrt N\,p^{11/2}\Bigl\{\rho^{-p} + \bigl[(1+2K)\Gamma_\eta + K\gamma_\eta\eta^{-1}\bigr]\bigl(\rho^{-p} + \eta\,(\rho^\ast/\rho)^p\bigr)\Bigr\}.
\end{equation}
    \item[(ii)] With $a' := \log\rho/\log\rho^\ast \in (0,1]$, if $\eta \le \rho^{\ast\,-2}$ and $N \ge 2 + \log(1/\eta)/\log\rho^\ast$, then
    \begin{equation}\label{eq:unconditional_floor}
\norm{u - \hat u_N}_{L^2(I)} \le C\sqrt N\,\bigl(1 + \log(1/\eta)\bigr)^{11/2}\bigl(1 + (1+2K)\Gamma_\eta + K\gamma_\eta\eta^{-1}\bigr)\,\eta^{a'}.
\end{equation}
    \item[(iii)] If $\rho = \rho^\ast$, then under the conditions of (ii)
    \[
    \norm{u - \hat u_N}_{L^2(I)} \le C\sqrt N\,(1 + \log(1/\eta))^{11/2}\bigl(\eta + (1+3K)C_\Phi\bigr),
    \]
    a bound independent of the conditioning of the system.
\end{enumerate}
\end{corollary}

Theorem~\ref{thm:competitor} and Corollary~\ref{cor:unconditional} are proved in Section~\ref{sec:competitor}. By (i) with $p = N-1$ the computed solution converges geometrically, at the rate $\rho$ of the analyticity of the data, until the second term takes over; the choice of $p$ behind (ii) balances the two terms and identifies the floor as $\eta^{a'}$ times the two constants, with $a' = 1$ when the data are analytic on the largest region $\Omega_{\rho^\ast}$, whose real trace is the whole extended domain. Item (iii), obtained from (ii) and Proposition~\ref{prop:apriori}, is the qualitative rescue: for such data the error of a backward-stable solve of the exponentially ill-conditioned square system is bounded by a polynomial in $N$, uniformly in the conditioning, whereas the exact solution of the same rounded data carries the rounding errors amplified by up to $\Gamma_0$ (Remark~\ref{rem:divergence}), which for the Poisson operator grows at the rate $\kappa_\Lambda$ (Corollary~\ref{cor:lebesgue_rate}).

\subsubsection{The Tail of the Spectrum}
\label{sec:res_tail}

By Corollary~\ref{cor:regimes_two_constant}(iii) the accuracy of a backward-stable solve is decided by $\gamma_t$ and $\Gamma_t$ at the level $t = \eta$ of the backward error. Corollary~\ref{cor:exponent} and Proposition~\ref{prop:classical}(iv) identify the exponent $a(\omega)$ at the end of the spectrum, and the exponents fitted to the computed tails are close to it (Sections~\ref{sec:num_bound} and \ref{sec:num_potential}). This subsection determines the exponential rate of the pair of constants along the whole tail, at the thresholds $t_N := e^{-N\tau}$ with $0 < \tau < \kappa_\sigma$, which lie between the uniform bounds of Proposition~\ref{prop:apriori} at $\tau = 0$ and the end of the spectrum at $\tau = \kappa_\sigma$. The lower and upper rates
\begin{equation}\label{eq:rates_tail}
\begin{aligned}
r^-(\tau) &:= \liminf_{N\to\infty}\frac1N\log\max\{\gamma_{t_N},\ t_N\Gamma_{t_N}\}, \\
r^+(\tau) &:= \limsup_{N\to\infty}\frac1N\log\max\{\gamma_{t_N},\ t_N\Gamma_{t_N}\}
\end{aligned}
\end{equation}
are well defined (Lemma~\ref{lem:monotone}). We say that the \emph{sharp power law} holds if there are $N_0$ and $C$ such that
\begin{equation}\label{eq:sharp_law}
\gamma_t \le C\,t^{\,a(\omega)}, \qquad \Gamma_t \le C\,t^{\,a(\omega)-1} \qquad\text{for all } N \ge N_0 \text{ and all } t > 0.
\end{equation}
Its first inequality implies the second with another constant (Lemma~\ref{lem:two_forms}), and it implies $r^+(\tau) \le -a(\omega)\tau$ for every $\tau \ge 0$, since at $t = t_N$ it gives $\max\{\gamma_{t_N},\ t_N\Gamma_{t_N}\} \le Ct_N^{\,a(\omega)} = Ce^{-Na(\omega)\tau}$: the rate lies on or below the \emph{chord} joining the rate $0$ at $\tau = 0$ to the rate $\kappa_\Lambda - \kappa_\sigma = -a(\omega)\kappa_\sigma$ at $\tau = \kappa_\sigma$, where Proposition~\ref{prop:classical}(iv) applies. The results of this subsection concern thresholds tied to $N$; at a fixed threshold $t > 0$, as $N \to \infty$, they add nothing to Proposition~\ref{prop:apriori}.

\paragraph{The equilibrium problem.} The rate is given by a constrained energy problem in the angle $\theta$ of Section~\ref{sec:variable}. Let $\Theta_I := [\theta_0, \pi-\theta_0]$ and $\Theta_E := [0,\pi]\setminus\Theta_I$ be the arcs of the physical interval and of the extension. We identify a finite positive Borel measure $\lambda$ on $[0,\pi]$ with its image under $\theta \mapsto \cos\theta$, so that the node measure \eqref{eq:node_measure} is $\mu_\omega = \omega^{-1}d\theta$ on $\Theta_I$, and call
\begin{equation}\label{eq:measure_envelope}
E_\lambda(\vartheta) := \int\log|\cos\vartheta - \cos\varphi|\,d\lambda(\varphi) \quad (\vartheta \in [0,\pi]), \qquad J_0(\lambda) := -\int E_\lambda\,d\lambda ,
\end{equation}
its envelope and its energy: the negative logarithmic potential and the logarithmic energy of the image of $\lambda$ \cite{saff1997logarithmic}. A monic polynomial of degree $N$ whose zeros are distributed like $N\lambda$ has, away from its zeros, size $e^{N(E_\lambda(\vartheta) + o(1))}$ at $\cos\vartheta$.

\begin{definition}[Equilibrium Problem of the Tail]
\label{def:equilibrium}
Let $\mathscr A_\omega$ be the set of Borel probability measures $\lambda$ on $[0,\pi]$ with $\lambda \le \omega^{-1}d\theta$. For $\tau \ge 0$ and $\lambda \in \mathscr A_\omega$ let $J_\tau(\lambda) := J_0(\lambda) - 2\tau\lambda(\Theta_I)$.
\end{definition}

The bound $\omega^{-1}d\theta$ is the node density, extended to the whole of $[0,\pi]$, and the term $-2\tau\lambda(\Theta_I)$ is the energy of the external field equal to $-\tau$ on the physical arc and to $0$ on the extension, which favours zeros on the physical arc, where the nodal values are small. For $\tau = 0$ the minimizer is the arcsine measure $\pi^{-1}d\theta$, the equilibrium measure of the extended interval, and for $\tau \ge \kappa_\sigma$ it is the node measure (proof of Proposition~\ref{prop:equilibrium}).

\begin{proposition}[The Equilibrium Measure of the Tail]
\label{prop:equilibrium}
For every $\tau \ge 0$ the functional $J_\tau$ has a unique minimizer $\lambda_\tau$ on $\mathscr A_\omega$. It is invariant under $\vartheta\mapsto\pi-\vartheta$, its envelope $E_\tau := E_{\lambda_\tau}$ is continuous on $[0,\pi]$, and with $E_\tau^\ast := \max_{[0,\pi]}E_\tau$:
\begin{enumerate}
    \item[(i)] $E_\tau^\ast - \tau \le E_\tau \le E_\tau^\ast$ on $[0,\pi]$;
    \item[(ii)] $E_\tau = E_\tau^\ast$ on the support of $\lambda_\tau|_{\Theta_E}$, and $E_\tau = E_\tau^\ast - \tau$ on the support of $(\omega^{-1}d\theta - \lambda_\tau)|_{\Theta_I}$, the set of points of the physical arc near which $\lambda_\tau$ does not coincide with the node density;
    \item[(iii)] if $\tau < \kappa_\sigma$, then $\lambda_\tau(\Theta_E) > 0$ and there is a unique $\theta_\tau \in [\theta_0,\pi/2)$ such that $\lambda_\tau = \omega^{-1}d\theta$ on $[\theta_0,\theta_\tau]\cup[\pi-\theta_\tau,\pi-\theta_0]$ and $E_\tau = E_\tau^\ast - \tau$ on $[\theta_\tau,\pi-\theta_\tau]$; moreover $\max_{\Theta_I}E_\tau = E_\tau(\theta_0)$;
    \item[(iv)] if $\tau \ge \kappa_\sigma$, then $\lambda_\tau = \mu_\omega$.
\end{enumerate}
\end{proposition}

Thus $\lambda_\tau$ describes the zeros of a polynomial that vanishes at every node of a \emph{layer} $[\theta_0,\theta_\tau]$ next to each end of the physical interval and whose other zeros are free: in the \emph{centre} $[\theta_\tau,\pi-\theta_\tau]$, where their density is at most that of the nodes and equal to it on no subinterval, and in the extension. For $0 < \tau < \kappa_\sigma$ the layer is nonempty, and the support in the extension ends at a positive distance from it (Corollary~\ref{cor:layer}). Normalized to have maximum one, its size $e^{N(E_\tau - E_\tau^\ast)}$ is one on the support in the extension and $e^{-N\tau}$ on the centre, and on the physical interval it is largest at the ends. The computed tail singular functions have this structure (Figure~\ref{fig:tail}a,b). Put
\begin{equation}\label{eq:tail_law}
r_\omega(\tau) := E_\tau(\theta_0) - E_\tau^\ast = \max_{\Theta_I}E_\tau - \max_{[0,\pi]}E_\tau \qquad (0 < \tau < \kappa_\sigma),
\end{equation}
where the two expressions agree by (iii).

\begin{theorem}[The Tail Law]
\label{thm:tail_law}
Let $\Lop = \Lop_P$. For every $\tau \in (0,\kappa_\sigma)$,
\begin{equation}\label{eq:tail_limit}
\lim_{N\to\infty}\frac1N\log\max\{\gamma_{t_N},\ t_N\Gamma_{t_N}\} = r_\omega(\tau), \qquad t_N = e^{-N\tau}.
\end{equation}
\end{theorem}

We call $r_\omega$ the \emph{tail law}. It depends only on $\omega$, that is, on the ratio $\delta/L$, and on the level $\tau$, and it determines $\max\{\gamma_t, t\Gamma_t\}$ at every threshold of the tail up to a factor $e^{o(N)}$. The following properties concern the equilibrium problem alone and are proved without the frame, from the extremal problem of Proposition~\ref{prop:extremal} (Section~\ref{sec:tail_law_proof}).

\begin{proposition}[Properties of the Tail Law]
\label{prop:tail_properties}
The function $r_\omega$ is nonincreasing on $(0,\kappa_\sigma)$, satisfies $|r_\omega(\tau) - r_\omega(\tau')| \le |\tau - \tau'|$, and
\begin{equation}\label{eq:tail_law_bounds}
-\tau \le r_\omega(\tau) \le \min\{0,\ \kappa_\Lambda - \tau\}, \qquad r_\omega(\tau) > \kappa_\Lambda - \kappa_\sigma ;
\end{equation}
in particular $r_\omega(\tau) \to 0$ as $\tau \to 0$ and $r_\omega(\tau) \to \kappa_\Lambda - \kappa_\sigma = -a(\omega)\kappa_\sigma$ as $\tau \to \kappa_\sigma$.
\end{proposition}

The tail law also determines the growth of the amplification factor $\Gamma_0$ of data errors in the exact solution (Remark~\ref{rem:divergence}) and of the Lebesgue constant of the test space, for which Proposition~\ref{prop:classical}(iii) gives only an upper bound.

\begin{corollary}[Growth Rate of the Test-Space Lebesgue Constant]
\label{cor:lebesgue_rate}
Let $\Lop = \Lop_P$. Then
\[
\lim_{N\to\infty}\frac{\log\Gamma_0}{N} = \lim_{N\to\infty}\frac{\log\Lambda_N^\Psi}{N} = \kappa_\Lambda(\omega).
\]
\end{corollary}

\paragraph{Duality.} The tail laws of complementary angles determine each other.

\begin{theorem}[Duality of the Tail Law]
\label{thm:tail_duality}
For $\omega \in (0,\pi)$ and $0 < \bar\tau < 1$ let $\bar r_\omega(\bar\tau) := r_\omega(\bar\tau\kappa_\sigma(\omega))/\kappa_\sigma(\omega)$ be the tail law in units of $\kappa_\sigma(\omega)$. Then
\begin{equation}\label{eq:tail_duality}
\bar r_\omega(\bar\tau) + \bar r_{\pi-\omega}(\bar\tau) = -\bar\tau \qquad (0 < \bar\tau < 1).
\end{equation}
\end{theorem}

The angles $\omega$ and $\pi - \omega$ belong to the extensions $\delta$ and $L^2/(4\delta)$, for instance $\delta = L$ and $\delta = L/4$. As $\bar\tau \to 1$, the bounds \eqref{eq:tail_law_bounds} turn \eqref{eq:tail_duality} into Proposition~\ref{prop:duality}, which it extends from the end of the spectrum to the whole tail; for the self-dual extension it determines the tail law.

\begin{corollary}[The Self-Dual Extension]
\label{cor:self_dual}
For $\delta = L/2$, $r_{\pi/2}(\tau) = -\tau/2$ for $0 < \tau < \kappa_\sigma$. Under the hypotheses of Theorem~\ref{thm:tail_law}, $\max\{\gamma_{t_N},\ t_N\Gamma_{t_N}\} = t_N^{\,1/2 + o(1)}$ as $N \to \infty$ for every $\tau \in (0,\kappa_\sigma)$, and $1/2 = a(\pi/2)$.
\end{corollary}

\paragraph{The start of the tail.} For $\delta = L/2$ the tail law is the chord. For $\delta \ne L/2$ it is not: its initial slope, given by Theorem~\ref{thm:initial_slope} below, differs from the slope $-a(\omega)$ of the chord by the following proposition, and the sharp power law fails for every $\delta < L/2$ (Corollary~\ref{cor:critical}). In the computations of Section~\ref{sec:num_potential} the tail law is convex for $\delta = L$ and $\delta = 2L$, and hence, by Theorem~\ref{thm:tail_duality}, concave for their duals $\delta = L/4$ and $\delta = L/8$.

\begin{proposition}[Endpoint Exponent and Extension Fraction]
\label{prop:critical}
For $\omega \in (0,\pi)$ put $s(\omega) := 1 - \omega/\pi$. Then $a(\omega) - s(\omega)$ is positive for $\omega \in (\pi/2, \pi)$, zero for $\omega = \pi/2$ and negative for $\omega \in (0, \pi/2)$; in terms of the extension, $a(\omega) > 2\delta/(L + 2\delta)$ if and only if $\delta < L/2$.
\end{proposition}

\begin{theorem}[The Start of the Tail]
\label{thm:initial_slope}
For every $\omega \in (0,\pi)$ there are $C$ and $\tau_0 \in (0,\kappa_\sigma)$ such that
\begin{equation}\label{eq:initial_slope}
\bigl|r_\omega(\tau) + s(\omega)\,\tau\bigr| \le C\tau^{3/2} \qquad (0 < \tau \le \tau_0).
\end{equation}
\end{theorem}

The tail law thus starts with the slope $-s(\omega) = -2\delta/(L+2\delta)$, minus the fraction of the extended domain occupied by the extension; the initial slopes of dual extensions add up to $-1$, in accordance with Theorem~\ref{thm:tail_duality}. The lower bound in \eqref{eq:initial_slope} is proved with admissible zero distributions for the extremal problem of Proposition~\ref{prop:extremal}, built from the explicit solution of a boundary-layer problem at the ends of the physical interval (Appendix~\ref{app:start}): on the scale $\tau$ and up to $O(\tau)$, their density vanishes on a gap of width $g\tau$ in the extension next to $\theta_0$ and equals the node density on a layer of width $l\tau$ in the physical interval, where
\begin{equation}\label{eq:layer_widths}
g = \frac{2\omega\sin^2\theta_0}{\pi\sin\omega}, \qquad l = \frac{2\omega\cos^2\theta_0}{\pi\sin\omega},
\end{equation}
and their envelope is maximal on the extension beyond the gap and smaller by $\tau$ on the physical interval beyond the layer. The upper bound follows by duality. The computed equilibrium measures have a gap and a layer of these widths (Section~\ref{sec:num_potential}); that the gap and the layer of $\lambda_\tau$, which exist by Corollary~\ref{cor:layer} below, have these widths is not proved. Together with Proposition~\ref{prop:critical}, the theorem decides on which side of the chord the tail law starts.

\begin{corollary}[The Critical Extension]
\label{cor:critical}
\leavevmode
\begin{enumerate}
    \item[(i)] If $\delta < L/2$, then $r_\omega(\tau) > -a(\omega)\tau$ for all sufficiently small $\tau > 0$, and for $\Lop = \Lop_P$ there are no $N_1$ and $C$ such that $\gamma_t \le Ct^{\,a(\omega)}$ for all $N \ge N_1$ and all $t > 0$; in particular the sharp power law \eqref{eq:sharp_law} fails.
    \item[(ii)] If $\delta > L/2$, then $r_\omega(\tau) < -a(\omega)\tau$ for all sufficiently small $\tau > 0$.
\end{enumerate}
\end{corollary}

The theorem also shows that the equilibrium measure has a layer and a gap at every level of the tail.

\begin{corollary}[The Layer and the Gap]
\label{cor:layer}
For $0 < \tau < \kappa_\sigma$, $-\tau < r_\omega(\tau) < 0$. Consequently $\theta_\tau > \theta_0$ in Proposition~\ref{prop:equilibrium}(iii), and $\lambda_\tau$ vanishes on $(\theta_0 - \eps, \theta_0)$ for some $\eps > 0$.
\end{corollary}

\subsubsection{Explicit Bounds Along the Tail}
\label{sec:res_bounds}

Theorem~\ref{thm:tail_law} is proved by matching two families of explicit bounds, valid for all $N \ge N_0$, which also serve for the certified evaluations of Section~\ref{sec:num_certified}. The lower bounds come from polynomials whose zeros are described by the following configurations: a layer of zeros at the nodes next to the ends of the physical interval, free zeros of piecewise-constant density on the centre and on part of the extension, and none in between. Their envelope $E$ is the negative logarithmic potential of the zero distribution, so that $e^{NE}$ is the size of the polynomial; $\ell(\mathcal C) \le 0$ is the exponent of its size at the ends of the physical interval relative to its maximum, and $\tau(\mathcal C) \ge 0$ the exponent by which its maximum on the centre lies below its overall maximum.

\begin{definition}[Piecewise-Uniform Configuration]
\label{def:configuration}
Let $0 < \xi \le \theta_0/2$. A \emph{configuration} $\mathcal C$ consists of a layer parameter $\theta_s \in (\theta_0, \pi/2)$ and finitely many arcs $A_k = [\alpha_k, \beta_k]$, $1 \le k \le Q$, with densities $\rho_k > 0$, such that
\begin{enumerate}
    \item[(i)] $\xi \le \alpha_k < \beta_k \le \pi/2$, the arcs have pairwise disjoint interiors, and each arc lies either in $[\xi, \theta_0)$ (an \emph{extension arc}, so that $\beta_k < \theta_0$) or in $[\theta_s, \pi/2]$ (a \emph{centre arc});
    \item[(ii)] with $A_k^\ast := A_k \cup (\pi - A_k)$, of length $|A_k^\ast| = 2(\beta_k - \alpha_k)$ if $\beta_k < \pi/2$ and $|A_k^\ast| = \pi - 2\alpha_k$ if $\beta_k = \pi/2$, the total mass $m(\mathcal C) := 2(\theta_s - \theta_0)/\omega + \sum_{k=1}^Q\rho_k|A_k^\ast|$ is less than one.
\end{enumerate}
With $S := [\theta_0, \theta_s]$ (the \emph{layer}) and $S^\ast := S \cup (\pi - S)$, the \emph{envelope} of $\mathcal C$ is
\begin{equation}\label{eq:envelope}
E(\vartheta) := \frac1\omega\int_{S^\ast}\log|\cos\vartheta - \cos\varphi|\,d\varphi + \sum_{k=1}^Q\rho_k\int_{A_k^\ast}\log|\cos\vartheta - \cos\varphi|\,d\varphi, \qquad \vartheta \in [0,\pi],
\end{equation}
and we put
\begin{gather*}
E_{\max} := \max_{[0,\pi]}E, \qquad E_c := \max_{[\theta_s,\ \pi-\theta_s]}E, \\
\ell(\mathcal C) := E(\theta_0) - E_{\max} \le 0, \qquad \tau(\mathcal C) := E_{\max} - E_c \ge 0 .
\end{gather*}
\end{definition}

\begin{theorem}[Explicit Lower Bound on the Leakage Along the Tail]
\label{thm:tail_lower}
Let $\Lop = \Lop_P$, let $\mathcal C$ be a configuration with $\ell := \ell(\mathcal C)$, $\tau := \tau(\mathcal C)$, and assume $\ell > -(\kappa_\sigma - \kappa_\Lambda)$. Then there are $N_0$ and $C$, depending only on $\delta$, $L$ and $\mathcal C$, and, for every $N \ge N_0$, a vector $\vect w_N \in \mathbb R^n\setminus\{0\}$ such that for all $s \in [0,1]$
\begin{equation}\label{eq:tail_quotient}
\frac1N\log\frac{\norm{\Phi\mat D\vect w_N}_{L^2(I)}}{\norm{\mat B_D\vect w_N}_2^{\,s}\,\norm{\vect w_N}_2^{\,1-s}} \ \ge\ \ell + s\tau - C\,\frac{\log N}{N},
\end{equation}
where $\norm{\mat B_D\vect w_N}_2 > 0$ by Proposition~\ref{prop:unisolvence_poisson}.
\end{theorem}

The upper bounds come from Lagrange interpolation: for the Poisson operator a function of $\Psi_N \subset V_N$ is $\sin\theta$ times a polynomial of degree $N-1$ in the Chebyshev variable, determined by its values at any $N$ distinct points of $(0,\pi)$, and for the functions that enter $\gamma_t$ and $\Gamma_t$ the values are small at the interior nodes and bounded on the whole extended interval. The interpolation sets are described by the following configurations, in which a node arc carries a subset of the nodes of prescribed density and a virtual arc a lattice of points that need not be nodes.

Interpolation at a subset of an equispaced grid that follows a prescribed distribution is a standard way of bounding a polynomial by its values on the grid, the mock-Chebyshev subsets being the best-known case \cite{boyd2009divergence}, see also \cite[Sec.~5.3]{adcock2019optimal}; the virtual points use in addition the bound on the extended interval. Adcock and Shadrin \cite[Thm.~2.4]{adcock2023fast} combine the same two bounds, by a Markov-type inequality and the bound of Coppersmith and Rivlin \cite{coppersmith1992growth}: a polynomial of degree $n$ bounded by one at $m + 1$ equispaced nodes of an interval and by $1/\eps$ on an extended interval is bounded on the interval by a numerical constant under linear oversampling, $m = c\,n\log(1/\eps)$ with $c$ depending on the extension. In the square system $m = n$, and at the levels $t = e^{-N\tau}$ the ratio of the two bounds is exponential in $N$, so that result does not apply and the bounds of Theorem~\ref{thm:dual} depend exponentially on $N$.

\begin{definition}[Interpolation Configuration]
\label{def:interp_configuration}
Let $0 < \xi \le \theta_0/2$. An \emph{interpolation configuration} $\mathcal S$ consists of finitely many pairwise disjoint arcs $A_k = [\alpha_k,\beta_k]$, $\alpha_k < \beta_k$, with densities $\rho_k > 0$, $1 \le k \le Q$, such that
\begin{enumerate}
    \item[(i)] the \emph{node arcs} $A_k$, $1 \le k \le Q_1$, with $Q_1 \ge 1$, lie in $[\theta_0, \pi-\theta_0]$, and each carries a modulus $q_k \in \mathbb N$ and a nonempty set of residues $R_k \subset \{0, 1, \dots, q_k-1\}$, with $\rho_k := |R_k|/(q_k\omega) \le 1/\omega$;
    \item[(ii)] the \emph{virtual arcs} $A_k$, $Q_1 < k \le Q$, with $Q > Q_1$, lie in $[\xi, \pi-\xi]$;
    \item[(iii)] the total mass $\sum_{k=1}^Q\rho_k(\beta_k - \alpha_k)$ equals one.
\end{enumerate}
Its \emph{envelope} is
\begin{equation}\label{eq:interp_envelope}
E(\vartheta) := \sum_{k=1}^Q\rho_k\int_{A_k}\log|\cos\vartheta - \cos\varphi|\,d\varphi, \qquad \vartheta \in [0,\pi],
\end{equation}
a continuous function (Lemma~\ref{lem:E_modulus}), and with $E_I := \max_{[\theta_0,\pi-\theta_0]}E$, $E_1$ the minimum of $E$ over the node arcs and $E_2$ its minimum over the virtual arcs, we put
\begin{gather*}
A_1(\mathcal S) := E_I - E_1 \ge 0, \qquad A_2(\mathcal S) := E_I - E_2, \\
\Phi(\mathcal S,\tau) := \max\{A_1(\mathcal S) - \tau,\ A_2(\mathcal S)\} \qquad (\tau \ge 0).
\end{gather*}
\end{definition}

\begin{theorem}[Upper Bound Along the Tail by Lagrange Interpolation]
\label{thm:dual}
Let $\Lop = \Lop_P$ and let $\mathcal S$ be an interpolation configuration. There are $N_0$ and $C$, depending only on $\delta$, $L$ and $\mathcal S$, such that for all $N \ge N_0$ and all $t \ge 0$
\begin{equation}\label{eq:dual_bound}
\max\{\gamma_t,\ t\Gamma_t\} \le N^C\bigl(t\,e^{NA_1(\mathcal S)} + e^{NA_2(\mathcal S)}\bigr);
\end{equation}
consequently, for every $\tau \ge 0$ and $t_N := e^{-N\tau}$,
\begin{equation}\label{eq:dual_rate}
\frac1N\log\max\{\gamma_{t_N},\ t_N\Gamma_{t_N}\} \le \Phi(\mathcal S,\tau) + C\,\frac{\log N}{N}.
\end{equation}
\end{theorem}

Like Theorem~\ref{thm:tail_lower}, Theorem~\ref{thm:dual} is proved for the Poisson operator only; the generalization of both to convection--diffusion is open (Section~\ref{sec:conclusion}). The quantity $A_1(\mathcal S)$ is the loss at the nodes, the size on the physical interval of the Lagrange polynomials of the node part of the set relative to their size at the nodes, and $A_2(\mathcal S)$ the loss at the virtual points. Both families of bounds are sharp.

\begin{corollary}[Sharpness of the Explicit Bounds]
\label{cor:sharpness}
For every $\tau \in (0,\kappa_\sigma)$,
\begin{equation}\label{eq:sharpness}
r_\omega(\tau) = \sup_{\mathcal C}\,\ell(\mathcal C) = \inf_{\mathcal S}\,\Phi(\mathcal S,\tau),
\end{equation}
the supremum being taken over the configurations $\mathcal C$ with $\tau(\mathcal C) \ge \tau$ and $\ell(\mathcal C) > -(\kappa_\sigma - \kappa_\Lambda)$ and the infimum over the interpolation configurations $\mathcal S$.
\end{corollary}

The configurations and the interpolation configurations that nearly attain the supremum and the infimum are discretizations of the same measure $\lambda_\tau$: its restriction to the layer gives the zeros of the configuration at the nodes and the rest of it the free zeros, while its restriction to the physical arc gives the nodes of the interpolation set and its restriction to the extension the virtual points. Propositions~\ref{prop:equilibrium} and \ref{prop:tail_properties}, Theorems~\ref{thm:tail_law}, \ref{thm:tail_duality}, \ref{thm:initial_slope}, \ref{thm:tail_lower} and \ref{thm:dual} and Corollaries~\ref{cor:lebesgue_rate}, \ref{cor:self_dual}, \ref{cor:critical}, \ref{cor:layer} and \ref{cor:sharpness} are proved in Section~\ref{sec:tail}, with technical lemmas in Appendices~\ref{app:tail} and \ref{app:start}, and Proposition~\ref{prop:critical} in Section~\ref{sec:potential}. Section~\ref{sec:tail} also states the non-asymptotic forms of the two bounds (Corollaries~\ref{cor:either_or}--\ref{cor:certified}), on which the certified results of Section~\ref{sec:num_certified} rest.

\subsubsection{Dependence Between the Results}
\label{sec:res_dependence}

Every theorem, proposition, lemma and corollary of Sections~\ref{sec:setting}--\ref{sec:multi} and of the appendices is proved without recourse to computation; the one statement that rests on a computation, the certified evaluation of explicit functions described in Appendix~\ref{app:tail}, is Corollary~\ref{cor:certified_instances} of Section~\ref{sec:num_certified}. Computed quantities are reported in Section~\ref{sec:numerics}. The main ingredients of the results are the following.
\begin{center}
\small
\setlength{\tabcolsep}{3pt}
\begin{tabular}{>{\raggedright\arraybackslash}p{0.38\textwidth}>{\raggedright\arraybackslash}p{0.56\textwidth}}
Result & Main ingredients \\ \hline
Theorem~\ref{thm:representation}, Corollaries~\ref{cor:error} and \ref{cor:greens} & Lemma~\ref{lem:dirichlet_1d}, Assumption~\ref{ass:unisolvence} \\
Proposition~\ref{prop:impossibility} & \cite[Thm.~3.1]{platte2011impossibility} \\
Proposition~\ref{prop:unisolvence_poisson} & Lemmas~\ref{lem:unisolvence} and \ref{lem:correction}, Proposition~\ref{prop:lagrange_norms} \\
Theorem~\ref{thm:rates}, Corollaries~\ref{cor:exponent} and \ref{thm:poisson_illcond} & Lemmas~\ref{lem:chebyshev_variable} and \ref{lem:clausen}, Appendix~\ref{app:rates} \\
Propositions~\ref{prop:duality} and \ref{prop:critical} & the closed form \eqref{eq:exponent}, Lemma~\ref{lem:cl2}; Proposition~\ref{prop:duality} for Proposition~\ref{prop:critical} \\
Theorem~\ref{thm:two_constant}, Corollary~\ref{cor:regimes_two_constant} & linear algebra; Lemma~\ref{lem:tsvd} for (ii) \\
Proposition~\ref{prop:apriori} & Appendix~\ref{app:apriori} \\
Proposition~\ref{prop:classical} & Assumption~\ref{ass:unisolvence}, Lemma~\ref{lem:dirichlet_box}(iii); for (ii)--(iv), Lemma~\ref{lem:correction} and Proposition~\ref{prop:lagrange_norms} \\
Theorem~\ref{thm:competitor}, Corollary~\ref{cor:unconditional} & the Chebyshev coefficient bound of \cite[Thm.~8.1]{trefethen2019approximation}, Lemmas~\ref{lem:cheb_ellipse} and \ref{lem:coefficient_norm}; Corollary~\ref{cor:regimes_two_constant}(iii) and Proposition~\ref{prop:apriori} \\
Theorem~\ref{thm:tail_lower} & Proposition~\ref{prop:unisolvence_poisson}, Lemmas~\ref{lem:unisolvence}, \ref{lem:construction} and \ref{lem:correction}, Appendices~\ref{app:rates}--\ref{app:tail} \\
Theorem~\ref{thm:dual} & Lemmas~\ref{lem:chebyshev_variable}, \ref{lem:dirichlet_box}(iii), \ref{lem:test_norm} and \ref{lem:lattice_lagrange} \\
Proposition~\ref{prop:equilibrium} & Lemmas~\ref{lem:cl2}, \ref{lem:clausen}, \ref{lem:cheb}, \ref{lem:minimum} and \ref{lem:env} \\
Proposition~\ref{prop:tail_properties}, Corollary~\ref{cor:sharpness} & Propositions~\ref{prop:equilibrium} and \ref{prop:extremal}, Lemmas~\ref{lem:clausen}, \ref{lem:minimum} and \ref{lem:env}; Lemma~\ref{lem:discretize} for Corollary~\ref{cor:sharpness} \\
Theorem~\ref{thm:tail_law} & Propositions~\ref{prop:equilibrium} and \ref{prop:tail_properties}, Theorems~\ref{thm:tail_lower} and \ref{thm:dual}, Lemmas~\ref{lem:monotone} and \ref{lem:discretize} \\
Corollary~\ref{cor:lebesgue_rate} & Theorem~\ref{thm:tail_law}, Propositions~\ref{prop:unisolvence_poisson}, \ref{prop:classical} and \ref{prop:tail_properties} \\
Proposition~\ref{prop:extremal} & Proposition~\ref{prop:equilibrium}, Lemma~\ref{lem:minimum} \\
Theorem~\ref{thm:tail_duality}, Corollary~\ref{cor:self_dual} & Proposition~\ref{prop:equilibrium}, Lemmas~\ref{lem:cheb} and \ref{lem:cl2}(iv), the closed form of $\kappa_\sigma$ in Theorem~\ref{thm:rates}; Theorem~\ref{thm:tail_law} for the statement on $\gamma_t$ and $\Gamma_t$ \\
Theorem~\ref{thm:initial_slope}, Corollaries~\ref{cor:critical} and \ref{cor:layer} & Propositions~\ref{prop:extremal} and \ref{prop:critical}, Theorem~\ref{thm:tail_duality}, Lemmas~\ref{lem:inner}, \ref{lem:shape} and \ref{lem:env}; Proposition~\ref{prop:tail_properties} for Corollary~\ref{cor:layer}, and Theorem~\ref{thm:tail_law}, Proposition~\ref{prop:unisolvence_poisson} and Lemma~\ref{lem:two_forms} for Corollary~\ref{cor:critical}(i) \\
Theorem~\ref{thm:regularized}, Corollary~\ref{cor:regimes} & Lemma~\ref{lem:tsvd} \\
Corollary~\ref{cor:certified_instances} & Corollaries~\ref{cor:no_power_law} and \ref{cor:certified}; certified evaluation (Appendix~\ref{app:tail}) \\
\end{tabular}
\end{center}

\section{Exact Structure of the Interpolatory Method: Proofs of Theorem~\ref{thm:representation} and Proposition~\ref{prop:impossibility}}
\label{sec:structure}

Throughout this section $\Lop u = -u'' + k(x)u'$ with $k \in C(\bar I)$; the operators \eqref{eq:phi_poisson}--\eqref{eq:phi_cd} are the cases $k \equiv 0$ and $k$ constant. All statements hold in exact arithmetic. The facts about the continuous Dirichlet problem that we use, the maximum principle, the Green's function \eqref{eq:green_general} and the inverse norms, are Lemma~\ref{lem:dirichlet_1d} of Appendix~\ref{app:standard}.

\begin{lemma}[Dimensions]
\label{lem:dimensions}
For $N \ge 3$, $\dim \Phi_N = \dim \Psi_N = N-2$, $\Lop: \Phi_N \to \Psi_N$ is a bijection, and Assumption~\ref{ass:unisolvence} holds if and only if the matrix $\mat A$ of \eqref{eq:square_system} is nonsingular.
\end{lemma}
\begin{proof}
With $s_j = \sin(j\theta_0)$ as in Section~\ref{sec:constants}, one has $w_j(0) = \sqrt{2/\tilde L}\,s_j$ and $w_j(L) = \sqrt{2/\tilde L}\sin(j\pi - j\pi\delta/\tilde L) = (-1)^{j+1}\sqrt{2/\tilde L}\,s_j$. The matrix of the two boundary functionals $v \mapsto v(0)$, $v \mapsto v(L)$ restricted to $\operatorname{span}\{w_1,w_2\}$ is therefore $\sqrt{2/\tilde L}\bigl(\begin{smallmatrix} s_1 & s_2 \\ s_1 & -s_2\end{smallmatrix}\bigr)$, with determinant $-4s_1s_2/\tilde L \ne 0$ because $0 < \pi\delta/\tilde L < 2\pi\delta/\tilde L < \pi$. Hence the two functionals are linearly independent on $V_N$ and $\dim\Phi_N = N-2$. If $v \in \Phi_N$ and $\Lop v = 0$ then $v = 0$ on $\bar I$ by Lemma~\ref{lem:dirichlet_1d}(i), hence $v = 0$ in $V_N$, since its polynomial factor in Lemma~\ref{lem:chebyshev_variable} vanishes on $[-c_0,c_0]$; so $\Lop$ is injective on $\Phi_N$, and $\dim\Psi_N = \dim\Phi_N$. Finally, $\mat A\vect c = 0$ means that $v := \sum_jc_jw_j$ lies in $\Phi_N$ and that $\Lop v \in \Psi_N$ vanishes at the interior nodes. Under Assumption~\ref{ass:unisolvence} this forces $\Lop v = 0$, hence $v = 0$ and $\vect c = 0$. Conversely, if $\mat A$ is nonsingular and $\Lop v$, $v \in \Phi_N$, vanishes at the interior nodes, then the coefficient vector of $v$ lies in the kernel of $\mat A$, so $v = 0$; interpolation at the $N-2$ interior nodes is therefore injective on the $(N-2)$-dimensional space $\Psi_N$, hence unisolvent.
\end{proof}

\begin{lemma}[Unisolvence Criterion for the Poisson Operator]
\label{lem:unisolvence}
Let $\Lop = \Lop_P = -d^2/dx^2$, let $\tilde\Lop_P^{-1}: V_N \to V_N$ be its inverse on $V_N$, $\tilde\Lop_P^{-1}w_j = \lambda_j^{-1}w_j$, and let $\tilde G(x,y) = (x+\delta)(L+\delta-y)/\tilde L$ for $x \le y$, $\tilde G(x,y) = \tilde G(y,x)$ for $x > y$, be the Green's function of $\Lop_P$ with Dirichlet conditions on $\tilde I$, so that $(\tilde\Lop_P^{-1}\psi)(x) = \int_{\tilde I}\tilde G(x,y)\psi(y)\,dy$ for $\psi \in V_N$. In the notation of Section~\ref{sec:variable} ($\theta = \pi(x+\delta)/\tilde L$, $z = \cos\theta$, $z_m = \cos\theta_m$) put $w_N^\circ(z) := \prod_{m=2}^{N-1}(z - z_m)$, $\psi_0 := \sin\theta\; w_N^\circ(z)$ and $\psi_1 := \sin\theta\; z\,w_N^\circ(z)$. Then Assumption~\ref{ass:unisolvence} holds if and only if
\begin{equation}\label{eq:unisolvence_criterion}
\int_{-\delta}^{L+\delta}\tilde G(0,y)\,\psi_0(y)\,dy \ne 0 \quad\text{and}\quad \int_{-\delta}^{L+\delta}\tilde G(0,y)\,\psi_1(y)\,dy \ne 0 .
\end{equation}
\end{lemma}
\begin{proof}
Let $K_N := \{\psi \in V_N : \psi(x_m) = 0,\ 2 \le m \le N-1\}$. By Lemma~\ref{lem:chebyshev_variable}, $\psi \in V_N$ has the form $\sqrt{2/\tilde L}\sin\theta\,q(z)$ with $q \in \mathbb P_{N-1}$, and $\psi \in K_N$ if and only if $q$ vanishes at the $N-2$ distinct points $z_2, \dots, z_{N-1}$ (as $\sin\theta_m \ne 0$), i.e.\ $q = w_N^\circ\cdot(\alpha z + \beta)$; hence $K_N = \operatorname{span}\{\psi_0, \psi_1\}$. By Lemma~\ref{lem:dimensions}, $\dim\Psi_N = N-2$, so Assumption~\ref{ass:unisolvence} is equivalent to $\Psi_N \cap K_N = \{0\}$.

Let $(Rv)(x) := v(L-x)$. In the variable $\theta$ this is $\theta \mapsto \pi-\theta$, so $z \mapsto -z$, and $Rw_j = (-1)^{j+1}w_j$: $R$ maps $V_N$ and $\Phi_N$ onto themselves, commutes with $\Lop_P$ and $\tilde\Lop_P^{-1}$ on $V_N$, and hence maps $\Psi_N = \Lop_P(\Phi_N)$ onto itself; since $x_{N+1-m} = L - x_m$, it also maps $K_N$ onto itself. As $w_N^\circ(-z) = (-1)^N w_N^\circ(z)$, one has $R\psi_0 = (-1)^N\psi_0$ and $R\psi_1 = (-1)^{N+1}\psi_1$. If $\Psi_N \cap K_N$ contains $\psi = \alpha\psi_0 + \beta\psi_1 \ne 0$, then, being $R$-invariant, it contains $\tfrac12(\psi \pm R\psi)$, i.e.\ $\alpha\psi_0$ and $\beta\psi_1$, so it contains $\psi_0$ or $\psi_1$; the converse is trivial. Thus Assumption~\ref{ass:unisolvence} fails if and only if $\psi_0 \in \Psi_N$ or $\psi_1 \in \Psi_N$.

For $\psi \in V_N$ put $v := \tilde\Lop_P^{-1}\psi \in V_N$; then $\psi \in \Psi_N$ if and only if $v \in \Phi_N$, i.e.\ $v(0) = v(L) = 0$, because $\Lop_P: V_N \to V_N$ is a bijection. If $R\psi = \pm\psi$ then $Rv = \pm v$ and $v(L) = \pm v(0)$, so $\psi \in \Psi_N$ if and only if $v(0) = 0$. Finally $v(0) = \int_{\tilde I}\tilde G(0,y)\psi(y)\,dy$, because $\tilde G$ inverts $-d^2/dx^2$ with Dirichlet conditions on $\tilde I$ and the $w_j$ are its eigenfunctions with eigenvalues $\lambda_j$. Applying this to $\psi_0$ and $\psi_1$ gives \eqref{eq:unisolvence_criterion}.
\end{proof}

The two integrals in \eqref{eq:unisolvence_criterion} are real-analytic functions of $\delta \in (0,\infty)$ for fixed $N$ and $L$, so for the Poisson operator and fixed $N$ the assumption can fail at most on a discrete set of extension parameters unless one of the integrals vanishes identically, which we have not excluded; for fixed $\delta$, on the other hand, it holds for all sufficiently large $N$ (Proposition~\ref{prop:unisolvence_poisson}). For the convection--diffusion operator no such reduction is available, and the assumption is the nonsingularity of $\mat A$; in all one-dimensional multiple-precision computations of Section~\ref{sec:numerics} the matrix was found nonsingular.

\subsection{The Representation Theorem}

\begin{proof}[Proof of Theorem~\ref{thm:representation}]
The boundary rows give $u_N(0) = u_N(L) = 0$, so $u_N \in \Phi_N$ and $\Lop u_N \in \Psi_N$. The interior rows give $(\Lop u_N)(x_m) = f(x_m)$, $2 \le m \le N-1$; by unisolvence $\Lop u_N = \Pop_N f$. Since $u_N \in C^2(\bar I)$ satisfies $\Lop u_N = \Pop_N f$ on $I$ and $u_N(0) = u_N(L) = 0$, the uniqueness in Lemma~\ref{lem:dirichlet_1d}(ii) gives $u_N = \Lop^{-1}\Pop_N f$.
\end{proof}

The argument uses only that $u_N \in \Phi_N$ and that $\Lop u_N$ interpolates the interior data; it therefore applies verbatim to any vector of interior nodal data $\vect g \in \mathbb R^{N-2}$ in place of $(f(x_m))_m$, with $\Pop_N\vect g$ denoting the interpolant in $\Psi_N$ of these values.

Inhomogeneous data are reduced to this case by subtracting any $v_g \in V_N$ with $v_g(0) = g_0$, $v_g(L) = g_L$, which exists by the proof of Lemma~\ref{lem:dimensions}: $u_N = v_g + \Lop^{-1}\Pop_N(f - \Lop v_g)$.

\begin{proof}[Proof of Corollary~\ref{cor:error}]
The identity follows from \eqref{eq:representation}. For every $\psi \in \Psi_N$, $\Pop_N\psi = \psi$, so $f - \Pop_N f = (f - \psi) - \Pop_N(f - \psi)$ and $\norm{f - \Pop_N f}_{L^\infty(I)} \le (1 + \Lambda_N^\Psi)\norm{f - \psi}_{L^\infty(I)}$; take the infimum over $\psi$ and apply $\norm{\Lop^{-1}}_\infty$ (Lebesgue's lemma, cf.\ \cite{cheney1966}).
\end{proof}

\subsection{What the Operator Contributes}

The explicit Green's functions of the two operators and their inverse norms are recorded in Proposition~\ref{prop:continuous_greens} of Appendix~\ref{app:standard}: $G_P$ is a hat function, and for $k > 0$ the kernel $G_{CD}$ decays exponentially upstream and plateaus at the level $1/k$ downstream. Whether any of this reaches the discrete method is decided by the test space, whose structure is elementary.

\begin{lemma}[Modal Phase Structure of the Test Space]
\label{lem:phase}
For $v = \sum_{j=1}^N c_j w_j \in V_N$ and constant $k \ge 0$,
\begin{equation}\label{eq:phase}
\Lop_{CD} v = \sqrt{\tfrac{2}{\tilde L}} \sum_{j=1}^N c_j\, \mu_j \sqrt{\mu_j^2 + k^2}\; \sin\bigl(\mu_j(x+\delta) + \vartheta_j\bigr), \qquad \vartheta_j = \arctan\frac{k}{\mu_j} \in \Bigl[0, \frac{k \tilde L}{j\pi}\Bigr].
\end{equation}
\end{lemma}
\begin{proof}
$\lambda_j \sin\phi + k\mu_j\cos\phi = \mu_j\sqrt{\mu_j^2+k^2}\,\sin(\phi+\vartheta_j)$ with $\tan\vartheta_j = k/\mu_j$, and $\arctan t \le t$.
\end{proof}

\begin{remark}[Modal P\'eclet Number and Operator Independence]
\label{rem:peclet}
The ratio $k/\mu_j$ in \eqref{eq:phase} is a modal P\'eclet number. Convection shifts the phases of all modes of the test space $\Psi_N(k) := \Lop_{CD}(\Phi_N)$ relative to those of $\Psi_N(0) = \Lop_P(\Phi_N)$, by $\vartheta_j \le k\tilde L/(j\pi)$, substantially only those of the convection-dominated modes $j \lesssim k\tilde L/\pi$. The ill-conditioning of the frame is a property of its high modes: heuristically, the modes above the transition index $NL/\tilde L$ are those whose restrictions to $I$ are nearly linearly dependent (Section~\ref{sec:potential}). For $N \gg k\tilde L/\pi$ these modes are the same in $\Psi_N(k)$ and $\Psi_N(0)$ up to phase shifts $O(k\tilde L/N)$, so the interpolation operators of the two test spaces can be expected to have the same ill-conditioning and, by \eqref{eq:GN_entries}, the discrete Green's functions the same structure; the exact computations of Section~\ref{sec:num_1d} find nearly the same exponential rate of $\Lambda_N^\Psi$ for $k \le 50$. For $N \lesssim k\tilde L/\pi$ convection reshapes the whole test space; in the experiments of Section~\ref{sec:num_solvers} with $k$ up to $10^5$ the convergence rate is nevertheless unchanged, and only the floor of the unnormalized system rises, with the norm of the matrix. The physical upstream decay of $G_{CD}$ is a low-mode property, and the interpolation error $f - \Pop_N f$, the only quantity on which $\Lop^{-1}$ acts by Corollary~\ref{cor:error}, is concentrated in the high modes, on which $\Lop_{CD}^{-1}$ and $\Lop_P^{-1}$ act alike: the moduli of their symbols $1/(\mu_j^2 + ik\mu_j)$ and $1/\mu_j^2$ agree to relative accuracy $(k/\mu_j)^2/2$, and their phases differ by $\vartheta_j$.
\end{remark}

\begin{remark}[Variable Coefficients]
\label{rem:variable_k}
Heuristically, for variable $k$ the high modes of $\Psi_N$ are perturbations of those of $\Psi_N(0)$ of relative size $O(\norm{k}_\infty\tilde L/(j\pi))$, so that the conclusion of Remark~\ref{rem:peclet} can be expected to persist; the variable-coefficient experiment of Section~\ref{sec:num_2d} agrees with this.
\end{remark}
\subsection{Exact Arithmetic: Impossibility, Amplification of Data Errors, and the Discrete Green's Function}
\label{sec:exact_arith}
Corollary~\ref{cor:error} places the interpolatory method inside the framework of \cite{platte2011impossibility}: the collocation error is the interpolation error of $f$ in $\Psi_N$ at equispaced nodes, smoothed by a bounded operator.

\begin{proof}[Proof of Proposition~\ref{prop:impossibility}]
Put $\Pop_N := 0$ for $2 \le N < N_1$. For $N \ge N_1$, $\Pop_N f$ depends only on the values of $f$ at the interior nodes, hence only on its values on the equispaced $N$-point grid $x_1, \dots, x_N$ of $\bar I$; therefore $\{\Pop_N\}_{N\ge2}$ is an approximation procedure based on equispaced grids in the sense of \cite[Sec.~1]{platte2011impossibility}, transferred from $[-1,1]$ to $\bar I$ by the affine map. For $N \ge N_1$, since $\Pop_N$ is linear and its Lagrange functions are real, its condition number \cite[Eq.~(1.1)]{platte2011impossibility}, the supremum over complex nodal data of the ratio between the maximum norm of the output on $\bar I$ and the maximum norm of the data, is $\max_{x\in\bar I}\sum_m|\ell_m^\Psi(x)| = \Lambda_N^\Psi$. Let $E'$ be a compact set with $E \subset \operatorname{int}E'$. Every function continuous on $E'$ and analytic in its interior is analytic in a neighbourhood of $E$ and satisfies $\norm{f}_{L^\infty(E)} \le \norm{f}_{L^\infty(E')}$, so the hypothesis, together with $\norm{f}_{L^\infty(I)} \le \norm{f}_{L^\infty(E')}$ for $N < N_1$, gives $\norm{f - \Pop_N f}_{L^\infty(I)} \le C'\sigma^{-N}\norm{f}_{L^\infty(E')}$ for all such $f$ and all $N \ge 2$, with $C' := \max\{C, \sigma^{N_1}\}$; this is hypothesis (3.1) of \cite[Thm.~3.1]{platte2011impossibility} for the set $E'$. That theorem yields $\beta > 1$ and $N_0 \ge N_1$ with $\Lambda_N^\Psi \ge \beta^N$ for all $N \ge N_0$. Since $\Lambda_N^\Psi \ge 1$ for every $N \ge N_1$ (each Lagrange function equals one at its node), the claim holds for all $N \ge N_1$ with $c := \min\{1, \min_{N_1 \le N < N_0}\Lambda_N^\Psi\beta^{-N}\}$.
\end{proof}

Proposition~\ref{prop:impossibility} requires exponential growth of $\Lambda_N^\Psi$ whenever interpolation in $\Psi_N$ converges geometrically for all data analytic in a fixed region. For the Poisson operator $\Lambda_N^\Psi$ grows with the exponential rate $\kappa_\Lambda$ (Corollary~\ref{cor:lebesgue_rate}, proved in Section~\ref{sec:tail}); for convection--diffusion no lower bound is proved, and the computations of Section~\ref{sec:num_1d} show growth at a rate close to $\kappa_\Lambda$. The consequences for the method in exact arithmetic are the following, and they are verified in multiple-precision arithmetic in Section~\ref{sec:numerics}.

\begin{remark}[Amplification of Data Errors and Divergence Beyond a Threshold]
\label{rem:divergence}
The exact solution of the collocation system depends on $f$ only through its nodal values, which a computation knows only to finite precision. Let Assumption~\ref{ass:unisolvence} hold, and let $\tilde u_N$ be the exact solution of \eqref{eq:square_system} with $g_0 = g_L = 0$ and interior data $f(x_m) + e_m$, $2 \le m \le N-1$. Theorem~\ref{thm:representation}, applied to these nodal data as noted after its proof, gives $\tilde u_N = u_N + \Lop^{-1}\sum_{m=2}^{N-1}e_m\ell_m^\Psi$, and Corollary~\ref{cor:error} gives
\begin{equation}\label{eq:perturbed_data}
\norm{u - \tilde u_N}_{L^\infty(I)} \le \norm{\Lop^{-1}}_\infty\Bigl[\bigl(1 + \Lambda_N^\Psi\bigr)\dist_\infty(f,\Psi_N) + \Lambda_N^\Psi\max_{m}|e_m|\Bigr].
\end{equation}
The first term bounds the error for exact data and the second the effect of the data errors. In $L^2(I)$ the amplification factor is exactly the constant $\Gamma_0$ of Definition~\ref{def:two_constants} with $M = n$: for $\vect e := (e_m)_{m=2}^{N-1} \in \mathbb R^n$, $\Lop^{-1}\sum_m e_m\ell_m^\Psi = \Phi\mat D\mat B_D^{-1}\vect e$, so that $\sup_{\vect e\ne0}\norm{\Lop^{-1}\sum_m e_m\ell_m^\Psi}_{L^2(I)}/\norm{\vect e}_2 = \Gamma_0$, and for the Poisson operator $\Gamma_0$ grows with the exponential rate $\kappa_\Lambda$ (Corollary~\ref{cor:lebesgue_rate}). For data rounded to a fixed precision the second term eventually dominates the bound, and beyond a threshold in $N$ the exact solution is observed to diverge (Section~\ref{sec:num_1d}).
\end{remark}

The same representation gives the discrete Green's function of the method.

\begin{corollary}[Discrete Green's Function]
\label{cor:greens}
Let Assumption~\ref{ass:unisolvence} hold, and let $\mat{G}_N := \mat{W}_N \mat{A}^{-1}$, with the synthesis matrix $\mat W_N$ of Section~\ref{sec:res_rates}, so that $\vect u = \mat G_N \vect f$ maps nodal data to nodal values of $u_N$. For $1 \le i \le N$ and $2 \le m \le N-1$,
\begin{equation}\label{eq:GN_entries}
\begin{aligned}
(\mat{G}_N)_{im} &= \int_0^L G(x_i, y)\, \ell_m^\Psi(y)\, dy, \\
|(\mat{G}_N)_{im}| &\le \norm{G(x_i,\cdot)}_{L^\infty}\, \norm{\ell_m^\Psi}_{L^1(I)}, \qquad
\sum_{m} |(\mat{G}_N)_{im}| \le \norm{\Lop^{-1}}_\infty\, \Lambda_N^\Psi .
\end{aligned}
\end{equation}
\end{corollary}
\begin{proof}[Proof of Corollary~\ref{cor:greens}]
Column $m$ of $\mat A^{-1}$ is the coefficient vector of the collocation solution for the data $\vect e_m$, whose test-space interpolant is $\ell_m^\Psi$; apply Theorem~\ref{thm:representation}, $G \ge 0$ and $\sum_m|\ell_m^\Psi| \le \Lambda_N^\Psi$.
\end{proof}

\begin{remark}[The Discrete Green's Function Does Not Resemble the Continuous Kernel]
\label{rem:mass}
By \eqref{eq:GN_entries}, row $i$ of $\mat G_N$ consists of the moments of the Lagrange functions $\ell_m^\Psi$ against the fixed kernel $G(x_i,\cdot)$. For a well-conditioned interpolation space $\norm{\ell_m^\Psi}_{L^1} \approx h$, and $\mat G_N$ inherits the continuous kernel, $(\mat G_N)_{im} \approx h\,G(x_i,x_m)$; here the masses are exponentially large, and the rows of $\mat G_N$ follow their profile for both operators (Section~\ref{sec:num_greens}). Neither the hat shape of $G_P$ nor the upstream decay of $G_{CD}$ is visible in $\mat G_N$: the apparent decay of a row away from the diagonal is the decay of the Lagrange-function mass towards the boundary nodes.
\end{remark}

\section{Potential Theory in the Chebyshev Variable: Proof of Theorem~\ref{thm:rates}}
\label{sec:potential}

The tool is the logarithmic potential theory of polynomial interpolation \cite[Chs.~12 and 13]{trefethen2019approximation}, which also explains the impossibility theorem \cite[Sec.~4]{platte2011impossibility}; for the general theory see \cite{saff1997logarithmic}. This section evaluates $U_\omega$ in closed form, states the estimate of the Lagrange functions of the trial frame on which the theorem rests (Proposition~\ref{prop:lagrange_norms}, proved in Appendix~\ref{app:rates}), and deduces Theorem~\ref{thm:rates}, Corollary~\ref{cor:exponent}, Propositions~\ref{prop:duality} and \ref{prop:critical} and the exponential ill-conditioning of Corollary~\ref{thm:poisson_illcond}.

\begin{lemma}[The Node Potential in Closed Form]
\label{lem:clausen}
For $z = \cos\theta$, $0 \le \theta \le \pi$,
\begin{equation}\label{eq:U_closed}
U_\omega(\cos\theta) = \log 2 + \frac{1}{\omega}\Bigl[\mathrm{Cl}_2(\theta-\theta_0) - \mathrm{Cl}_2(\theta+\theta_0) + \mathrm{Cl}_2(\pi-\theta-\theta_0) - \mathrm{Cl}_2(\pi-\theta+\theta_0)\Bigr].
\end{equation}
In particular
\begin{equation}\label{eq:U_values}
U_\omega(1) = \log 2 - \frac{\mathrm{Cl}_2(\pi-\omega)}{\omega}, \quad
U_\omega(c_0) = \log 2 + \frac{\mathrm{Cl}_2(2\omega)}{2\omega}, \quad
U_\omega(0) = \log 2 + \frac{\mathrm{Cl}_2(\omega)}{\omega},
\end{equation}
and $U_\omega(\cos\theta)$ is increasing on $[0,\pi/2]$ and invariant under $\theta \mapsto \pi-\theta$, so that $U_\omega$ is even in $z$ and decreasing in $|z|$: its maximum is $U_\omega(0)$, its minimum on the physical interval is $U_\omega(\pm c_0)$, and its minimum on $[-1,1]$ is $U_\omega(\pm 1)$.
\end{lemma}
\begin{proof}
\emph{Step 1: reduction to $(0,\theta_0)$.} By Lemma~\ref{lem:cl2}(iv) and the substitution $\varphi \mapsto \pi - \varphi$, which maps $(\pi - \theta_0, \pi)$ onto $(0, \theta_0)$ and $\cos\varphi$ to $-\cos\varphi$,
\[
\begin{aligned}
\int_{\theta_0}^{\pi-\theta_0}\log|\cos\theta - \cos\varphi|\,d\varphi &= -\pi\log 2 \\
&\quad - \int_0^{\theta_0}\bigl(\log|\cos\theta - \cos\varphi| + \log|\cos\theta + \cos\varphi|\bigr)\,d\varphi .
\end{aligned}
\]

\emph{Step 2: evaluation.} By
\[
\cos\theta - \cos\varphi = -2\sin\tfrac{\theta+\varphi}{2}\,\sin\tfrac{\theta-\varphi}{2}, \qquad
\cos\theta + \cos\varphi = 2\sin\tfrac{\pi-\theta-\varphi}{2}\,\sin\tfrac{\pi-\theta+\varphi}{2},
\]
each of the two logarithms equals $-\log 2$ plus two terms of the form $\log|2\sin(t/2)|$. With $\mathrm{Cl}_2(b) - \mathrm{Cl}_2(a) = -\int_a^b\log|2\sin(t/2)|\,dt$, the evenness of the integrand, and the substitutions $t = \theta \pm \varphi$ and $t = \pi - \theta \pm \varphi$,
\begin{align*}
\int_0^{\theta_0}\log\bigl|2\sin\tfrac{\theta+\varphi}{2}\bigr|\,d\varphi &= \mathrm{Cl}_2(\theta) - \mathrm{Cl}_2(\theta+\theta_0), \\
\int_0^{\theta_0}\log\bigl|2\sin\tfrac{\theta-\varphi}{2}\bigr|\,d\varphi &= \mathrm{Cl}_2(\theta-\theta_0) - \mathrm{Cl}_2(\theta), \\
\int_0^{\theta_0}\log\bigl|2\sin\tfrac{\pi-\theta-\varphi}{2}\bigr|\,d\varphi &= \mathrm{Cl}_2(\pi-\theta-\theta_0) - \mathrm{Cl}_2(\pi-\theta), \\
\int_0^{\theta_0}\log\bigl|2\sin\tfrac{\pi-\theta+\varphi}{2}\bigr|\,d\varphi &= \mathrm{Cl}_2(\pi-\theta) - \mathrm{Cl}_2(\pi-\theta+\theta_0).
\end{align*}
The four right-hand sides add up to the bracket $[\cdots]$ of \eqref{eq:U_closed}, so the integral in Step~1 equals $-\pi\log 2 + 2\theta_0\log 2 - [\cdots] = -\omega\log 2 - [\cdots]$, and $U_\omega(\cos\theta) = -\omega^{-1}$ times this integral is \eqref{eq:U_closed}.

\emph{Step 3: special values.} At $\theta = 0$ the bracket is $-2\mathrm{Cl}_2(\theta_0) + \mathrm{Cl}_2(\pi-\theta_0) - \mathrm{Cl}_2(\pi+\theta_0) = 2\bigl[\mathrm{Cl}_2(\pi-\theta_0) - \mathrm{Cl}_2(\theta_0)\bigr] = -\mathrm{Cl}_2(2\theta_0)$ by Lemma~\ref{lem:cl2}(i) and (ii), and $2\theta_0 = \pi-\omega$. At $\theta = \theta_0$ it is $-\mathrm{Cl}_2(2\theta_0) + \mathrm{Cl}_2(\pi-2\theta_0) = \mathrm{Cl}_2(\omega) - \mathrm{Cl}_2(\pi-\omega) = \tfrac12\mathrm{Cl}_2(2\omega)$ by (i), $\mathrm{Cl}_2(0) = 0$ and (ii). At $\theta = \pi/2$ it is $2\bigl[\mathrm{Cl}_2(\omega/2) - \mathrm{Cl}_2(\pi-\omega/2)\bigr] = \mathrm{Cl}_2(\omega)$ by (ii). This proves \eqref{eq:U_values}.

\emph{Step 4: monotonicity.} Differentiating \eqref{eq:U_closed} with $\mathrm{Cl}_2'(x) = -\log|2\sin(x/2)|$ and combining the four logarithms with $2\sin a\cos a = \sin 2a$ gives, for $0 < \theta < \pi/2$, $\theta \ne \theta_0$,
\[
\frac{d}{d\theta}U_\omega(\cos\theta) = \frac{1}{\omega}\log\frac{\sin(\theta+\theta_0)}{|\sin(\theta-\theta_0)|} .
\]
For such $\theta$, both $|\theta - \theta_0|$ and $\min\{\theta+\theta_0, \pi-\theta-\theta_0\}$ lie in $[0,\pi/2]$, and the second exceeds the first because $\theta, \theta_0 > 0$ and $2\max\{\theta,\theta_0\} < \pi$; as $\sin$ is increasing on $[0,\pi/2]$ and $\sin(\theta+\theta_0) = \sin(\pi-\theta-\theta_0)$, the derivative is positive. The function $\theta \mapsto U_\omega(\cos\theta)$ is continuous (the derivative has an integrable singularity at $\theta_0$) and hence increasing on $[0,\pi/2]$, and \eqref{eq:U_closed} is invariant under $\theta \mapsto \pi-\theta$; this proves the remaining claims.
\end{proof}

\begin{lemma}[Unisolvence in the Trial Frame]
\label{lem:lebesgue_condition_link}
Interpolation at the $N$ nodes is unisolvent in $V_N$; equivalently, the synthesis matrix $\mat W_N$ of Section~\ref{sec:res_rates} is nonsingular.
\end{lemma}
\begin{proof}
By Lemma~\ref{lem:chebyshev_variable}, $V_N = \sqrt{2/\tilde L}\,\sin\theta\cdot\mathbb P_{N-1}(z)$ with $\theta = \pi(x+\delta)/\tilde L$, $z = \cos\theta$, and the nodes have distinct values $z_m$ and $\sin\theta_m \ne 0$; interpolation at the $N$ nodes is therefore unisolvent in $V_N$, and $\mat W_N$ is nonsingular.
\end{proof}

The proof of Theorem~\ref{thm:rates} rests on the following estimate, in which $\ell_m^V \in V_N$ denotes the Lagrange function of $V_N$ at the node $x_m$, $1 \le m \le N$, which exists by Lemma~\ref{lem:lebesgue_condition_link}, and the notation $z_m = \cos\theta_m$, $c_0 = \cos\theta_0$ is that of Section~\ref{sec:variable}, and $\Delta := \omega/(N-1)$ is the spacing of the angles $\theta_m$. It is proved in Appendix~\ref{app:rates}.

\begin{proposition}[Norms of the Lagrange Functions of $V_N$]
\label{prop:lagrange_norms}
For $N \ge N_0$, with $N_0$ depending only on $\theta_0$, and uniformly in $1 \le m \le N$, with implied constants depending only on $\delta$ and $L$,
\begin{equation}\label{eq:lagrange_norms}
\begin{aligned}
N^{-1}\log\norm{\ell_m^V}_{L^2(\tilde I)} &= U_\omega(z_m) - U_\omega(1) + O(\log N/N), \\
N^{-1}\log\norm{\ell_m^V}_{L^2(I)} &= U_\omega(z_m) - U_\omega(c_0) + O(\log N/N), \\
N^{-1}\log\norm{\ell_m^V}_{L^\infty(I)} &= U_\omega(z_m) - U_\omega(c_0) + O(\log N/N).
\end{aligned}
\end{equation}
\end{proposition}

\begin{proof}[Proof of Theorem~\ref{thm:rates}]
Let $m^\ast$ be an index with $|\theta_{m^\ast} - \pi/2| \le \Delta/2$, so that $|U_\omega(z_{m^\ast}) - U_\omega(0)| \le C\Delta\log(1/\Delta)$ by \eqref{eq:U_modulus}; since $U_\omega(z_m) \le U_\omega(0)$ for every $m$ by Lemma~\ref{lem:clausen}, Proposition~\ref{prop:lagrange_norms} gives
\begin{gather*}
N^{-1}\log\max_m\norm{\ell_m^V}_{L^\infty(I)} = U_\omega(0) - U_\omega(c_0) + O(\log N/N), \\
N^{-1}\log\max_m\norm{\ell_m^V}_{L^2(\tilde I)} = U_\omega(0) - U_\omega(1) + O(\log N/N).
\end{gather*}
For the first limit in \eqref{eq:rates}, $\max_m\norm{\ell_m^V}_{L^\infty(I)} \le \Lambda_N \le N\max_m\norm{\ell_m^V}_{L^\infty(I)}$ by the definition of the Lebesgue constant. For the second, $\mat W_N$ maps coefficient vectors to nodal values, so the $m$-th column of $\mat W_N^{-1}$ is the coefficient vector of $\ell_m^V$, whose Euclidean norm is $\norm{\ell_m^V}_{L^2(\tilde I)}$ by \eqref{eq:chebyshev_norms}; hence $\max_m\norm{\ell_m^V}_{L^2(\tilde I)} \le \norm{\mat W_N^{-1}}_2 \le \norm{\mat W_N^{-1}}_F \le \sqrt N\max_m\norm{\ell_m^V}_{L^2(\tilde I)}$, and $\sigma_{\min}(\mat W_N) = 1/\norm{\mat W_N^{-1}}_2$. The closed forms in \eqref{eq:rates} are \eqref{eq:U_values}.
\end{proof}

\begin{proof}[Proof of Corollary~\ref{cor:exponent}]
Here $\ell^\ast = \ell^V_{m^\ast}$ with $m^\ast$ as in the proof of Theorem~\ref{thm:rates}: the condition $|\theta_{m^\ast} - \pi/2| \le \Delta/2$, which is $|x_{m^\ast} - L/2| \le h/2$, holds for the index $m^\ast$ of the statement. Proposition~\ref{prop:lagrange_norms} and \eqref{eq:U_modulus} give $N^{-1}\log\norm{\ell^\ast}_{L^2(\tilde I)} = U_\omega(0) - U_\omega(1) + O(\log N/N) = \kappa_\sigma + O(\log N/N)$ and $N^{-1}\log\bigl(\norm{\ell^\ast}_{L^2(I)}/\norm{\ell^\ast}_{L^2(\tilde I)}\bigr) = U_\omega(1) - U_\omega(c_0) + O(\log N/N) = -(\kappa_\sigma - \kappa_\Lambda) + O(\log N/N)$. The ratio of the two logarithms therefore tends to $(\kappa_\sigma - \kappa_\Lambda)/\kappa_\sigma = a(\omega)$. Since $\ell^\ast$ has the nodal data $\vect e_{m^\ast}$, $t(\ell^\ast) = 1/\norm{\ell^\ast}_{L^2(\tilde I)}$, and the first statement of the corollary is the first of these rates. The values and limits of $a(\omega)$ quoted in Section~\ref{sec:res_rates} follow from Lemma~\ref{lem:cl2}: $\mathrm{Cl}_2(\pi) = 0$ gives $a(\pi/2) = 1/2$; $\mathrm{Cl}_2(\pi/3) = \tfrac32\mathrm{Cl}_2(2\pi/3)$ gives $a(\pi/3) = 3/5$ and $a(2\pi/3) = 2/5$; and (iii) gives the limits: as $\omega \to 0$ both numerator and denominator are $\omega(1 - \log\omega) + O(\omega)$, and as $\omega \to \pi$ the numerator is $(\pi-\omega)\log 2 + O((\pi-\omega)^3)$ while the denominator is $(\pi-\omega)(1 + \log 2 - \log(\pi-\omega)) + O((\pi-\omega)^3)$.
\end{proof}

For comparison, equispaced nodes on the full interval $[-1,1]$ have a potential drop of $\log 2$ between the centre and the endpoints \cite[Ch.~13]{trefethen2019approximation}, which is the classical growth rate $2^N$ of the Lebesgue constant of equispaced polynomial interpolation \cite[Thm.~2]{trefethen1991two}; the frame replaces the endpoints of the node interval by those of the larger interval and produces two rates instead of one.

\begin{proof}[Proof of Proposition~\ref{prop:duality}]
Replace $\omega$ by $\pi-\omega$ in \eqref{eq:exponent} and use $\mathrm{Cl}_2(2\pi - 2\omega) = -\mathrm{Cl}_2(2\omega)$ (Lemma~\ref{lem:cl2}(i)); the two numerators add up to the common denominator.
\end{proof}

\begin{proof}[Proof of Proposition~\ref{prop:critical}]
Since $s(\omega) + s(\pi-\omega) = 1$ and $a(\omega) + a(\pi-\omega) = 1$ (Proposition~\ref{prop:duality}), the function $a - s$ is odd about $\pi/2$ and vanishes there, so it suffices to prove $a(\omega) > s(\omega)$ for $\omega \in (\pi/2,\pi)$. Put $x := \pi - \omega \in (0,\pi/2)$. By Lemma~\ref{lem:cl2}(i), $\mathrm{Cl}_2(2\omega) = \mathrm{Cl}_2(2\pi - 2x) = -\mathrm{Cl}_2(2x)$, so \eqref{eq:exponent} reads
\[
a(\omega) = \frac{\mathrm{Cl}_2(x) - \tfrac12\mathrm{Cl}_2(2x)}{\mathrm{Cl}_2(x) + \mathrm{Cl}_2(\pi - x)}, \qquad s(\omega) = \frac{x}{\pi} .
\]
The Clausen function is positive on $(0,\pi)$: its derivative $-\log(2\sin(t/2))$ is positive on $(0,\pi/3)$ and negative on $(\pi/3,\pi)$, and $\mathrm{Cl}_2(0) = \mathrm{Cl}_2(\pi) = 0$. Hence the denominator is positive, and $a(\omega) > s(\omega)$ is equivalent to $\pi\bigl[\mathrm{Cl}_2(x) - \tfrac12\mathrm{Cl}_2(2x)\bigr] > x\bigl[\mathrm{Cl}_2(x) + \mathrm{Cl}_2(\pi-x)\bigr]$. By the duplication formula of Lemma~\ref{lem:cl2}(ii), $\tfrac12\mathrm{Cl}_2(2x) = \mathrm{Cl}_2(x) - \mathrm{Cl}_2(\pi - x)$, so the left-hand side equals $\pi\,\mathrm{Cl}_2(\pi - x)$, and the inequality becomes
\[
D(x) := (\pi - x)\,\mathrm{Cl}_2(\pi - x) - x\,\mathrm{Cl}_2(x) > 0, \qquad 0 < x < \pi/2 .
\]
The function $D$ is continuous on $[0,\pi/2]$ with $D(0) = \pi\,\mathrm{Cl}_2(\pi) = 0$ and $D(\pi/2) = 0$. On $(0,\pi/2)$ it is twice differentiable, and with $\mathrm{Cl}_2'(y) = -\log(2\sin(y/2))$ one finds $D'(x) = -\mathrm{Cl}_2(\pi-x) - \mathrm{Cl}_2(x) + (\pi - x)\log(2\cos(x/2)) + x\log(2\sin(x/2))$ and
\begin{gather*}
D''(x) = 2\log\tan\frac x2 + \frac12\Bigl[x\cot\frac x2 - (\pi - x)\tan\frac x2\Bigr] = F(x/2), \\
F(u) := 2\log\tan u + u\cot u - \Bigl(\frac\pi2 - u\Bigr)\tan u .
\end{gather*}
Here $F(\pi/4) = 0$ and, for $0 < u < \pi/4$,
\begin{gather*}
F'(u) = \frac{3}{\sin u\cos u} - \frac{u}{\sin^2u} - \frac{\pi/2 - u}{\cos^2u} = \frac{G(2u)}{\sin^2u\cos^2u}, \\
G(v) := \frac32\sin v - \frac v2\cos v - \frac\pi4(1 - \cos v),
\end{gather*}
because $3\sin u\cos u = \tfrac32\sin 2u$ and $-u\cos^2u - (\pi/2 - u)\sin^2u = -u\cos 2u - \tfrac\pi4(1 - \cos 2u)$. Now $G(0) = 0$, $G'(v) = \cos v + (v/2 - \pi/4)\sin v$ with $G'(\pi/2) = 0$, and $G''(v) = -\tfrac12\sin v + (v/2 - \pi/4)\cos v < 0$ on $(0,\pi/2)$; hence $G' > 0$ and $G > 0$ on $(0,\pi/2)$, so $F' > 0$ on $(0,\pi/4)$, $F < F(\pi/4) = 0$ there, and $D'' < 0$ on $(0,\pi/2)$. A continuous function that is strictly concave on $(0,\pi/2)$ and vanishes at both ends of $[0,\pi/2]$ is positive inside, which proves $D > 0$ and the proposition.
\end{proof}

\subsection{Exponential Ill-Conditioning}

Theorem~\ref{thm:rates} determines the baseline ill-conditioning of the method, through the factorization $\mat{A}_P = \mat{W}_N \mathbf{\Lambda}$ of the bulk collocation matrix of the Poisson operator. The functions $\{w_j\}_{j\le N}$ restricted to $I$ are odd Fourier modes of the doubled extended domain restricted to a subinterval, and the singular values of such restricted Fourier systems decay exponentially beyond a transition index proportional to $NL/\tilde L$, at a rate that increases with the redundancy $\tilde L/L$ \cite{slepian1978prolate, huybrechs2010fourier, adcock2014numerical}; for the present frame the smallest singular value decays at the exponential rate $\kappa_\sigma$ in $N$ (Theorem~\ref{thm:rates}).

\begin{lemma}[Exponential Decay of the Smallest Singular Value]
\label{lem:sigma_min}
Let $\kappa_\sigma(\omega) > 0$ be the constant of Theorem~\ref{thm:rates}, $\omega = \pi L/\tilde L$. For every $\alpha \in (0, \kappa_\sigma(\omega))$ there is $C_1 = C_1(\alpha, \delta, L) > 0$ such that $\sigma_{\min}(\mat{W}_N) \le C_1 e^{-\alpha N}$ for all $N \ge 2$.
\end{lemma}
\begin{proof}
By Theorem~\ref{thm:rates}, $N^{-1}\log\bigl(1/\sigma_{\min}(\mat W_N)\bigr) \to \kappa_\sigma(\omega) > \alpha$, so there is $N_1$ with $\sigma_{\min}(\mat W_N) \le e^{-\alpha N}$ for all $N \ge N_1$. Take $C_1 := \max\bigl\{1, \max_{2 \le N < N_1}\sigma_{\min}(\mat W_N)e^{\alpha N}\bigr\}$.
\end{proof}

\begin{proof}[Proof of Corollary~\ref{thm:poisson_illcond}]
The matrix $\mat A_P = \mat W_N\mathbf\Lambda$ is nonsingular, since $\mat W_N$ is (Lemma~\ref{lem:lebesgue_condition_link}) and $\lambda_j > 0$, and $\sigma_{\min}(\mat{A}_P) \le \lambda_N\sigma_{\min}(\mat{W}_N)$ with $\lambda_N = (N\pi/\tilde L)^2 = \norm{\mathbf\Lambda}_2$: for a unit vector $\vect v$ with $\norm{\mat W_N\vect v}_2 = \sigma_{\min}(\mat W_N)$, the vector $\vect x := \mathbf\Lambda^{-1}\vect v$ satisfies $\norm{\mat A_P\vect x}_2 = \sigma_{\min}(\mat W_N)$ and $\norm{\vect x}_2 \ge 1/\lambda_N$, so $\sigma_{\min}(\mat A_P) \le \norm{\mat A_P\vect x}_2/\norm{\vect x}_2 \le \lambda_N\sigma_{\min}(\mat W_N)$. Since $\theta_0 = \pi\delta/\tilde L \in (0,\pi/2)$ and $\pi(x+\delta)/\tilde L \in [\theta_0, \pi-\theta_0]$ for $x \in \bar I$, $w_1(x) = \sqrt{2/\tilde L}\sin\bigl(\pi(x+\delta)/\tilde L\bigr) \ge \sqrt{2/\tilde L}\sin\theta_0$ on $\bar I$, so $\sigma_{\max}(\mat{A}_P) \ge \norm{\mat{A}_P \vect e_1}_2 = \lambda_1 \norm{\mat{W}_N\vect e_1}_2 \ge \lambda_1\sqrt{2N/\tilde L}\,\sin\theta_0$. With Lemma~\ref{lem:sigma_min},
$\kappa_2(\mat{A}_P) \ge \lambda_1\sqrt{2/\tilde L}\,\sin\theta_0\,\sqrt N\,e^{\alpha N}/(C_1\lambda_N)$, and $\lambda_1/\lambda_N = N^{-2}$; the claim holds with $C := \sqrt{2/\tilde L}\,\sin\theta_0/C_1$.
\end{proof}

We do not prove the corresponding bound for the convection--diffusion matrix; in the computations of Section~\ref{sec:num_1d} the test-space Lebesgue constants of the two operators grow at nearly the same rate.

\section{Floating-Point Arithmetic: Proof of Theorem~\ref{thm:two_constant}}
\label{sec:fp}

This section provides the exact identity that governs the computed solution and proves Theorem~\ref{thm:two_constant}, Corollary~\ref{cor:regimes_two_constant}, Proposition~\ref{prop:classical}(i), and Theorem~\ref{thm:competitor} with Corollary~\ref{cor:unconditional}; Proposition~\ref{prop:apriori} is proved in Appendix~\ref{app:apriori}. The notation is that of Section~\ref{sec:constants}.

\subsection{Backward-Stable Solvers}

\begin{proposition}[Computed Solution as Exact Collocation of Perturbed Data]
\label{prop:backward}
Let Assumption~\ref{ass:unisolvence} hold, let $\mat B$ be the square interior collocation matrix of the recombined basis ($M = n$), let $\hat{\vect c} \in \mathbb R^n$ satisfy $(\mat B + \mat E)\hat{\vect c} = \vect f$ for some perturbation $\mat E$, $\vect f = (f(x_m))_{m=2}^{N-1}$, and put $\hat u_N := \sum_j \hat c_j\varphi_j \in \Phi_N$. Then $\hat u_N$ is the exact interpolatory collocation solution for the nodal data $\vect f - \mat E\hat{\vect c}$:
\begin{equation}\label{eq:backward}
\hat u_N = \Lop^{-1}\Pop_N\bigl(\vect f - \mat E\hat{\vect c}\bigr), \qquad u - \hat u_N = \Lop^{-1}\bigl(f - \Pop_N(\vect f - \mat E\hat{\vect c})\bigr).
\end{equation}
\end{proposition}
\begin{proof}
For $2 \le m \le N-1$, $(\mat B\hat{\vect c})_m = (\Lop\hat u_N)(x_m) = f(x_m) - (\mat E\hat{\vect c})_m$, so $\Lop\hat u_N \in \Psi_N$ interpolates the data $\vect f - \mat E\hat{\vect c}$, and $\hat u_N \in \Phi_N$ because the $\varphi_j$ vanish at $0$ and $L$. Theorem~\ref{thm:representation}, applied to these data as noted after its proof, gives the first identity, and $u = \Lop^{-1}f$ the second.
\end{proof}

If the full square system \eqref{eq:square_system} is solved instead, the perturbation of the two boundary rows adds to $\hat u_N$ the collocation solution with zero interior data and boundary values $-(\mat E\hat{\vect c})_1$, $-(\mat E\hat{\vect c})_N$; that function satisfies the differential equation only at the nodes, so Lemma~\ref{lem:dirichlet_1d} does not control it. The floating-point theory of this section is therefore stated for the recombined system, in which the boundary conditions hold exactly by construction; the experiments of Sections~\ref{sec:num_solvers} and \ref{sec:num_bound} use this system, and Section~\ref{sec:num_solvers} also reports the full system for comparison.

Both direct solvers used in this paper produce a perturbation $\mat E$ as in Proposition~\ref{prop:backward}. Let $\eps_{\mathrm{mach}}$ denote the unit roundoff ($2^{-53}$ in IEEE double precision) and $c$ the small integer constant of \cite[Eq.~(3.8)]{higham2002accuracy}. If the nonsingular system $\mat B\vect c = \vect f$, $\vect f \ne 0$, is solved by Householder QR factorization and $c\,n^2\eps_{\mathrm{mach}} \le 1/4$, then $(\mat B + \mat E_0)\hat{\vect c} = \vect f + \Delta\vect f$ with $\norm{\mat E_0}_F \le \tilde\gamma_{n^2}\norm{\mat B}_F$ and $\norm{\Delta\vect f}_2 \le \tilde\gamma_{n^2}\norm{\vect f}_2$, where $\tilde\gamma_{n^2} = c\,n^2\eps_{\mathrm{mach}}/(1 - c\,n^2\eps_{\mathrm{mach}}) \le 1/3$ \cite[Thm.~19.5]{higham2002accuracy}. Then $\hat{\vect c} \ne 0$ and $(1 - \tilde\gamma_{n^2})\norm{\vect f}_2 \le (1 + \tilde\gamma_{n^2})\norm{\mat B}_F\norm{\hat{\vect c}}_2$, so that $\mat E := \mat E_0 - \Delta\vect f\,\hat{\vect c}^T/\norm{\hat{\vect c}}_2^2$ satisfies $(\mat B + \mat E)\hat{\vect c} = \vect f$ and $\norm{\mat E}_2 \le 2\tilde\gamma_{n^2}\norm{\mat B}_F/(1 - \tilde\gamma_{n^2}) \le 4c\,n^2\eps_{\mathrm{mach}}\norm{\mat B}_F$. If it is solved by Gaussian elimination with partial pivoting, with computed factors $\hat{\mat L}$ and $\hat{\mat U}$ of $\mat P\mat B$ for a permutation matrix $\mat P$, and $3n\eps_{\mathrm{mach}} \le 1/2$, then $(\mat B + \mat E)\hat{\vect c} = \vect f$ with $|\mat E| \le 6n\eps_{\mathrm{mach}}\mat P^T|\hat{\mat L}||\hat{\mat U}|$ entrywise \cite[Sec.~9.3 and Thm.~9.4]{higham2002accuracy}, and, since $|\hat{\mat L}| \le 1$ entrywise and the entries of $\hat{\mat U}$ are bounded by $\max_{i,j}|(\mat B)_{ij}| \le \norm{\mat B}_\infty$ times the growth factor of the elimination, $\norm{\mat E}_\infty$ is at most $6n^3\eps_{\mathrm{mach}}\norm{\mat B}_\infty$ times that growth factor. For the normalized system $\mat B_D\hat{\vect y} = \vect f$ the same bounds hold with $\mat B_D$ in place of $\mat B$; since the columns of $\mat B_D$ have unit norm, Householder QR then has backward error $\norm{\mat E}_2 \le 4c\,n^{5/2}\eps_{\mathrm{mach}}$, however ill-conditioned $\mat B_D$ is. By Proposition~\ref{prop:backward} the computed solution is the exact interpolatory solution of nearby data; it differs from $u_N = \Lop^{-1}\Pop_N\vect f$ by $\Lop^{-1}\Pop_N(\mat E\hat{\vect c})$, which Lebesgue's bound controls only by $\norm{\Lop^{-1}}_\infty\Lambda_N^\Psi\norm{\mat E\hat{\vect c}}_\infty$, and $\Lambda_N^\Psi$ grows exponentially (Corollary~\ref{cor:lebesgue_rate} and Section~\ref{sec:num_1d}). Yet $\hat{\vect c}$ remains bounded and $\hat u_N$ is accurate to a level that decreases with $\delta$, for all $N$ computed (Sections~\ref{sec:num_1d} and \ref{sec:num_solvers}); the next subsection proves a bound that accounts for this.

\subsection{The Two-Constant Bound}
\label{sec:two_constant}

The splitting of the coefficient space in Definition~\ref{def:two_constants}, the argument of \cite[Thm.~3.7]{adcock2020frames2}, is the whole proof of the two-constant bound.

\begin{proof}[Proof of Theorem~\ref{thm:two_constant}]
$\hat u_N - u_{\vect y} = \Phi\mat D\vect d$. Split $\vect d = \vect d_t + \vect d_t^{\perp}$ with $\vect d_t \in \mathcal V_t$, $\vect d_t^\perp \in \mathcal V_t^\perp$. By definition, $\norm{\Phi\mat D\vect d_t}_{L^2(I)} \le \gamma_t\norm{\vect d_t}_2 \le \gamma_t\norm{\vect d}_2$ and $\norm{\Phi\mat D\vect d_t^\perp}_{L^2(I)} \le \Gamma_t\norm{\mat B_D\vect d_t^\perp}_2 \le \Gamma_t(\norm{\mat B_D\vect d}_2 + \norm{\mat B_D\vect d_t}_2)$, where $\norm{\mat B_D\vect d_t}_2 \le t\norm{\vect d_t}_2 \le t\norm{\vect d}_2$ because $\vect d_t$ lies in the span of singular vectors with $\sigma_i \le t$. Finally $\mat B_D\vect d = (\mat B_D\hat{\vect y} - \vect f) + (\vect f - \mat B_D\vect y)$, so $\norm{\mat B_D\vect d}_2 \le \hat\rho + \norm{\vect f - \mat B_D\vect y}_2$. Adding $\norm{u - u_{\vect y}}_{L^2(I)}$ completes the proof.
\end{proof}

\begin{proof}[Proof of Corollary~\ref{cor:regimes_two_constant}]
For (i), $u_N = \Phi\mat D\hat{\vect y}$ with $\mat B_D\hat{\vect y} = \vect f$ by Theorem~\ref{thm:representation} and the definition of $\mat B$, so $\hat\rho = 0$, and $\mathcal V_0 = \{0\}$ because $\mat B_D$ is nonsingular (Assumption~\ref{ass:unisolvence}); with $t = 0$, \eqref{eq:two_constant_bound} reduces to the stated bound, and $\vect v_n \in \mathcal V_0^\perp$ with $\norm{\mat B_D\vect v_n}_2 = \sigma_n$ gives the lower bound on $\Gamma_0$. For (ii), $u_N^\eps = \Phi\mat D\hat{\vect y}$ with $\hat{\vect y} := \mat D^{-1}\vect c_\eps$, so $\hat\rho = \norm{\vect f - \mat B\vect c_\eps}_2$; write $r := \norm{\vect f - \mat B_D\vect y}_2$; Lemma~\ref{lem:tsvd} with $\vect z = \mat D\vect y$ gives $\hat\rho \le r + \eps\norm{\vect y}_2$ and $\norm{\hat{\vect y}}_2 = \norm{\mat D^{-1}\vect c_\eps}_2 \le \norm{\vect y}_2 + \eps^{-1}r$, hence $\norm{\vect d}_2 \le 2\norm{\vect y}_2 + \eps^{-1}r$; substituting into \eqref{eq:two_constant_bound} with $t = \eps$ gives $\Gamma_\eps(r + r + \eps\norm{\vect y}_2 + 2\eps\norm{\vect y}_2 + r) + \gamma_\eps(2\norm{\vect y}_2 + \eps^{-1}r) \le (3\Gamma_\eps + 2\gamma_\eps\eps^{-1})(r + \eps\norm{\vect y}_2)$. For (iii), $(\mat B_D + \mat E)\vect d = (\vect f - \mat B_D\vect y) - \mat E\vect y$ gives $\norm{\vect d}_2 \le (K/\eta)(r + \eta\norm{\vect y}_2)$, and $\hat\rho = \norm{\mat E\hat{\vect y}}_2 \le \eta(\norm{\vect y}_2 + \norm{\vect d}_2)$; substituting into \eqref{eq:two_constant_bound} with $t = \eta$ gives $\Gamma_\eta(r + \eta\norm{\vect y}_2 + 2\eta\norm{\vect d}_2) + \gamma_\eta\norm{\vect d}_2 \le [(1 + 2K)\Gamma_\eta + K\gamma_\eta\eta^{-1}](r + \eta\norm{\vect y}_2)$.
\end{proof}

The only hypothesis of Corollary~\ref{cor:regimes_two_constant}(iii) beyond backward stability is the lower bound on $\sigma_{\min}(\mat B_D + \mat E)$, which fails only if the rounding perturbation happens to leave $\mat B_D + \mat E$ almost exactly singular, the worst case noted in \cite[Rem.~19]{herremans2025function}. For a Gaussian perturbation $\mat E$ of variance $\varsigma^2$ in place of the rounding perturbation, the probability that $\sigma_{\min}(\mat B_D + \mat E) < \eta/K$ is at most $2.35\sqrt n\,\eta/(K\varsigma)$ by \cite[Thm.~3.3]{sankar2006smoothed}; rounding errors are not Gaussian, so this is only an analogy, whereas the following certification is exact: for Gaussian elimination with partial pivoting and $2n\eps_{\mathrm{mach}} \le 1$, the proof of \cite[Thm.~9.4]{higham2002accuracy} gives $\mat B_D + \mat E = \mat P^T(\hat{\mat L} + \Delta\mat L)(\hat{\mat U} + \Delta\mat U)$ with $|\Delta\mat L| \le 2n\eps_{\mathrm{mach}}|\hat{\mat L}|$ and $|\Delta\mat U| \le 2n\eps_{\mathrm{mach}}|\hat{\mat U}|$, so that\[\norm{(\mat B_D + \mat E)^{-1}}_2 \le \frac{\norm{\hat{\mat L}^{-1}}_2\,\norm{\hat{\mat U}^{-1}}_2}{\bigl(1 - 2n\eps_{\mathrm{mach}}\norm{|\hat{\mat L}^{-1}||\hat{\mat L}|}_2\bigr)\bigl(1 - 2n\eps_{\mathrm{mach}}\norm{|\hat{\mat U}^{-1}||\hat{\mat U}|}_2\bigr)}\]whenever both factors in the denominator are positive, a bound computable from the computed factors. Alternatively, $\norm{\vect d}_2 \le \norm{\hat{\vect y}}_2 + \norm{\vect y}_2$ can simply be read off from the computed vector, in which case Theorem~\ref{thm:two_constant}, with the residual $\hat\rho$ evaluated with a rigorous bound on its rounding error, is an a posteriori error bound with no hypothesis on the solver.

\begin{proof}[Proof of Proposition~\ref{prop:classical}(i)]
For $\vect w \in \mathbb R^n$ put $v := \Phi\mat D\vect w \in \Phi_N$ and $\psi := \Lop v \in \Psi_N$, whose values at the interior nodes are the entries of $\mat B_D\vect w$. By Assumption~\ref{ass:unisolvence}, $\psi = \sum_m(\mat B_D\vect w)_m\ell_m^\Psi$, so $\norm{\psi}_{L^2(I)} \le \sum_m|(\mat B_D\vect w)_m|\,\norm{\ell_m^\Psi}_{L^2(I)} \le \sqrt n\,\norm{\mat B_D\vect w}_2\,\Lambda^\Psi_{N,2}$ by the Cauchy--Schwarz inequality. Since $v \in C^2(\bar I)$, $v(0) = v(L) = 0$ and $\Lop v = \psi$, Lemma~\ref{lem:dirichlet_box}(iii) with $d = 1$ gives $\norm{v}_{L^2(I)} \le (L/\pi)^2\norm{\psi}_{L^2(I)}$. For $\vect w \in \mathcal V_t$ one has $\norm{\mat B_D\vect w}_2 \le t\norm{\vect w}_2$, which gives the bound on $\gamma_t$; the bound on $\Gamma_t$ is the same inequality without this step, and \eqref{eq:classical_lu} follows from \eqref{eq:lu_bound} as in Proposition~\ref{prop:apriori}. Finally $\norm{\ell_m^\Psi}_{L^2(I)} \le \sqrt L\norm{\ell_m^\Psi}_{L^\infty(I)} \le \sqrt L\,\Lambda_N^\Psi$.
\end{proof}

Proposition~\ref{prop:classical} is the bound obtained without the singular-value splitting: roundoff amplified by a Lebesgue-type constant of the test space. Section~\ref{sec:num_bound} compares its factor $(L/\pi)^2\sqrt n\,\Lambda^\Psi_{N,2}$ with the computed factor $\Gamma_\eta + \gamma_\eta\eta^{-1}$.

\subsection{Competitors of Moderate Coefficient Norm: Proof of Theorem~\ref{thm:competitor}}
\label{sec:competitor}

The competitor is built from the exact solution $u$ rather than from the data: an element of $\Phi_N$ that approximates $u$ together with two derivatives on $I$ has an image under $\Lop$ that approximates $f$ on $I$, and the boundary conditions are corrected by the two lowest modes. The construction takes place in the Chebyshev variable, where $\Phi_N \subset V_N = \sqrt{2/\tilde L}\,\sin\theta\cdot\mathbb P_{N-1}(z)$, and the size of the competitor on the extended interval is controlled by the growth of Chebyshev polynomials of $[-c_0,c_0]$ on the larger interval $[-1,1]$, which is governed by the ellipse $E_{\rho^\ast}$ through $\pm1$. Throughout, $T_\nu$ denotes the Chebyshev polynomial of the first kind, $T_\nu(\tfrac12(s + s^{-1})) = \tfrac12(s^\nu + s^{-\nu})$, and $\mathcal E_R := \{\tfrac12(s + s^{-1}) : 1 \le |s| \le R\}$, $R \ge 1$, is the closed Bernstein ellipse of $[-1,1]$ with parameter $R$, so that $\mathcal E_1 = [-1,1]$ and $E_\rho = c_0\,\mathrm{int}\,\mathcal E_\rho$.

\begin{lemma}[Chebyshev Polynomials on Bernstein Ellipses]
\label{lem:cheb_ellipse}
For $R \ge 1$, $\nu \ge 1$ and $w \in \mathcal E_R$,
\begin{equation}\label{eq:cheb_ellipse}
|T_\nu(w)| \le R^\nu, \qquad |T_\nu'(w)| \le 4e\,\nu^2R^{\nu-1}, \qquad |T_\nu''(w)| \le 32e\,\nu^4R^{\nu-2}.
\end{equation}
\end{lemma}
\begin{proof}
For $R_1 \ge 1$ and $w = \tfrac12(s + s^{-1})$ with $1 \le |s| \le R_1$, $|T_\nu(w)| = \tfrac12|s^\nu + s^{-\nu}| \le R_1^\nu$; with $R_1 = R$ this is the first bound. Put $R' := R(1 + 1/\nu)$. For $w = \tfrac12(s_1 + s_1^{-1}) \in \mathcal E_R$ and $\zeta = \tfrac12(s_2 + s_2^{-1})$ with $|s_2| = R'$ on the boundary of $\mathcal E_{R'}$, the identity $\zeta - w = \tfrac12(s_2 - s_1)(1 - 1/(s_1s_2))$ gives $|\zeta - w| \ge \tfrac12(R' - R)(1 - 1/R') \ge \tfrac12\,\frac R\nu\,\frac{1}{\nu+1} \ge \frac{R}{4\nu^2} =: d$, because $|s_2 - s_1| \ge |s_2| - |s_1| \ge R' - R$, $|1 - 1/(s_1s_2)| \ge 1 - 1/(|s_1||s_2|) \ge 1 - 1/R'$ and $1 - 1/R' \ge 1 - 1/(1 + 1/\nu) = 1/(\nu+1)$. Hence the open disc of radius $d$ about $w$ lies in $\mathcal E_{R'}$, on which $|T_\nu| \le R'^\nu$ by the first step with $R_1 = R'$, and Cauchy's estimate gives $|T_\nu^{(j)}(w)| \le j!\,R'^\nu d^{-j} = j!\,(1 + 1/\nu)^\nu R^\nu(4\nu^2/R)^j \le e\,j!\,(4\nu^2)^jR^{\nu-j}$ for $j = 1, 2$.
\end{proof}

\begin{proof}[Proof of Theorem~\ref{thm:competitor}]
\emph{Step 1: analytic continuation of $u$.} Since $-u'' + ku' = f$ on $I$, $(e^{-kx}u')' = -e^{-kx}f$, so $u'(x) = e^{kx}\bigl(u'(0) - \int_0^x e^{-kt}f(t)\,dt\bigr)$ and $u(x) = \int_0^x u'(t)\,dt$ for $x \in \bar I$. Both integrals, taken along paths in $\mathcal D$, define analytic functions on the simply connected set $\mathcal D$, which extend $u$ analytically to $\mathcal D$. In particular $u$ is bounded on the compact set $\bar\Omega_\rho \subset \mathcal D$, and so are $u'$ and $u''$ on $\bar I$.

\emph{Step 2: the function $U$ and its Chebyshev coefficients.} Let $x(z) := (\tilde L/\pi)\arccos z - \delta$, analytic on $\mathbb C\setminus\bigl((-\infty,-1]\cup[1,\infty)\bigr) \supset E_{\rho^\ast}$ and mapping $E_\rho$ onto $\Omega_\rho$ and $[-c_0,c_0]$ onto $\bar I$, and put $U(z) := u(x(z))/\sqrt{1 - z^2}$ with the branch of the square root that is positive on $(-1,1)$; for $z = \cos\theta \in [-c_0,c_0]$, $\sqrt{1-z^2} = \sin\theta \ge \sin\theta_0$, so $U$ is analytic on a neighbourhood of $[-c_0,c_0]$, and $U(\pm c_0) = 0$ because $u(0) = u(L) = 0$. Let $U(c_0w) = \sum_{\nu\ge0}b_\nu T_\nu(w)$, $w \in [-1,1]$, be the Chebyshev series of the function $w \mapsto U(c_0w)$, which is analytic on a neighbourhood of $[-1,1]$; the series converges absolutely and uniformly \cite[Thm.~3.1]{trefethen2019approximation}. We claim that
\begin{equation}\label{eq:bk_bound}
|b_\nu| \le C\,(\nu+1)^{1/2}\rho^{-\nu}, \qquad \nu \ge 0 .
\end{equation}
If $\rho < \rho^\ast$, then $\pm1 \notin \bar E_\rho$, so $U$ is analytic on a neighbourhood of $\bar E_\rho$ and bounded on $E_\rho$; the function $w \mapsto U(c_0w)$ is then analytic and bounded by $\norm{U}_{L^\infty(E_\rho)}$ in the open Bernstein ellipse of $[-1,1]$ with parameter $\rho$, and \cite[Thm.~8.1]{trefethen2019approximation} gives $|b_\nu| \le 2\norm{U}_{L^\infty(E_\rho)}\rho^{-\nu}$. If $\rho = \rho^\ast$, then $U$ is analytic on $E_{\rho^\ast}$, and for $1 < \rho' < \rho^\ast$ we bound it on $E_{\rho'}$. For $z = \tfrac{c_0}{2}(s + s^{-1}) \in E_{\rho'}$, $1 \le |s| < \rho'$, and with $1 = \tfrac{c_0}{2}(\rho^\ast + 1/\rho^\ast)$, the identity of the previous proof gives $|1 - z| = \tfrac{c_0}{2}|\rho^\ast - s|\,|1 - 1/(s\rho^\ast)| \ge \tfrac{c_0}{2}(\rho^\ast - \rho')(1 - 1/\rho^\ast) =: c_1(\rho^\ast - \rho')$, and likewise $|1 + z| \ge c_1(\rho^\ast - \rho')$; since $|1 - z| + |1 + z| \ge 2$, one of the two factors is at least one, so $|1 - z^2| \ge c_1(\rho^\ast - \rho')$ and $|U(z)| \le \norm{u}_{L^\infty(\bar\Omega_{\rho^\ast})}\bigl(c_1(\rho^\ast - \rho')\bigr)^{-1/2}$ on $E_{\rho'}$. Let $\nu_0 := \lceil 2\rho^\ast/(\rho^\ast - 1)\rceil$, which is at least $3$. For $\nu \ge \nu_0$ put $\rho' := \rho^\ast(1 - 1/\nu)$, which satisfies $\rho' \ge (1 + \rho^\ast)/2 > 1$, $\rho^\ast - \rho' = \rho^\ast/\nu$ and $\rho'^{-\nu} = \rho^{\ast\,-\nu}(1 - 1/\nu)^{-\nu} \le 4\rho^{\ast\,-\nu}$; the function $w \mapsto U(c_0w)$ is analytic and bounded by $\norm{u}_{L^\infty(\bar\Omega_{\rho^\ast})}(c_1\rho^\ast/\nu)^{-1/2}$ in the open Bernstein ellipse of $[-1,1]$ with parameter $\rho'$, so \cite[Thm.~8.1]{trefethen2019approximation} gives $|b_\nu| \le 8\norm{u}_{L^\infty(\bar\Omega_{\rho^\ast})}(c_1\rho^\ast)^{-1/2}\nu^{1/2}\rho^{\ast\,-\nu}$. For the finitely many $\nu < \nu_0$ the same theorem with $\rho' := (1 + \rho^\ast)/2$ bounds $|b_\nu|$ by a constant. In both cases \eqref{eq:bk_bound} holds. By \eqref{eq:bk_bound} and Lemma~\ref{lem:cheb_ellipse} with $R = 1$, the series $\sum_\nu b_\nu T_\nu(w)$ and its first two termwise derivatives converge uniformly on $[-1,1]$, so the expansion may be differentiated twice term by term there.

\emph{Step 3: the competitor.} Let $q_p(z) := \sum_{\nu=0}^pb_\nu T_\nu(z/c_0) \in \mathbb P_p$ and $\tau_p := U - q_p$ on $[-c_0,c_0]$, and define $v_p^\circ(x) := \sin\theta\,q_p(\cos\theta)$, $\theta = \pi(x+\delta)/\tilde L$, which lies in $V_{p+1} \subset V_N$ by Lemma~\ref{lem:chebyshev_variable}. On $\bar I$, $u - v_p^\circ = \sin\theta\,\tau_p(\cos\theta)$. Since $v_p^\circ(0) = \sin\theta_0\,q_p(c_0) = -\sin\theta_0\,\tau_p(c_0)$ and likewise at $x = L$, and since the boundary functionals are linearly independent on $\operatorname{span}\{w_1,w_2\}$ (proof of Lemma~\ref{lem:dimensions}), there are $\alpha_p, \beta_p$ with $|\alpha_p| + |\beta_p| \le C\bigl(|\tau_p(c_0)| + |\tau_p(-c_0)|\bigr)$ such that $v_p := v_p^\circ - \alpha_pw_1 - \beta_pw_2$ vanishes at $0$ and $L$; then $v_p \in \Phi_{p+1} \subset \Phi_N$.

\emph{Step 4: the residual on $I$.} For $j = 0, 1, 2$, Lemma~\ref{lem:cheb_ellipse} with $R = 1$ and \eqref{eq:bk_bound} give, on $[-c_0,c_0]$,
\[
\bigl|\tau_p^{(j)}(z)\bigr| \le \sum_{\nu>p}|b_\nu|\,c_0^{-j}\,32e\,(\nu+1)^{2j} \le C\sum_{\nu>p}(\nu+1)^{2j+1/2}\rho^{-\nu} \le Cp^{2j+1/2}\rho^{-p},
\]
where the last step uses $\nu + 1 \le (p+1)(1+i)$ for $\nu = p + i$, so that $(\nu+1)^{2j+1/2} \le (p+1)^{2j+1/2}(1+i)^{2j+1/2}$, and $\sum_{i\ge1}(1+i)^{2j+1/2}\rho^{-i} < \infty$. Since $u - v_p^\circ = \sin\theta\,\tau_p(\cos\theta)$ on $\bar I$ and $dx = (\tilde L/\pi)\,d\theta$, differentiation gives $(u - v_p^\circ)' = (\pi/\tilde L)\bigl[z\tau_p - (1 - z^2)\tau_p'\bigr]$ and $(u - v_p^\circ)'' = -(\pi/\tilde L)^2\sin\theta\bigl[\tau_p + 3z\tau_p' - (1 - z^2)\tau_p''\bigr]$ at $z = \cos\theta$, so $\norm{u - v_p^\circ}_{L^\infty(I)} \le Cp^{1/2}\rho^{-p}$, $\norm{(u - v_p^\circ)'}_{L^\infty(I)} \le Cp^{5/2}\rho^{-p}$ and $\norm{(u - v_p^\circ)''}_{L^\infty(I)} \le Cp^{9/2}\rho^{-p}$. Since $u - v_p = (u - v_p^\circ) + \alpha_pw_1 + \beta_pw_2$ with $|\alpha_p| + |\beta_p| \le Cp^{1/2}\rho^{-p}$ and the $w_i$ are fixed smooth functions, $\norm{u - v_p}_{L^\infty(I)} \le Cp^{1/2}\rho^{-p}$ and $\norm{f - \Lop v_p}_{L^\infty(I)} = \norm{\Lop(u - v_p)}_{L^\infty(I)} \le \norm{(u-v_p)''}_{L^\infty(I)} + k\norm{(u-v_p)'}_{L^\infty(I)} \le Cp^{9/2}\rho^{-p}$, which is (i).

\emph{Step 5: the size on the extended interval.} For $z \in [-1,1]$ the point $w = z/c_0$ lies in $[-1/c_0, 1/c_0] \subset \mathcal E_{\rho^\ast}$, because $\tfrac12(\rho^\ast + 1/\rho^\ast) = 1/c_0$. Lemma~\ref{lem:cheb_ellipse} with $R = \rho^\ast$ and \eqref{eq:bk_bound} give, on $[-1,1]$ and for $j = 0,1,2$,
\[
\begin{aligned}
\bigl|q_p^{(j)}(z)\bigr| &\le \sum_{\nu=0}^p|b_\nu|\,c_0^{-j}\,32e\,(\nu+1)^{2j}\rho^{\ast\,\nu} \\
&\le C\sum_{\nu=0}^p(\nu+1)^{2j+1/2}(\rho^\ast/\rho)^\nu \le Cp^{2j+3/2}(\rho^\ast/\rho)^p ,
\end{aligned}
\]
since $\rho^\ast/\rho \ge 1$. With $r(z) := zq_p(z) - (1 - z^2)q_p'(z)$ one has, as in Step~4, $(v_p^\circ)' = (\pi/\tilde L)\,r(\cos\theta)$ and $(v_p^\circ)'' = -(\pi/\tilde L)^2\sin\theta\,r'(\cos\theta)$ on $\tilde I$, where $r' = q_p + 3zq_p' - (1 - z^2)q_p''$; hence $\norm{(v_p^\circ)'}_{L^\infty(\tilde I)} \le Cp^{7/2}(\rho^\ast/\rho)^p$ and $\norm{(v_p^\circ)''}_{L^\infty(\tilde I)} \le Cp^{11/2}(\rho^\ast/\rho)^p$. The correction $\alpha_pw_1 + \beta_pw_2$ contributes at most $Cp^{1/2}\rho^{-p}$ to both norms, and $\norm{\cdot}_{L^2(\tilde I)} \le \sqrt{\tilde L}\,\norm{\cdot}_{L^\infty(\tilde I)}$; this proves (ii), and (iii) follows from Lemma~\ref{lem:coefficient_norm}.
\end{proof}

\begin{proof}[Proof of Corollary~\ref{cor:unconditional}]
Apply Corollary~\ref{cor:regimes_two_constant}(iii) with the competitor $\vect y := \vect y_p$ of Theorem~\ref{thm:competitor}, for which $u_{\vect y} = v_p$. Then $\norm{u - u_{\vect y}}_{L^2(I)} \le \sqrt L\,\norm{u - v_p}_{L^\infty(I)} \le C\sqrt L\,p^{1/2}\rho^{-p}$, $\norm{\vect f - \mat B_D\vect y}_2 = \bigl\|\bigl(f(x_m) - (\Lop v_p)(x_m)\bigr)_m\bigr\|_2$ by the definition of $\mat B$, which is at most $\sqrt n\,\norm{f - \Lop v_p}_{L^\infty(I)} \le C\sqrt N\,p^{9/2}\rho^{-p}$, and $\eta\norm{\vect y}_2 \le C_BC\sqrt N\,p^{11/2}\eta(\rho^\ast/\rho)^p$ by Theorem~\ref{thm:competitor}(iii); inserting these into \eqref{eq:lu_bound} gives \eqref{eq:unconditional}. For \eqref{eq:unconditional_floor} take $p := \lceil\log(1/\eta)/\log\rho^\ast\rceil$, which satisfies ${2 \le p \le N-1}$ under the stated conditions on $\eta$ and $N$, $\rho^{\ast\,p} \ge 1/\eta$ and $\rho^{\ast\,p} \le \rho^\ast/\eta$. Then $\rho^{-p} \le \rho^{-\log(1/\eta)/\log\rho^\ast} = \eta^{a'}$ and $\eta(\rho^\ast/\rho)^p = \eta\rho^{\ast\,p}\rho^{-p} \le \rho^\ast\eta^{a'}$, while $p^{11/2} \le \bigl(1 + \log(1/\eta)/\log\rho^\ast\bigr)^{11/2}$; \eqref{eq:unconditional} becomes \eqref{eq:unconditional_floor} after the constants are absorbed into $C$. The replacement of the two constants by their a priori bounds is Proposition~\ref{prop:apriori}, and for $a' = 1$ the product $\eta^{-1}\eta^{a'}$ equals one.
\end{proof}

\section{Proofs for the Tail of the Spectrum}
\label{sec:tail}

This section proves Propositions~\ref{prop:equilibrium}, \ref{prop:unisolvence_poisson} and \ref{prop:tail_properties}, Theorems~\ref{thm:tail_lower}, \ref{thm:dual}, \ref{thm:tail_law}, \ref{thm:tail_duality} and \ref{thm:initial_slope}, parts (ii)--(iv) of Proposition~\ref{prop:classical} and Corollaries~\ref{cor:lebesgue_rate}, \ref{cor:self_dual}, \ref{cor:critical}, \ref{cor:layer} and \ref{cor:sharpness}, and it states and proves the non-asymptotic Corollaries~\ref{cor:either_or}, \ref{cor:no_power_law} and \ref{cor:certified} and the characterization of the tail law by an extremal problem over zero distributions (Proposition~\ref{prop:extremal}). Throughout, the collocation nodes are the interior nodes of the $N$-grid ($M = n$) and $\Lop = \Lop_P$; Lemmas~\ref{lem:two_forms} and \ref{lem:monotone} and their proofs hold for either operator. The notation of Appendices~\ref{app:rates} and \ref{app:apriori} is in force: $\theta_m$, $z_m$, $\Delta = \omega/(N-1)$, the Lagrange polynomials $\ell_m$ and the Lagrange functions $\ell_m^V$ of $V_N$, and the reflection $(Rv)(x) = v(L-x)$, the inverse $\tilde\Lop_P^{-1}$ on $V_N$ and the Green's function $\tilde G$ of Lemma~\ref{lem:unisolvence}, and the two elementary facts stated at the beginning of Appendix~\ref{app:rates}. The constant $C$ depends on $\theta_0$, $L$ and on the configuration of Definition~\ref{def:configuration} or \ref{def:interp_configuration} in question, and $N_0$ is enlarged finitely many times. We begin with the equilibrium problem, which involves neither $N$ nor the nodes.

\subsection{The Equilibrium Problem: Proof of Proposition~\ref{prop:equilibrium}}
\label{sec:equilibrium_proof}

In this subsection all measures on $[0,\pi]$ have bounded densities with respect to $d\theta$. For $f \in L^2(0,\pi)$ we write $\hat f_k := \int_0^\pi f(\varphi)\cos k\varphi\,d\varphi$, $k \ge 0$, and $E_f := E_{f\,d\theta}$, $J_0(f) := J_0(f\,d\theta)$, defined by \eqref{eq:measure_envelope} also for signed $f$. With the even and $2\pi$-periodic function $h(t) := \log|2\sin(t/2)|$,
\begin{equation}\label{eq:kernel_split}
\log|\cos\vartheta - \cos\varphi| = -\log 2 + h(\vartheta+\varphi) + h(\vartheta-\varphi),
\end{equation}
as in the proof of Lemma~\ref{lem:cl2}(iv). The potential of a finite positive Borel measure $\mu$ with compact support in $\mathbb R$ is $U^\mu(z) := -\int\log|z-s|\,d\mu(s)$, $z \in \mathbb C$; for $\lambda = f\,d\theta$ we write $U^\lambda$ for the potential of its image, so that $U^\lambda(\cos\vartheta) = -E_f(\vartheta)$.

\begin{lemma}[Chebyshev Expansion of the Envelope]
\label{lem:cheb}
For $f \in L^2(0,\pi)$ and every $\vartheta \in [0,\pi]$,
\begin{equation}\label{eq:cheb}
E_f(\vartheta) = -\hat f_0\log 2 - \sum_{k\ge1}\frac2k\,\hat f_k\cos k\vartheta, \qquad J_0(f) = \hat f_0^{\,2}\log 2 + \sum_{k\ge1}\frac2k\,\hat f_k^{\,2},
\end{equation}
the first series converging absolutely and uniformly. Consequently $E_f$ is continuous, $f$ is determined by $E_f$, and $J_0(f) > 0$ if $\hat f_0 = 0$ and $f \ne 0$.
\end{lemma}
\begin{proof}
For $k \ge 1$, integration by parts with $h'(t) = \tfrac12\cot(t/2)$ gives
\[
\int_0^\pi h(t)\cos kt\,dt = -\frac1{2k}\int_0^\pi\cot(t/2)\sin kt\,dt = -\frac{\pi}{2k} :
\]
the boundary terms vanish because $h(t)\sin kt \to 0$ as $t \to 0$ and $\sin k\pi = 0$, and the last integral equals $\pi$ by the identity $\cot(t/2)\sin kt = 1 + \cos kt + 2\sum_{j=1}^{k-1}\cos jt$. Since also $\int_0^\pi h = 0$ by Lemma~\ref{lem:cl2}(i), the even function $h \in L^2(-\pi,\pi)$ has the Fourier series $-\sum_{k\ge1}k^{-1}\cos kt$. For fixed $\vartheta$ let $g := \log|\cos\vartheta - \cos\cdot| \in L^2(0,\pi)$. By \eqref{eq:kernel_split}, the substitution $\varphi \mapsto -\varphi$ in the term $h(\vartheta - \varphi)$ and the periodicity of $h$, $\hat g_k = \int_{-\pi}^{\pi}h(\vartheta+\varphi)\cos k\varphi\,d\varphi = \int_{-\pi}^{\pi}h(t)\cos k(t-\vartheta)\,dt = -(\pi/k)\cos k\vartheta$ for $k \ge 1$, and $\hat g_0 = -\pi\log 2$ by Lemma~\ref{lem:cl2}(iv). Parseval's identity $\int_0^\pi gf = \pi^{-1}\hat g_0\hat f_0 + 2\pi^{-1}\sum_{k\ge1}\hat g_k\hat f_k$ gives the first identity of \eqref{eq:cheb}. By Bessel's inequality $\sum_{k\ge1}\hat f_k^{\,2} \le (\pi/2)\norm{f}_{L^2}^2$, so $\sum_k k^{-1}|\hat f_k| \le (\pi^3/12)^{1/2}\norm{f}_{L^2}$; the series therefore converges absolutely and uniformly, and integrating it termwise against $-f$ gives the second identity. Being uniformly convergent, the series is the cosine series of $E_f$, whose coefficients $-\pi\hat f_0\log 2$ and $-(\pi/k)\hat f_k$ determine $f$ by the completeness of $\{\cos k\varphi\}_{k\ge0}$ in $L^2(0,\pi)$; the same completeness gives $J_0(f) > 0$ when $\hat f_0 = 0 \ne f$.
\end{proof}

\begin{lemma}[A Minimum Principle for Potentials]
\label{lem:minimum}
Let $\mu_1$ and $\mu_2$ be finite positive Borel measures with compact supports in $\mathbb R$ such that $\mu_2(\mathbb R) \le \mu_1(\mathbb R)$ and $U^{\mu_1}$ is continuous on $\mathbb C$, and let $K \subset \mathbb R$ be a nonempty compact set such that the positive part of $\mu_1 - \mu_2$ is carried by $K$. If $c \in \mathbb R$ and $u := U^{\mu_2} - U^{\mu_1} \ge c$ on $K$, then $u \ge c$ on $\mathbb C$. If moreover $u(z_0) = c$ at some $z_0 \in \mathbb C \setminus K$, then $c = 0$, $\mu_1(\mathbb R) = \mu_2(\mathbb R)$ and $u = 0$ on $\mathbb C\setminus K$.
\end{lemma}
\begin{proof}
The potential of a finite positive measure with compact support is superharmonic on $\mathbb C$ and harmonic outside its support \cite[Thm.~3.1.2]{ransford1995potential}. Since $\mu_2 - \mu_1 = (\mu_2 - \mu_1)^+ - (\mu_1 - \mu_2)^+$ and the second measure is carried by $K$, the function $u$ is superharmonic on $D := \mathbb C\setminus K$, which is connected because $K \subset \mathbb R$, and $u$ is lower semicontinuous on $\mathbb C$ with values in $(-\infty,\infty]$. Let $R_0$ exceed the moduli of the points of both supports. For $|z| > R_0$, $u(z) = (\mu_1(\mathbb R) - \mu_2(\mathbb R))\log|z| + H(1/z)$ with $H(\zeta) := \mathrm{Re}\int\log(1 - s\zeta)\,d(\mu_1-\mu_2)(s)$, a harmonic function on $|\zeta| < 1/R_0$ with $H(0) = 0$.

Suppose that $m := \inf_{\mathbb C}u < c$. The infimum of $u$ over a closed disc is attained, since $u$ is lower semicontinuous. If $\mu_1(\mathbb R) > \mu_2(\mathbb R)$, then $u \to +\infty$ as $z \to \infty$, and $m$ is attained. If the masses are equal and $m$ is not attained, then $m = \lim_{z\to\infty}u(z) = 0 = H(0)$, so $H \ge H(0)$ near $0$; by the maximum principle for the subharmonic function $-H$ \cite[Thm.~2.3.1]{ransford1995potential}, $H$ vanishes near $0$, and $m$ is attained after all. In either case $m$ is attained at a point of $D$, since $u \ge c > m$ on $K$, and the maximum principle for $-u$ on $D$ shows that $u = m$ on $D$. This is impossible if $\mu_1(\mathbb R) > \mu_2(\mathbb R)$; if the masses are equal, it gives $m = \lim_{z\to\infty}u(z) = 0$, and, $D$ being dense in $\mathbb C$, the lower semicontinuity of $u$ gives $u \le 0 < c$ on $K$, a contradiction. Hence $u \ge c$ on $\mathbb C$. If $u(z_0) = c$ with $z_0 \in D$, then $u$ attains its infimum over $D$ at $z_0$, so $u = c$ on $D$; as above this forces $\mu_1(\mathbb R) = \mu_2(\mathbb R)$ and $c = 0$.
\end{proof}

\begin{proof}[Proof of Proposition~\ref{prop:equilibrium}]
\emph{Existence, uniqueness and symmetry.} We identify $\mathscr A_\omega$ with the set of $f \in L^\infty(0,\pi)$ with $0 \le f \le 1/\omega$ and $\hat f_0 = 1$, a convex set that is compact in the weak-$\ast$ topology of $L^\infty(0,\pi) = L^1(0,\pi)^\ast$. This topology is metrizable on bounded sets, $L^1(0,\pi)$ being separable. If $f_n \to f$ in it, then $\hat f_{n,k} \to \hat f_k$ for every $k$ and $\int_{\Theta_I}f_n \to \int_{\Theta_I}f$, and since $\sum_{k>K}k^{-1}\hat f_{n,k}^{\,2} \le (\pi/(2K))\norm{f_n}_{L^2}^2 \le \pi^2/(2K\omega^2)$ by Bessel's inequality, Lemma~\ref{lem:cheb} gives $J_0(f_n) \to J_0(f)$. So $J_\tau$ is weak-$\ast$ continuous on the compact set $\mathscr A_\omega$ and attains its minimum. For $\lambda \ne \lambda'$ in $\mathscr A_\omega$ the bilinearity of $J_0$ gives $J_\tau(\tfrac12(\lambda + \lambda')) = \tfrac12J_\tau(\lambda) + \tfrac12J_\tau(\lambda') - \tfrac14J_0(\lambda - \lambda')$, and $J_0(\lambda - \lambda') > 0$ by Lemma~\ref{lem:cheb}; hence the minimizer $\lambda_\tau$ is unique. The reflection $\vartheta \mapsto \pi - \vartheta$ leaves $\mathscr A_\omega$, $\Theta_I$ and the kernel $\log|\cos\vartheta - \cos\varphi|$ invariant, hence $J_\tau$, and therefore fixes $\lambda_\tau$. The envelope $E := E_\tau$ is continuous by Lemma~\ref{lem:cheb}. Below $\lambda := \lambda_\tau = f\,d\theta$.

\emph{Equilibrium conditions.} For $\mu \in \mathscr A_\omega$ and $0 < \eps \le 1$, the bilinearity of $J_0$ and the symmetry of the kernel give $J_\tau(\lambda + \eps(\mu - \lambda)) = J_\tau(\lambda) + 2\eps\int G\,d(\mu-\lambda) + \eps^2J_0(\mu - \lambda)$ with $G := -E - \tau\mathbf 1_{\Theta_I}$, so the minimality of $\lambda$ forces $\int G\,d\mu \ge \int G\,d\lambda$ for all $\mu \in \mathscr A_\omega$. We claim that there is $w \in \mathbb R$ with
\begin{equation}\label{eq:bathtub}
G \le -w \ \text{ a.e.\ on } \{f > 0\}, \qquad G \ge -w \ \text{ a.e.\ on } \{f < 1/\omega\}.
\end{equation}
Both sets have positive measure, because $\hat f_0 = 1 < \pi/\omega$. If the essential supremum of $G$ on the first set exceeded the essential infimum of $G$ on the second, there would be $b < b'$ such that $A := \{f < 1/\omega,\ G < b\}$ and $B := \{f > 0,\ G > b'\}$ have positive measure, and for small $\eta > 0$ so would $A_\eta := A \cap \{f < 1/\omega - \eta\}$ and $B_\eta := B \cap \{f > \eta\}$. With $\eps := \eta\min\{|A_\eta|, |B_\eta|\}$, the measure $\mu := \lambda + \eps\,(\mathbf 1_{A_\eta}/|A_\eta| - \mathbf 1_{B_\eta}/|B_\eta|)\,d\theta$ lies in $\mathscr A_\omega$ and satisfies $\int G\,d\mu - \int G\,d\lambda \le \eps(b - b') < 0$, a contradiction; so \eqref{eq:bathtub} holds with $-w$ the essential supremum of $G$ on $\{f > 0\}$. Conversely, if some $\lambda' = f'\,d\theta \in \mathscr A_\omega$ satisfies \eqref{eq:bathtub} with its own function $G'$ and some $w'$, then for every $\mu \in \mathscr A_\omega$ the function $(G' + w')(d\mu/d\theta - f')$ is nonnegative a.e., and $J_\tau(\mu) - J_\tau(\lambda') = 2\int(G' + w')\,d(\mu - \lambda') + J_0(\mu - \lambda') \ge 0$, so $\lambda' = \lambda_\tau$. Every point of the support of $\lambda|_{\Theta_E}$ is a limit of points of $\Theta_E \cap \{f > 0\}$ at which $G \le -w$, i.e.\ $E \ge w$, and similarly for the other three sets below; since $E$ is continuous, \eqref{eq:bathtub} gives
\begin{equation}\label{eq:pointwise}
\begin{aligned}
E &\ge w \ \text{ on } \operatorname{supp}\lambda|_{\Theta_E}, & E &\le w \ \text{ on } \operatorname{supp}(\omega^{-1}d\theta - \lambda)|_{\Theta_E}, \\
E + \tau &\ge w \ \text{ on } \operatorname{supp}\lambda|_{\Theta_I}, & E + \tau &\le w \ \text{ on } \operatorname{supp}(\omega^{-1}d\theta - \lambda)|_{\Theta_I}.
\end{aligned}
\end{equation}

\emph{Two global bounds.} Let $\lambda_0 := \pi^{-1}d\theta$, the arcsine measure, whose envelope is $-\log 2$ by Lemma~\ref{lem:cl2}(iv); it belongs to $\mathscr A_\omega$ and minimizes $J_0$ there, because $J_0(\lambda) = \hat f_0^{\,2}\log 2 + \sum_{k\ge1}(2/k)\hat f_k^{\,2} \ge \log 2 = J_0(\lambda_0)$ for $\lambda = f\,d\theta \in \mathscr A_\omega$ by Lemma~\ref{lem:cheb}. Suppose first $\lambda \ne \lambda_0$, and let $S^\pm$ be the supports of $(f - \pi^{-1})^\pm d\theta$, both nonempty since $\lambda$ and $\lambda_0$ have the same mass. Every neighbourhood of a point of $S^-$ meets $\{f < 1/\pi\} \subset \{f < 1/\omega\}$ in a set of positive measure, so the point lies in the support of $(\omega^{-1}d\theta - \lambda)|_{\Theta_E}$ or of $(\omega^{-1}d\theta - \lambda)|_{\Theta_I}$, and $E \le w$ on $S^-$ by \eqref{eq:pointwise}; likewise every point of $S^+$ lies in the support of $\lambda|_{\Theta_E}$ or of $\lambda|_{\Theta_I}$, and $E \ge w - \tau$ on $S^+$. The potentials $U^\lambda$ and $U^{\lambda_0}$ are continuous on $\mathbb C$ by Lemma~\ref{lem:env}(d), and $U^{\lambda_0} = \log 2$ on $[-1,1]$. Lemma~\ref{lem:minimum} with $(\mu_1,\mu_2,K) = (\lambda_0,\lambda,\cos S^-)$ and $u = U^\lambda - U^{\lambda_0} = -E - \log 2 \ge -w - \log 2$ on $S^-$ gives $E \le w$ on $[0,\pi]$, and with $(\mu_1,\mu_2,K) = (\lambda,\lambda_0,\cos S^+)$ and $u = U^{\lambda_0} - U^\lambda = E + \log 2 \ge w - \tau + \log 2$ on $S^+$ it gives $E \ge w - \tau$ on $[0,\pi]$. If $\lambda = \lambda_0$, then $E = -\log 2$, and \eqref{eq:pointwise}, in which all four supports are the closures of $\Theta_E$ and $\Theta_I$, forces $w = -\log 2$ and $\tau = 0$. In both cases
\begin{equation}\label{eq:two_sided}
w - \tau \le E \le w \qquad\text{on } [0,\pi].
\end{equation}
If $\lambda(\Theta_E) > 0$, then $E = w$ on the nonempty support of $\lambda|_{\Theta_E}$ by \eqref{eq:pointwise} and \eqref{eq:two_sided}, so $w = E_\tau^\ast$; then \eqref{eq:two_sided} is (i), and (ii) follows from \eqref{eq:pointwise} and (i).

\emph{Part (iv), and $\lambda_\tau(\Theta_E) > 0$ for $\tau < \kappa_\sigma$.} By Lemma~\ref{lem:clausen}, $E_{\mu_\omega}(\vartheta) = -U_\omega(\cos\vartheta)$ attains its maximum $-U_\omega(1)$ over $[0,\pi]$ at $\vartheta \in \{0,\pi\}$ and its minimum $-U_\omega(0)$ over $\Theta_I$ at $\pi/2$. For $\tau \ge \kappa_\sigma = U_\omega(0) - U_\omega(1)$ the node measure satisfies \eqref{eq:bathtub} with $w := -U_\omega(1)$: its density is $1/\omega$ on $\Theta_I$, where $G = U_\omega(\cos\vartheta) - \tau \le U_\omega(0) - \tau \le U_\omega(1)$, and $0$ on $\Theta_E$, where $G = U_\omega(\cos\vartheta) \ge U_\omega(1)$. Hence $\lambda_\tau = \mu_\omega$, and (i) and (ii) hold with $E_\tau^\ast = -U_\omega(1)$, since $E_{\mu_\omega} \ge -U_\omega(0) \ge E_\tau^\ast - \tau$ and $(\omega^{-1}d\theta - \mu_\omega)|_{\Theta_I} = 0$. Conversely, if $\lambda_\tau(\Theta_E) = 0$, then $\lambda_\tau = \mu_\omega$, the only element of $\mathscr A_\omega$ carried by $\Theta_I$, and \eqref{eq:pointwise} gives $-U_\omega(1) = \max_{\bar\Theta_E}E \le w \le \min_{\Theta_I}E + \tau = -U_\omega(0) + \tau$, that is $\tau \ge \kappa_\sigma$. So $\lambda_\tau(\Theta_E) > 0$ for $\tau < \kappa_\sigma$, which completes (i) and (ii).

\emph{The layer.} Let $\tau < \kappa_\sigma$. If $A \subset (0,\pi)$ is an open interval with $f = 1/\omega$ a.e.\ on $A$, then $\nu := \omega^{-1}d\theta - \lambda$ is a positive measure of mass $\pi/\omega - 1 > 0$ with $\nu(A) = 0$, and $E = E_{\omega^{-1}d\theta} - E_\nu = -(\pi/\omega)\log 2 - E_\nu$ by Lemma~\ref{lem:cl2}(iv); for $z$ in the open interval $\cos A$, $\frac{d^2}{dz^2}\int\log|z - \cos\varphi|\,d\nu(\varphi) = -\int(z - \cos\varphi)^{-2}\,d\nu(\varphi) < 0$, so $z \mapsto E(\arccos z)$ is strictly convex on $\cos A$. Let $O$ be the union of the relatively open subsets of $\Theta_I$ on which $f = 1/\omega$ a.e.; by the Lindel\"of property it is a countable union of such sets, so that $f = 1/\omega$ a.e.\ on $O$. Every point of $\Theta_I \setminus O$ lies in the support of $(\omega^{-1}d\theta - \lambda)|_{\Theta_I}$, where $E = E_\tau^\ast - \tau$ by (ii). A component of $O$ that contains neither $\theta_0$ nor $\pi - \theta_0$ is therefore an open interval $(\alpha,\beta)$ with $E(\alpha) = E(\beta) = E_\tau^\ast - \tau$, and strict convexity gives $E < E_\tau^\ast - \tau$ inside, contrary to (i). A component that contains both $\theta_0$ and $\pi - \theta_0$ is $\Theta_I$, which would give $\lambda = \mu_\omega$. By the symmetry of $\lambda$, $O = [\theta_0,\theta_\tau) \cup (\pi - \theta_\tau, \pi - \theta_0]$ for some $\theta_\tau \in [\theta_0,\pi/2]$, with $O = \emptyset$ if $\theta_\tau = \theta_0$, and $\theta_\tau < \pi/2$ since otherwise $f = 1/\omega$ a.e.\ on $\Theta_I$. Then $\lambda = \omega^{-1}d\theta$ on $[\theta_0,\theta_\tau]\cup[\pi-\theta_\tau,\pi-\theta_0]$, and $[\theta_\tau, \pi - \theta_\tau] \subset \Theta_I\setminus O$, where $E = E_\tau^\ast - \tau$. Moreover, $z \mapsto E(\arccos z)$ is continuous on $[\cos\theta_\tau, c_0]$ and convex inside, so $\max_{[\theta_0,\theta_\tau]}E = \max\{E(\theta_0), E(\theta_\tau)\} = E(\theta_0)$, because $E(\theta_\tau) = E_\tau^\ast - \tau \le E(\theta_0)$ by (i); by symmetry $\max_{\Theta_I}E = E(\theta_0)$. Finally, $\theta_\tau$ is unique: if $\theta' < \theta''$ both had the properties in (iii), then $f = 1/\omega$ a.e.\ and $E = E_\tau^\ast - \tau$ on $(\theta',\theta'')$, contrary to the strict convexity of $z \mapsto E(\arccos z)$ on $\cos(\theta',\theta'')$ shown above. This completes (iii).
\end{proof}

\subsection{Configurations and Lattice Products}

We now turn to the configurations of Definition~\ref{def:configuration}. Their envelope \eqref{eq:envelope} is $E_\nu$ for the measure $\nu$ of density $1/\omega$ on $S^\ast$ and $\rho_k$ on $A_k^\ast$: on the layer $\nu$ coincides with the node measure, and the arcs carry the free zeros. The envelope has a closed form: for $0 \le \alpha < \beta \le \pi$ and $\vartheta \in [0,\pi]$,
\begin{equation}\label{eq:arc_integral}
\begin{aligned}
\int_\alpha^\beta\log|\cos\vartheta - \cos\varphi|\,d\varphi = -(\beta-\alpha)\log 2 &+ \mathrm{Cl}_2(\vartheta+\alpha) - \mathrm{Cl}_2(\vartheta+\beta) \\ &+ \mathrm{Cl}_2(\vartheta-\beta) - \mathrm{Cl}_2(\vartheta-\alpha),
\end{aligned}
\end{equation}
which follows from $\log|\cos\vartheta - \cos\varphi| = -\log 2 + \log|2\sin\tfrac12(\vartheta+\varphi)| + \log|2\sin\tfrac12(\vartheta-\varphi)|$ and $\mathrm{Cl}_2'(x) = -\log|2\sin(x/2)|$ by the substitutions $u = \vartheta \pm \varphi$, exactly as in Step~2 of the proof of Lemma~\ref{lem:clausen}. The maxima in Definition~\ref{def:configuration} exist because $E$ is continuous (Lemma~\ref{lem:E_modulus}).

\begin{lemma}[Lattice Products]
\label{lem:construction}
Let $\mathcal C$ be a configuration. For $N \ge N_0$ let $Z_N$ be the finite multiset of points of $[0,\pi]$, counted with multiplicity, consisting of
\begin{enumerate}
    \item[(i)] the layer nodes, i.e.\ the nodes $\theta_m \in S^\ast$;
    \item[(ii)] for each arc $A_k$ with $\beta_k < \pi/2$, the lattice $\varphi^{(k)}_i := \alpha_k + (i-1)d_k$, $1 \le i \le K_k$, together with its mirror image $\pi - \varphi^{(k)}_i$, where $K_k$ is the smallest even integer with $K_k - 1 \ge \rho_k(\beta_k - \alpha_k)N$ and $d_k := (\beta_k-\alpha_k)/(K_k-1)$; for an arc with $\beta_k = \pi/2$, the lattice $\varphi^{(k)}_i := \alpha_k + (i-1)d_k$, $1 \le i \le K_k$, on $A_k^\ast = [\alpha_k, \pi-\alpha_k]$, where $K_k$ is the smallest even integer with $K_k - 1 \ge \rho_k(\pi-2\alpha_k)N$ and $d_k := (\pi - 2\alpha_k)/(K_k-1)$;
    \item[(iii)] the point $\pi/2$ if $N$ is even.
\end{enumerate}
(A point that belongs to two of these families, for instance a node at $\theta_s$ that is also the first point of a centre lattice, or the common endpoint of two adjacent arcs, is counted once for each family.) Then $Z_N$ is invariant under $\zeta \mapsto \pi - \zeta$ with multiplicities, its number of points $|Z_N|$ counted with multiplicity satisfies $|Z_N| \equiv N-1 \pmod 2$ and $|Z_N| \le N-1$, and the polynomial $P_N(z) := \prod_{\zeta\in Z_N}(z - \cos\zeta)$, the product being taken over the multiset, satisfies $P_N(-z) = (-1)^{N-1}P_N(z)$ and, with $E$ the envelope of $\mathcal C$,
\begin{enumerate}
    \item[(a)] $N^{-1}\log|P_N(\cos\vartheta)| \le E(\vartheta) + C\log N/N$ for every $\vartheta \in [0,\pi]$;
    \item[(b)] $N^{-1}\log|P_N(\cos\vartheta)| \ge E(\vartheta) - C\log N/N$ whenever $\vartheta \in [0,\pi]$ has distance at least $\Delta/4$ from every layer node, at least $d_k/4$ from every point of the $k$-th lattice for each $k$, and at least $1/N$ from $\pi/2$.
\end{enumerate}
\end{lemma}
\begin{proof}
\emph{Symmetry and counting.} The layer nodes come in pairs $\theta_m \leftrightarrow \pi - \theta_m = \theta_{N+1-m}$ with $\theta_m < \pi/2$, because $\theta_s < \pi/2$; each lattice in (ii) is symmetric by construction and, having an even number of points, does not contain $\pi/2$. Hence the multiset $Z_N$ without the point of (iii) is symmetric with multiplicities and has an even number of points, so $|Z_N| \equiv N - 1 \pmod 2$, and $P_N(-z) = \prod_{\zeta\in Z_N}(-1)\bigl(z - \cos(\pi-\zeta)\bigr) = (-1)^{|Z_N|}P_N(z)$ because $\cos(\pi - \zeta) = -\cos\zeta$ and $\zeta \mapsto \pi - \zeta$ permutes the multiset. The layer contains at most $2(\theta_s - \theta_0)N/\omega + 2$ nodes, and the smallest even integer with $K_k - 1 \ge x$ is less than $x + 3$, so $|Z_N| \le m(\mathcal C)N + 6Q + 3$, which is at most $N-1$ once $N(1 - m(\mathcal C)) \ge 6Q + 4$.

\emph{Potential estimates.} We have $\log|P_N(\cos\vartheta)| = \sum_{\zeta\in Z_N}g_\vartheta(\zeta)$. The layer nodes in $S$ are $\theta_1, \dots, \theta_q$ with $\theta_q \le \theta_s < \theta_q + \Delta$, and $q \ge 2$ once $\Delta \le \theta_s - \theta_0$. If $\vartheta$ is at distance at least $\Delta/4$ from them, Lemma~\ref{lem:riemann_arc} with $[\alpha,\beta] = [\theta_0, \theta_q] \subset [\xi, \pi - \xi]$, $K = q$ and $d = \Delta \le \xi/4 = \min\{\xi,1\}/4$ (which holds for $N \ge N_0$, as $\xi < 1$) gives
\[
\sum_{m=1}^q g_\vartheta(\theta_m) = \frac1\Delta\int_{\theta_0}^{\theta_q}g_\vartheta + O(\log N) = \frac N\omega\int_{\theta_0}^{\theta_s}g_\vartheta + O(\log N),
\]
since $1/\Delta = (N-1)/\omega$, $|\int_{\theta_0}^{\theta_s}g_\vartheta| \le C$, and $|\int_{\theta_q}^{\theta_s}g_\vartheta| \le C\Delta\log(1/\Delta)$ by the estimate of the end half-cells in the proof of Lemma~\ref{lem:riemann}, $[\theta_q, \theta_s]$ being an interval of length at most $\Delta$ inside $[\xi/2, \pi - \xi/2]$. The mirror layer is treated in the same way. For the $k$-th lattice on an arc $[\alpha,\beta]$ (one of $A_k$, $\pi - A_k$, or $A_k^\ast$), Lemma~\ref{lem:riemann_arc} with $d = d_k \le 1/(\rho_kN) \le \min\{\xi,1\}/4$ (for $N \ge N_0$) gives $\sum_ig_\vartheta(\varphi_i^{(k)}) = d_k^{-1}\int_\alpha^\beta g_\vartheta + O(\log N)$ when $\vartheta$ is at distance at least $d_k/4$ from its points, and $d_k^{-1} = (K_k-1)/(\beta - \alpha) \in [\rho_kN, \rho_kN + 2/(\beta-\alpha))$ together with $|\int_\alpha^\beta g_\vartheta| \le C$ gives $d_k^{-1}\int_\alpha^\beta g_\vartheta = \rho_kN\int_\alpha^\beta g_\vartheta + O(1)$. Adding the finitely many lattices and, for even $N$, the term $g_\vartheta(\pi/2) = \log|\cos\vartheta|$, we obtain
\[
\log|P_N(\cos\vartheta)| = NE(\vartheta) + O(\log N) + \begin{cases}\log|\cos\vartheta|, & N \text{ even},\\ 0, & N \text{ odd},\end{cases}
\]
under the hypotheses of (b); there $|\cos\vartheta| = |\sin(\vartheta - \pi/2)| \ge 2/(\pi N)$ by the first elementary fact, since $1/N \le |\vartheta - \pi/2| \le \pi/2$, which proves (b). For (a), each lattice contains at most one point within a quarter spacing of $\vartheta$, its points being one spacing apart; omitting these points, whose factors satisfy $g_\vartheta(\zeta) \le \log 2$, and using case (b) of Lemma~\ref{lem:riemann_arc} for the lattices concerned and $\log|\cos\vartheta| \le 0$ gives the upper bound.
\end{proof}

\subsection{From Polynomials to Coefficient Vectors}

\begin{lemma}[Boundary Correction]
\label{lem:correction}
Let $\chi_N^{\pm} := \ell_1^V \pm \ell_N^V \in V_N$. For $N \ge N_0$ and both signs:
\begin{enumerate}
    \item[(i)] $\chi_N^{\pm}$ vanishes at the interior nodes and $R\chi_N^{\pm} = \pm \chi_N^{\pm}$;
    \item[(ii)] $\norm{\chi_N^{\pm}}_{L^\infty(I)} \le N^C$ and $\norm{\chi_N^{\pm}}_{L^2(\tilde I)} \le \exp\bigl(N(\kappa_\sigma - \kappa_\Lambda) + C\log N\bigr)$;
    \item[(iii)] $p_N^{\pm} := (\tilde\Lop_P^{-1}\chi_N^{\pm})(0) = \int_{\tilde I}\tilde G(0,y)\,\chi_N^{\pm}(y)\,dy \ge \exp\bigl(N(\kappa_\sigma-\kappa_\Lambda) - C\log N\bigr)$.
\end{enumerate}
Consequently, for every $g_0 \in V_N$ with $Rg_0 = \pm g_0$, the function $g := g_0 - \alpha \chi_N^{\pm}$ with $\alpha := (\tilde\Lop_P^{-1}g_0)(0)/p_N^{\pm}$, the sign being that of $g_0$, belongs to $\Psi_N$, coincides with $g_0$ at the interior nodes, and
\begin{equation}\label{eq:alpha_bound}
|\alpha| \le C\norm{g_0}_{L^2(\tilde I)}\exp\bigl(-N(\kappa_\sigma - \kappa_\Lambda) + C\log N\bigr).
\end{equation}
\end{lemma}
\begin{proof}
(i) $\ell_1^V$ and $\ell_N^V$ vanish at the interior nodes by definition. By the proof of Lemma~\ref{lem:unisolvence}, $R$ maps $V_N$ onto itself and permutes the nodes, $x_{N+1-m} = L - x_m$, so $R\ell_1^V$ is the Lagrange function of $V_N$ at $x_N$: $R\ell_1^V = \ell_N^V$, $R\ell_N^V = \ell_1^V$, and $R\chi_N^{\pm} = \ell_N^V \pm \ell_1^V = \pm \chi_N^{\pm}$.

(ii) By Proposition~\ref{prop:lagrange_norms} with $m = 1$ and $z_1 = c_0$, $\norm{\ell_1^V}_{L^\infty(I)} \le \exp(C\log N)$ and $\norm{\ell_1^V}_{L^2(\tilde I)} \le \exp\bigl(N[U_\omega(c_0) - U_\omega(1)] + C\log N\bigr)$, and $U_\omega(c_0) - U_\omega(1) = \kappa_\sigma - \kappa_\Lambda$ by \eqref{eq:rates}; the same bounds hold for $\ell_N^V = R\ell_1^V$, and the triangle inequality gives (ii).

(iii) \emph{Sign structure.} By the proof of Proposition~\ref{prop:lagrange_norms}, $\ell_1^V(x) = (\sin\vartheta/\sin\theta_0)\,\ell_1(\cos\vartheta)$ and $\ell_N^V(x) = (\sin\vartheta/\sin\theta_0)\,\ell_N(\cos\vartheta)$ with $\vartheta = \pi(x+\delta)/\tilde L$, since $\sin\theta_N = \sin\theta_0$. The two Lagrange polynomials are related by
\begin{equation}\label{eq:lagrange_mirror}
(-1)^{N-1}(z + c_0)\,\ell_N(z) = (z - c_0)\,\ell_1(z):
\end{equation}
indeed $\ell_N(z)/\ell_1(z) = \frac{z - z_1}{z - z_N}\prod_{k\ge2}(z_1 - z_k)\big/\prod_{k\le N-1}(z_N - z_k)$, and since $z_N = -c_0$ and $z_{N+1-k} = -z_k$, $\prod_{k\le N-1}(z_N - z_k) = (-1)^{N-1}\prod_{k\le N-1}(c_0 + z_k) = (-1)^{N-1}\prod_{k\ge2}(c_0 - z_k)$, so that $\ell_N(z) = (-1)^{N-1}\frac{z - c_0}{z + c_0}\,\ell_1(z)$. For $x \in (-\delta, 0)$, i.e.\ $\vartheta \in (0,\theta_0)$ and $z = \cos\vartheta \in (c_0, 1)$, every factor of $\ell_1(z) = \prod_{k\ge2}(z - z_k)/(z_1 - z_k)$ is positive, so $\ell_1^V(x) > 0$, and $0 < (z - c_0)/(z + c_0) < 1$; by \eqref{eq:lagrange_mirror},
\[
\chi_N^{\pm} = \ell_1^V\Bigl(1 \pm (-1)^{N-1}\frac{z - c_0}{z + c_0}\Bigr) \ge \frac{2c_0}{1 + c_0}\,\ell_1^V > 0 \quad\text{on } (-\delta, 0),
\]
because $1 + (z-c_0)/(z+c_0) = 2z/(z+c_0) \ge 1$ and $1 - (z-c_0)/(z+c_0) = 2c_0/(z+c_0) \ge 2c_0/(1+c_0)$ for $z \in (c_0,1)$; and $\chi_N^{\pm}(L - x) = \pm \chi_N^{\pm}(x)$ by (i).

\emph{Lower bound.} Split $p_N^{\pm}$ into the integrals over $(-\delta,0)$, $(L, L+\delta)$ and $I$. By (ii) and $0 \le \tilde G(0,y) \le \delta(L+\delta)/\tilde L$, the integral over $I$ is at most $N^C$ in modulus. In the integral over $(L,L+\delta)$ substitute $y = L - y'$ and use the symmetry of $\chi_N^{\pm}$; with $\tilde G(0,y') = (y'+\delta)(L+\delta)/\tilde L$ and $\tilde G(0, L-y') = \delta(y'+\delta)/\tilde L$ for $y' \in (-\delta,0)$, which follow from the formula for $\tilde G$ in Lemma~\ref{lem:unisolvence}, the two extension integrals add up to
\[
\int_{-\delta}^0\frac{(L+\delta) \pm \delta}{\tilde L}\,(y'+\delta)\,\chi_N^{\pm}(y')\,dy' \ \ge\ \frac{L}{\tilde L}\int_{-\delta}^0(y'+\delta)\,\chi_N^{\pm}(y')\,dy' ,
\]
since $\chi_N^{\pm} > 0$ there and $(L+\delta) \pm \delta \ge L$. Restrict the last integral to the $y'$ with $\vartheta = \pi(y'+\delta)/\tilde L \in [1/N, 2/N]$, an interval of length $\tilde L/(\pi N)$ on which $y' + \delta \ge \tilde L/(\pi N)$. For $N \ge N_0$ such $\vartheta$ is at distance at least $\Delta/4$ from every node, so Lemma~\ref{lem:lagrange_polynomials}(b) and Lemma~\ref{lem:U_modulus} give $|\ell_1(\cos\vartheta)| \ge \exp\bigl(N[U_\omega(c_0) - U_\omega(\cos\vartheta)] - C\log N\bigr) \ge \exp\bigl(N[U_\omega(c_0) - U_\omega(1)] - C\log N\bigr)$, and $\sin\vartheta \ge 2/(\pi N)$; hence $\chi_N^{\pm} \ge c_0\,\ell_1^V \ge \exp\bigl(N(\kappa_\sigma-\kappa_\Lambda) - C\log N\bigr)$ there. Collecting the factors, $p_N^{\pm} \ge \exp\bigl(N(\kappa_\sigma-\kappa_\Lambda) - C\log N\bigr) - N^C$, which gives (iii) for $N \ge N_0$ because $\kappa_\sigma > \kappa_\Lambda$.

\emph{Consequence.} By the proof of Lemma~\ref{lem:unisolvence}, $\psi \in V_N$ lies in $\Psi_N$ if and only if $v := \tilde\Lop_P^{-1}\psi$ vanishes at $0$ and at $L$, and $v(0) = \int_{\tilde I}\tilde G(0,y)\psi(y)\,dy$. Since $R$ commutes with $\tilde\Lop_P^{-1}$ on $V_N$ and $Rg = \pm g$, one has $v(L) = \pm v(0)$ for $v = \tilde\Lop_P^{-1}g$, and $v(0) = (\tilde\Lop_P^{-1}g_0)(0) - \alpha p_N^{\pm} = 0$ by the choice of $\alpha$; so $g \in \Psi_N$. The nodal values are those of $g_0$ by (i), and \eqref{eq:alpha_bound} follows from $|(\tilde\Lop_P^{-1}g_0)(0)| \le \norm{\tilde G(0,\cdot)}_{L^2(\tilde I)}\norm{g_0}_{L^2(\tilde I)}$ and (iii).
\end{proof}

\begin{proof}[Proof of Proposition~\ref{prop:unisolvence_poisson}]
Let $N \ge N_0$ of Lemma~\ref{lem:correction}. The functions $\chi_N^{\pm}$ vanish at the interior nodes and are linearly independent, since $\chi_N^+ + \chi_N^- = 2\ell_1^V$ and $\chi_N^+ - \chi_N^- = 2\ell_N^V$; so they span the two-dimensional space $K_N$ of the functions of $V_N$ that vanish at the interior nodes (proof of Lemma~\ref{lem:unisolvence}). Let $\psi = a\chi_N^+ + b\chi_N^- \in \Psi_N$. As $R$ maps $\Psi_N$ onto itself, $a\chi_N^+ = \frac12(\psi + R\psi)$ and $b\chi_N^- = \frac12(\psi - R\psi)$ lie in $\Psi_N = \Lop_P(\Phi_N)$, so $\tilde\Lop_P^{-1}$ maps them into $\Phi_N$, whose elements vanish at $0$: $ap_N^+ = bp_N^- = 0$. By Lemma~\ref{lem:correction}(iii), $a = b = 0$. Hence $\Psi_N \cap K_N = \{0\}$, which is Assumption~\ref{ass:unisolvence} by the proof of Lemma~\ref{lem:unisolvence}.
\end{proof}

\begin{proof}[Proof of Proposition~\ref{prop:classical}(ii)--(iv)]
Fix $2 \le m \le N-1$ and put $g_0 := \ell_m^V$ and $g_0^{\pm} := \tfrac12(g_0 \pm Rg_0) = \tfrac12(\ell_m^V \pm \ell_{N+1-m}^V)$, so that $g_0 = g_0^+ + g_0^-$, $Rg_0^{\pm} = \pm g_0^{\pm}$, and $\norm{g_0^{\pm}}_{L^2(\tilde I)} \le \norm{g_0}_{L^2(\tilde I)}$ because $R$ is an isometry of $L^2(\tilde I)$. Lemma~\ref{lem:correction}, applied to $g_0^+$ with the sign $+$ and to $g_0^-$ with the sign $-$, gives $\alpha^{\pm}$ such that $g := g_0 - \alpha^+\chi_N^+ - \alpha^-\chi_N^- \in \Psi_N$ coincides with $g_0$ at the interior nodes, where $g_0(x_k) = \delta_{mk}$; by Assumption~\ref{ass:unisolvence}, $g = \ell_m^\Psi$. By Proposition~\ref{prop:lagrange_norms}, $\norm{\ell_m^V}_{L^\infty(I)} \le \exp\bigl(N[U_\omega(z_m) - U_\omega(c_0)] + C\log N\bigr)$ and $\norm{g_0}_{L^2(\tilde I)} \le \exp\bigl(N[U_\omega(z_m) - U_\omega(1)] + C\log N\bigr)$, so \eqref{eq:alpha_bound} and Lemma~\ref{lem:correction}(ii) give
\begin{align*}
|\alpha^{\pm}|\,\norm{\chi_N^{\pm}}_{L^\infty(I)} &\le N^C\exp\bigl(N[U_\omega(z_m) - U_\omega(1)] - N(\kappa_\sigma - \kappa_\Lambda)\bigr) \\
&= N^C\exp\bigl(N[U_\omega(z_m) - U_\omega(c_0)]\bigr),
\end{align*}
since $\kappa_\sigma - \kappa_\Lambda = U_\omega(c_0) - U_\omega(1)$ by \eqref{eq:rates}. The triangle inequality proves (ii). For (iii), $\Lambda_N^\Psi = \max_{x\in\bar I}\sum_m|\ell_m^\Psi(x)| \le (N-2)\max_m\norm{\ell_m^\Psi}_{L^\infty(I)}$, and $U_\omega(z_m) \le U_\omega(0)$ for every $m$ by Lemma~\ref{lem:clausen}, with $U_\omega(0) - U_\omega(c_0) = \kappa_\Lambda$; the bound on $\Lambda^\Psi_{N,2}$ follows from $\Lambda^\Psi_{N,2} \le \sqrt L\,\Lambda_N^\Psi$. For (iv), insert (iii) into \eqref{eq:classical}; at $t = e^{-N\kappa_\sigma}$, $e^{N\kappa_\Lambda}t = e^{-N(\kappa_\sigma-\kappa_\Lambda)} = t^{a(\omega)}$ and $e^{N\kappa_\Lambda} = t^{-\kappa_\Lambda/\kappa_\sigma} = t^{a(\omega)-1}$ by the definition $a(\omega) = 1 - \kappa_\Lambda/\kappa_\sigma$.
\end{proof}

\subsection{Proof of Theorem~\ref{thm:tail_lower} and of Corollaries~\ref{cor:either_or} and \ref{cor:no_power_law}}

\begin{proof}[Proof of Theorem~\ref{thm:tail_lower}]
Let $Z_N$ and $P_N$ be as in Lemma~\ref{lem:construction} and $g_0 := \sin\theta\,P_N(\cos\theta)$, a function of $V_N$ by Lemma~\ref{lem:chebyshev_variable} since $|Z_N| \le N-1$. Under $R$, $\theta \mapsto \pi - \theta$, $\sin\theta$ is invariant and $z \mapsto -z$, so $Rg_0 = (-1)^{N-1}g_0$ by the symmetry of $P_N$. Lemma~\ref{lem:correction}, with the sign $\pm = (-1)^{N-1}$ of $g_0$, gives $g := g_0 - \alpha\chi_N^{\pm} \in \Psi_N$; put $v := \tilde\Lop_P^{-1}g \in \Phi_N$ and let $\vect w_N$ be its coefficient vector, $v = \Phi\mat D\vect w_N$, which exists and is unique because $\{\varphi_j\}$ is a basis of $\Phi_N$ (Lemma~\ref{lem:dimensions}). Since $(\mat B\mat D\vect w_N)_m = (\Lop_Pv)(x_m) = g(x_m)$, the vector $\mat B_D\vect w_N$ consists of the values of $g$, i.e.\ of $g_0$, at the interior nodes.

\emph{Nodal values.} At the layer nodes $g_0$ vanishes. Every other interior node satisfies $\theta_m \in (\theta_s, \pi - \theta_s)$, so Lemma~\ref{lem:construction}(a) gives $|g_0(x_m)| \le |P_N(\cos\theta_m)| \le \exp(NE_c + C\log N)$, and $\norm{\mat B_D\vect w_N}_2 \le \exp(NE_c + C\log N)$.

\emph{Coefficient norm.} By Lemma~\ref{lem:construction}(a), $|g_0| \le \exp(NE_{\max} + C\log N)$ on $\tilde I$, hence $\norm{g_0}_{L^2(\tilde I)} \le \exp(NE_{\max} + C\log N)$. By Lemma~\ref{lem:correction}(ii) and \eqref{eq:alpha_bound}, $|\alpha|\norm{\chi_N^{\pm}}_{L^2(\tilde I)} \le N^C\norm{g_0}_{L^2(\tilde I)}$, so Lemma~\ref{lem:coefficient_norm} gives $\norm{\vect w_N}_2 \le C_B\sqrt N\bigl(\norm{g_0}_{L^2(\tilde I)} + |\alpha|\norm{\chi_N^{\pm}}_{L^2(\tilde I)}\bigr) \le \exp(NE_{\max} + C\log N)$.

\emph{Size on $I$.} Let $[x_a,x_b]$ be the image of $[\theta_0 + \Delta/4, \theta_0 + 3\Delta/4]$ under $x = \tilde L\theta/\pi - \delta$, an interval of length $h_b = \tilde L\Delta/(2\pi)$ inside $I$, between the layer nodes $\theta_1 = \theta_0$ and $\theta_2 = \theta_0 + \Delta$ (both lie in $S$ for $N \ge N_0$). Its points are at distance at least $\Delta/4$ from every node, at distance at least $\theta_0 - \max_k\beta_k > 0$ (maximum over the extension arcs, if there are any) from every extension lattice, and at distance at least $\theta_s - \theta_0 - \Delta$ from every centre lattice and from $\pi/2$; for $N \ge N_0$ these distances exceed the respective quarter spacings and $1/N$, so Lemma~\ref{lem:construction}(b) applies, and with \eqref{eq:E_modulus} for $|\vartheta - \theta_0| \le 3\Delta/4$ it gives $|P_N(\cos\vartheta)| \ge \exp\bigl(NE(\vartheta) - C\log N\bigr) \ge \exp\bigl(NE(\theta_0) - C\log N\bigr)$ on $[x_a,x_b]$. No point of $Z_N$ lies in $(\theta_1,\theta_2)$, so $P_N(\cos\theta)$ has constant sign on $[x_a, x_b]$, and so has $g_0$, with $|g_0| \ge \sin\theta_0\exp\bigl(NE(\theta_0) - C\log N\bigr) =: 2g_\ast$. On $I$ the correction satisfies $|\alpha\chi_N^{\pm}| \le \exp\bigl(N(E_{\max} - (\kappa_\sigma - \kappa_\Lambda)) + C\log N\bigr)$ by Lemma~\ref{lem:correction}(ii) and \eqref{eq:alpha_bound}, which is at most $g_\ast$ for $N \ge N_0$ because $E(\theta_0) - E_{\max} = \ell > -(\kappa_\sigma - \kappa_\Lambda)$. Hence $g = g_0 - \alpha\chi_N^{\pm}$ has the sign of $g_0$ and $|g| \ge g_\ast$ on $[x_a,x_b]$. As $v \in \Phi_N$ is smooth with $-v'' = \Lop_Pv = g$ on $\tilde I$, Lemma~\ref{lem:bump} gives $\norm{v}_{L^2(I)} \ge g_\ast h_b^{5/2}/(32\pi^2) \ge \exp\bigl(NE(\theta_0) - C\log N\bigr)$; in particular $v \ne 0$ and $\vect w_N \ne 0$.

\emph{Conclusion.} For $s \in [0,1]$ the three estimates give
\[
\begin{aligned}
\log\frac{\norm{\Phi\mat D\vect w_N}_{L^2(I)}}{\norm{\mat B_D\vect w_N}_2^{\,s}\norm{\vect w_N}_2^{\,1-s}} &\ge N\bigl[E(\theta_0) - sE_c - (1-s)E_{\max}\bigr] - C\log N \\
&= N(\ell + s\tau) - C\log N,
\end{aligned}
\]
since $E(\theta_0) - sE_c - (1-s)E_{\max} = (E(\theta_0) - E_{\max}) + s(E_{\max} - E_c)$.
\end{proof}

\begin{lemma}[Two Forms of a Sharp Power Law]
\label{lem:two_forms}
Let Assumption~\ref{ass:unisolvence} hold and let $b \in (0,1)$.
\begin{enumerate}
    \item[(i)] If $\gamma_t \le Ct^{\,b}$ for all $t > 0$, then $\Gamma_t \le C_bt^{\,b-1}$ for all $t > 0$, with $C_b := 2^b(1 - 4^{b-1})^{-1/2}C$.
    \item[(ii)] If $\gamma_t \le Ct^{\,b}$ and $\Gamma_t \le C't^{\,b-1}$ for all $t > 0$, then
    \[
    \norm{\Phi\mat D\vect w}_{L^2(I)} \le (C + C')\norm{\mat B_D\vect w}_2^{\,b}\norm{\vect w}_2^{\,1-b} \qquad (\vect w \in \mathbb R^n);
    \]
    conversely this inequality with a constant $C''$ in place of $C + C'$ implies $\gamma_t \le C''t^{\,b}$ and $\Gamma_t \le C''t^{\,b-1}$ for all $t > 0$.
\end{enumerate}
\end{lemma}
\begin{proof}
(i) Let $t > 0$ and $\vect w \in \mathcal V_t^\perp\setminus\{0\}$, and for $k \ge 0$ let $\vect w_k$ be the orthogonal projection of $\vect w$ onto the span of the $\vect v_i$ with $t2^k < \sigma_i \le t2^{k+1}$, so that $\vect w = \sum_k\vect w_k$, a finite sum. Since $\vect w_k \in \mathcal V_{t2^{k+1}}$, $\norm{\Phi\mat D\vect w_k}_{L^2(I)} \le C(t2^{k+1})^b\norm{\vect w_k}_2$, and $\norm{\mat B_D\vect w_k}_2 \ge t2^k\norm{\vect w_k}_2$; the vectors $\mat B_D\vect w_k$ are pairwise orthogonal, as they lie in the spans of disjoint sets of left singular vectors. By the Cauchy--Schwarz inequality and $\sum_{k\ge0}4^{k(b-1)} = (1 - 4^{b-1})^{-1}$,
\[
\norm{\Phi\mat D\vect w}_{L^2(I)} \le \sum_kC2^b(t2^k)^{b-1}\norm{\mat B_D\vect w_k}_2 \le C_bt^{b-1}\norm{\mat B_D\vect w}_2 .
\]
(ii) Let $\vect w \ne 0$; then $\mat B_D\vect w \ne 0$ because $\mat B_D$ is nonsingular under Assumption~\ref{ass:unisolvence}. Split $\vect w = \vect w_t + \vect w_t^\perp$ as in the proof of Theorem~\ref{thm:two_constant}. Since $\mat B_D\vect w_t$ and $\mat B_D\vect w_t^\perp$ are orthogonal, $\norm{\mat B_D\vect w_t^\perp}_2 \le \norm{\mat B_D\vect w}_2$, so $\norm{\Phi\mat D\vect w}_{L^2(I)} \le \gamma_t\norm{\vect w_t}_2 + \Gamma_t\norm{\mat B_D\vect w_t^\perp}_2 \le Ct^b\norm{\vect w}_2 + C't^{b-1}\norm{\mat B_D\vect w}_2$, and $t := \norm{\mat B_D\vect w}_2/\norm{\vect w}_2$ gives the first claim. The converse is immediate from the definitions \eqref{eq:two_constants}, since $\norm{\mat B_D\vect w}_2 \le t\norm{\vect w}_2$ on $\mathcal V_t$ and $\norm{\vect w}_2 \le \norm{\mat B_D\vect w}_2/t$ on $\mathcal V_t^\perp$.
\end{proof}

At the level of the two constants, Theorem~\ref{thm:tail_lower} has the following non-asymptotic form.

\begin{corollary}[Lower Bound at the Nodal Level]
\label{cor:either_or}
Let $\Lop = \Lop_P$ and let $\mathcal C$ be a configuration with $\ell := \ell(\mathcal C) > -(\kappa_\sigma - \kappa_\Lambda)$ and $\tau := \tau(\mathcal C)$. Then, with $N_0$ and $C$ as in Theorem~\ref{thm:tail_lower}, for $N \ge N_0$ and $t_N := e^{-N\tau}$,
\begin{equation}\label{eq:either_or}
\max\{\gamma_{t_N},\ t_N\Gamma_{t_N}\} \ \ge\ \tfrac12\exp\bigl(N\ell - C\log N\bigr).
\end{equation}
\end{corollary}
\begin{proof}[Proof of Corollary~\ref{cor:either_or}]
Let $\vect w := \vect w_N$ be the vector of Theorem~\ref{thm:tail_lower} and $t := t_N$. The cases $s = 0$ and $s = 1$ of \eqref{eq:tail_quotient} read $\norm{\vect w}_2 \le e^{-N\ell + C\log N}\norm{\Phi\mat D\vect w}_{L^2(I)}$ and $\norm{\mat B_D\vect w}_2 \le e^{-N(\ell+\tau) + C\log N}\norm{\Phi\mat D\vect w}_{L^2(I)}$. Split $\vect w = \vect w_t + \vect w_t^\perp$ with $\vect w_t \in \mathcal V_t$ and $\vect w_t^\perp \in \mathcal V_t^\perp$. Then $\norm{\vect w_t}_2 \le \norm{\vect w}_2$, and $\norm{\mat B_D\vect w_t^\perp}_2 \le \norm{\mat B_D\vect w}_2$ because $\mat B_D\vect w_t$ and $\mat B_D\vect w_t^\perp$ lie in the spans of the left singular vectors with $\sigma_i \le t$ and with $\sigma_i > t$, which are orthogonal. By \eqref{eq:two_constants},
\[
\norm{\Phi\mat D\vect w}_{L^2(I)} \le \gamma_t\norm{\vect w}_2 + \Gamma_t\norm{\mat B_D\vect w}_2 \le \bigl(\gamma_te^{-N\ell} + \Gamma_te^{-N(\ell+\tau)}\bigr)e^{C\log N}\norm{\Phi\mat D\vect w}_{L^2(I)} .
\]
Dividing by $\norm{\Phi\mat D\vect w}_{L^2(I)}$, which is positive because $\vect w \ne 0$, the $\varphi_j$ are linearly independent and a nonzero element of $V_N$, a trigonometric polynomial, does not vanish identically on $I$, and using $e^{-N\tau} = t$ gives $\gamma_te^{-N\ell} + t\Gamma_te^{-N\ell} \ge e^{-C\log N}$; the larger of the two terms is at least half the sum.
\end{proof}

\begin{corollary}[Configurations Refute the Sharp Power Law]
\label{cor:no_power_law}
Let $\Lop = \Lop_P$ and let $\mathcal C$ be a configuration with $\ell(\mathcal C) > -(\kappa_\sigma - \kappa_\Lambda)$ and $\ell(\mathcal C) + a(\omega)\tau(\mathcal C) > 0$. Then there are no $N_1$ and $C$ such that $\gamma_t \le Ct^{\,a(\omega)}$ for all $N \ge N_1$ and all $t > 0$; in particular the sharp power law \eqref{eq:sharp_law} fails.
\end{corollary}
\begin{proof}[Proof of Corollary~\ref{cor:no_power_law}]
Write $\ell := \ell(\mathcal C)$, $\tau := \tau(\mathcal C)$ and $a := a(\omega) \in (0,1)$, and let $N_0$ and $C_0$ be the constants of Theorem~\ref{thm:tail_lower}, so that for every $N \ge N_0$ there is a vector $\vect w_N \ne 0$ with \eqref{eq:tail_quotient}. Suppose that $\gamma_t \le Ct^{\,a}$ for all $N \ge N_1$ and all $t > 0$. For $N \ge \max\{N_0, N_1\}$, Lemma~\ref{lem:two_forms}, whose hypothesis holds by Proposition~\ref{prop:unisolvence_poisson}, gives $\Gamma_t \le C_at^{\,a-1}$ for all $t > 0$ and then $\norm{\Phi\mat D\vect w}_{L^2(I)} \le (C + C_a)\norm{\mat B_D\vect w}_2^{\,a}\norm{\vect w}_2^{\,1-a}$ for all $\vect w$, whereas \eqref{eq:tail_quotient} with $s = a$ shows that the quotient $\norm{\Phi\mat D\vect w_N}_{L^2(I)}/\bigl(\norm{\mat B_D\vect w_N}_2^{\,a}\norm{\vect w_N}_2^{\,1-a}\bigr)$ is at least $\exp\bigl(N(\ell + a\tau) - C_0\log N\bigr)$, which exceeds $C + C_a$ for all large $N$ because $\ell + a\tau > 0$.
\end{proof}

\subsection{The Upper Bound: Proof of Theorem~\ref{thm:dual}}
\label{sec:dual}

In this subsection $\Lop = \Lop_P$. Let $\mathcal S$ be an interpolation configuration with envelope $E$, and let $S_N$ be the set of $N' \ge N$ distinct points of Lemma~\ref{lem:lattice_lagrange}, the interior nodes of the node arcs selected by the residues together with lattices on the virtual arcs, with the Lagrange polynomials $\ell_s$, $s \in S_N$, of degree $N'-1$ in $z$, which satisfy \eqref{eq:lattice_lagrange} for $N \ge N_0$.

\begin{proof}[Proof of Theorem~\ref{thm:dual}]
\emph{Reduction.} We show that for $N \ge N_0$, every $t \ge 0$ and every $\vect w \in \mathbb R^n$ with $\norm{\vect w}_2 \le 1$ and $\norm{\mat B_D\vect w}_2 \le t$,
\begin{equation}\label{eq:dual_core}
\norm{\Phi\mat D\vect w}_{L^2(I)} \le N^C\bigl(t\,e^{NA_1(\mathcal S)} + e^{NA_2(\mathcal S)}\bigr).
\end{equation}
This gives \eqref{eq:dual_bound}: a unit vector $\vect w \in \mathcal V_t$ satisfies $\norm{\mat B_D\vect w}_2 \le t$, so $\gamma_t$ is at most the right-hand side; and for $\vect w \in \mathcal V_t^\perp\setminus\{0\}$ and $t > 0$ one has $\norm{\mat B_D\vect w}_2 \ge t\norm{\vect w}_2 > 0$, as in the proof of Proposition~\ref{prop:apriori}, so $\vect w' := t\vect w/\norm{\mat B_D\vect w}_2$ satisfies $\norm{\vect w'}_2 \le 1$ and $\norm{\mat B_D\vect w'}_2 = t$, and \eqref{eq:dual_core} for $\vect w'$ reads $t\norm{\Phi\mat D\vect w}_{L^2(I)}/\norm{\mat B_D\vect w}_2 \le N^C(te^{NA_1(\mathcal S)} + e^{NA_2(\mathcal S)})$, whence the bound on $t\Gamma_t$; for $t = 0$ the term $t\Gamma_t$ vanishes. Estimate \eqref{eq:dual_rate} follows from \eqref{eq:dual_bound} with $t = e^{-N\tau}$, since then $te^{NA_1(\mathcal S)} + e^{NA_2(\mathcal S)} \le 2e^{N\Phi(\mathcal S,\tau)}$.

\emph{Norms.} Let $\vect w$ be as in \eqref{eq:dual_core}, $v := \Phi\mat D\vect w \in \Phi_N$ and $\psi := \Lop v \in \Psi_N \subset V_N$. Since $v \in C^2(\bar I)$ vanishes at $0$ and $L$ and $\Lop v = \psi$, Lemma~\ref{lem:dirichlet_box}(iii) with $d = 1$ gives $\norm{v}_{L^2(I)} \le (L/\pi)^2\norm{\psi}_{L^2(I)} \le (L/\pi)^2\sqrt L\,\norm{\psi}_{L^\infty(I)}$. The values of $\psi$ at the interior nodes are the entries of $\mat B_D\vect w$, so $|\psi(x_m)| \le t$ for $2 \le m \le N-1$. By Lemma~\ref{lem:test_norm}, $\norm{\psi}_{L^2(\tilde I)} \le C_\Psi\norm{\vect w}_2/\sqrt N \le C_\Psi/\sqrt N$, and since $\psi = \sum_{j=1}^Nc_jw_j$ with $\sum_jc_j^2 = \norm{\psi}^2_{L^2(\tilde I)}$ and $|w_j| \le \sqrt{2/\tilde L}$, the Cauchy--Schwarz inequality gives $|\psi(x)| \le \sqrt{2N/\tilde L}\,\norm{\psi}_{L^2(\tilde I)} \le C$ for all $x \in \tilde I$.

\emph{Interpolation.} By Lemma~\ref{lem:chebyshev_variable}, $\psi(x) = \sin\theta\,p(\cos\theta)$ with $p \in \mathbb P_{N-1}$ and $\theta = \pi(x+\delta)/\tilde L$. At a node $s = \theta_m$ of $S_N$, which is an interior node, $|p(\cos s)| = |\psi(x_m)|/\sin\theta_m \le t/\sin\theta_0$; at a point $s$ of a virtual lattice, $s \in [\xi,\pi-\xi]$ and $|p(\cos s)| = |\psi(x)|/\sin s \le C/\sin\xi$. Since $S_N$ consists of $N' \ge N$ distinct points and $\deg p \le N-1 \le N'-1$, Lagrange interpolation reproduces $p$: $p = \sum_{s\in S_N}p(\cos s)\,\ell_s$. For $\vartheta \in [\theta_0,\pi-\theta_0]$ we have $E(\vartheta) \le E_I$, and $E(s) \ge E_1$ for the nodes of $S_N$ and $E(s) \ge E_2$ for the virtual points, because the former lie in the node arcs and the latter in the virtual arcs; so \eqref{eq:lattice_lagrange} gives $|\ell_s(\cos\vartheta)| \le N^Ce^{NA_1(\mathcal S)}$ for the nodes and $|\ell_s(\cos\vartheta)| \le N^Ce^{NA_2(\mathcal S)}$ for the virtual points. Summing the $N' \le N + C$ terms,
\[
|\psi(x)| \le |p(\cos\vartheta)| \le N^C\Bigl(\frac{t}{\sin\theta_0}\,e^{NA_1(\mathcal S)} + \frac{C}{\sin\xi}\,e^{NA_2(\mathcal S)}\Bigr), \qquad x \in \bar I,\ \vartheta = \pi(x+\delta)/\tilde L,
\]
and the bound on $\norm{v}_{L^2(I)}$ of the previous step gives \eqref{eq:dual_core}.
\end{proof}

\begin{corollary}[Certified Levels of the Sharp Exponent]
\label{cor:certified}
Under the hypotheses of Theorem~\ref{thm:dual} let $a := a(\omega)$, which lies in $(0,1)$, and suppose that $A_1(\mathcal S)/(1-a) \le -A_2(\mathcal S)/a$. Then for every $\tau$ with
\begin{equation}\label{eq:certified_interval}
\frac{A_1(\mathcal S)}{1-a} \le \tau \le \frac{-A_2(\mathcal S)}{a},
\end{equation}
every $N \ge N_0$ and $t = e^{-N\tau}$, one has $\gamma_t \le 2N^Ct^{\,a}$ and $\Gamma_t \le 2N^Ct^{\,a-1}$, with $N_0$ and $C$ from Theorem~\ref{thm:dual}.
\end{corollary}
\begin{proof}[Proof of Corollary~\ref{cor:certified}]
By Lemma~\ref{lem:clausen}, $U_\omega(1) < U_\omega(c_0) < U_\omega(0)$, so $0 < \kappa_\Lambda < \kappa_\sigma$ and $a \in (0,1)$. For $\tau$ in \eqref{eq:certified_interval}, $A_1(\mathcal S) - \tau \le -a\tau$ because $\tau \ge A_1(\mathcal S)/(1-a)$, and $A_2(\mathcal S) \le -a\tau$ because $\tau \le -A_2(\mathcal S)/a$; hence $\Phi(\mathcal S,\tau) \le -a\tau$, and \eqref{eq:dual_bound} with $t = e^{-N\tau}$ gives $\max\{\gamma_t, t\Gamma_t\} \le 2N^Ce^{-Na\tau} = 2N^Ct^{\,a}$.
\end{proof}

\subsection{The Tail Law: Proofs of Theorem~\ref{thm:tail_law}, Proposition~\ref{prop:tail_properties} and Corollaries~\ref{cor:lebesgue_rate} and \ref{cor:sharpness}}
\label{sec:tail_law_proof}

The tail law is the value of an extremal problem over zero distributions, the continuous form of the supremum in \eqref{eq:sharpness}; its properties and the sharpness of the explicit bounds follow from this problem without the frame.

\begin{proposition}[An Extremal Problem for the Tail]
\label{prop:extremal}
For $0 < \tau < \kappa_\sigma$ let
\begin{equation}\label{eq:reduced_problem}
m_\omega(\tau) := \sup\Bigl\{\max_{[-c_0,c_0]}\bigl(b - U^{\nu}\bigr)\Bigr\},
\end{equation}
the supremum being over all closed sets $F \subset [-c_0,c_0]$, constants $b \in \mathbb R$ and compactly supported positive measures $\nu$ on $\mathbb R$ with $\nu(\mathbb R) \le 1$ such that
\begin{equation}\label{eq:reduced_constraints}
b - U^\nu \le 0 \ \text{ on } [-1,1], \qquad b - U^\nu \le -\tau \ \text{ on } F, \qquad \nu \ge \mu_\omega \ \text{ on } [-c_0,c_0]\setminus F .
\end{equation}
Then $m_\omega(\tau) = r_\omega(\tau)$.
\end{proposition}

Here $\nu$ stands for the limiting distribution of the zeros of a polynomial of degree $N$, the mass $1 - \nu(\mathbb R)$ escaping to infinity, and $e^{N(b - U^\nu)}$ for its size, normalized to be at most one on the extended interval and at most $e^{-N\tau}$ on the free set $F$; on the saturated set $[-c_0,c_0]\setminus F$ the polynomial has a zero at every node. Problem \eqref{eq:reduced_problem} is a two-level variant of the extremal problem by which Rakhmanov \cite[Sec.~4.2]{rakhmanov2007bounds} describes the polynomials of degree $n$ bounded by one at $N > n$ given points: there the size of a polynomial at a point is maximized relative to its maximum on the set where its zero distribution lies below the distribution of the points, over zero distributions not exceeding that distribution, and the maximizer is the equilibrium measure with this upper constraint \cite[Thm.~3]{rakhmanov1996equilibrium}, see also \cite{dragnev1997constrained, kuijlaars1999extremal}. Here the polynomial is in addition normalized on the whole extended interval, which carries no nodes outside $[-c_0,c_0]$, and required to be smaller by the factor $e^{-N\tau}$ on the free set, and the maximizer is again an equilibrium measure with an upper constraint, now in the external field of Definition~\ref{def:equilibrium}. For the least-squares Fourier extension from equispaced samples, the constrained equilibrium measure is saturated next to the ends of the physical interval and free in its centre \cite{geronimo2020fourier}, the structure of $\lambda_\tau$ in Proposition~\ref{prop:equilibrium}(iii).

\begin{proof}[Proof of Proposition~\ref{prop:extremal}]
Let $\lambda := \lambda_\tau$, identified with its image on $[-1,1]$, and $\theta_\tau$ as in Proposition~\ref{prop:equilibrium}(iii). The triple $F := \cos([\theta_\tau, \pi - \theta_\tau])$, $b := -E_\tau^\ast$, $\nu := \lambda$ is admissible: $b - U^\nu = E_\tau - E_\tau^\ast$ is at most $0$ on $[-1,1]$ and equals $-\tau$ on $F$ by Proposition~\ref{prop:equilibrium}(i) and (iii), and $\nu = \mu_\omega$ on $[-c_0,c_0]\setminus F$, the image of the layer. Its value is $\max_{\Theta_I}E_\tau - E_\tau^\ast = r_\omega(\tau)$, so $m_\omega(\tau) \ge r_\omega(\tau)$. Conversely, let $(F, b, \nu)$ be admissible, and let $K$ be the union of the supports of $\lambda$ restricted to $[-1,1]\setminus[-c_0,c_0]$ and to $F$, a compact set that is nonempty because $\lambda_\tau(\Theta_E) > 0$. Since $\nu \ge \mu_\omega \ge \lambda$ on $[-c_0,c_0]\setminus F$, the positive part of $\lambda - \nu$ is carried by $K$. On the first support $-U^\lambda = E_\tau^\ast$ by Proposition~\ref{prop:equilibrium}(ii) and $U^\nu \ge b$; on the second $U^\nu \ge b + \tau$ and $-U^\lambda \ge E_\tau^\ast - \tau$ by Proposition~\ref{prop:equilibrium}(i). So $U^\nu - U^\lambda \ge b + E_\tau^\ast$ on $K$, and Lemma~\ref{lem:minimum} with $(\mu_1,\mu_2) = (\lambda,\nu)$, applicable by Lemma~\ref{lem:env}(d), extends this to $\mathbb C$. On $[-c_0,c_0]$ it gives $b - U^\nu \le -U^\lambda - E_\tau^\ast \le r_\omega(\tau)$ by Proposition~\ref{prop:equilibrium}(iii), so $m_\omega(\tau) \le r_\omega(\tau)$.
\end{proof}

\begin{proof}[Proof of Proposition~\ref{prop:tail_properties}]
\emph{Monotonicity and the Lipschitz bound.} Let $0 < \tau \le \tau' < \kappa_\sigma$. A triple that is admissible in \eqref{eq:reduced_problem} at the level $\tau'$ is admissible at the level $\tau$; and if $(F, b, \nu)$ is admissible at the level $\tau$, then $(F, b - (\tau' - \tau), \nu)$ is admissible at the level $\tau'$, with a value smaller by $\tau' - \tau$. Hence $m_\omega(\tau') \le m_\omega(\tau) \le m_\omega(\tau') + \tau' - \tau$, and Proposition~\ref{prop:extremal} gives the same inequalities for $r_\omega$.

\emph{The bounds.} Proposition~\ref{prop:equilibrium}(i) gives $-\tau \le r_\omega(\tau) \le 0$. Below, $\lambda := \lambda_\tau$ and $\mu_\omega$ are identified with their images on $[-1,1]$; both potentials are continuous by Lemma~\ref{lem:env}(d), and $U^{\mu_\omega} = U_\omega$. For the upper bound $\kappa_\Lambda - \tau$ let $K'$ be the image of the support of $(\omega^{-1}d\theta - \lambda)|_{\Theta_I}$; it is nonempty because $\lambda(\Theta_I) < 1$ by Proposition~\ref{prop:equilibrium}(iii), and it carries the positive part of $\mu_\omega - \lambda$, the image of $(\omega^{-1}d\theta - \lambda)|_{\Theta_I}$. On $K'$, $E_\tau = E_\tau^\ast - \tau$ by Proposition~\ref{prop:equilibrium}(ii) and $U_\omega \le U_\omega(0)$ by Lemma~\ref{lem:clausen}, so that $u := U^\lambda - U_\omega = -E_\tau - U_\omega \ge \tau - E_\tau^\ast - U_\omega(0)$. Lemma~\ref{lem:minimum} with $(\mu_1,\mu_2) = (\mu_\omega,\lambda)$, whose masses are equal, extends this to $\mathbb C$; at $c_0$ it reads $r_\omega(\tau) = E_\tau(\theta_0) - E_\tau^\ast \le U_\omega(0) - U_\omega(c_0) - \tau = \kappa_\Lambda - \tau$ by \eqref{eq:rates}.

For the strict lower bound let $K$ be the support of $\lambda|_{\Theta_E}$, nonempty by Proposition~\ref{prop:equilibrium}(iii), and apply Lemma~\ref{lem:minimum} with $(\mu_1,\mu_2) = (\lambda,\mu_\omega)$: the masses are equal, and the positive part of $\lambda - \mu_\omega$ is carried by $K$ because $\lambda \le \mu_\omega$ on $\Theta_I$. On $K$, $u := U^{\mu_\omega} - U^\lambda = U_\omega + E_\tau \ge U_\omega(1) + E_\tau^\ast$ by Proposition~\ref{prop:equilibrium}(ii) and Lemma~\ref{lem:clausen}. Hence $u \ge U_\omega(1) + E_\tau^\ast$ on $\mathbb C$, and at $c_0$ this reads $r_\omega(\tau) = E_\tau(\theta_0) - E_\tau^\ast \ge U_\omega(1) - U_\omega(c_0) = \kappa_\Lambda - \kappa_\sigma$ by \eqref{eq:rates}. Suppose equality held. Then $c_0 \notin K$, since otherwise $E_\tau(\theta_0) = E_\tau^\ast$ and $r_\omega(\tau) = 0 > \kappa_\Lambda - \kappa_\sigma$; so $u(c_0) = U_\omega(1) + E_\tau^\ast$ with $c_0 \notin K$, and Lemma~\ref{lem:minimum} gives $u = 0$ on $\mathbb C\setminus K$, hence on $\mathbb C$ by continuity. Then $E_\tau = E_{\mu_\omega}$, and $\lambda_\tau = \mu_\omega$ by Lemma~\ref{lem:cheb}, which contradicts $\lambda_\tau(\Theta_E) > 0$. The limits follow from \eqref{eq:tail_law_bounds}.
\end{proof}

\begin{lemma}[Monotonicity of the Rates]
\label{lem:monotone}
Put $m_N(t) := \max\{\gamma_t,\ t\Gamma_t\}$ for $t > 0$. Then $m_N(t) > 0$, so that $r^\pm(\tau) \in [-\infty,\infty]$ are well defined for $\tau \ge 0$, and for $0 \le \tau \le \tau'$
\begin{equation}\label{eq:monotone}
r^\pm(\tau') \le r^\pm(\tau) \le r^\pm(\tau') + \tau' - \tau .
\end{equation}
\end{lemma}
\begin{proof}
Since $\Phi\mat D\vect w \ne 0$ in $L^2(I)$ for $\vect w \ne 0$ (proof of Corollary~\ref{cor:either_or}), $\gamma_t > 0$ if $\mathcal V_t \ne \{0\}$, and otherwise $\mathcal V_t^\perp = \mathbb R^n$ and $t\Gamma_t > 0$. For $t > 0$ let $Q_N(t) := \sup\{\norm{\Phi\mat D\vect w}_{L^2(I)} : \norm{\vect w}_2 \le 1,\ \norm{\mat B_D\vect w}_2 \le t\}$. The reduction step in the proof of Theorem~\ref{thm:dual} shows $m_N(t) \le Q_N(t)$, and for $\vect w$ as in the definition of $Q_N(t)$ the splitting in the proof of Corollary~\ref{cor:either_or} gives $\norm{\Phi\mat D\vect w}_{L^2(I)} \le \gamma_t + t\Gamma_t \le 2m_N(t)$; so $Q_N(t)/2 \le m_N(t) \le Q_N(t)$. The function $Q_N$ is nondecreasing, and $Q_N(t) \ge (t/t')Q_N(t')$ for $0 < t \le t'$, because $\vect w \mapsto (t/t')\vect w$ maps the admissible set of $Q_N(t')$ into that of $Q_N(t)$. With $t = e^{-N\tau'}$ and $t' = e^{-N\tau}$ this gives $m_N(e^{-N\tau'}) \le 2m_N(e^{-N\tau})$ and $m_N(e^{-N\tau'}) \ge \tfrac12e^{-N(\tau'-\tau)}m_N(e^{-N\tau})$, and \eqref{eq:monotone} follows.
\end{proof}

\begin{lemma}[Discretization of the Equilibrium Measure]
\label{lem:discretize}
Let $0 < \tau < \kappa_\sigma$ and $\eps > 0$. There are a configuration $\mathcal C$ with $\tau(\mathcal C) \ge \tau - \eps$ and $\ell(\mathcal C) \ge r_\omega(\tau) - \eps$, and an interpolation configuration $\mathcal S$ with $\Phi(\mathcal S,\tau) \le r_\omega(\tau) + \eps$.
\end{lemma}
\begin{proof}
Let $\lambda := \lambda_\tau = f\,d\theta$, $E := E_\tau$, $E^\ast := E_\tau^\ast$, and $\theta_\tau$ as in Proposition~\ref{prop:equilibrium}(iii). For $0 < \eta \le \theta_0/2$ with $\eta < \min\{1/4, \omega/2\}$ let $\mathcal P$ be a partition of $[0,\pi]$ into arcs with pairwise disjoint interiors and lengths in $[\eta/2,\eta]$, invariant under $\vartheta \mapsto \pi - \vartheta$, of which $\theta_0$, $\pi/2$ and $\pi - \theta_0$ are endpoints; it exists because $[0,\theta_0]$ and $[\theta_0,\pi/2]$ have lengths at least $\eta$. For $A \in \mathcal P$ put $f_A := \lambda(A)/|A| \le 1/\omega$, and let $\bar\lambda$ have density $f_A$ on each $A \in \mathcal P$. By Lemma~\ref{lem:env}, $|E_{\bar\lambda} - E| \le 8\omega^{-1}\Omega(\eta)$ on $[0,\pi]$, $|E(\vartheta) - E(\vartheta')| \le \omega^{-1}\Omega(\eta)$ for $|\vartheta - \vartheta'| \le \eta$, and $|E_g| \le MC_0|B|^{1/2}$ on $[0,\pi]$ for every $g$ with $|g| \le M$ that vanishes outside a Borel set $B$. Below, $\eps_1, \eps_2, \dots$ denote quantities that depend only on $\eta$, $\omega$ and $\tau$ and tend to $0$ as $\eta \to 0$.

\emph{The configuration.} The set $K := \{\vartheta \in [\theta_0,\pi/2] : E(\vartheta) + \tau \ge E^\ast + \eps/4\}$ is compact and, as $E + \tau = E^\ast$ on $[\theta_\tau,\pi/2]$, contained in $[\theta_0,\theta_\tau)$. If $\theta_\tau > \theta_0$, let $\eta$ be so small that some endpoint $\theta_s$ of $\mathcal P$ satisfies $\max(K \cup\{\theta_0\}) < \theta_s < \theta_\tau$; if $\theta_\tau = \theta_0$, let $\theta_s$ be the endpoint of $\mathcal P$ next to $\theta_0$ on its right. Let $\xi \le \eta$ be the smallest positive endpoint and $\theta_0 - \zeta$, $\zeta \le \eta$, the endpoint next to $\theta_0$ on its left. The configuration $\mathcal C$ has the layer parameter $\theta_s$, the extension arcs $A \in \mathcal P$ with $A \subset [\xi,\theta_0-\zeta]$ and $f_A > 0$, the centre arcs $A \in \mathcal P$ with $A \subset [\theta_s,\pi/2]$ and $f_A > 0$, and on these arcs the densities $(1 - \eta')f_A$ with $\eta' := \eta(1 + 2/\omega)$; for $\eta' < 1$ it satisfies Definition~\ref{def:configuration}, its total mass being at most $\max\{\lambda(S^\ast) + (1-\eta')(1 - \lambda(S^\ast)),\ 2\eta/\omega + 1 - \eta'\} < 1$. Its envelope is $E_\nu$ for the measure $\nu$ of density $1/\omega$ on $S^\ast$ and $(1-\eta')f_A$ on the arcs. The density of $\nu - \bar\lambda$ vanishes on the layer if $\theta_s < \theta_\tau$, where $f = 1/\omega$, is bounded by $1/\omega$ on the layer if $\theta_\tau = \theta_0$, where $|S^\ast| \le 2\eta$, and on $[0,\xi) \cup (\theta_0 - \zeta,\theta_0)$ and its mirror image, of measure at most $4\eta$, and is bounded by $\eta'/\omega$ elsewhere. Hence $|E_\nu - E| \le \eps_1 := 8\omega^{-1}\Omega(\eta) + \omega^{-1}C_0\bigl((6\eta)^{1/2} + \eta'\sqrt\pi\bigr)$ on $[0,\pi]$. Since $E^\ast = \max E$, $E \le E^\ast - \tau + \eps/4$ on $[\theta_s,\pi-\theta_s]$ and $E(\theta_0) - E^\ast = r_\omega(\tau)$, this gives $\tau(\mathcal C) \ge \tau - \eps/4 - 2\eps_1$ and $\ell(\mathcal C) \ge r_\omega(\tau) - 2\eps_1$, and $2\eps_1 \le 3\eps/4$ for small $\eta$.

\emph{The interpolation configuration.} Let $q \in \mathbb N$ with $q \ge 1/\eta$, and shrink every arc of $\mathcal P$ by $\eta^2$ at both ends. For each $A \in \mathcal P$ with $A \subset \Theta_I$ let $r_A := \lfloor(1-\eta)q\omega f_A\rfloor \in \{0,\dots,q-1\}$; the shrunk arcs with $r_A \ge 1$ are the node arcs, with modulus $q$, residues $\{0,\dots,r_A - 1\}$ and density $\rho_A = r_A/(q\omega) \in [(1-\eta)f_A - \eta/\omega,\ (1-\eta)f_A]$. There are node arcs for small $\eta$: by Lemma~\ref{lem:cheb}, $J_0(\lambda) \ge \log 2 = J_0(\pi^{-1}d\theta)$, and $J_\tau(\lambda) \le J_\tau(\pi^{-1}d\theta)$ then gives $\lambda(\Theta_I) \ge \omega/\pi$, so $f_A \ge 1/\pi$ for some $A \subset \Theta_I$, and $r_A \ge 1$ for this $A$ once $(1 - \eta)\omega \ge \pi\eta$. The shrunk arcs of the $A \in \mathcal P$ with $A \subset \overline{\Theta_E}\cap[\xi,\pi-\xi]$ and $f_A > 0$ are the virtual arcs, with density $(1 + \varkappa)f_A$, where $\varkappa \ge 0$ makes the total mass one. Since $\lambda(\Theta_E) > 0$ by Proposition~\ref{prop:equilibrium}(iii), virtual arcs exist for small $\eta$ and $\varkappa \le \eps_2$, because the mass removed by the floors, which lower the density by at most $2\eta/\omega$ on $\Theta_I$, by the shrinking, which removes at most $4\pi\eta$ of length, and with the arcs in $[0,\xi)\cup(\pi-\xi,\pi]$ is at most $2\eta + (4\pi + 2)\eta/\omega$, while the mass of $\lambda$ on the virtual arcs tends to $\lambda(\Theta_E)$. The arcs are pairwise disjoint, so $\mathcal S$ is an interpolation configuration, and as before its envelope satisfies $|E_{\mathcal S} - E| \le \eps_3$ on $[0,\pi]$. By Proposition~\ref{prop:equilibrium}(i), $E \ge E^\ast - \tau$ everywhere, so the minimum of $E_{\mathcal S}$ over the node arcs is at least $E^\ast - \tau - \eps_3$. Every virtual arc lies in some $A \in \mathcal P$ with $A \subset \overline{\Theta_E}$ and $\lambda(A) > 0$, hence within $\eta$ of the support of $\lambda|_{\Theta_E}$, where $E = E^\ast$ by Proposition~\ref{prop:equilibrium}(ii); so the minimum of $E_{\mathcal S}$ over the virtual arcs is at least $E^\ast - \omega^{-1}\Omega(\eta) - \eps_3$. Finally $\max_{\Theta_I}E_{\mathcal S} \le E(\theta_0) + \eps_3$ by Proposition~\ref{prop:equilibrium}(iii). Therefore $A_1(\mathcal S) - \tau \le r_\omega(\tau) + 2\eps_3$ and $A_2(\mathcal S) \le r_\omega(\tau) + 2\eps_3 + \omega^{-1}\Omega(\eta)$, so that $\Phi(\mathcal S,\tau) \le r_\omega(\tau) + \eps$ for small $\eta$.
\end{proof}

\begin{proof}[Proof of Theorem~\ref{thm:tail_law}]
Fix $\tau \in (0,\kappa_\sigma)$, let $m_N$ be as in Lemma~\ref{lem:monotone}, and let $\eps > 0$.

\emph{Upper bound.} For the interpolation configuration $\mathcal S$ of Lemma~\ref{lem:discretize}, \eqref{eq:dual_rate} gives $N^{-1}\log m_N(t_N) \le \Phi(\mathcal S,\tau) + C\log N/N$ for $N \ge N_0$, so $r^+(\tau) \le \Phi(\mathcal S,\tau) \le r_\omega(\tau) + \eps$.

\emph{Lower bound.} Let $0 < \eps < \min\{\tau,\ r_\omega(\tau) - (\kappa_\Lambda - \kappa_\sigma)\}$, which is possible by Proposition~\ref{prop:tail_properties}, and let $\mathcal C$ be the configuration of Lemma~\ref{lem:discretize}; then $\ell(\mathcal C) \ge r_\omega(\tau) - \eps > -(\kappa_\sigma - \kappa_\Lambda)$. By Corollary~\ref{cor:either_or}, $m_N(e^{-N\tau(\mathcal C)}) \ge \tfrac12\exp(N\ell(\mathcal C) - C\log N)$ for $N \ge N_0$, so $r^-(\tau(\mathcal C)) \ge \ell(\mathcal C)$. As $\tau - \eps \le \tau(\mathcal C)$, Lemma~\ref{lem:monotone} gives $r^-(\tau) \ge r^-(\tau - \eps) - \eps \ge r^-(\tau(\mathcal C)) - \eps \ge r_\omega(\tau) - 2\eps$.

Since $\eps$ is arbitrary, $r^-(\tau) \ge r_\omega(\tau) \ge r^+(\tau)$, which proves \eqref{eq:tail_limit}.
\end{proof}

\begin{proof}[Proof of Corollary~\ref{cor:lebesgue_rate}]
Let $N \ge N_0$ of Proposition~\ref{prop:unisolvence_poisson}, so that $\mat B_D$ is nonsingular and $\mathcal V_0 = \{0\}$. For $t > 0$ and $\vect w \in \mathcal V_t$, $\norm{\Phi\mat D\vect w}_{L^2(I)} \le \Gamma_0\norm{\mat B_D\vect w}_2 \le t\Gamma_0\norm{\vect w}_2$, and for $\vect w \in \mathcal V_t^\perp$, $\norm{\Phi\mat D\vect w}_{L^2(I)} \le \Gamma_0\norm{\mat B_D\vect w}_2$; hence $\max\{\gamma_t, t\Gamma_t\} \le t\Gamma_0$. For $\tau \in (0,\kappa_\sigma)$ and $t = t_N = e^{-N\tau}$, Theorem~\ref{thm:tail_law} therefore gives $\liminf_{N\to\infty}N^{-1}\log\Gamma_0 \ge r_\omega(\tau) + \tau$, and letting $\tau \to \kappa_\sigma$, Proposition~\ref{prop:tail_properties} gives $\liminf_{N\to\infty}N^{-1}\log\Gamma_0 \ge \kappa_\Lambda$. Conversely, Proposition~\ref{prop:classical}(i) with $t = 0$ and $\Lambda^\Psi_{N,2} \le \sqrt L\,\Lambda_N^\Psi$ give $\Gamma_0 \le (L^2/\pi^2)\sqrt{nL}\,\Lambda_N^\Psi$, and Proposition~\ref{prop:classical}(iii) gives $\Lambda_N^\Psi \le N^Ce^{N\kappa_\Lambda}$. Hence
\[
\kappa_\Lambda \le \liminf_{N\to\infty}\frac{\log\Gamma_0}{N} \le \liminf_{N\to\infty}\frac{\log\Lambda_N^\Psi}{N} \le \limsup_{N\to\infty}\frac{\log\Lambda_N^\Psi}{N} \le \kappa_\Lambda ,
\]
and $\limsup_{N\to\infty}N^{-1}\log\Gamma_0 \le \limsup_{N\to\infty}N^{-1}\log\Lambda_N^\Psi = \kappa_\Lambda$.
\end{proof}

\begin{proof}[Proof of Corollary~\ref{cor:sharpness}]
Let $\tau \in (0,\kappa_\sigma)$ and $\eps > 0$ with $\tau + \eps < \kappa_\sigma$.

\emph{Configurations.} A configuration $\mathcal C$ with $\tau(\mathcal C) \ge \tau$ gives a triple that is admissible in \eqref{eq:reduced_problem} at the level $\tau$: $\nu$ the image of the measure of density $1/\omega$ on $S^\ast$ and $\rho_k$ on $A_k^\ast$, of mass $m(\mathcal C) < 1$, $F$ the image of $[\theta_s, \pi - \theta_s]$, and $b := -E_{\max}$. Indeed $b - U^\nu = E - E_{\max}$ is at most $0$ on $[-1,1]$ and at most $-\tau(\mathcal C) \le -\tau$ on $F$, and $\nu = \mu_\omega$ on $[-c_0,c_0]\setminus F$, the image of the layer. The value of the triple is at least $E(\theta_0) - E_{\max} = \ell(\mathcal C)$, so $\ell(\mathcal C) \le r_\omega(\tau)$ by Proposition~\ref{prop:extremal}. Conversely, Lemma~\ref{lem:discretize} at the level $\tau + \eps$ gives a configuration with $\tau(\mathcal C) \ge \tau$ and $\ell(\mathcal C) \ge r_\omega(\tau + \eps) - \eps \ge r_\omega(\tau) - 2\eps$, which exceeds $-(\kappa_\sigma - \kappa_\Lambda)$ for small $\eps$, by Proposition~\ref{prop:tail_properties}.

\emph{Interpolation configurations.} Let $\mathcal S$ be an interpolation configuration with envelope $E$, let $\sigma$ be the image of the measure of density $\rho_k$ on $A_k$, $1 \le k \le Q$, a probability measure with $-U^\sigma(\cos\vartheta) = E(\vartheta)$, and let $(F, b, \nu)$ be admissible at the level $\tau$. Let $K$ be the union of the images of the virtual arcs and of the intersections of $F$ with the images of the node arcs, a nonempty compact set. The node arcs lie in $\Theta_I$ and carry densities at most $1/\omega$, so $\sigma \le \mu_\omega \le \nu$ on their images outside $F$, and the positive part of $\sigma - \nu$ is carried by $K$. On the images of the virtual arcs, $U^\nu \ge b$ and $-U^\sigma \ge E_2$; on the images of the node arcs in $F$, $U^\nu \ge b + \tau$ and $-U^\sigma \ge E_1$; hence $U^\nu - U^\sigma \ge b + \min\{E_2,\ E_1 + \tau\}$ on $K$, and Lemma~\ref{lem:minimum} with $(\mu_1,\mu_2) = (\sigma,\nu)$, applicable by Lemma~\ref{lem:env}(d), extends this to $\mathbb C$. On $[-c_0,c_0]$ it gives
\[
b - U^\nu \le E - \min\{E_2,\ E_1 + \tau\} \le E_I - \min\{E_2,\ E_1 + \tau\} = \Phi(\mathcal S,\tau),
\]
so that $r_\omega(\tau) = m_\omega(\tau) \le \Phi(\mathcal S,\tau)$ by Proposition~\ref{prop:extremal}. Conversely, Lemma~\ref{lem:discretize} gives $\mathcal S$ with $\Phi(\mathcal S,\tau) \le r_\omega(\tau) + \eps$.
\end{proof}

\subsection{Duality: Proof of Theorem~\ref{thm:tail_duality} and Corollary~\ref{cor:self_dual}}

Let $\omega' := \pi - \omega$ and $c := \omega/\omega'$. The arcs of $\omega'$ are $\Theta_I' := [\theta_0',\pi-\theta_0']$, with $\theta_0' = \omega/2 = \pi/2 - \theta_0$, and $\Theta_E' := [0,\pi]\setminus\Theta_I'$, and Definition~\ref{def:equilibrium} and Proposition~\ref{prop:equilibrium} with $\omega'$ in place of $\omega$ define $\mathscr A_{\omega'}$, $J'_\tau$ and the minimizers $\lambda'_\tau$. For a density $f$ on $[0,\pi]$ that is symmetric about $\pi/2$ let $(\mathcal Rf)(\varphi) := f(|\varphi - \pi/2|)$; if $[0,\pi]$ is viewed as half of the unit circle, $\mathcal R$ is the rotation by $\pi/2$. It maps symmetric densities to symmetric densities and fixes constants, and $(\mathcal Rf)(\varphi)$ is a value of $f$ on $\Theta_I$ if $\varphi \in \overline{\Theta_E'}$ and on $\Theta_E$ if $\varphi$ lies in the interior of $\Theta_I'$. For symmetric $f$, $\hat f_k = 0$ for odd $k$, and the substitution $\varphi = \pi/2 \pm \psi$ gives $\widehat{\mathcal Rf}_k = \cos(k\pi/2)\,\hat f_k$ for all $k \ge 0$; by Lemma~\ref{lem:cheb}, therefore, $J_0(\mathcal Rf) = J_0(f)$ and $E_{\mathcal Rf}(\vartheta) = E_f(|\vartheta - \pi/2|)$.

\begin{proof}[Proof of Theorem~\ref{thm:tail_duality}]
For symmetric $\lambda = f\,d\theta \in \mathscr A_\omega$ put $T\lambda := c\,(\omega^{-1} - \mathcal Rf)\,d\theta$. Its density lies in $[0,c/\omega] = [0,1/\omega']$, its mass is $c(\pi/\omega - 1) = 1$, and it is symmetric, so $T$ maps the symmetric elements of $\mathscr A_\omega$ onto those of $\mathscr A_{\omega'}$, the inverse being $\mu \mapsto \mathcal R(\omega^{-1} - c^{-1}\,d\mu/d\theta)\,d\theta$. The density of $T\lambda$ has the cosine coefficients $1$ and $-c\cos(k\pi/2)\hat f_k$, $k \ge 1$, so Lemma~\ref{lem:cheb} gives $J_0(T\lambda) = \log 2 + c^2(J_0(\lambda) - \log 2)$; moreover $(T\lambda)(\Theta_I') = c(\omega'/\omega - \lambda(\Theta_E)) = 1 - c + c\lambda(\Theta_I)$. Hence $J'_{c\tau}(T\lambda) = c^2J_\tau(\lambda) + (1 - c^2)\log 2 - 2c(1-c)\tau$ for symmetric $\lambda$, so $T\lambda_\tau$ minimizes $J'_{c\tau}$ over the symmetric elements of $\mathscr A_{\omega'}$; as the minimizer $\lambda'_{c\tau}$ of $J'_{c\tau}$ over $\mathscr A_{\omega'}$ is symmetric and unique (Proposition~\ref{prop:equilibrium}), $T\lambda_\tau = \lambda'_{c\tau}$. By Lemma~\ref{lem:cl2}(iv), $E_{T\lambda}(\vartheta) = -c\bigl((\pi/\omega)\log 2 + E_\lambda(|\vartheta - \pi/2|)\bigr)$; with $\lambda = \lambda_\tau$ this gives $\max_{[0,\pi]}E_{T\lambda} = -c\bigl((\pi/\omega)\log 2 + \min_{[0,\pi]}E_\tau\bigr)$ and, as $|\theta_0' - \pi/2| = \theta_0$, $E_{T\lambda}(\theta_0') = -c\bigl((\pi/\omega)\log 2 + E_\tau(\theta_0)\bigr)$. Let $0 < \tau < \kappa_\sigma(\omega)$. Then $\min E_\tau = E_\tau^\ast - \tau$ by Proposition~\ref{prop:equilibrium}(i) and (iii), and $c\tau \in (0,\kappa_\sigma(\omega'))$ because $\omega\kappa_\sigma(\omega) = \mathrm{Cl}_2(\omega) + \mathrm{Cl}_2(\pi-\omega)$ is symmetric in $\omega$ and $\pi-\omega$. Therefore $r_{\omega'}(c\tau) = E_{T\lambda}(\theta_0') - \max E_{T\lambda} = c\bigl(\min E_\tau - E_\tau(\theta_0)\bigr) = -c\bigl(\tau + r_\omega(\tau)\bigr)$. Dividing by $c\kappa_\sigma(\omega) = \kappa_\sigma(\omega')$ gives \eqref{eq:tail_duality} at $\bar\tau = \tau/\kappa_\sigma(\omega)$.
\end{proof}

\begin{proof}[Proof of Corollary~\ref{cor:self_dual}]
For $\omega = \pi/2$, $\omega' = \omega$, and \eqref{eq:tail_duality} reads $2\bar r_{\pi/2}(\bar\tau) = -\bar\tau$, i.e.\ $r_{\pi/2}(\tau) = -\tau/2$. The second statement is Theorem~\ref{thm:tail_law}, and $a(\pi/2) = 1/2$ (Section~\ref{sec:res_rates}).
\end{proof}

\subsection{The Start of the Tail: Proof of Theorem~\ref{thm:initial_slope} and Corollaries~\ref{cor:critical} and \ref{cor:layer}}
\label{sec:initial_slope_proof}

The construction is outlined after Theorem~\ref{thm:initial_slope}. We use the notation of Appendix~\ref{app:start}, write $s := s(\omega) = 2\theta_0/\pi$, and let $C_1$, $C_2$ and $C'$ denote positive constants that depend only on $\omega$.

\begin{proof}[Proof of Theorem~\ref{thm:initial_slope}]
\emph{The profile.} Let $r_0 := \min\{\theta_0,\omega\}/4$, choose $x_\ast < -g$ with $m \ge 1/(2\pi)$ on $(-\infty,x_\ast]$, which is possible by Lemma~\ref{lem:inner}(iii), and let $0 < \tau \le \tau_1$, where $\tau_1 \in (0,\min\{1,\kappa_\sigma\})$ satisfies $2\tau_1(l - x_\ast) \le r_0$. Define on $[0,\pi]$
\[
e_\tau(\vartheta) := \frac\omega\pi + 1 + e\Bigl(\frac{\vartheta - \theta_0}{\tau}\Bigr) + e\Bigl(\frac{\pi - \theta_0 - \vartheta}{\tau}\Bigr).
\]
By Lemma~\ref{lem:inner}(i), $e_\tau$ is absolutely continuous, $e_\tau' \in L^2(0,\pi)$, and $e_\tau'$ is H\"older continuous near every point of $(0,\pi)\setminus\{\theta_0,\pi-\theta_0\}$. Since $-1 \le e \le 0$, $e = 0$ on $(-\infty,-g]$ and $e = -1$ on $[l,\infty)$, and since $(\pi - \theta_0 - \vartheta)/\tau \ge \omega/(2\tau) \ge l$ for $\vartheta \le \pi/2$, we have $e_\tau(\pi - \vartheta) = e_\tau(\vartheta)$,
\begin{equation}\label{eq:profile}
e_\tau \le \frac\omega\pi \ \text{ on } [0,\pi], \qquad e_\tau = \frac\omega\pi - 1 \ \text{ on } F_\tau := [\theta_0 + l\tau,\ \pi - \theta_0 - l\tau], \qquad e_\tau(\theta_0) = \frac\omega\pi - s,
\end{equation}
and $e_\tau(\theta_0 + \tau x) = \omega/\pi + e(x)$ for $|\tau x| \le r_0$.

\emph{The density.} Let $f_0 := \pi^{-1}(1 - \tau\mathcal Ne_\tau)$ with $\mathcal N$ as in Lemma~\ref{lem:shape}. By Lemma~\ref{lem:shape}(i), $\int_0^\pi f_0 = 1$ and $E_{f_0} = -\log 2 + \tau(e_\tau - \bar e_\tau)$ on $[0,\pi]$, with $\bar e_\tau := \pi^{-1}\int_0^\pi e_\tau$. Let $A$ be the set of $\vartheta \in [0,\pi]$ with $\dist(\vartheta,\{\theta_0,\pi-\theta_0\}) > r_0$. We claim that $f_0(\pi - \vartheta) = f_0(\vartheta)$ and that, with $C_1 := \pi^{-2}\bigl(1 + \cot(r_0/4) + \cot(\omega/4) + 2\cot(\theta_0/4)\bigr)$, for a.e.\ $\vartheta \in [0,\pi]$
\begin{equation}\label{eq:f0_inner}
\begin{aligned}
\Bigl|f_0(\vartheta) - m\Bigl(\frac{\vartheta - \theta_0}{\tau}\Bigr)\Bigr| &\le C_1\tau \quad \text{if } |\vartheta - \theta_0| \le r_0, \\
\Bigl|f_0(\vartheta) - \frac1\pi\Bigr| &\le C_1\tau \quad \text{if } \vartheta \in A .
\end{aligned}
\end{equation}
The symmetry holds since the cosine coefficients of $e_\tau$ vanish for odd $k$. By Lemma~\ref{lem:shape}(ii) and the identity
\[
\frac{2\sin\varphi}{\cos\varphi - \cos\vartheta} = \cot\frac{\vartheta-\varphi}2 - \cot\frac{\vartheta+\varphi}2 ,
\]
for almost every $\vartheta$
\[
\tau\mathcal Ne_\tau(\vartheta) = \frac\tau{2\pi}\,\mathrm{p.v.}\!\int_0^\pi e_\tau'(\varphi)\Bigl(\cot\frac{\vartheta-\varphi}2 - \cot\frac{\vartheta+\varphi}2\Bigr)d\varphi ,
\]
where $e_\tau' = e_1 + e_2$ with $e_1(\varphi) := \tau^{-1}e'((\varphi - \theta_0)/\tau)$, supported in $V := [\theta_0 - g\tau, \theta_0 + l\tau]$, and $e_2(\varphi) := -e_1(\pi - \varphi)$, supported in $\pi - V$; by Lemma~\ref{lem:inner}(i), $\int|e_1| = \int|e_2| = 1$. For $\varphi \in V\cup(\pi - V)$ and $\vartheta \in [0,\pi]$ we have $\theta_0/2 \le \vartheta + \varphi \le 2\pi - \theta_0/2$, so $|\cot\frac{\vartheta+\varphi}2| \le \cot(\theta_0/4)$, and $|\cot\frac{\vartheta-\varphi}2| \le \cot(\rho/2)$ whenever $|\vartheta - \varphi| \ge \rho$, $0 < \rho \le \pi$. If $\vartheta \in A$, then $|\vartheta - \varphi| \ge r_0/2$ on $V \cup (\pi - V)$, since $(g + l)\tau \le r_0/2$, and $|\tau\mathcal Ne_\tau(\vartheta)| \le \pi^{-1}\tau(\cot(r_0/4) + \cot(\theta_0/4))$, which gives the second estimate. Let $|\vartheta - \theta_0| \le r_0$ and $u := (\vartheta - \theta_0)/\tau$. On $\pi - V$, $|\vartheta - \varphi| \ge \omega - r_0 - l\tau \ge \omega/2$, so the contribution of $e_2$ is at most $(2\pi)^{-1}\tau(\cot(\omega/4) + \cot(\theta_0/4))$, and that of $e_1$ with the second cotangent is at most $(2\pi)^{-1}\tau\cot(\theta_0/4)$. In the remaining term we write $\frac12\cot\frac x2 = \frac1x + k(x)$, where $k$ is odd and decreasing on $(-\pi,\pi)$, so that $|k| \le 1/\pi$ there, and substitute $\varphi = \theta_0 + \tau y$; as $|\vartheta - \varphi| \le r_0 + (g+l)\tau < \pi$ for $\varphi \in V$, Lemma~\ref{lem:inner}(ii) gives
\[
\begin{aligned}
\frac\tau{2\pi}\,\mathrm{p.v.}\!\int e_1(\varphi)\cot\frac{\vartheta-\varphi}2\,d\varphi &= \frac1\pi\,\mathrm{p.v.}\!\int_{\mathbb R}\frac{e'(y)}{u - y}\,dy + \frac\tau\pi\int e_1(\varphi)\,k(\vartheta-\varphi)\,d\varphi \\
&= 1 - \pi m(u) + R_1,
\end{aligned}
\]
with $|R_1| \le \tau/\pi^2$. Hence $|\tau\mathcal Ne_\tau(\vartheta) - 1 + \pi m(u)| \le \pi C_1\tau$, which is the first estimate.

\emph{The correction.} Let $Z_1 := [\theta_0 + x_\ast\tau,\ \theta_0 + l\tau]$, $Z := Z_1 \cup (\pi - Z_1)$ and $Y := [\theta_0, \theta_0 + l\tau] \cup [\pi - \theta_0 - l\tau, \pi - \theta_0] \subset Z$, so that $A \cap Z = \emptyset$ and $|A| = \pi - 4r_0 \ge \pi/2$. By \eqref{eq:f0_inner}, the symmetry of $f_0$ and Lemma~\ref{lem:inner}(iii), a.e.\ $f_0 \ge -C_1\tau$ on $Z$, $f_0 \ge 1/\omega - C_1\tau$ on $Y$, where $m = 1/\omega$, and $f_0 \ge 1/(2\pi) - C_1\tau$ on $[0,\pi]\setminus Z$, where $m \ge 1/(2\pi)$ in the first case of \eqref{eq:f0_inner}. Put
\[
f := f_0 + C_1\tau\,\mathbf 1_Z - \frac{C_1\tau|Z|}{|A|}\,\mathbf 1_A .
\]
Then $\int_0^\pi f = 1$, and for $\tau \le \tau_2 := \min\{\tau_1, 1/(4\pi C_1)\}$, since $|Z| \le |A|$, $f \ge 0$ a.e.\ on $[0,\pi]$ and $f \ge 1/\omega$ a.e.\ on $Y$. As $|Z| = 2(l - x_\ast)\tau$, $\tau \le 1$ and $|A| \ge \pi/2$, Lemma~\ref{lem:env}(b) gives $|E_f - E_{f_0}| \le C_1\tau C_0|Z|^{1/2} + C_1\tau|Z||A|^{-1/2}C_0 \le C_2\tau^{3/2}$ on $[0,\pi]$, with $C_2 := C_0C_1\bigl((2(l - x_\ast))^{1/2} + 2(2/\pi)^{1/2}(l - x_\ast)\bigr)$.

\emph{The admissible triple.} Let $\nu$ be the image of $f\,d\theta$ under $\theta \mapsto \cos\theta$, a positive measure on $[-1,1]$ of mass one, $F := \cos(F_\tau) \subset [-c_0,c_0]$ and $b_\tau := \log 2 + \tau\bar e_\tau - \tau\omega/\pi - C_2\tau^{3/2}$. Then
\[
b_\tau - U^\nu(\cos\vartheta) = b_\tau + E_f(\vartheta) = \tau\Bigl(e_\tau(\vartheta) - \frac\omega\pi\Bigr) + (E_f - E_{f_0})(\vartheta) - C_2\tau^{3/2},
\]
which by \eqref{eq:profile} is at most $0$ on $[0,\pi]$ and at most $-\tau$ on $F_\tau$; and $[-c_0,c_0]\setminus F$ is the image of $Y$ without its inner endpoints, where $f \ge 1/\omega$, so that $\nu \ge \mu_\omega$ there. Thus $(F, b_\tau, \nu)$ satisfies \eqref{eq:reduced_constraints}, and Proposition~\ref{prop:extremal} gives
\[
r_\omega(\tau) \ge b_\tau + E_f(\theta_0) \ge \tau\Bigl(e_\tau(\theta_0) - \frac\omega\pi\Bigr) - 2C_2\tau^{3/2} = -s\tau - 2C_2\tau^{3/2} \qquad (0 < \tau \le \tau_2).
\]

\emph{The upper bound.} Let $\omega' := \pi - \omega$ and $c := \omega/\omega'$. The lower bound for $\omega'$ gives $C'$ and $\tau' > 0$ with $r_{\omega'}(\sigma) \ge -s(\omega')\sigma - C'\sigma^{3/2}$ for $0 < \sigma \le \tau'$. Since $c\kappa_\sigma(\omega) = \kappa_\sigma(\omega')$, Theorem~\ref{thm:tail_duality}, applied at $\bar\tau = \tau/\kappa_\sigma(\omega)$, reads $r_\omega(\tau) = -\tau - c^{-1}r_{\omega'}(c\tau)$ for $0 < \tau < \kappa_\sigma(\omega)$, and $s(\omega') = \omega/\pi = 1 - s$; so for $0 < \tau \le \tau'/c$
\[
r_\omega(\tau) \le -\tau + s(\omega')\tau + C'c^{1/2}\tau^{3/2} = -s\tau + C'c^{1/2}\tau^{3/2}.
\]
Both bounds hold for $0 < \tau \le \tau_0 := \min\{\tau_2, \tau'/c\}$, which proves \eqref{eq:initial_slope} with the constant $\max\{2C_2, c^{1/2}C'\}$.
\end{proof}

\begin{proof}[Proof of Corollary~\ref{cor:critical}]
Let $a := a(\omega)$, $s := s(\omega)$, and $C$ and $\tau_0$ as in Theorem~\ref{thm:initial_slope}; recall that $\omega/\pi = L/(L + 2\delta)$. (i) As $\delta < L/2$, $\omega > \pi/2$, and $a > s$ by Proposition~\ref{prop:critical}; so $r_\omega(\tau) \ge -s\tau - C\tau^{3/2} > -a\tau$ for $0 < \tau \le \tau_0$ with $C\tau^{1/2} < a - s$. Suppose that $\gamma_t \le C_\ast t^{\,a}$ for all $N \ge N_1$ and all $t > 0$. By Proposition~\ref{prop:unisolvence_poisson} and Lemma~\ref{lem:two_forms}(i), $t\Gamma_t \le C_at^{\,a}$ with $C_a := 2^a(1 - 4^{a-1})^{-1/2}C_\ast$ for all sufficiently large $N$, so that $r^+(\tau) \le -a\tau$ for every $\tau > 0$, whereas $r^+(\tau) = r_\omega(\tau) > -a\tau$ for small $\tau > 0$ by Theorem~\ref{thm:tail_law}. (ii) As $\delta > L/2$, $\omega < \pi/2$ and $a < s$, so that $r_\omega(\tau) \le -s\tau + C\tau^{3/2} < -a\tau$ for $0 < \tau \le \tau_0$ with $C\tau^{1/2} < s - a$.
\end{proof}

\begin{proof}[Proof of Corollary~\ref{cor:layer}]
Let $s := s(\omega) \in (0,1)$, and $C$ and $\tau_0$ as in Theorem~\ref{thm:initial_slope}. For $0 < \tau \le \tau_0$ with $C\tau^{1/2} < \min\{s, 1-s\}$, \eqref{eq:initial_slope} gives $-\tau < r_\omega(\tau) < 0$. By Proposition~\ref{prop:tail_properties}, $r_\omega$ is nonincreasing and $\tau \mapsto r_\omega(\tau) + \tau$ is nondecreasing, so both strict inequalities hold for all $\tau \in (0,\kappa_\sigma)$. Fix such $\tau$ and let $E := E_\tau$, $E^\ast := E_\tau^\ast$. Since $E(\theta_0) = E^\ast + r_\omega(\tau) < E^\ast$ and $E$ is continuous, $E < E^\ast$ on $(\theta_0 - \eps, \theta_0]$ for some $\eps > 0$, so this interval contains no point of the support of $\lambda_\tau|_{\Theta_E}$, on which $E = E^\ast$ by Proposition~\ref{prop:equilibrium}(ii); that is, $\lambda_\tau$ vanishes on $(\theta_0 - \eps, \theta_0)$. Since $E(\theta_0) > E^\ast - \tau$, likewise $E > E^\ast - \tau$ on $[\theta_0, \theta_0 + \eps)$ after decreasing $\eps$, so this interval contains no point of the support of $(\omega^{-1}d\theta - \lambda_\tau)|_{\Theta_I}$, on which $E = E^\ast - \tau$; hence $\lambda_\tau = \omega^{-1}d\theta$ on $[\theta_0, \theta_0 + \eps)$, this relatively open subset of $\Theta_I$ lies in the set $O$ of the proof of Proposition~\ref{prop:equilibrium}(iii), and $\theta_\tau \ge \theta_0 + \eps$.
\end{proof}

\section{Regularized Oversampled Collocation}
\label{sec:regularized}

For comparison with the interpolatory method, we treat the oversampled formulation with truncated singular value decomposition \eqref{eq:tsvd}, with $M \ge n$ nodes. It is controlled by the following constant of the test space.

\begin{definition}[Marcinkiewicz--Zygmund Constant of the Test Space]
\label{def:mz}
Let $\mat B$ have full column rank at the nodes $x_1,\dots,x_M$, and put
\begin{equation}\label{eq:mz}
C_M := \sup_{\psi \in \Psi_N\setminus\{0\}} \frac{\norm{\psi}_{L^2(I)}}{(L/M)^{1/2}\,\bigl(\sum_{m=1}^M \psi(x_m)^2\bigr)^{1/2}} .
\end{equation}
\end{definition}

A bound on $C_M$ independent of $N$ is the lower half of a Marcinkiewicz--Zygmund inequality for $\Psi_N$ at the nodes, as in \cite[Eq.~(1.5)]{adcock2020frames2}; $C_M$ is computable as the square root of a generalized eigenvalue (Section~\ref{sec:num_solvers}).

\begin{theorem}[Error of Regularized Oversampled Collocation]
\label{thm:regularized}
Let $\mat B$ have full column rank, let $u = \Lop^{-1}f$ and let $u_N^\eps$ be defined by \eqref{eq:tsvd}. Then for every $v = \sum_{j=1}^n z_j\varphi_j \in \Phi_N$, with coefficient vector $\vect z = (z_1, \dots, z_n)^T$,
\begin{equation}\label{eq:regularized_bound}
\begin{aligned}
\norm{u - u_N^\eps}_{L^2(I)} \le \norm{\Lop^{-1}}_{2\to 2}\Bigl[&(1 + 2C_M)\sqrt L\, \norm{f - \Lop v}_{L^\infty(I)} \\
&+ C_M (L/M)^{1/2}\,\eps\,\norm{\mat D^{-1}\vect z}_2\Bigr],
\end{aligned}
\end{equation}
where $\norm{\Lop^{-1}}_{2\to 2}$ is the norm of $\Lop^{-1}$ on $L^2(I)$ and $\mat D^{-1}\vect z$ is the coefficient vector of $v$ in the normalized basis, $v = \Phi\mat D(\mat D^{-1}\vect z)$.
\end{theorem}

\begin{corollary}[Interpolatory and Oversampled Regimes]
\label{cor:regimes}
\leavevmode
\begin{enumerate}
    \item[(i)] Let Assumption~\ref{ass:unisolvence} hold, let $M = n$ with the interior nodes of the $N$-grid, and let $0 < \eps < \sigma_n$. Then $u_N^\eps = u_N$, the interpolatory solution of Theorem~\ref{thm:representation}, and
    \begin{equation}\label{eq:interp_l2}
    \norm{u - u_N}_{L^2(I)} \le \norm{\Lop^{-1}}_{2\to2}(1+2C_{n})\sqrt L\,\dist_\infty(f,\Psi_N),
    \end{equation}
    the $L^2$ analogue of \eqref{eq:error_bound}.
    \item[(ii)] Let $M \ge n$, and let the nodes be such that $\mat B$ has full column rank and $C_M \le C_\ast$ for all $N$. Then \eqref{eq:regularized_bound} holds with $C_\ast$ in place of $C_M$, uniformly in $N$, for every $\eps > 0$ and every $v = \sum_{j=1}^n z_j\varphi_j \in \Phi_N$, with coefficient vector $\vect z = (z_1, \dots, z_n)^T$, and the coefficient vector $\vect c_\eps$ of $u_N^\eps$ in \eqref{eq:tsvd} satisfies 
\[
\norm{\mat D^{-1}\vect c_\eps}_2 \le \norm{\mat D^{-1}\vect z}_2 + \eps^{-1}\sqrt M\,\norm{f - \Lop v}_{L^\infty(I)}.
\]
\end{enumerate}
\end{corollary}

Theorem~\ref{thm:regularized} is the collocation analogue of the regularized-frame bounds of \cite{adcock2019frames, adcock2020frames2}: the PDE error is bounded by the $\eps$-regularized best approximation error of the data in the test space, transferred through $\Lop^{-1}$ and through $C_M$. The content of the theorem therefore depends entirely on $C_M$, which Section~\ref{sec:num_solvers} computes for the interpolatory and the oversampled node sets.

The two properties of truncated singular value decomposition that the proofs use, the control of the nodal residual and of the coefficients, are the discrete counterparts of \cite[Thms.~13 and 14]{adcock2019frames} and are proved as Lemma~\ref{lem:tsvd} in Appendix~\ref{app:standard}.

\begin{proof}[Proof of Theorem~\ref{thm:regularized}]
Put $\psi = \Lop v \in \Psi_N$ and $r = f - \psi$. Since $u_N^\eps \in \Phi_N$, $u - u_N^\eps = \Lop^{-1}(f - \Lop u_N^\eps)$ and $f - \Lop u_N^\eps = r + (\psi - \Lop u_N^\eps)$ with $\psi - \Lop u_N^\eps \in \Psi_N$. The nodal values of the latter are $\mat B\vect z - \mat B\vect c_\eps = (\mat B\vect z - \vect f) + (\vect f - \mat B\vect c_\eps)$, so by Lemma~\ref{lem:tsvd}
\begin{align*}
\norm{\psi - \Lop u_N^\eps}_{L^2(I)} &\le C_M (L/M)^{1/2}\bigl(2\norm{\vect f - \mat B\vect z}_2 + \eps\norm{\mat D^{-1}\vect z}_2\bigr) \\
&\le C_M\bigl(2\sqrt L\norm{r}_{L^\infty(I)} + (L/M)^{1/2}\eps\norm{\mat D^{-1}\vect z}_2\bigr),
\end{align*}
using $\norm{\vect f - \mat B\vect z}_2 = \norm{(r(x_m))_m}_2 \le \sqrt M\norm{r}_\infty$. Add $\norm{r}_{L^2(I)} \le \sqrt L\norm{r}_\infty$ and apply $\Lop^{-1}$.
\end{proof}

\begin{proof}[Proof of Corollary~\ref{cor:regimes}]
(i) For $M = n$ the matrix $\mat B$ is square and nonsingular by Assumption~\ref{ass:unisolvence}, so $\sigma_n > 0$, and for $\eps < \sigma_n$ no singular value is truncated in \eqref{eq:tsvd}: $\vect c_\eps = \mat D\mat V\mathbf\Sigma^{-1}\mat U^T\vect f = \mat D(\mat B\mat D)^{-1}\vect f = \mat B^{-1}\vect f$. This is the coefficient vector of the interpolatory solution $u_N = \Lop^{-1}\Pop_Nf$ of Theorem~\ref{thm:representation} in the basis $\{\varphi_j\}$, because $u_N \in \Phi_N$ and the interior collocation equations for $u_N = \sum_jc_j\varphi_j$ read $\mat B\vect c = \vect f$. Hence $u_N^\eps = u_N$ for every $\eps \in (0,\sigma_n)$, and \eqref{eq:regularized_bound} holds for every such $\eps$ with $C_M = C_n$; letting $\eps \to 0$ for fixed $v$ removes the second term, and $\inf_{v\in\Phi_N}\norm{f - \Lop v}_{L^\infty(I)} = \dist_\infty(f,\Psi_N)$ because $\Psi_N = \Lop(\Phi_N)$. (ii) The first claim is \eqref{eq:regularized_bound} with $C_M \le C_\ast$, and the bound on the coefficients is the second inequality of Lemma~\ref{lem:tsvd} with $\norm{\vect f - \mat B\vect z}_2 = \bigl\lVert\bigl(f(x_m) - (\Lop v)(x_m)\bigr)_m\bigr\rVert_2 \le \sqrt M\norm{f - \Lop v}_{L^\infty(I)}$.
\end{proof}

\section{Higher Dimensions and Curved Boundaries}
\label{sec:multi}

Let $R_0 = (0, L_1) \times \dots \times (0, L_d)$ be embedded in $R = \prod_k (-\delta_k, L_k+\delta_k)$, with the tensor-product basis $w_{\mathbf j}(\mathbf x) = \prod_k w_{j_k}(x_k)$, $1 \le j_k \le N$, where $w_{j_k}$ is the function \eqref{eq:basis_func} of the interval $(-\delta_k, L_k+\delta_k)$; these are the eigenfunctions of $-\Delta$ on $R$ with eigenvalues $\lambda_{\mathbf j} = \sum_k \lambda_{j_k}$, and they span the trial space $V_{N,d}$. The nodes form the tensor-product grid of the $N$ equispaced nodes of the intervals $[0,L_k]$, $N^d$ in number; the interior nodes are those in $R_0$, the boundary nodes those on $\partial R_0$. The synthesis matrix $\mat W_d$ of $V_{N,d}$ at the nodes is the Kronecker product of the one-dimensional synthesis matrices, which are nonsingular by Lemma~\ref{lem:lebesgue_condition_link}, and $\mat A$ denotes the matrix of the interpolatory collocation system, which imposes the boundary values at the boundary nodes and the equation at the interior nodes. For $\delta_k = \delta$ and $L_k = L$, $\mat W_d = \mat W_N\otimes\dots\otimes\mat W_N$, and the bulk Poisson matrix, the equation enforced at every node, factors as $\mat A_P^{(d)} = \mat W_d\mathbf\Lambda_d$ with $\mathbf\Lambda_d := \operatorname{diag}(\lambda_{\mathbf j})$.

\begin{theorem}[Compounding of the Frame Instability]\label{thm:d_dim_instability}
With $\delta_k = \delta$, $L_k = L$, for every $\alpha \in (0,\kappa_\sigma(\omega))$ there is $C>0$ independent of $N$ such that $\kappa_2(\mat A_P^{(d)}) \ge C N^{-3/2} e^{\alpha d N}$ for all $N \ge 2$.
\end{theorem}
\begin{proof}
The singular values of a Kronecker product are the products of the singular values of the factors, so $\sigma_{\min}(\mat W_d) = \sigma_{\min}(\mat W_N)^d \le C_1^de^{-\alpha dN}$ by Lemma~\ref{lem:sigma_min}, and $\sigma_{\min}(\mat A_P^{(d)}) \le \sigma_{\min}(\mat W_d)\max_{\mathbf j}\lambda_{\mathbf j} = d\lambda_N\sigma_{\min}(\mat W_N)^d$. On the other hand 
\[
\sigma_{\max}(\mat A_P^{(d)}) \ge \norm{\mat A_P^{(d)}\vect e_{\mathbf 1}}_2 = d\lambda_1\norm{\mat W_N\vect e_1}_2^d \ge d\lambda_1(2N\sin^2\theta_0/\tilde L)^{d/2},
\] 
as in the proof of Corollary~\ref{thm:poisson_illcond}. Dividing, $\kappa_2(\mat A_P^{(d)}) \ge N^{-2}(2\sin^2\theta_0/\tilde L)^{d/2}N^{d/2}e^{\alpha dN}/C_1^d$, and $N^{d/2-2} \ge N^{-3/2}$ for $d \ge 1$.
\end{proof}

\begin{theorem}[Exact Representation in $d$ Dimensions]
\label{thm:representation_d}
Let $\Lop = -\Delta + \mathbf k(\mathbf x)\cdot\nabla$ with $\mathbf k \in C(\bar R_0)^d$, let $f \in C(\bar R_0)$, and let $u \in C^2(R_0)\cap C(\bar R_0)$ satisfy $\Lop u = f$ in $R_0$ and $u = 0$ on $\partial R_0$. Let $\Phi_{N,d} := \{v \in V_{N,d}: v|_{\partial R_0} = 0\}$ and $\Psi_{N,d} := \Lop(\Phi_{N,d})$, and suppose that interpolation at the interior nodes is unisolvent in $\Psi_{N,d}$, with interpolation operator $\Pop_{N,d}$ and Lebesgue constant $\Lambda(\Psi_{N,d})$. Then the interpolatory solution $u_N \in V_{N,d}$ with homogeneous Dirichlet data at the boundary nodes is the unique function in $C^2(R_0)\cap C(\bar R_0)$ with $\Lop u_N = \Pop_{N,d}f$ in $R_0$ and $u_N = 0$ on $\partial R_0$, and with $\alpha := \max_{\bar R_0}|k_1| + 1$,
\begin{equation}\label{eq:error_bound_d}
\begin{aligned}
\norm{u - u_N}_{L^\infty(R_0)} &\le (e^{\alpha L_1} - 1)\bigl(1 + \Lambda(\Psi_{N,d})\bigr)\dist_\infty(f, \Psi_{N,d}), \\
\sum_m|(\mat G_N)_{im}| &\le (e^{\alpha L_1} - 1)\Lambda(\Psi_{N,d}),
\end{aligned}
\end{equation}
where $\mat G_N = \mat W_d\mat A^{-1}$ maps nodal data to nodal values, $\mathbf x_i$ is any interior node and the sum runs over the interior nodes $\mathbf x_m$. If moreover $\mathbf k$ is constant and $u \in C^2(\bar R_0)$, then Theorem~\ref{thm:regularized} holds with $R_0$ in place of $I$, the products $\varphi_{j_1}(x_1)\cdots\varphi_{j_d}(x_d)$ of the one-dimensional bases of Section~\ref{sec:constants} in place of the $\varphi_j$, $|R_0|$ in place of $L$ in the definition \eqref{eq:mz} of $C_M$ and in \eqref{eq:regularized_bound}, and $\norm{\Lop^{-1}}_{2\to2}$ replaced by $(\sum_i\pi^2/L_i^2)^{-1}$.
\end{theorem}
\begin{proof}
Let $v \in V_{N,d}$ vanish at the boundary nodes. On each face of $R_0$ the trace of $v$ lies in the tensor product of the one-dimensional spaces $V_N$ of the other $d-1$ directions and vanishes at the $N^{d-1}$ nodes of the face, at which interpolation in that space is unisolvent, its synthesis matrix being a Kronecker product of matrices that are nonsingular by Lemma~\ref{lem:lebesgue_condition_link}; hence $v = 0$ on $\partial R_0$, i.e.\ $v \in \Phi_{N,d}$. If moreover $\Lop v$ vanishes at the interior nodes, then $\Lop v = 0$ by the unisolvence, $v = 0$ on $\bar R_0$ by Lemma~\ref{lem:dirichlet_box}(ii), and the coefficients of $v$ vanish because $\mat W_d$ is nonsingular; so $\mat A$ is nonsingular, and $u_N$ exists and lies in $\Phi_{N,d}$. The interior collocation equations then give $\Lop u_N = \Pop_{N,d}f$, exactly as in the proof of Theorem~\ref{thm:representation}, and uniqueness is Lemma~\ref{lem:dirichlet_box}(ii). The function $u - u_N \in C^2(R_0)\cap C(\bar R_0)$ vanishes on $\partial R_0$ and satisfies $\Lop(u - u_N) = f - \Pop_{N,d}f$, so Lemma~\ref{lem:dirichlet_box}(ii) gives $\norm{u - u_N}_\infty \le (e^{\alpha L_1}-1)\norm{f - \Pop_{N,d}f}_\infty$, and $\norm{f - \Pop_{N,d}f}_\infty \le (1 + \Lambda(\Psi_{N,d}))\dist_\infty(f,\Psi_{N,d})$ as in Corollary~\ref{cor:error}. For the second bound, $\sum_m|(\mat G_N)_{im}| = \sum_m s_m(\mat G_N)_{im}$ with $s_m := \operatorname{sign}(\mat G_N)_{im}$ is the value at $\mathbf x_i$ of the collocation solution for the interior data $(s_m)_m$, whose interpolant $\sum_m s_m\ell_m^\Psi$ in $\Psi_{N,d}$ has supremum norm at most $\Lambda(\Psi_{N,d})$; apply Lemma~\ref{lem:dirichlet_box}(ii) to that solution. For the last claim, the same face argument shows that $v \in V_{N,d}$ lies in $\Phi_{N,d}$ if and only if it vanishes at the $N^d - (N-2)^d$ boundary nodes, whose evaluation functionals are linearly independent on $V_{N,d}$ because $\mat W_d$ is nonsingular; so $\dim\Phi_{N,d} = (N-2)^d$, and the $(N-2)^d$ linearly independent products of the one-dimensional bases, which vanish on $\partial R_0$, form a basis of $\Phi_{N,d}$. Moreover $\Lop^{-1}$ enters the proof of Theorem~\ref{thm:regularized} only through the identity $u - u_N^\eps = \Lop^{-1}(f - \Lop u_N^\eps)$, which here expresses uniqueness, and through its $L^2$ norm, which Lemma~\ref{lem:dirichlet_box}(iii) bounds; the remaining steps are unchanged.
\end{proof}

Heuristically (Remark~\ref{rem:peclet}), Lemma~\ref{lem:phase} applied direction by direction for constant $\mathbf k$ suggests that the test spaces of $-\Delta$ and of $-\Delta + \mathbf k\cdot\nabla$ have nearly the same high modes.

\begin{remark}[Curved Boundaries by Node Relocation]
\label{rem:relocation}
For a domain $\Omega \subset R_0$ with smooth boundary $\Gamma$, the extended-domain solver of \cite{wu2018development} keeps the tensor-product basis and the tensor-product grid, moves the grid nodes adjacent to $\Gamma$ onto $\Gamma$, imposes the boundary data at the moved nodes, and enforces the differential equation at all remaining nodes of $R_0$, inside and outside $\Omega$. For every grid edge that $\Gamma$ crosses, the endpoint nearer to $\Gamma$ is moved, in \cite{wu2018development} to a nearby point of a fine parametrization of $\Gamma$ and in the computations of Section~\ref{sec:num_disk} to the crossing point itself, in both cases by at most half a grid spacing. This requires the forcing to be defined on all of $R_0$; for manufactured solutions it is, and for general data a smooth extension of $f$ to $R_0$ must be constructed first, for instance a periodic $C^\infty$ extension to a rectangle \cite{boyd2005fourier, bueno2006fourier} or a $C^k$ extension obtained by solving a polyharmonic problem outside $\Omega$ \cite[Sec.~2.1]{stein2016immersed}. The continuous problem behind this formulation is uniquely solvable in $\Omega$ but not outside, and Theorem~\ref{thm:representation_d} does not cover it: the boundary data are imposed only at the relocated nodes, not on all of $\Gamma$, and the problem outside $\Omega$ carries no boundary condition. Node relocation changes the sampling set of the frame, not the frame; Section~\ref{sec:num_disk} tests its effect.
\end{remark}

\section{Numerical Experiments}
\label{sec:numerics}

All comparisons between operators are \emph{controlled}: within each experiment the domain, the extension parameter $\delta$ and the exact solution are fixed while the convection coefficient varies, $k = 0$ being the Poisson operator. Quantities meaningful only in exact arithmetic (Lebesgue constants, Marcinkiewicz--Zygmund constants, Lagrange-function masses, the exact discrete Green's function, singular spectra, exact solves) are computed in multiple-precision arithmetic with at least $48$ digits, the precision increasing with $N$ where the smallest singular values require it; everything else in IEEE double precision, with LU factorization with partial pivoting unless stated otherwise. Errors are maximum-norm errors on a fine grid unless stated otherwise. The floors of double-precision computations depend on the order of the floating-point operations, for instance on the number of threads of the linear-algebra library, by up to an order of magnitude.

\begin{remark}[A Pitfall]
\label{rem:pitfall}
The convergence rate and the accuracy floor of a Fourier extension depend on the extension ratio $\delta/L$ and on the extendability of the solution, so a comparison of operators on different domains or with different exact solutions cannot isolate the operator. A more insidious error is a solution that lies in the trial space: on $(0,1)$ the function $\sin(2\pi x)$ is a multiple of $w_{2(1+2\delta)}$ whenever $2\delta \in \N$, and $\sin(\pi x)$ one of $w_{1+2\delta}$ whenever $\delta\in\N$, so both are reproduced to machine precision by any operator as soon as $N$ reaches the mode index, regardless of conditioning. All solutions used below lie in no $V_N$.
\end{remark}

\subsection{Controlled One-Dimensional Comparison}
\label{sec:num_1d}

We solve $-u'' + ku' = f$ on $I=(0,1)$ with $u(x) = e^x\sin(\pi x)$, for $k \in \{0, 10, 50\}$ and $\delta \in \{0.5, 1, 2\}$, with the interpolatory system \eqref{eq:square_system} (Figure~\ref{fig:controlled_1d}). For each $\delta$ the errors for the three operators coincide in the geometric phase, whose rate is set by $\delta$; the median floors decrease as $\delta$ grows and rise with $k$ by at most about an order of magnitude, and the coefficient vectors stay of moderate size. The exact test-space Lebesgue constants grow at nearly the same exponential rate for the three operators, $0.58$ per mode for $\delta = 0.5$ and $0.64$ to $0.65$ for $\delta = 1$, against $\kappa_\Lambda = 0.583$ and $0.646$, the rate of Corollary~\ref{cor:lebesgue_rate} for the Poisson operator; the convection--diffusion constants are larger at equal $N$, and their fitted rates slightly smaller. Solved exactly with the double-precision data that the floating-point solve receives, the system attains its minimum error near $N^\ast \approx 40$ and then diverges, with coefficients of size $10^{8}$ to $10^{10}$ at $N = 64$, whereas with exactly evaluated data it converges geometrically, to below $10^{-16}$ at $N = 64$ with coefficients of size $1.4$; the threshold is the same for $k = 0$ and $k = 50$ (panel d). Both exact solves follow the bound \eqref{eq:perturbed_data} of Remark~\ref{rem:divergence}. With exactly evaluated data only its first term is present, and for the Poisson operator Proposition~\ref{prop:classical}(iii) and Theorem~\ref{thm:competitor}(i), with $p = N-1$ and $\rho = \rho^\ast = 1 + \sqrt2$ since $u$ is entire, bound it by $CN^{C'}e^{-N(\log\rho^\ast - \kappa_\Lambda)}$ with $C$ and $C'$ independent of $N$; it decays because $\log\rho^\ast = 0.8814 > \kappa_\Lambda = 0.5831$. With rounded data the second term, the rounding errors amplified by $\Lambda_N^\Psi$, takes over beyond $N^\ast$.

\begin{figure}[ht!]
    \centering
    \includegraphics[width=\textwidth]{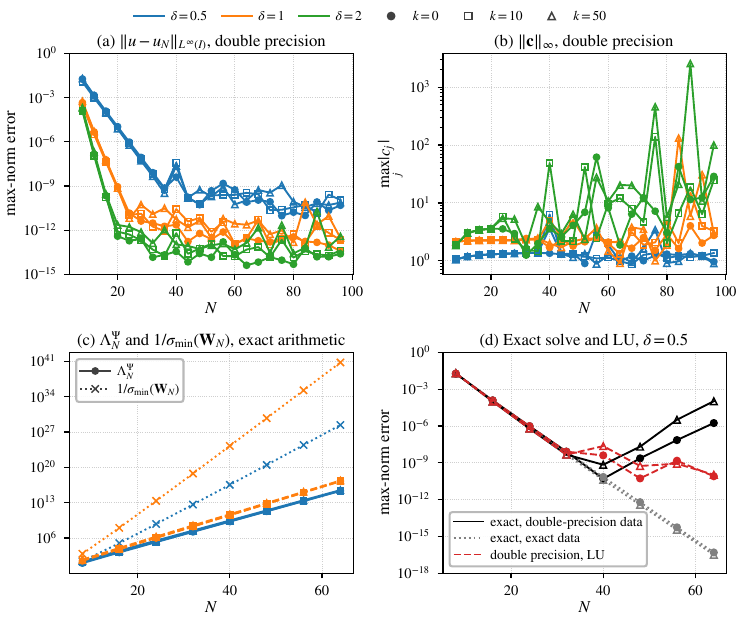}
    \caption{Controlled one-dimensional comparison on $I = (0,1)$ with $u(x) = e^x\sin(\pi x)$; filled symbols Poisson ($k=0$), open symbols convection--diffusion ($k=10$ squares, $k=50$ triangles), colours distinguish $\delta$. (a) Maximum-norm error of the interpolatory system in double precision. (b) Maximum coefficient magnitude. (c) Exact test-space Lebesgue constants $\Lambda_N^\Psi$ for $\delta = 0.5$ (solid) and $\delta = 1$ (dashed), and $1/\sigma_{\min}(\mat W_N)$ (dotted, crosses). (d) Error of the exact solve with the double-precision data (black, solid) and with exactly evaluated data (grey, dotted), and of the double-precision LU solve (red, dashed), for $\delta = 0.5$ and $k = 0$ (filled) and $k = 50$ (open).}
    \label{fig:controlled_1d}
\end{figure}

\subsection{The Exact Discrete Green's Function}
\label{sec:num_greens}

Figure~\ref{fig:greens_square} shows the exact discrete Green's function $\mat G_N = \mat W_N\mat A^{-1}$ for $N = 48$, $\delta = 0.5$, for the Poisson operator and for $k = 20$. (At this resolution $\kappa_2(\mat A)\approx 1.2\times10^{21}$ and the double-precision $\mat G_N$ is not reliable; in the well-conditioned regime, e.g.\ $\delta = 0.05$, $N = 96$, $\kappa_2(\mat A) \approx 4\times10^{8}$, the double-precision and exact $\mat G_N$ agree to $10^{-8}$.) In row $i = 25$ (panel c) its entries exceed $h\,G(x_i,x_m)$ by up to seven (Poisson) and nine ($k = 20$) orders of magnitude, and the Poisson and convection--diffusion rows nearly coincide, although the continuous kernels differ by up to four orders of magnitude for $x_m > x_i$; the rows follow the profile of the $L^1$ masses of the Lagrange functions, which rise from about $10h$ next to the boundary to $2.4\times10^{8}$ at the centre, as \eqref{eq:GN_entries} predicts.

\begin{figure}[ht!]
    \centering
    \includegraphics[width=\textwidth]{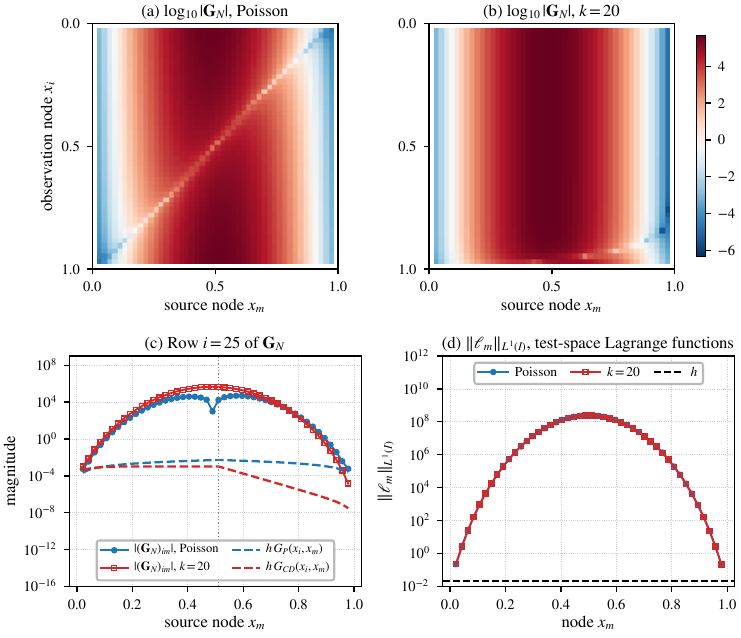}
    \caption{Exact discrete Green's function of the interpolatory system, $N = 48$, $\delta = 0.5$. (a, b) $\log_{10}|\mat G_N|$ for Poisson and for $k = 20$ (interior nodes). (c) Row $i = 25$ against the source node, compared with the continuous kernels $h\,G_P(x_i,x_m)$ and $h\,G_{CD}(x_i,x_m)$. (d) $L^1$ masses of the test-space Lagrange functions, with $h$ for reference.}
    \label{fig:greens_square}
\end{figure}

\subsection{Solvers, Oversampling, Marcinkiewicz--Zygmund Constants, and Large P\'eclet Numbers}
\label{sec:num_solvers}

Figure~\ref{fig:solvers} examines the floating-point mechanism for the problem of Section~\ref{sec:num_1d} with $k = 0$, computed in seven ways: the square system in the recombined basis with normalized columns, $\mat B_D\vect y = \vect f$, solved by LU and by Householder QR; the full square system \eqref{eq:square_system}, with its two boundary rows, solved by LU; the recombined system solved by the truncated singular value decomposition \eqref{eq:tsvd} with $\eps = 10^{-16}\sigma_{\max}$, with raw and with normalized columns; and the regularized oversampled method \eqref{eq:tsvd} with $M = 2N$ and $M = 4N$ equispaced nodes and normalized columns.

The three unregularized solves agree to within a factor of about $20$ for all $N$. Truncation of the raw system is worse than LU by factors of $6$ to $60$ for $\delta = 0.5$ and $N \ge 48$, whereas truncation of the normalized system matches the unregularized solves. The normalization is not incidental: the columns of $\mat B$ scale as $\lambda_j = (j\pi/\tilde L)^2$, and Lemma~\ref{lem:tsvd} and Definition~\ref{def:two_constants} are stated in the norm of the normalized frame. Oversampling by a factor of two or four converges faster, at about $0.8$ to $0.9$ per mode against $0.6$ for LU ($\delta = 0.5$, $16 \le N \le 32$), and lowers the floor to between $2.5\times10^{-15}$ and $7\times10^{-15}$, reached for $N \ge 48$ ($\delta = 0.5$) and $N \ge 24$ ($\delta = 2$), for $\delta = 0.5$ more than three orders of magnitude below the floor of the interpolatory method.

The exact Marcinkiewicz--Zygmund constants of Definition~\ref{def:mz} (panel c) explain the difference. For the interpolatory node set $C_{N-2}$ grows at the rate of the Lebesgue constant, as $e^{0.57N}$ for $\delta = 0.5$ and $e^{0.66N}$ for $\delta = 2$ (fits over $16 \le N \le 48$). Oversampling lowers the rate but does not remove the growth: for $M = 2N$ the constant reaches $30$ and $233$ at $N = 48$, and for $M = 4N$ it stays below $2$ ($\delta = 0.5$) and $3$ ($\delta = 2$) for $N \le 48$ but is $3.5$ and $10$ at $N = 64$. As for polynomials sampled at equispaced points \cite{coppersmith1992growth, rakhmanov2007bounds}, no fixed oversampling factor can be expected to bound $C_M$, since in the Chebyshev variable the nodes have bounded density at the ends $\pm c_0$ of the physical interval; in general the oversampling needed depends on the frame and on the sampling points \cite[Sec.~1]{adcock2020frames2}. Over the range of Figure~\ref{fig:solvers}c the fourfold-oversampled method satisfies the hypothesis of Corollary~\ref{cor:regimes}(ii) with a moderate constant, and the interpolatory method does not.

For $k$ up to $10^{5}$ (panel d, $\delta = 2$, full square system solved by LU) the interpolatory method converges at the same rate for all $k$, and its floor rises from $10^{-14}$ to $10^{-10}$ roughly in proportion to $\eps_{\mathrm{mach}}\norm{\mat A}$, since $\norm{\mat A}$ scales with $k\mu_N$; the regularized oversampled method with normalized columns keeps its floor at $3\times10^{-15}$ to $7\times10^{-15}$, because column normalization removes the growth of the matrix norm with $k$. In this range convection neither destabilizes nor stabilizes the method.

\begin{figure}[ht!]
    \centering
    \includegraphics[width=\textwidth]{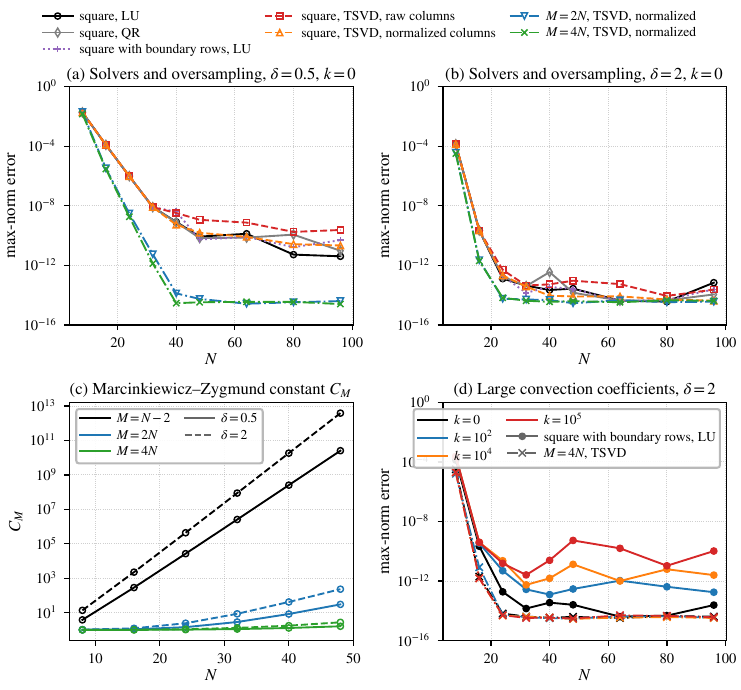}
    \caption{Floating-point behaviour, $u(x) = e^x\sin(\pi x)$ on $(0,1)$. (a, b) Error of the square system in the recombined normalized form solved by LU and by QR, of the full square system with boundary rows solved by LU, of the recombined system solved by truncated singular value decomposition (TSVD) with raw and with normalized columns, and of the regularized oversampled method with $M = 2N$ and $M = 4N$, for $\delta = 0.5$ and $\delta = 2$, $k = 0$. (c) Exact Marcinkiewicz--Zygmund constants $C_M$ of the test space for the interpolatory node set and for $M = 2N$, $4N$. (d) Full square system solved by LU (solid) and regularized $M=4N$ (dash-dotted) for $k$ up to $10^{5}$, $\delta = 2$.}
    \label{fig:solvers}
\end{figure}

\subsection{Evaluation of the Two-Constant Bound}
\label{sec:num_bound}

Figure~\ref{fig:floating_point} evaluates Theorem~\ref{thm:two_constant} for the interpolatory system in its recombined normalized form $\mat B_D\hat{\vect y} = \vect f$, solved by LU in double precision, with $\delta = 0.5$ and $N = 24$ to $96$, in $70$- to $144$-digit arithmetic. For the right singular vectors $\vect v_i$ of $\mat B_D$ in the tail of the spectrum, $\sigma_i < 10^{-2}$, the $L^2(I)$ norms $g_i$ of the trial function $\Phi\mat D\vect v_i$, in the metric of the solution, and of the test function $\Lop\Phi\mat D\vect v_i$, in that of the data, follow power laws
\begin{equation}\label{eq:leakage_law}
g_i \approx c\,\sigma_i^{\,a}, \qquad\text{and in the trial metric}\quad \gamma_t \approx c\,t^{a},
\end{equation}
down to $\sigma_i = 10^{-44}$, with fitted exponents between $0.46$ and $0.49$ in the trial metric and between $0.44$ and $0.48$ in the test metric, increasing with $N$, near $a(\pi/2) = 1/2$; the trial-metric constant $c$ is smaller than the test-metric one by three to four orders of magnitude, increasing with $N$ (panel a). At $t = 10^{-15}$ and $N \ge 48$ neither $\gamma_t$ nor $\Gamma_t$ grows with $N$ (panel b), and the a priori bounds $\gamma_t \le C_\Phi$ and $\Gamma_t \le C_\Phi/t$ of Proposition~\ref{prop:apriori} exceed them by many orders of magnitude. The tail is the same for the convection--diffusion operator: for $k = 10$ and $k = 50$ at $N = 48$, $64$ and $96$ the fitted exponents differ from the Poisson ones by at most $0.004$, the dimensions of $\mathcal V_t$ at $t = 10^{-15}$, $10^{-20}$, $10^{-25}$ and $10^{-30}$ by at most one, and, where both subspaces are nonempty, the values of $\gamma_t$ and $\Gamma_t$ by at most factors of three and five.

These values bear on two conjectures made for the truncated singular value decomposition of the Fourier extension. The boundedness of $\Gamma_t$ at fixed $t$ over the range computed is the counterpart, in the solution metric, of the boundedness of the constant $C_1$ conjectured in \cite[Eq.~(5.18)]{adcock2014numerical}, and substituted into \eqref{eq:lu_bound} the power law \eqref{eq:leakage_law} makes the term $\gamma_\eta\eta^{-1}\cdot\eta\norm{\vect y}_2 \approx c\,\eta^{a}\norm{\vect y}_2$ the floor, the $\eps^{1-a(\gamma;T)}$ accuracy conjectured in \cite[Eq.~(5.22)]{adcock2014numerical}, with their exponent $a(\gamma;T)$ at $\gamma = 1$ in the role of $1 - a$, here for a direct solver and for the error of the solution of the boundary value problem. The small trial-metric constant $c$ reflects the smoothing of the high modes of the test space by $\Lop^{-1}$, which is why the collocation solution is far more accurate than the residual $f - \Lop\hat u_N$ suggests.

Minimized over $t$, with the fourfold-oversampled regularized solution as competitor $\vect y$, the bound \eqref{eq:two_constant_bound} exceeds the actual $L^2(I)$ error of the LU solution by factors $45$ and $14$ at $N = 24$ and $32$, before the floor is reached, and by $4.4$, $6.9$, $8.5$ and $7.7$ at the floor, $N = 48, 64, 80, 96$, and by $5.5$ and $11$ for $\delta = 1$ and $\delta = 2$ at $N = 48$, with fitted exponents $0.59$ and $0.68$ in the trial metric against $a = 0.6$ and $0.685$; it reproduces the flat floor (panel c). With $\hat\rho = 0$ and the exact, exponentially large $\norm{\vect d}_2$ it diverges with the exact solution of the same data beyond $N^\ast$, but overestimates its error by factors growing from $6$ at $N = 48$ to $10^{12}$ at $N = 96$, since the term $\gamma_t\norm{\vect d}_2$ ignores the direction of $\vect d$. Both bounds are evaluated for the data rounded to double precision, which the LU solve receives, that is, for the continuous problem with a right-hand side $\tilde f$ that takes the rounded values at the nodes and satisfies $\norm{f - \tilde f}_{L^\infty(I)} \le 9.4\times10^{-15}$; its solution differs from $u$ by less than $2\times10^{-15}$ in $L^2(I)$. By contrast, the Lebesgue-type factor $(L/\pi)^2\sqrt n\,\Lambda^\Psi_{N,2}$ of Proposition~\ref{prop:classical}, computed exactly, grows at the rate of the Lebesgue constant from $1.2\times10^{3}$ at $N = 24$ to $9.2\times10^{20}$ at $N = 96$, whereas the computed factor $\Gamma_\eta + \gamma_\eta\eta^{-1}$, which multiplies the same quantity in \eqref{eq:lu_bound} up to constants depending on $K$, stays between $3$ and $1.2\times10^{3}$ at the thresholds of panel c. The splitting of the singular values removes the exponential growth; at $N = 96$ the two factors differ by eighteen orders of magnitude.

The competitor of Theorem~\ref{thm:competitor} is available in closed form for this problem, since $u$ is entire and $\rho = \rho^\ast = 1 + \sqrt2$ for $\delta = 0.5$: its residual $\norm{f - \Lop v_p}_{L^\infty(I)}$ decays from $1.8\times10^{-3}$ at $p = 20$ to $3.4\times10^{-17}$ at $p = 60$, at the rate $\rho^\ast$ up to a polynomial factor, while $\norm{v_p''}_{L^2(\tilde I)}$ grows only from $47$ to $176$ between $p = 10$ and $p = 80$; for $\delta = 2$, where $\rho^\ast = 6.31$, the residual is $2.1\times10^{-19}$ at $p = 30$.

\begin{figure}[ht!]
    \centering
    \includegraphics[width=\textwidth]{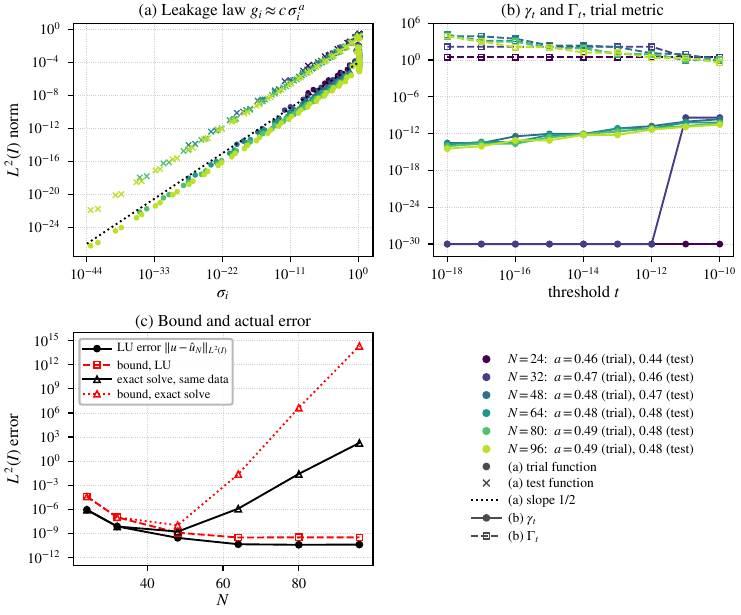}
    \caption{The two-constant bound for the interpolatory system, $\delta = 0.5$. (a) $L^2(I)$ norms of the trial functions (circles) and test functions (crosses) of the right singular vectors of $\mat B_D$ against the singular values, with fitted exponents $a$; the dotted line has slope $1/2$. (b) Leakage $\gamma_t$ (solid) and restricted Marcinkiewicz--Zygmund constant $\Gamma_t$ (dashed) in the trial metric against the threshold $t$; a value of $10^{-30}$ marks an empty subspace $\mathcal V_t$ (no singular value below $t$), which is the case for $N = 24$ at all $t$ shown and for $N = 32$ below $t = 10^{-11}$. (c) Bound \eqref{eq:two_constant_bound} against the actual $L^2(I)$ error, for the LU solution and for the exact solution of the same double-precision data.}
    \label{fig:floating_point}
\end{figure}

\subsection{Potential-Theoretic Rates}
\label{sec:num_potential}

Figure~\ref{fig:potential} tests Theorem~\ref{thm:rates} and Corollary~\ref{cor:exponent}. The local growth rates $\Delta\log\Lambda_N/\Delta N$ and $\Delta\log(1/\sigma_{\min}(\mat W_N))/\Delta N$ of the exactly computed trial-space quantities approach $\kappa_\Lambda$ and $\kappa_\sigma$ from below (panel a): for $\delta = 0.5$ they reach $0.567$ and $1.153$ at $N \approx 112$, against $2\,\mathrm{Cl}_2(\pi/2)/\pi = 0.5831$ and $4\,\mathrm{Cl}_2(\pi/2)/\pi = 1.1662$, and for $\delta = 1$ they reach $0.629$ and $1.602$, against $0.6461$ and $1.6153$. The approach from below is consistent with an algebraic prefactor in the asymptotics, as for the classical rate $\log 2$ of equispaced polynomial interpolation \cite[Eq.~(4.8)]{trefethen1991two}. As the extension grows, $\kappa_\sigma$ grows without bound while $\kappa_\Lambda$ saturates (panel b), which is why the exponent $a(\omega)$ increases with $\delta$. The exponents fitted to the tails of the exactly computed spectra lie within $0.02$ of $a(\omega)$ wherever the tail contains more than thirty directions (panel c); for $\delta = 0.5$ they approach it from below as $N$ grows (Section~\ref{sec:num_bound}). The ratio of Corollary~\ref{cor:exponent} itself converges slowly, at the rate $\log N/N$ of Proposition~\ref{prop:lagrange_norms}: for the central Lagrange function of the trial frame it decreases from $0.880$ at $N = 16$ to $0.553$ at $N = 128$ for $\delta = 0.5$, towards $1/2$.

\begin{figure}[ht!]
    \centering
    \includegraphics[width=\textwidth]{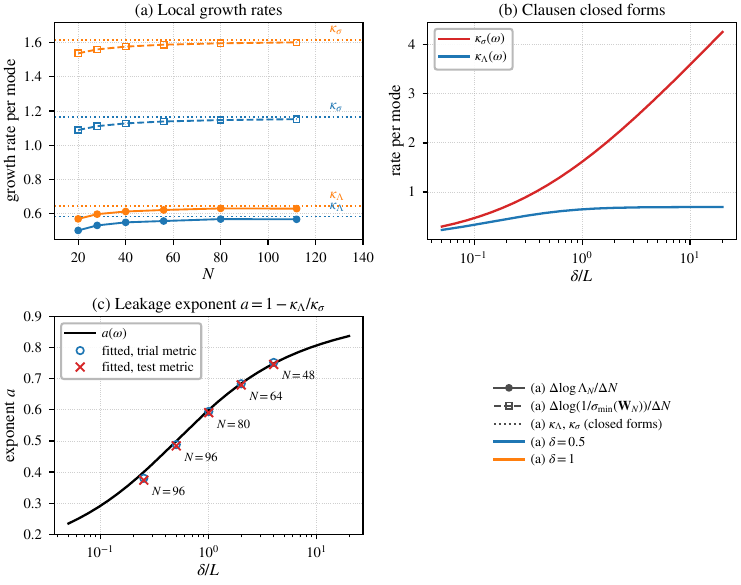}
    \caption{Potential theory of the equispaced sine frame. (a) Local growth rates of $\Lambda_N$ (circles) and of $1/\sigma_{\min}(\mat W_N)$ (squares) for $\delta = 0.5$ and $\delta = 1$, with the predicted constants $\kappa_\Lambda$ and $\kappa_\sigma$ of Theorem~\ref{thm:rates} as dotted lines. (b) $\kappa_\Lambda(\omega)$ and $\kappa_\sigma(\omega)$ against $\delta/L$. (c) The leakage exponent $a(\omega) = 1 - \kappa_\Lambda/\kappa_\sigma$ of Corollary~\ref{cor:exponent} (line) and the exponents fitted to the tails of the exact spectra in the trial (circles) and test (crosses) metrics, labelled by $N$.}
    \label{fig:potential}
\end{figure}

\paragraph{The tail singular functions.} Figure~\ref{fig:tail}a,b shows the zeros and the nodal values of tail singular functions of the test space for $N = 48$, $\delta = 0.5$. All their zeros are real, as the multiple-precision roots of the polynomial factors show. In an outer layer of the physical interval, whose width grows along the tail, they lie within a tenth of a spacing of the nodes; in the centre they form a lattice sparser than the nodes, and the function is flat at the nodal level, at nodes and between them alike; the remaining zeros, which the figure does not show, lie in the extension, fewer along the tail, and keep a gap next to the ends of the physical interval. This is the structure that $\lambda_\tau$ describes (Proposition~\ref{prop:equilibrium}, Corollary~\ref{cor:layer}).

\paragraph{The tail law.} We solve the equilibrium problem of Definition~\ref{def:equilibrium} numerically. Restricted to densities that are constant on $130$ to $150$ cells of $[0,\pi/2]$, mirrored and graded geometrically towards $\theta_0$ on the scale of $\tau$, the functional $J_\tau$ is a convex quadratic function of the cell densities, whose coefficients are double integrals of the kernel in closed form through $\mathrm{Cl}_2$ and the Clausen function $\mathrm{Cl}_3$ with $\mathrm{Cl}_3' = -\mathrm{Cl}_2$; it is minimized under the bounds of Definition~\ref{def:equilibrium} and the constraint on the mass. From the computed measure we build the configuration and the interpolation configuration of Lemma~\ref{lem:discretize}, with the saturated cells as the layer, and evaluate $\ell(\mathcal C)$ and $\Phi(\mathcal S,\tau)$ on a grid of $2\times10^4$ points. By Corollary~\ref{cor:sharpness} and the Lipschitz bound of Proposition~\ref{prop:tail_properties}, the two values bracket $r_\omega(\tau)$ up to $|\tau(\mathcal C) - \tau| \le 7\times10^{-4}$ and to the grid; these computations are not certified. In Table~\ref{tab:dual} the width of each bracket measures the accuracy of the computation, and each row is read against the theory as follows.

For $\delta = L/2$ the bracket contains the chord at all three levels and the layer and the gap have equal widths; the computed $E_\tau(\theta_0)$, not listed, equals $-\log 2$ to five digits, as the proof of Theorem~\ref{thm:tail_duality} with $\omega = \omega' = \pi/2$ requires. Divided by $\kappa_\sigma(\omega)$, the brackets of the dual pairs $(L, L/4)$ and $(2L, L/8)$ satisfy \eqref{eq:tail_duality} to within their widths. For $\delta = L$ and $\delta = 2L$ the whole bracket lies below the chord, and for their duals above it. The mass of the nodes increases with the level towards one, the mass of the node measure $\mu_\omega$, which is the minimizer for $\tau \ge \kappa_\sigma$ (Proposition~\ref{prop:equilibrium}(iv)). On the physical interval the envelope is largest at $\theta_0$, and the saturated set is a single layer at each end, as Proposition~\ref{prop:equilibrium}(iii) states.

\begin{table}[ht!]
\centering\small
\caption{The tail law from the equilibrium problem: the lower bound $\ell(\mathcal C)$ and the upper bound $\Phi(\mathcal S,\tau)$ of the configuration and the interpolation configuration built from the computed equilibrium measure $\lambda_\tau$ (Lemma~\ref{lem:discretize}), which bracket $r_\omega(\tau)$; the chord $-a(\omega)\tau$; the mass $\lambda_\tau(\Theta_I)$ of the nodes; and the widths of the layer, $\theta_\tau - \theta_0$, and of the gap of the support in the extension next to $\theta_0$, in units of $\tau$. The pairs $\delta = L$, $L/4$ and $\delta = 2L$, $L/8$ are dual (Theorem~\ref{thm:tail_duality}). Computed, not certified.}
\label{tab:dual}
\setlength{\tabcolsep}{4pt}
\begin{tabular}{cccccccc}
$\delta/L$ & $\tau/\kappa_\sigma$ & $\ell(\mathcal C)$ & $\Phi(\mathcal S,\tau)$ & $-a\tau$ & $\lambda_\tau(\Theta_I)$ & layer$/\tau$ & gap$/\tau$ \\ \hline
$1/2$ & $0.2$ & $-0.11718$ & $-0.11546$ & $-0.11662$ & $0.681$ & $0.47$ & $0.47$ \\
$1/2$ & $0.5$ & $-0.29178$ & $-0.29120$ & $-0.29156$ & $0.843$ & $0.51$ & $0.51$ \\
$1/2$ & $0.8$ & $-0.46654$ & $-0.46645$ & $-0.46650$ & $0.953$ & $0.56$ & $0.56$ \\
$1$ & $0.2$ & $-0.21027$ & $-0.20841$ & $-0.19384$ & $0.574$ & $0.19$ & $0.52$ \\
$1$ & $0.5$ & $-0.50822$ & $-0.50667$ & $-0.48460$ & $0.789$ & $0.22$ & $0.52$ \\
$1$ & $0.8$ & $-0.79031$ & $-0.78992$ & $-0.77536$ & $0.936$ & $0.25$ & $0.55$ \\
$1/4$ & $0.2$ & $-0.05697$ & $-0.05607$ & $-0.06461$ & $0.787$ & $1.11$ & $0.39$ \\
$1/4$ & $0.5$ & $-0.15003$ & $-0.14973$ & $-0.16153$ & $0.894$ & $1.04$ & $0.45$ \\
$1/4$ & $0.8$ & $-0.25117$ & $-0.25075$ & $-0.25845$ & $0.968$ & $1.10$ & $0.50$ \\
$2$ & $0.5$ & $-0.78836$ & $-0.78678$ & $-0.73510$ & $0.739$ & $0.09$ & $0.49$ \\
$1/8$ & $0.5$ & $-0.07143$ & $-0.07127$ & $-0.08457$ & $0.935$ & $1.97$ & $0.36$ \\
\end{tabular}
\end{table}

\paragraph{The local exponent.} Figure~\ref{fig:tail}c shows the local slope $-dr_\omega/d\tau$, computed from the extremal problem \eqref{eq:reduced_problem} of Proposition~\ref{prop:extremal}, and the local exponents of the exact spectra in the test metric, the slopes $\Delta\log g_i/\Delta\log\sigma_i$ with $g_i$ as in \eqref{eq:leakage_law}. For a free set that is an interval, $F = [-c,c]$, and a fixed evaluation point, \eqref{eq:reduced_problem} is a linear program in $(\nu, b)$, which we discretize with piecewise-constant densities on $[-1.6,1.6]$, exact cell potentials and constraints at cell centres and edges, maximizing over $c$; at $\tau = \kappa_\sigma$ it reproduces $\kappa_\Lambda - \kappa_\sigma$ to within $0.007$ for $\delta = 0.5$, $1$ and $2$. For $\delta = L/2$ the local slope stays within $0.007$ of $1/2$ for $\tau/\kappa_\sigma \le 0.9$, as Corollary~\ref{cor:self_dual} requires, and the exact local exponents lie between $0.48$ and $0.50$ away from the first and last few ($N = 96$). For $\delta = L$ and $\delta = 2L$ the tail law is convex, its local slope decreasing along the tail through the endpoint exponent $a(\omega)$; the exact spectra show the same variation up to two hundredths below the curve, approaching it from below as $N$ grows. For $\delta = L/4$ it is concave, the local slope increasing through $a(\omega)$, in the exact spectra as well, and the slopes for $\delta = L$ and $\delta = L/4$ add up to one to within $0.014$ for $\tau/\kappa_\sigma \le 0.9$, as Theorem~\ref{thm:tail_duality} requires. The chord of $r_\omega$ over $[0,\kappa_\sigma]$ has slope $-a(\omega)$, which is why a single exponent fits the tail to two digits for $\delta \ge L/2$; but the exponent is a property of the endpoint and of the tail as a whole, not of each level, except in the self-dual case.

\paragraph{The start of the tail.} At $\tau/\kappa_\sigma = 0.02$ and $0.05$, with the cells next to $\theta_0$ scaled with $\tau$, the brackets give $r_\omega(\tau)/\tau$ in $[-0.342, -0.329]$, $[-0.669, -0.658]$ and $[-0.802, -0.786]$ for $\delta = L/4$, $L$ and $2L$, against $-s(\omega) = -1/3$, $-2/3$ and $-4/5$ in Theorem~\ref{thm:initial_slope}. The ratios of the widths of the gap and of the layer of $\lambda_\tau$, $0.33$, $3.05$ and $9.3$--$9.6$, match $g/l = \tan^2\theta_0 = 1/3$, $3$ and $9.47$ in \eqref{eq:layer_widths}, and their sums are within $8\%$ of $g + l = 2\omega/(\pi\sin\omega)$, the cells resolving the widths to a few per cent. These values are computed, not certified; the certified statements are collected in Section~\ref{sec:num_certified}.

\begin{figure}[ht!]
    \centering
    \includegraphics[width=\textwidth]{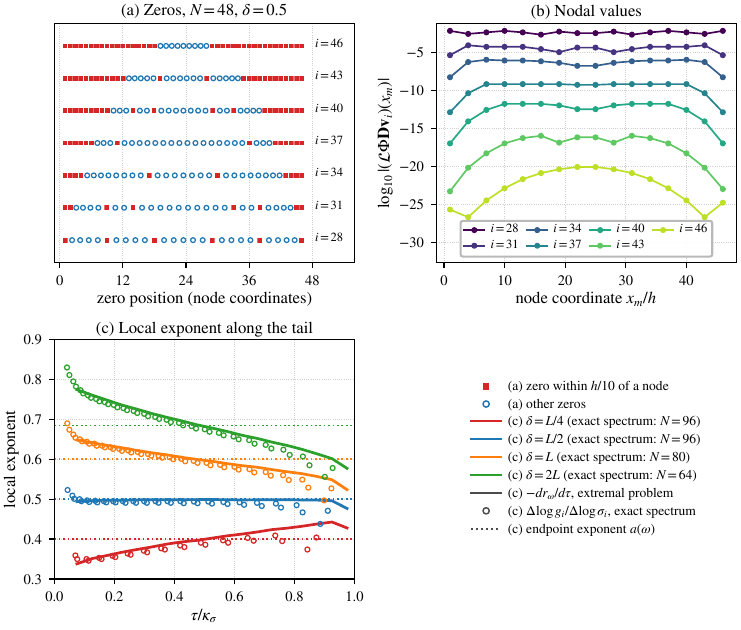}
    \caption{The tail of the spectrum. (a) Real zeros in the physical interval of tail singular functions of the test space ($N = 48$, $\delta = 0.5$), in node coordinates; red squares are zeros within a tenth of a spacing of a node, open circles the other zeros. (b) Their nodal values $\log_{10}|(\Lop\Phi\mat D\vect v_i)(x_m)|$ along the physical interval. (c) Local exponent along the tail: the slope $-dr_\omega/d\tau$ of the tail law, computed from the extremal problem \eqref{eq:reduced_problem}, for four extension ratios (lines), the local exponents $\Delta\log g_i/\Delta\log\sigma_i$ of the exact spectra in the test metric (circles, five-point fits), and the endpoint exponents $a(\omega)$ of Corollary~\ref{cor:exponent} (dotted).}
    \label{fig:tail}
\end{figure}

\subsection{Certified Configurations on Both Sides of the Critical Extension}
\label{sec:num_certified}

The two constructions of Section~\ref{sec:tail} reduce the sharp power law at a given level to the evaluation of explicit functions of Clausen type: a configuration $\mathcal C$ with $\ell(\mathcal C) > -(\kappa_\sigma - \kappa_\Lambda)$ and $\ell(\mathcal C) + a\tau(\mathcal C) > 0$ shows that the law fails (Corollary~\ref{cor:no_power_law}), and an interpolation configuration $\mathcal S$ with $A_1(\mathcal S)/(1-a) \le -A_2(\mathcal S)/a$ proves its bounds up to a power of $N$ at the levels of the interval \eqref{eq:certified_interval} (Corollary~\ref{cor:certified}). The configurations below were found by the linear programs of Appendix~\ref{app:tail} and evaluated with certified error bounds, from grids and the modulus of continuity of the envelope, as described in Appendix~\ref{app:tail}.

\begin{corollary}[Certified Configurations on Both Sides of the Critical Extension]
\label{cor:certified_instances}
Let the nodes be the interior nodes of the $N$-grid ($M = n$), and let $\Lop = \Lop_P$.
\begin{enumerate}
    \item[(i)] For $\delta = L/4$ the configuration $\mathcal C$ of Table~\ref{tab:config_quarter} satisfies $\ell(\mathcal C) + a(\omega)\tau(\mathcal C) \ge 0.0084$, so there are no $N_1$ and $C$ with $\gamma_t \le Ct^{\,a(\omega)}$ for all $N \ge N_1$ and all $t > 0$; in particular the sharp power law \eqref{eq:sharp_law} fails.
    \item[(ii)] There are $N_0$ and $C$, depending only on $L$, such that for all $N \ge N_0$ and $t = e^{-N\tau}$: for $\delta = L$ ($\omega = \pi/3$, $a = 3/5$) and every $\tau \in [0.097, 0.913]\,\kappa_\sigma(\pi/3)$, $\gamma_t \le 2N^Ct^{\,3/5}$ and $\Gamma_t \le 2N^Ct^{-2/5}$; for $\delta = 2L$ ($\omega = \pi/5$) and every $\tau \in [0.080, 0.960]\,\kappa_\sigma(\pi/5)$, $\gamma_t \le 2N^Ct^{\,a(\pi/5)}$ and $\Gamma_t \le 2N^Ct^{\,a(\pi/5)-1}$.
\end{enumerate}
\end{corollary}
\begin{proof}
(i) The configuration satisfies Definition~\ref{def:configuration} with $\xi = 1/50$, and the certified bounds on $\ell$ and $\tau$, together with the hypothesis $\ell > -(\kappa_\sigma - \kappa_\Lambda)$, are established in Appendix~\ref{app:tail}; the claim is Corollary~\ref{cor:no_power_law}. (ii) The interpolation configurations of Tables~\ref{tab:certified_L} and \ref{tab:certified_2L} satisfy Definition~\ref{def:interp_configuration} with $\xi = 1/50$. For each of them the certified values $\bar A_1 \ge A_1(\mathcal S)$ and $\bar A_2 \ge A_2(\mathcal S)$ of Appendix~\ref{app:tail} satisfy $\bar A_1/(1-a) \le -\bar A_2/a$, so that the interval $[\bar A_1/(1-a), -\bar A_2/a]$ is contained in \eqref{eq:certified_interval}, and these intervals cover the two ranges of levels; Corollary~\ref{cor:certified} applies to each of them, and the constants are the largest of the finitely many constants of Theorem~\ref{thm:dual}.
\end{proof}

Part (i) is contained in Corollary~\ref{cor:critical}(i), which covers every $\delta < L/2$; the configuration is an explicit witness at the level $\tau(\mathcal C)$, and it does not use the equilibrium problem. By Corollary~\ref{cor:sharpness} and Theorem~\ref{thm:tail_duality}, the certificates of (ii) show that the tail laws of $\delta = L$ and $\delta = 2L$ lie strictly below the chord inside the certified intervals, and those of their duals $\delta = L/4$ and $\delta = L/8$ strictly above it; Corollary~\ref{cor:critical} shows this only at the start of the tail. For $\delta = L/2$ the tail law is the chord (Corollary~\ref{cor:self_dual}). Whether the sharp power law holds for $\delta = L$ and $\delta = 2L$, with a constant independent of $N$ at all levels, is not decided (Section~\ref{sec:conclusion}).

\subsection{Controlled Two-Dimensional Comparison on the Box}
\label{sec:num_2d}

On $R_0 = (0,1)^2$ with $u(x,y) = e^{x+y/2}\sin\bigl(\pi(x+y)/2\bigr)\cos y$ we compare the Poisson operator, $\mathbf k = (10,10)$ and $\mathbf k = (5\cos\pi x, 5\cos\pi y)$ for $\delta \in \{0.5, 1, 2\}$, with the interpolatory tensor-product system. In Figure~\ref{fig:controlled_2d} the three histories coincide in the geometric phase, whose rate depends on $\delta$; for $\delta = 1$ and $\delta = 2$ the floors of the three operators lie within a factor of about $25$ of one another, and saturation sets in at about the same $N$ as in one dimension (Figure~\ref{fig:controlled_1d}a).

\begin{figure}[ht!]
    \centering
    \includegraphics[width=\textwidth]{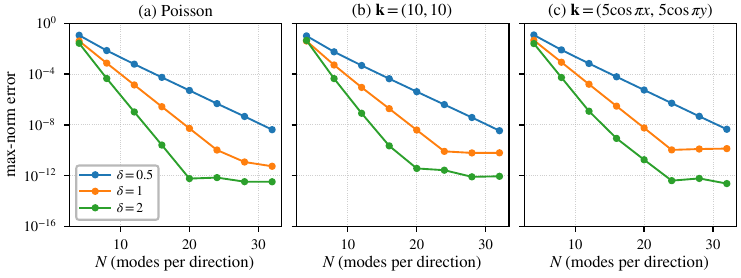}
    \caption{Controlled two-dimensional comparison on $(0,1)^2$: maximum-norm error of the interpolatory system for (a) Poisson, (b) $\mathbf k = (10,10)$, (c) $\mathbf k = (5\cos\pi x, 5\cos\pi y)$.}
    \label{fig:controlled_2d}
\end{figure}

\subsection{A Disk with Relocated Boundary Nodes}
\label{sec:num_disk}

Finally we apply the relocation procedure of Remark~\ref{rem:relocation} to the disk $\Omega = \{|\mathbf x| < 0.9\}$ in $R_0 = (-1,1)^2$, with the same exact solution as in Section~\ref{sec:num_2d}, for the Poisson operator and for $\mathbf k = (10,10)$. Figure~\ref{fig:disk}a shows the relocated grid for $N = 24$: $60$ nodes are moved onto the circle, by at most $0.0397$ against $h/2 = 0.0435$. Panels (b, c) show the maximum-norm error inside $\Omega$. Convergence is geometric on average but not monotone, because the relocated node set changes with $N$; for $\delta = 1$ and $\delta = 2$ the smallest errors of the two operators lie within a factor of five of one another, so that again the two operators behave alike. In these computations the relocation of the sampling set preserves the floating-point behaviour of the frame, which is what makes the extended-domain method usable on curved domains.

\begin{figure}[ht!]
    \centering
    \includegraphics[width=\textwidth]{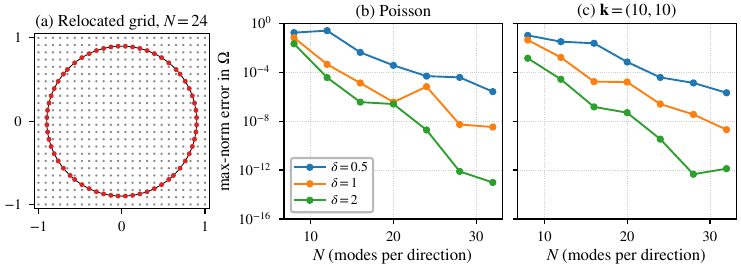}
    \caption{Disk $\Omega=\{|\mathbf x|<0.9\}$ embedded in $(-1,1)^2$. (a) Relocated grid, $N = 24$: PDE nodes (grey) and boundary nodes moved onto the circle (red). (b, c) Maximum-norm error in $\Omega$ for Poisson and for $\mathbf k = (10,10)$.}
    \label{fig:disk}
\end{figure}

\section{Conclusion}
\label{sec:conclusion}

Interpolatory extended-domain spectral collocation solves a square system built from a frame without regularization, a configuration that the impossibility theorem for equispaced data declares unstable and that the theory of regularized frame approximation does not cover. One identity explains its behaviour: the collocation solution is the continuous solution operator applied to the interpolant of the data in the test space $\Psi_N = \Lop(\Phi_N)$. The exact solution therefore amplifies errors in the data by up to the Lebesgue constant of this space, which must grow exponentially for geometric convergence and, for the Poisson operator, grows at the explicit rate of the trial frame; a backward-stable floating-point solve, if the perturbed matrix has no singular value far below the backward error, attains the accuracy of a frame approximation regularized at the level of the backward error, with an error controlled by two constants of the sampled frame. For the Poisson operator these constants follow, along the whole tail of the spectrum and at the exponential scale, a law given by a constrained equilibrium problem of logarithmic potential theory, which satisfies an exact duality between the extensions $\delta$ and $L^2/(4\delta)$, is linear in the level for the self-dual extension $\delta = L/2$, and starts with slope minus the fraction of the extended domain occupied by the extension. In all computations the differential operator plays no essential role in the stability of the method, as Lemma~\ref{lem:phase} suggests: convection shifts the phases of the test-space modes by amounts that are small for the high modes, where the ill-conditioning of the frame lives.

The main hypotheses of the theorems are the nonsingularity of the collocation matrix, which for the Poisson operator holds for all sufficiently large $N$, the Poisson operator for the Lebesgue-type rates of the test space and for the tail law, analyticity of the data for the competitor, and, for the bounds on a computed solution, backward stability of the solver with a lower bound on the smallest singular value of the perturbed matrix. What a sharp a priori theory would still have to supply is the size of the two constants at the fixed roundoff threshold, which sets the accuracy floor: the tail law determines them at thresholds that decrease exponentially with $N$, not at a fixed threshold $t > 0$ as $N \to \infty$; they are computed exactly (Section~\ref{sec:num_bound}), and the a priori bounds, although uniform in $N$, lie far above them.

The practical consequences are that the extension ratio, the analyticity of the solution and the unit roundoff, not the operator, determine the accuracy of the interpolatory method, the operator entering at large P\'eclet numbers only through the norm of the matrix; that oversampling with a normalized truncated singular value decomposition is the formulation of choice when accuracy beyond the interpolatory floor is required; and that comparisons between operators must hold the domain, the extension and the solution fixed.

The tail law leaves two questions open. The first is the side of the chord. By Corollary~\ref{cor:critical} the tail law starts above the chord for $\delta < L/2$, so that the sharp power law $\gamma_t \le Ct^{a(\omega)}$, $\Gamma_t \le Ct^{a(\omega)-1}$, with $C$ independent of $N$ and $t$, fails, and below it for $\delta > L/2$. We conjecture that for $\delta > L/2$ the tail law lies below the chord along the whole tail, as the computations show and as certified interpolation configurations prove for $\delta = L$ and $\delta = 2L$ along most of the tail, and equivalently, by Theorem~\ref{thm:tail_duality}, that for $\delta < L/2$ it lies above the chord. Even then the sharp power law for $\delta > L/2$ would be established only at the exponential scale; a constant independent of $N$ requires control of the factors $e^{o(N)}$ as well. The second question is the structure of the equilibrium measure at the start of the tail: the proof of Theorem~\ref{thm:initial_slope} uses a gap of width $g\tau$ and a layer of width $l\tau$ at each end of the physical interval, and the computed measures $\lambda_\tau$ have these widths, but that $\lambda_\tau$ has them, beyond the existence of a layer and a gap (Corollary~\ref{cor:layer}), is not proved, and the order $\tau^{3/2}$ of the remainder in \eqref{eq:initial_slope} is not known to be sharp.

The tail law and the bounds along the tail are proved for the Poisson operator only. For convection--diffusion the computed tail is the same (Section~\ref{sec:num_bound}), but a proof is open. For a trial function $v = \sin\theta\,q(\cos\theta)$ the test function is $\Lop v = c^2\sin\theta\,r'(\cos\theta) + kc\,r(\cos\theta)$, with $r(z) = zq(z) - (1-z^2)q'(z)$ of degree $N$ and $c = \pi/\tilde L$. For $k = 0$ it is $\sin\theta$ times a polynomial of degree $N-1$, the structure that the explicit polynomials of Theorem~\ref{thm:tail_lower} and the Lagrange interpolation of Theorem~\ref{thm:dual} use; for $k > 0$ it is a trigonometric polynomial of degree $N$ without odd symmetry in $\theta$. Interpolation on the whole circle then needs about twice as many points, half of them where only a bound of order one is known, and small nodal values do not pass to the Poisson test function $-v''$, which differs from $\Lop v$ by $kv'$, an integral of $-v''$ over the whole extended interval. A sharp bound would have to use that the sine and the cosine part come from the same polynomial $r$.

The analysis of the two constants has not been carried over to the tensor-product domains of Section~\ref{sec:multi} either, and the same change of variables should give the rates of other frames of Laplacian eigenfunctions.

\appendix

\section{Standard Facts}
\label{app:standard}

This appendix collects the standard facts used in the paper, with proofs: the one-dimensional Dirichlet problem (Lemma~\ref{lem:dirichlet_1d}), the explicit Green's functions of the two model operators (Proposition~\ref{prop:continuous_greens}), the Dirichlet problem on a box (Lemma~\ref{lem:dirichlet_box}), the Clausen function (Lemma~\ref{lem:cl2}) and regularized least squares (Lemma~\ref{lem:tsvd}).

\begin{lemma}[The One-Dimensional Dirichlet Problem]
\label{lem:dirichlet_1d}
Let $k \in C(\bar I)$, $K(x) := \int_0^x k(s)\,ds$, $E(x) := \int_0^x e^{K(s)}\,ds$, and
\begin{equation}\label{eq:green_general}
G(x,y) := \frac{e^{-K(y)}\,E(\min\{x,y\})\,\bigl(E(L) - E(\max\{x,y\})\bigr)}{E(L)}, \qquad x, y \in \bar I .
\end{equation}
\begin{enumerate}
\item[(i)] (Maximum principle.) If $v \in C^2(I)\cap C(\bar I)$ and $\Lop v \le 0$ on $I$, then $\max_{\bar I} v = \max\{v(0), v(L)\}$. Consequently, if $\Lop v = 0$ on $I$ then $\min\{v(0),v(L)\} \le v \le \max\{v(0),v(L)\}$ on $\bar I$, and if in addition $v(0) = v(L) = 0$ then $v \equiv 0$.
\item[(ii)] (Solution formula.) For $g \in C(\bar I)$ the function $v := \int_0^L G(\cdot,y)\,g(y)\,dy$ belongs to $C^2(\bar I)$ and is the unique element of $C^2(I)\cap C(\bar I)$ with $\Lop v = g$ on $I$ and $v(0) = v(L) = 0$. We write $v = \Lop^{-1}g$.
\item[(iii)] $G \ge 0$ is continuous on $\bar I\times\bar I$,
\[
\norm{\Lop^{-1}}_\infty = \sup_{g\ne0}\frac{\norm{\Lop^{-1}g}_{L^\infty(I)}}{\norm{g}_{L^\infty(I)}} = \max_{x\in\bar I}\int_0^L G(x,y)\,dy = \max_{\bar I}\bigl(\Lop^{-1}\mathbf 1\bigr),
\]
and the norm $\norm{\Lop^{-1}}_{2\to2}$ of $\Lop^{-1}$ on $L^2(I)$ is at most $\norm{G}_{L^2(I\times I)}$.
\end{enumerate}
\end{lemma}
\begin{proof}
(i) Let $\alpha := \norm{k}_{L^\infty(I)} + 1$ and $v_\eps := v + \eps e^{\alpha x}$ for $\eps > 0$. Then $\Lop v_\eps = \Lop v + \eps\alpha\bigl(k(x) - \alpha\bigr)e^{\alpha x} \le -\eps\alpha e^{\alpha x} < 0$ on $I$. If $v_\eps$ attained its maximum over $\bar I$ at a point $x_0 \in I$, then $v_\eps'(x_0) = 0$ and $v_\eps''(x_0) \le 0$, so $\Lop v_\eps(x_0) = -v_\eps''(x_0) \ge 0$, a contradiction. Hence $\max_{\bar I}v \le \max_{\bar I}v_\eps = \max\{v_\eps(0), v_\eps(L)\} \le \max\{v(0),v(L)\} + \eps e^{\alpha L}$, and $\eps \to 0$ gives the first claim. If $\Lop v = 0$, the first claim applies to $v$ and to $-v$; if moreover $v(0) = v(L) = 0$, both bounds give $v \equiv 0$.

(ii) Uniqueness: the difference of two solutions satisfies (i) with zero boundary values. Existence: $E$ is strictly increasing with $E(0) = 0$, and $\Lop v = g$ is equivalent to $(e^{-K}v')' = -e^{-K}g$ because $K' = k$. Splitting the integral at $y = x$,
\[
v(x) = \frac{E(L) - E(x)}{E(L)}\int_0^x e^{-K(y)}E(y)\,g(y)\,dy + \frac{E(x)}{E(L)}\int_x^L e^{-K(y)}\bigl(E(L)-E(y)\bigr)g(y)\,dy .
\]
Differentiating, the two terms produced by the variable limits cancel, and with $E' = e^{K}$ one obtains $e^{-K(x)}v'(x) = E(L)^{-1}\bigl[-\int_0^x e^{-K}Eg + \int_x^L e^{-K}(E(L)-E)g\bigr] = c - \int_0^x e^{-K}g$, where $c := E(L)^{-1}\int_0^L e^{-K}(E(L) - E)g$. Differentiating once more gives $(e^{-K}v')' = -e^{-K}g$, so $\Lop v = g$; $v(0) = v(L) = 0$ because $E(0) = 0$ and $E(L) - E(L) = 0$; and $v' = e^{K}(c - \int_0^x e^{-K}g) \in C^1(\bar I)$, so $v \in C^2(\bar I)$.

(iii) $G \ge 0$ and continuous because $E$ is continuous and increasing with $E(0) = 0$. For $g \in C(\bar I)$, $|\Lop^{-1}g(x)| \le \norm{g}_\infty\int_0^L G(x,y)\,dy$ with equality for $g \equiv 1$, and $\int_0^L G(x,y)\,dy = (\Lop^{-1}\mathbf 1)(x)$ by (ii). The $L^2$ bound is the Cauchy--Schwarz inequality applied to the integral kernel.
\end{proof}

\begin{proposition}[Continuous Green's Functions and Inverse Norms]
\label{prop:continuous_greens}
On $I=(0,L)$ with homogeneous Dirichlet conditions,
\begin{align}
G_P(x,y) &= \begin{cases} x(L-y)/L, & x \le y, \\ y(L-x)/L, & x > y, \end{cases} & \norm{\Lop_P^{-1}}_\infty &= \frac{L^2}{8}, \label{eq:GP}\\
G_{CD}(x,y) &= \begin{cases} a(y)\,(e^{kx}-1), & x \le y, \\ b(y)\,(e^{kx} - e^{kL}), & x > y, \end{cases}
& \norm{\Lop_{CD}^{-1}}_\infty &\le \frac{L}{k}\quad (k>0), \label{eq:GCD}
\end{align}
with $b(y) = -(1-e^{-ky})/\bigl(k(e^{kL}-1)\bigr)$ and $a(y) = b(y) + e^{-ky}/k$. For $k>0$ and $x \le y$, $0 \le G_{CD}(x,y) \le k^{-1} e^{-k(y-x)}$, whereas for $x > y$, $G_{CD}(x,y) = k^{-1}(1-e^{-ky})(1 - e^{-k(L-x)})/(1 - e^{-kL})$.
\end{proposition}
\begin{proof}
Specialize \eqref{eq:green_general} to $K(x) = kx$: for $k = 0$, $E(x) = x$ and \eqref{eq:green_general} is \eqref{eq:GP}; for $k > 0$, $E(x) = (e^{kx}-1)/k$, and for $x \le y$ one obtains $G_{CD}(x,y) = (e^{kx}-1)(e^{k(L-y)}-1)/(k(e^{kL}-1)) = a(y)(e^{kx}-1)$, while for $x > y$, $G_{CD}(x,y) = (1-e^{-ky})(e^{kL}-e^{kx})/(k(e^{kL}-1)) = b(y)(e^{kx}-e^{kL})$, which is the stated closed form. The bound for $x \le y$ follows from $(e^{kx}-1)(e^{k(L-y)}-1) \le (e^{kL}-1)e^{-k(y-x)}$, which after expansion reduces to $e^{kx} + e^{k(L-y)} \ge 1 + e^{-k(y-x)}$ and holds because $e^{kx} \ge 1$, $e^{k(L-y)} \ge 1$ and $e^{-k(y-x)} \le 1$. By Lemma~\ref{lem:dirichlet_1d}(iii) the inverse norms are $\max_{\bar I}\Lop^{-1}\mathbf 1$; solving $\Lop v = 1$ with zero boundary values gives $\Lop_P^{-1}\mathbf 1 = x(L-x)/2$, whose maximum is $L^2/8$, and $\Lop_{CD}^{-1}\mathbf 1 = x/k - L(e^{kx}-1)/(k(e^{kL}-1)) \le x/k \le L/k$.
\end{proof}

\begin{lemma}[The Dirichlet Problem on a Box]
\label{lem:dirichlet_box}
Let $U \subset \mathbb R^d$ be a bounded open set contained in the slab $\{0 \le x_1 \le \ell\}$, let $\Lop = -\Delta + \mathbf k\cdot\nabla$ with a bounded function $\mathbf k: U \to \mathbb R^d$, and put $\alpha := \sup_U|k_1| + 1$.
\begin{enumerate}
\item[(i)] If $v \in C^2(U)\cap C(\bar U)$ and $\Lop v \le 0$ in $U$, then $\max_{\bar U}v = \max_{\partial U}v$.
\item[(ii)] If $v \in C^2(U)\cap C(\bar U)$, $\Lop v = g$ in $U$ with $g$ bounded, and $v = 0$ on $\partial U$, then $\sup_U|v| \le (e^{\alpha\ell} - 1)\sup_U|g|$. In particular the Dirichlet problem for $\Lop$ has at most one solution in $C^2(U)\cap C(\bar U)$.
\item[(iii)] If $U = R_0 = \prod_{i=1}^d(0,L_i)$, $\mathbf k$ is constant, $v \in C^2(\bar R_0)$, $\Lop v = g$ and $v = 0$ on $\partial R_0$, then $\norm{v}_{L^2(R_0)} \le \bigl(\sum_{i=1}^d\pi^2/L_i^2\bigr)^{-1}\norm{g}_{L^2(R_0)}$.
\end{enumerate}
\end{lemma}
\begin{proof}
(i) Put $v_\eps := v + \eps e^{\alpha x_1}$, $\eps > 0$. Then $\Lop v_\eps = \Lop v + \eps\alpha(k_1 - \alpha)e^{\alpha x_1} < 0$ in $U$. At an interior maximum of $v_\eps$ the gradient vanishes and the Hessian is negative semidefinite, so $\Lop v_\eps = -\Delta v_\eps \ge 0$ there, a contradiction; hence $\max_{\bar U}v \le \max_{\bar U}v_\eps = \max_{\partial U}v_\eps \le \max_{\partial U}v + \eps e^{\alpha\ell}$, and $\eps \to 0$ gives (i).

(ii) Let $\phi(\mathbf x) := e^{\alpha\ell} - e^{\alpha x_1}$, so that $\phi \ge 0$ on $\bar U$ and $\Lop\phi = \alpha(\alpha - k_1)e^{\alpha x_1} \ge \alpha \ge 1$ in $U$. With $\bar g := \sup_U|g|$, the functions $w_\pm := \bar g\phi \pm v$ satisfy $\Lop w_\pm \ge \bar g \pm g \ge 0$ in $U$ and $w_\pm \ge 0$ on $\partial U$, so (i) applied to $-w_\pm$ gives $w_\pm \ge 0$, i.e.\ $|v| \le \bar g\phi \le (e^{\alpha\ell}-1)\bar g$. Uniqueness is the case $g = 0$.

(iii) Multiplying $\Lop v = g$ by $v$ and integrating by parts, using $v = 0$ on $\partial R_0$, gives $\int_{R_0}|\nabla v|^2 + \int_{R_0}(\mathbf k\cdot\nabla v)v = \int_{R_0}gv$, and $\int_{R_0}(\mathbf k\cdot\nabla v)v = \tfrac12\int_{R_0}\mathbf k\cdot\nabla(v^2) = 0$ for constant $\mathbf k$. The functions $e_{\mathbf j}(\mathbf x) = \prod_i\sqrt{2/L_i}\sin(j_i\pi x_i/L_i)$ form an orthonormal basis of $L^2(R_0)$ of eigenfunctions of $-\Delta$ with eigenvalues $\lambda_{\mathbf j} = \sum_i j_i^2\pi^2/L_i^2 \ge \lambda^\ast := \sum_i\pi^2/L_i^2$; by Green's second identity, which has no boundary terms because $v$ and $e_{\mathbf j}$ vanish on $\partial R_0$, the coefficients of $-\Delta v$ are $\lambda_{\mathbf j}$ times those of $v$, so $\int_{R_0}|\nabla v|^2 = -\int_{R_0}v\Delta v = \sum_{\mathbf j}\lambda_{\mathbf j}|\langle v, e_{\mathbf j}\rangle|^2 \ge \lambda^\ast\norm{v}^2_{L^2(R_0)}$. Hence $\lambda^\ast\norm{v}^2 \le \norm{g}\norm{v}$.
\end{proof}

\begin{lemma}[Elementary Properties of the Clausen Function]
\label{lem:cl2}
The Clausen function $\mathrm{Cl}_2$ of Section~\ref{sec:variable} is odd and $2\pi$-periodic, and:
\begin{enumerate}
\item[(i)] $\mathrm{Cl}_2(\pi) = 0$ and $\mathrm{Cl}_2(2\pi - x) = -\mathrm{Cl}_2(x)$;
\item[(ii)] $\mathrm{Cl}_2(2x) = 2\mathrm{Cl}_2(x) - 2\mathrm{Cl}_2(\pi - x)$ for all $x$; in particular $\mathrm{Cl}_2(\pi/3) = \tfrac32\mathrm{Cl}_2(2\pi/3)$;
\item[(iii)] $\mathrm{Cl}_2(x) = x(1 - \log x) + O(x^3)$ and $\mathrm{Cl}_2(\pi - x) = x\log 2 + O(x^3)$ as $x \to 0^+$;
\item[(iv)] $\int_0^\pi\log|\cos\theta - \cos\varphi|\,d\varphi = -\pi\log 2$ for every $\theta \in [0,\pi]$.
\end{enumerate}
\end{lemma}
\begin{proof}
Items (i) and (ii) and the first expansion in (iii) are Eqs.~(4.10)--(4.12), (4.17), (4.18) and (4.28) of \cite{lewin1981polylogarithms}; we include short proofs. Oddness holds because the integrand $h(t) := \log|2\sin(t/2)|$ is even. The identity $\int_0^\pi\log\sin u\,du = -\pi\log 2$ follows from $I := \int_0^\pi\log\sin u\,du = 2\int_0^{\pi/2}\log\sin u\,du = 2\int_0^{\pi/2}\log\cos u\,du$, whence $2I = 2\int_0^{\pi/2}\log(\sin u\cos u)\,du = 2\int_0^{\pi/2}\log\bigl(\tfrac12\sin 2u\bigr)\,du = -\pi\log 2 + I$. Hence $\int_0^\pi h = \pi\log 2 + 2\int_0^{\pi/2}\log\sin u\,du = 0$, i.e.\ $\mathrm{Cl}_2(\pi) = 0$; since $h$ is $2\pi$-periodic and even, $\int_\pi^{2\pi}h = \int_0^\pi h = 0$, so $h$ has zero mean over a period, $\mathrm{Cl}_2$ is $2\pi$-periodic, and $\mathrm{Cl}_2(2\pi - x) = \mathrm{Cl}_2(-x) = -\mathrm{Cl}_2(x)$, which proves (i).

(ii) Substituting $t = 2s$ and using $2\sin s = (2\sin(s/2))(2\cos(s/2))$,
\begin{gather*}
\mathrm{Cl}_2(2x) = -2\int_0^x\log|2\sin s|\,ds = 2\mathrm{Cl}_2(x) - 2\int_0^x\log|2\cos(s/2)|\,ds, \\
\int_0^x\log|2\cos(s/2)|\,ds = \int_{\pi-x}^{\pi}\log|2\sin(u/2)|\,du = \mathrm{Cl}_2(\pi-x) - \mathrm{Cl}_2(\pi) ,
\end{gather*}
by $s = \pi - u$; with (i) this is the duplication formula, and $x = \pi/3$ gives $\mathrm{Cl}_2(2\pi/3) = 2\mathrm{Cl}_2(\pi/3) - 2\mathrm{Cl}_2(2\pi/3)$.

(iii) For $0 < t \le 1$, $\log(2\sin(t/2)) = \log t + \log\bigl(\sin(t/2)/(t/2)\bigr) = \log t + O(t^2)$, so $\mathrm{Cl}_2(x) = -\int_0^x\log t\,dt + O(x^3) = x - x\log x + O(x^3)$. Also $\mathrm{Cl}_2(\pi - x) = \mathrm{Cl}_2(\pi) + \int_{\pi-x}^{\pi}h = \int_0^x\log(2\cos(s/2))\,ds = \int_0^x\bigl(\log 2 + O(s^2)\bigr)ds = x\log 2 + O(x^3)$.

(iv) Since $|\cos\theta - \cos\varphi| = \tfrac12\,|2\sin\tfrac{\varphi+\theta}{2}|\,|2\sin\tfrac{\varphi-\theta}{2}|$,
\[
\int_0^\pi\log|\cos\theta - \cos\varphi|\,d\varphi = -\pi\log 2 + \int_\theta^{\pi+\theta}h(t)\,dt + \int_{-\theta}^{\pi-\theta}h(t)\,dt ,
\]
by the substitutions $t = \varphi + \theta$ and $t = \varphi - \theta$. The substitution $t \mapsto 2\pi - t$ and the evenness and periodicity of $h$ turn the first integral into $\int_{\pi-\theta}^{2\pi-\theta}h$, so the two integrals add up to $\int_{-\theta}^{2\pi-\theta}h$, the integral over a full period, which vanishes by (i).
\end{proof}

\begin{lemma}[Regularized Least Squares, a Discrete Counterpart of {\cite[Thms.~13 and 14]{adcock2019frames}}]
\label{lem:tsvd}
Let $M \ge n$, $\eps > 0$ and $\vect c_\eps$ be as in \eqref{eq:tsvd}. For every $\vect z \in \mathbb R^{n}$,
\begin{equation}\label{eq:tsvd_ineq}
\norm{\vect f - \mat B\vect c_\eps}_2 \le \norm{\vect f - \mat B\vect z}_2 + \eps\,\norm{\mat D^{-1}\vect z}_2, \qquad
\norm{\mat D^{-1}\vect c_\eps}_2 \le \norm{\mat D^{-1}\vect z}_2 + \eps^{-1}\norm{\vect f - \mat B\vect z}_2 .
\end{equation}
\end{lemma}
\begin{proof}
Write $\vect y = \mat D^{-1}\vect z$ and $\mat P_\eps = \mat U_\eps\mat U_\eps^T$, the orthogonal projector onto the left singular vectors with $\sigma_i > \eps$. Then $\vect f - \mat B\vect c_\eps = (\mat I - \mat P_\eps)\vect f = (\mat I-\mat P_\eps)(\vect f - \mat B\mat D\vect y) + (\mat I - \mat P_\eps)\mat B\mat D\vect y$, and $\norm{(\mat I-\mat P_\eps)\mat B\mat D\vect y}_2 = \bigl(\sum_{\sigma_i\le\eps}\sigma_i^2(\vect v_i\cdot\vect y)^2\bigr)^{1/2} \le \eps\norm{\vect y}_2$. For the second inequality, $\mat D^{-1}\vect c_\eps = \sum_{\sigma_i>\eps}\sigma_i^{-1}(\vect u_i\cdot\vect f)\vect v_i = \sum_{\sigma_i>\eps}\bigl[(\vect v_i\cdot\vect y) + \sigma_i^{-1}\vect u_i\cdot(\vect f - \mat B\mat D\vect y)\bigr]\vect v_i$; since the $\vect v_i$ are orthonormal and $\sigma_i > \eps$ in the sum, the two parts have norms at most $\norm{\vect y}_2$ and $\eps^{-1}\norm{\vect f - \mat B\mat D\vect y}_2$.
\end{proof}

\section{Norms of the Lagrange Functions of the Trial Frame: Proof of Proposition~\ref{prop:lagrange_norms}}
\label{app:rates}

Throughout this appendix $N \ge N_0$, where $N_0 = N_0(\theta_0) \ge 4$ is such that $\Delta := \omega/(N-1) \le \theta_0/4$ and $2/N \le \theta_0/2$ for $N \ge N_0$, and $C$ denotes a constant that depends only on $\theta_0$ (equivalently on $\delta/L$) and may change from line to line. The nodes are $\theta_m = \theta_0 + (m-1)\Delta$ and $z_m = \cos\theta_m$, $1 \le m \le N$, so that $\theta_1 = \theta_0$ and $\theta_N = \pi - \theta_0$. We write $\omega_N(z) := \prod_{m=1}^N (z - z_m)$, $U_N(z) := -N^{-1}\log|\omega_N(z)|$, and $U_N^{(j)}(z) := -N^{-1}\sum_{m\ne j}\log|z - z_m|$ for the discrete potential with the $j$-th node removed. Two elementary facts are used repeatedly: for $|s| \le \pi$, $|\sin(s/2)| \ge |s|/\pi$; and for $\vartheta \in [0,\pi]$ and $\varphi \in [\theta_0/2, \pi - \theta_0/2]$, $\tfrac12(\vartheta + \varphi) \in [\theta_0/4, \pi - \theta_0/4]$, so that $\sin\tfrac12(\vartheta+\varphi) \ge \sin(\theta_0/4)$.

\begin{lemma}[Riemann Sums with a Logarithmic Singularity]
\label{lem:riemann}
For $\vartheta \in [0,\pi]$ put $g_\vartheta(\varphi) := \log|\cos\vartheta - \cos\varphi|$. Then for all $N \ge N_0$ and all $\vartheta \in [0,\pi]$:
\begin{enumerate}
    \item[(a)] if $|\vartheta - \theta_m| \ge \Delta/4$ for every $m$, then
    $\bigl|\Delta\sum_{m=1}^N g_\vartheta(\theta_m) - \int_{\theta_0}^{\pi-\theta_0} g_\vartheta(\varphi)\,d\varphi\bigr| \le C\,\Delta\log(1/\Delta)$;
    \item[(b)] if $|\vartheta - \theta_j| < \Delta/4$ for some $j$, then the same bound holds for the sum with the term $m = j$ omitted.
\end{enumerate}
\end{lemma}
\begin{proof}
\emph{Decomposition.} Since $\cos\vartheta - \cos\varphi = 2\sin\tfrac12(\varphi+\vartheta)\sin\tfrac12(\varphi-\vartheta)$,
\[
g_\vartheta = \log 2 + h_1 + h_2, \qquad h_1(\varphi) := \log\bigl|\sin\tfrac12(\varphi - \vartheta)\bigr|, \quad h_2(\varphi) := \log\bigl|\sin\tfrac12(\varphi+\vartheta)\bigr| .
\]
Let $I_m := [\theta_m - \Delta/2, \theta_m + \Delta/2]$ and $J := \bigcup_{m=1}^N I_m = [\theta_0 - \Delta/2, \pi - \theta_0 + \Delta/2]$; the cells are adjacent and $\int_{I_m}\log 2 = \Delta\log 2$ for every $m$, so the constant term cancels cell by cell and
\begin{equation}\label{eq:riemann_split}
\begin{aligned}
\Delta\sum_{m=1}^N g_\vartheta(\theta_m) - \int_{\theta_0}^{\pi-\theta_0}g_\vartheta
&= \sum_{m=1}^N\Bigl(\Delta h_1(\theta_m) - \int_{I_m}h_1\Bigr) + \sum_{m=1}^N\Bigl(\Delta h_2(\theta_m) - \int_{I_m}h_2\Bigr) \\
&\quad + \Bigl(\int_J g_\vartheta - \int_{\theta_0}^{\pi-\theta_0}g_\vartheta\Bigr).
\end{aligned}
\end{equation}
In case (b) the term $m = j$ is removed from the left-hand sum; this removes $\Delta h_1(\theta_j)$ and $\Delta h_2(\theta_j)$ from the two sums and adds $-\Delta\log 2$ to the right-hand side. Note that in case (b) the index $j$ is unique, since the nodes are $\Delta$ apart. We use the midpoint rule $|\Delta\,h(\theta_m) - \int_{I_m}h| \le (\Delta^3/24)\max_{I_m}|h''|$ for $h \in C^2(I_m)$, and we bound the three brackets of \eqref{eq:riemann_split} in turn; $C$ denotes constants depending only on $\theta_0$.

\emph{The smooth term $h_2$.} Since $\Delta \le \theta_0$, $J \subset [\theta_0/2, \pi - \theta_0/2]$, and the second elementary fact gives $\sin\tfrac12(\varphi+\vartheta) \ge \sin(\theta_0/4)$ for $\varphi \in J$, $\vartheta \in [0,\pi]$. Hence $h_2 \in C^2(J)$ with $|h_2| \le \log(1/\sin(\theta_0/4))$ and $|h_2''(\varphi)| = \bigl(4\sin^2\tfrac12(\varphi+\vartheta)\bigr)^{-1} \le \bigl(4\sin^2(\theta_0/4)\bigr)^{-1}$, so the second bracket is at most $N\Delta^3C \le C\Delta^2$, because $N\Delta = \omega + \Delta \le 2\omega$. In case (b) the missing term contributes in addition $|\int_{I_j}h_2| \le C\Delta$.

\emph{The singular term $h_1$: cells away from $\vartheta$.} For $\varphi \in J \subset (0,\pi)$ and $\vartheta \in [0,\pi]$ one has $|\varphi - \vartheta| < \pi$, so $h_1$ is smooth on $J$ except at $\varphi = \vartheta$, and for $\varphi \ne \vartheta$ the first elementary fact gives $|h_1''(\varphi)| = \bigl(4\sin^2\tfrac12(\varphi-\vartheta)\bigr)^{-1} \le \pi^2/(4|\varphi-\vartheta|^2)$. Let $\mathcal M := \{m : |\theta_m - \vartheta| \le 3\Delta/2\}$, a set of at most four indices, since an interval of length $3\Delta$ contains at most four points spaced $\Delta$ apart. For $m \notin \mathcal M$ and $\varphi \in I_m$, $|\varphi - \vartheta| \ge |\theta_m - \vartheta| - \Delta/2 \ge \tfrac23|\theta_m - \vartheta| > 0$, so $h_1 \in C^2(I_m)$ and the midpoint error on $I_m$ is at most $(\Delta^3/24)\cdot\tfrac{9\pi^2}{16}|\theta_m - \vartheta|^{-2}$. On either side of $\vartheta$, the nodes outside $\mathcal M$ are at distances greater than $3\Delta/2$ that increase in steps of $\Delta$, so the $k$-th of them is at distance at least $k\Delta/2$; hence $\sum_{m\notin\mathcal M}|\theta_m - \vartheta|^{-2} \le 2\sum_{k\ge1}(k\Delta/2)^{-2} = (4\pi^2/3)\Delta^{-2} \le 14\Delta^{-2}$, and the cells outside $\mathcal M$ contribute at most $C\Delta$ to the first bracket.

\emph{The singular term $h_1$: cells near $\vartheta$.} For $m \in \mathcal M$ we bound $\Delta|h_1(\theta_m)|$ and $\int_{I_m}|h_1|$ separately. In case (a), $\Delta/4 \le |\theta_m - \vartheta| \le 3\Delta/2 \le \pi$, so the first elementary fact gives $|\sin\tfrac12(\theta_m-\vartheta)| \ge \Delta/(4\pi)$ and therefore $0 \le -h_1(\theta_m) \le \log(4\pi/\Delta)$; in case (b) the same holds for $m \in \mathcal M\setminus\{j\}$, because then $|\theta_m - \vartheta| \ge |\theta_m - \theta_j| - |\theta_j - \vartheta| \ge 3\Delta/4$, and the term $m = j$ does not occur. For $\varphi \in I_m$, $m \in \mathcal M$, one has $|\varphi - \vartheta| \le 2\Delta \le \pi$, hence $0 \le -h_1(\varphi) \le \log\bigl(\pi/|\varphi-\vartheta|\bigr)$ and
\[
\int_{I_m}|h_1| \le \int_{|s|\le 2\Delta}\log\frac{\pi}{|s|}\,ds = 4\Delta\Bigl(1 + \log\frac{\pi}{2\Delta}\Bigr).
\]
Hence the cells in $\mathcal M$ contribute at most $4\bigl[\Delta\log(4\pi/\Delta) + 4\Delta(1 + \log(\pi/(2\Delta)))\bigr] \le C\Delta\log(1/\Delta)$ to the first bracket, in both cases.

\emph{The ends and the omitted cell.} The third bracket of \eqref{eq:riemann_split} is the integral of $g_\vartheta$ over the two half-cells $[\theta_0 - \Delta/2, \theta_0]$ and $[\pi-\theta_0, \pi-\theta_0+\Delta/2]$, which lie in $J$; on each, $|g_\vartheta| \le \log 2 + |h_1| + C$, and the integral of $|h_1|$ over an interval of length $\Delta/2$ contained in $J$ is at most $\int_{|s|\le\Delta}\log(\pi/|s|)\,ds = 2\Delta(1 + \log(\pi/\Delta))$, because $-h_1(\varphi) \le \log(\pi/|\varphi - \vartheta|)$ on $J$ and the right-hand side is a decreasing function of $|\varphi-\vartheta|$. So the third bracket is at most $C\Delta\log(1/\Delta)$. In case (b) the additional term $\Delta\log 2$ is of the same order. Collecting the bounds proves the lemma.
\end{proof}

\begin{lemma}[Discrete Potentials]
\label{lem:discrete_potential}
For $N \ge N_0$, with $z = \cos\vartheta$, $\vartheta \in [0,\pi]$:
\begin{enumerate}
    \item[(a)] if $|\vartheta - \theta_m| \ge \Delta/4$ for every $m$, then $|U_N(z) - U_\omega(z)| \le C\log N/N$;
    \item[(b)] if $|\vartheta - \theta_j| < \Delta/4$ for some $j$, then $|U_N^{(j)}(z) - U_\omega(z)| \le C\log N/N$; in particular $N^{-1}\log|\omega_N'(z_j)| = -U_\omega(z_j) + O(\log N/N)$ uniformly in $j$.
\end{enumerate}
\end{lemma}
\begin{proof}
(a) Since $\log|z - z_m| = \log|\cos\vartheta - \cos\theta_m| = g_\vartheta(\theta_m)$, the definition gives $U_N(z) = -(N\Delta)^{-1}\,\Delta\sum_{m=1}^N g_\vartheta(\theta_m)$, and $N\Delta = N\omega/(N-1)$, i.e.\ $\omega/(N\Delta) = 1 - N^{-1}$. By \eqref{eq:potential}, $\int_{\theta_0}^{\pi-\theta_0}g_\vartheta = -\omega U_\omega(z)$, so Lemma~\ref{lem:riemann}(a) gives $\Delta\sum_m g_\vartheta(\theta_m) = -\omega U_\omega(z) + R$ with $|R| \le C\Delta\log(1/\Delta)$. Hence
\[
U_N(z) = (1 - N^{-1})\,U_\omega(z) - \frac{R}{N\Delta}, \qquad
|U_N(z) - U_\omega(z)| \le \frac{|U_\omega(z)|}{N} + \frac{C\Delta\log(1/\Delta)}{\omega} .
\]
Now $|U_\omega| \le \log 2 + 4\max|\mathrm{Cl}_2|/\omega$ on $[-1,1]$ by \eqref{eq:U_closed}, and $\Delta\log(1/\Delta) = \frac{\omega}{N-1}\log\frac{N-1}{\omega} \le C\frac{\log N}{N}$ for $N \ge N_0$; both bounds depend only on $\theta_0$, which proves (a).

(b) The same computation with the term $m = j$ omitted, that is, with
\[
U_N^{(j)}(z) = -(N\Delta)^{-1}\,\Delta\sum_{m\ne j}g_\vartheta(\theta_m)
\]
and with Lemma~\ref{lem:riemann}(b) in place of (a), gives $|U_N^{(j)}(z) - U_\omega(z)| \le C\log N/N$. For the last statement, $\omega_N'(z_j) = \prod_{m\ne j}(z_j - z_m)$, so $N^{-1}\log|\omega_N'(z_j)| = -U_N^{(j)}(z_j)$, and $\vartheta = \theta_j$ satisfies the hypothesis of (b); the constant is uniform in $j$ because the constant of Lemma~\ref{lem:riemann} is uniform in $\vartheta$.
\end{proof}

\begin{lemma}[Lagrange Polynomials]
\label{lem:lagrange_polynomials}
Let $\ell_m(z) := \omega_N(z)/\bigl((z-z_m)\,\omega_N'(z_m)\bigr)$, $1 \le m \le N$, be the Lagrange polynomials of the nodes $z_1,\dots,z_N$. For $N \ge N_0$ and $z = \cos\vartheta$:
\begin{enumerate}
    \item[(a)] for all $\vartheta \in [0,\pi]$ and all $m$, $N^{-1}\log|\ell_m(z)| \le U_\omega(z_m) - U_\omega(z) + C\log N/N$;
    \item[(b)] if $|\vartheta - \theta_k| \ge \Delta/4$ for every $k$, then for all $m$, $N^{-1}\log|\ell_m(z)| \ge U_\omega(z_m) - U_\omega(z) - C\log N/N$.
\end{enumerate}
\end{lemma}
\begin{proof}
Since $\omega_N(z)/(z - z_m) = \prod_{k\ne m}(z - z_k)$ and $\omega_N'(z_m) = \prod_{k\ne m}(z_m - z_k)$, one has $\ell_m(z_k) = \delta_{mk}$, and for $z \notin \{z_1, \dots, z_N\}$
\begin{equation}\label{eq:log_lagrange}
\begin{aligned}
\log|\ell_m(z)| &= \log|\omega_N(z)| - \log|z - z_m| - \log|\omega_N'(z_m)| \\
&= -NU_N(z) - \log|z - z_m| + NU_N^{(m)}(z_m),
\end{aligned}
\end{equation}
where $NU_N^{(m)}(z_m) = -\log|\omega_N'(z_m)| = NU_\omega(z_m) + O(\log N)$ uniformly in $m$ by Lemma~\ref{lem:discrete_potential}(b). If $z = z_k$ for some $k \ne m$, then $\ell_m(z) = 0$ and (a) holds trivially; this can only occur in Case~2 below.

\emph{Case 1: $|\vartheta - \theta_k| \ge \Delta/4$ for every $k$.} Then $-NU_N(z) = -NU_\omega(z) + O(\log N)$ by Lemma~\ref{lem:discrete_potential}(a). Moreover $|z - z_m| \le 2$, and by the two elementary facts of this appendix, applied with $|\vartheta - \theta_m| \le \pi$ and $\theta_m \in [\theta_0, \pi-\theta_0]$,
\begin{equation}\label{eq:node_distance}
|z - z_m| = 2\,\bigl|\sin\tfrac12(\vartheta-\theta_m)\bigr|\,\sin\tfrac12(\vartheta+\theta_m) \ge 2\cdot\frac{|\vartheta - \theta_m|}{\pi}\cdot\sin\frac{\theta_0}{4} \ge \frac{\Delta}{2\pi}\sin\frac{\theta_0}{4},
\end{equation}
so $|\log|z - z_m|| \le C\log N$. Inserting the three estimates into \eqref{eq:log_lagrange} gives $\log|\ell_m(z)| = N[U_\omega(z_m) - U_\omega(z)] + O(\log N)$, which proves (b) and the case of (a) in which $\vartheta$ is at distance at least $\Delta/4$ from every node.

\emph{Case 2: $|\vartheta - \theta_j| < \Delta/4$ for some $j$}, necessarily unique. If $j = m$, write $\log|\ell_m(z)| = \sum_{k\ne m}\bigl(\log|z - z_k| - \log|z_m - z_k|\bigr) = N\bigl[U_N^{(m)}(z_m) - U_N^{(m)}(z)\bigr]$ and apply Lemma~\ref{lem:discrete_potential}(b) at $z$ and at $z_m$: $\log|\ell_m(z)| = N[U_\omega(z_m) - U_\omega(z)] + O(\log N)$. If $j \ne m$, then $|\vartheta - \theta_m| \ge |\theta_j - \theta_m| - |\vartheta - \theta_j| \ge 3\Delta/4$, so \eqref{eq:node_distance} holds and gives $-\log|z - z_m| \le C\log N$; and $\log|\omega_N(z)| = \log|z - z_j| - NU_N^{(j)}(z) \le \log 2 - NU_\omega(z) + C\log N$ by Lemma~\ref{lem:discrete_potential}(b). Inserting into \eqref{eq:log_lagrange} gives $\log|\ell_m(z)| \le N[U_\omega(z_m) - U_\omega(z)] + C\log N$. This completes the proof of (a).
\end{proof}

\begin{lemma}[Modulus of Continuity of $U_\omega$]
\label{lem:U_modulus}
For all $\vartheta, \vartheta' \in [0,\pi]$ with $|\vartheta - \vartheta'| \le 1/2$,
\begin{equation}\label{eq:U_modulus}
|U_\omega(\cos\vartheta) - U_\omega(\cos\vartheta')| \le C|\vartheta - \vartheta'|\bigl(1 + \log(1/|\vartheta-\vartheta'|)\bigr).
\end{equation}
\end{lemma}
\begin{proof}
By the symmetry $U_\omega(\cos(\pi-\theta)) = U_\omega(\cos\theta)$, and by splitting at $\pi/2$ an interval that contains $\pi/2$, it suffices to take $\vartheta, \vartheta' \in [0,\pi/2]$. For $\theta \in (0,\pi/2)\setminus\{\theta_0\}$ the derivative formula in Step~4 of the proof of Lemma~\ref{lem:clausen} gives $|\tfrac{d}{d\theta}U_\omega(\cos\theta)| \le \omega^{-1}\bigl(|\log\sin(\theta+\theta_0)| + |\log\sin|\theta-\theta_0||\bigr)$. Here $\theta+\theta_0 \in [\theta_0, \pi/2+\theta_0]$, so $\sin(\theta+\theta_0) \ge \min\{\sin\theta_0, \cos\theta_0\}$ and the first term is bounded by a constant; and $|\theta-\theta_0| \le \pi/2$, so $\sin|\theta-\theta_0| \ge 2|\theta-\theta_0|/\pi$ and the second term is at most $\log\bigl(\pi/(2|\theta-\theta_0|)\bigr)$. Since $\theta \mapsto U_\omega(\cos\theta)$ is continuous on $[0,\pi/2]$ and continuously differentiable off $\theta_0$ with integrable derivative, it is absolutely continuous, and $|U_\omega(\cos\vartheta) - U_\omega(\cos\vartheta')|$ is at most the integral of this bound over the interval between $\vartheta$ and $\vartheta'$, of length $s := |\vartheta - \vartheta'| \le 1/2$. The integral of $\log(\pi/(2|\theta - \theta_0|))$ over an interval of length $s$ is largest when the interval is centred at $\theta_0$, where it equals $2\int_0^{s/2}\log\bigl(\pi/(2u)\bigr)\,du = s\bigl(1 + \log(\pi/s)\bigr)$. This gives \eqref{eq:U_modulus}.
\end{proof}

\begin{proof}[Proof of Proposition~\ref{prop:lagrange_norms}]
By Lemma~\ref{lem:chebyshev_variable}, $\ell_m^V(x) = (\sin\theta/\sin\theta_m)\,\ell_m(\cos\theta)$: this function lies in $V_N = \sqrt{2/\tilde L}\sin\theta\cdot\mathbb P_{N-1}(z)$ and equals $\delta_{mk}$ at $\theta_k$. Since $dx = (\tilde L/\pi)\,d\theta$ and $\sin\theta_0 \le \sin\theta_m \le 1$,
\begin{gather*}
\norm{\ell_m^V}_{L^2(\tilde I)}^2 = \frac{\tilde L}{\pi\sin^2\theta_m}\int_0^\pi \sin^2\theta\,\ell_m(\cos\theta)^2\,d\theta, \\
\norm{\ell_m^V}_{L^2(I)}^2 = \frac{\tilde L}{\pi\sin^2\theta_m}\int_{\theta_0}^{\pi-\theta_0}\sin^2\theta\,\ell_m(\cos\theta)^2\,d\theta .
\end{gather*}
\emph{Upper bounds.} By Lemma~\ref{lem:lagrange_polynomials}(a) and Lemma~\ref{lem:clausen}, $\max_{[-1,1]}|\ell_m|^2 \le \exp\bigl(2N[U_\omega(z_m) - U_\omega(1)] + C\log N\bigr)$ and $\max_{[-c_0,c_0]}|\ell_m|^2 \le \exp\bigl(2N[U_\omega(z_m) - U_\omega(c_0)] + C\log N\bigr)$, because $U_\omega$ attains its minimum over $[-1,1]$ at $\pm1$ and over $[-c_0,c_0]$ at $\pm c_0$. Both integrals are bounded by $\pi$ times these maxima, so that the squared norms are at most $\tilde L/\sin^2\theta_0$ times them, and $\norm{\ell_m^V}_{L^\infty(I)} \le \max_{[-c_0,c_0]}|\ell_m|/\sin\theta_0$.

\emph{Lower bounds.} For the extended interval, restrict the integral to $\theta \in [1/N, 2/N]$, which for $N \ge N_0$ lies outside $[\theta_0 - \Delta/4, \pi]$, so that Lemma~\ref{lem:lagrange_polynomials}(b) applies; there $\sin^2\theta \ge (2/(\pi N))^2$ and, by \eqref{eq:U_modulus}, $U_\omega(\cos\theta) \le U_\omega(1) + C\log N/N$. Hence, as $\tilde L/(\pi\sin^2\theta_m) \ge \tilde L/\pi$, $\norm{\ell_m^V}^2_{L^2(\tilde I)} \ge (\tilde L/\pi)N^{-1}(2/(\pi N))^2\exp\bigl(2N[U_\omega(z_m) - U_\omega(1)] - C\log N\bigr)$. For the physical interval, restrict to $\theta \in [\theta_0 + \Delta/4, \theta_0 + 3\Delta/4]$, an interval of length $\Delta/2$ inside $[\theta_0,\pi-\theta_0]$ at distance at least $\Delta/4$ from every node; there $\sin^2\theta \ge \sin^2\theta_0$ and $U_\omega(\cos\theta) \le U_\omega(c_0) + C\Delta\log(1/\Delta)$ by \eqref{eq:U_modulus}, so $\norm{\ell_m^V}^2_{L^2(I)} \ge (\tilde L/\pi)(\Delta/2)\sin^2\theta_0\exp\bigl(2N[U_\omega(z_m) - U_\omega(c_0)] - C\log N\bigr)$, and the same evaluation point gives the lower bound for the supremum norm. Taking logarithms and dividing by $N$ proves \eqref{eq:lagrange_norms}, the factors $\tilde L$, $\sin\theta_0$ and $\Delta$ contributing $O(\log N/N)$ with constants depending only on $\delta$ and $L$.
\end{proof}

\section{Column Norms of the Collocation Matrix: Proof of Proposition~\ref{prop:apriori}}
\label{app:apriori}

Throughout this appendix $\Lop = \Lop_P$ or $\Lop = \Lop_{CD}$ with constant $k \ge 0$, the nodes are the interior nodes $x_2, \dots, x_{N-1}$ of the $N$-grid, $\theta_m = \theta_0 + (m-1)\Delta$ with $\Delta = \omega/(N-1)$ as in Appendix~\ref{app:rates}, and $J := j+2$ is the frequency of the basis function $\varphi_j = w_J - a_jw_1 - b_jw_2$. By Lemma~\ref{lem:phase}, for every $i \ge 1$,
\begin{equation}\label{eq:Lw_nodes}
(\Lop w_i)(x_m) = \sqrt{2/\tilde L}\;\rho_i\,\sin(i\theta_m + \vartheta_i), \qquad \rho_i := \mu_i\sqrt{\mu_i^2 + k^2},\quad \vartheta_i := \arctan(k/\mu_i),
\end{equation}
with $\rho_i = \lambda_i$ and $\vartheta_i = 0$ for the Poisson operator; in all cases $\rho_i \ge \lambda_i = (i\pi/\tilde L)^2$.

\begin{lemma}[Nodal Norms of Shifted Sines]
\label{lem:nodal_sines}
There are $c_S > 0$ and $N_0$, depending only on $\theta_0$, such that for all $N \ge N_0$, all integers $3 \le J \le N$ and all $\beta \in \mathbb R$,
\begin{equation}\label{eq:nodal_sines}
\sum_{m=2}^{N-1}\sin^2(J\theta_m + \beta) \ge c_S\,N .
\end{equation}
\end{lemma}
\begin{proof}
We distinguish three ranges of $J$.

\emph{(i) $J\Delta \ge \pi/4$.} Since $J\Delta \le N\Delta = \omega + \Delta \le \pi - 7\theta_0/4$ for $\Delta \le \theta_0/4$, the number $J\Delta$ lies in $[\pi/4, \pi - 7\theta_0/4]$ (if this interval is empty, which happens for $\theta_0 > 3\pi/7$, there is no such $J$ and nothing to prove), and on it $\sin(J\Delta) \ge s^\ast := \min\{\sin(\pi/4), \sin(7\theta_0/4)\} > 0$, since $\sin$ is concave on $[0,\pi]$. Writing $\sin^2 x = \tfrac12 - \tfrac12\cos 2x$ and $\theta_m = \theta_2 + (m-2)\Delta$,
\[
\sum_{m=2}^{N-1}\sin^2(J\theta_m + \beta) = \frac{N-2}{2} - \frac12\,\mathrm{Re}\Bigl[e^{2i(J\theta_2+\beta)}\,\frac{e^{2iJ\Delta(N-2)} - 1}{e^{2iJ\Delta} - 1}\Bigr],
\]
where $e^{2iJ\Delta} \ne 1$ because $2J\Delta \in (0,2\pi)$; the bracket has modulus at most $2/|e^{2iJ\Delta} - 1| = 1/\sin(J\Delta) \le 1/s^\ast$. Hence the sum is at least $(N-2)/2 - 1/(2s^\ast) \ge (N-2)/4 \ge N/8$ for $N \ge \max\{4, 2 + 2/s^\ast\}$.

\emph{(ii) $J \ge 4\pi/\omega$ and $J\Delta < \pi/4$.} The set $\{\theta \in \mathbb R : \sin^2(J\theta + \beta) \ge 1/2\}$ is the union of the intervals $H_i := \bigl[(i\pi + \pi/4 - \beta)/J,\ (i\pi + 3\pi/4 - \beta)/J\bigr]$, $i \in \mathbb Z$, each of length $\pi/(2J)$, whose left endpoints form a lattice of spacing $\pi/J$. The interval $H_i$ lies inside the arc $[\theta_0,\pi-\theta_0]$ whenever its left endpoint lies in $[\theta_0, \pi - \theta_0 - \pi/(2J)]$, an interval of length $\omega - \pi/(2J) \ge 0$, which contains at least $\lfloor(\omega - \pi/(2J))J/\pi\rfloor \ge \omega J/\pi - 2 \ge \omega J/(2\pi)$ lattice points because $\omega J/\pi \ge 4$. Each such $H_i$ has length $\pi/(2J) \ge 2\Delta$ and therefore contains at least $\lfloor\pi/(2J\Delta)\rfloor \ge \pi/(4J\Delta)$ of the nodes $\theta_1,\dots,\theta_N$, which are the points of the lattice $\theta_0 + \Delta\mathbb Z$ in the arc (an interval of length $\ell \ge 2\Delta$ contains at least $\lfloor\ell/\Delta\rfloor \ge \ell/(2\Delta)$ points of a lattice of spacing $\Delta$). Hence at least $(\omega J/2\pi)\cdot\pi/(4J\Delta) = \omega/(8\Delta) = (N-1)/8$ nodes satisfy $\sin^2(J\theta_m+\beta) \ge 1/2$, all but at most two of them interior, and the sum is at least $\tfrac12\bigl((N-1)/8 - 2\bigr) \ge N/64$ for $N \ge 65$.

\emph{(iii) $3 \le J < 4\pi/\omega$.} There are finitely many such $J$, independently of $N$. For each of them and each $\beta$, the function $\phi_{J,\beta}(\theta) := \sin^2(J\theta + \beta)$ is continuous and nonnegative and vanishes only at isolated points, so $\mathcal I_J(\beta) := \int_{\theta_0}^{\pi-\theta_0}\phi_{J,\beta}\,d\theta > 0$; since $\mathcal I_J$ is continuous and $2\pi$-periodic in $\beta$, $I_J := \min_\beta \mathcal I_J(\beta) > 0$. The sum $\Delta\sum_{m=2}^{N-1}\phi_{J,\beta}(\theta_m)$ is the rectangle rule with cells of length $\Delta$ centred at the interior nodes, which cover $[\theta_0 + \Delta/2, \pi - \theta_0 - \Delta/2]$; its error against $\int_{\theta_0}^{\pi-\theta_0}\phi_{J,\beta}$ is at most $\Delta$ times the total variation of $\phi_{J,\beta}$ on the arc, plus the integral over the two uncovered half-cells, hence at most $\Delta(2J\omega/\pi + 3)$, because the extrema of $\phi_{J,\beta}$ are the points with $J\theta + \beta \in \tfrac{\pi}{2}\mathbb Z$, spaced $\pi/(2J)$ apart, so that $\phi_{J,\beta}$ has at most $2J\omega/\pi + 2$ monotone pieces on the arc, each of variation at most one, and $\phi_{J,\beta} \le 1$. Consequently $\sum_{m=2}^{N-1}\phi_{J,\beta}(\theta_m) \ge \bigl(I_J - \Delta(2J\omega/\pi + 3)\bigr)/\Delta \ge I_J/(2\Delta) = I_J(N-1)/(2\omega) \ge I_J N/(4\omega)$ as soon as $\Delta \le I_J/(2(2J\omega/\pi + 3))$ and $N \ge 2$.

The lemma holds with $c_S := \min\{1/64,\ \min_{3 \le J < 4\pi/\omega} I_J/(4\omega)\}$ and $N_0$ the largest of the thresholds in (i)--(iii).
\end{proof}

\begin{lemma}[Column Norms]
\label{lem:column_norms}
There are $0 < c_B \le C_B$ and $N_0$, depending only on $\theta_0$, $L$ and $k$, such that for all $N \ge N_0$ and $1 \le j \le n$,
\begin{equation}\label{eq:column_norms}
c_B\,\rho_{j+2}\sqrt N \le \norm{\mat B\vect e_j}_2 \le C_B\,\rho_{j+2}\sqrt N .
\end{equation}
\end{lemma}
\begin{proof}
The coefficients of Section~\ref{sec:constants} are $a_j = s_J/s_1$, $b_j = 0$ for odd $J$ and $a_j = 0$, $b_j = s_J/s_2$ for even $J$, with $s_i = \sin(i\theta_0)$; in particular $|a_j| \le 1/\sin\theta_0$, $|b_j| \le 1/\sin 2\theta_0$, and we put $A := \max\{1/\sin\theta_0, 1/\sin 2\theta_0\}$. By \eqref{eq:Lw_nodes},
\begin{gather*}
(\mat B)_{mj} = \psi_j(x_m) = \sqrt{2/\tilde L}\;\rho_J\bigl[\sin(J\theta_m + \vartheta_J) - g_J(\theta_m)\bigr], \\
g_J(\theta) := \frac{a_j\rho_1\sin(\theta + \vartheta_1) + b_j\rho_2\sin(2\theta + \vartheta_2)}{\rho_J},
\end{gather*}
and $|g_J| \le A(\rho_1 + \rho_2)/\rho_J \le A(\rho_1+\rho_2)\tilde L^2/(\pi^2J^2) =: A_k/J^2$, because $\rho_J \ge \lambda_J$. The upper bound in \eqref{eq:column_norms} follows with $C_B := \sqrt{2/\tilde L}\,(1 + A_k/9)$. For the lower bound, the triangle inequality in $\ell^2$ and Lemma~\ref{lem:nodal_sines} give
\[
\norm{\mat B\vect e_j}_2 \ge \sqrt{2/\tilde L}\;\rho_J\Bigl(\sqrt{c_SN} - \sqrt{N}\,\frac{A_k}{J^2}\Bigr) \ge \sqrt{2/\tilde L}\;\rho_J\,\frac{\sqrt{c_SN}}{2}
\]
for $J \ge J_1 := \bigl\lceil (2A_k/\sqrt{c_S})^{1/2}\bigr\rceil$.
For the finitely many $J$ with $3 \le J < J_1$ the function $\theta \mapsto \sin(J\theta + \vartheta_J) - g_J(\theta)$ does not depend on $N$ and is a nonzero trigonometric polynomial, its component of frequency $J$ having coefficient one, so $I'_J := \int_{\theta_0}^{\pi-\theta_0}\bigl(\sin(J\theta+\vartheta_J) - g_J(\theta)\bigr)^2d\theta > 0$, and the rectangle-rule estimate of case (iii) of the proof of Lemma~\ref{lem:nodal_sines}, applied to the square of this function, which is continuous, of bounded variation and bounded by $(1 + A_k/9)^2$ on the arc, with its own variation and maximum in place of those of $\phi_{J,\beta}$, gives $\sum_{m=2}^{N-1}\bigl(\sin(J\theta_m + \vartheta_J) - g_J(\theta_m)\bigr)^2 \ge I'_JN/(4\omega)$ for $N \ge N_J$, where $N_J$ depends only on $J$, $\theta_0$, $L$ and $k$. The lower bound in \eqref{eq:column_norms} holds with $c_B := \sqrt{2/\tilde L}\,\min\bigl\{\sqrt{c_S}/2,\ \min_{3\le J<J_1}(I'_J/(4\omega))^{1/2}\bigr\}$ and $N_0$ the maximum of the threshold of Lemma~\ref{lem:nodal_sines} and of the finitely many $N_J$.
\end{proof}

\begin{proof}[Proof of Proposition~\ref{prop:apriori}]
The Gram matrix of the basis $\{\varphi_j\}$ in $L^2(\tilde I)$ is $\mat G_{jk} = \langle\varphi_j, \varphi_k\rangle = \delta_{jk} + a_ja_k + b_jb_k$ by the orthonormality of the $w_i$, i.e.\ $\mat G = \mat I + \vect a\vect a^T + \vect b\vect b^T$, so $\norm{\mat G}_2 \le 1 + \norm{\vect a}_2^2 + \norm{\vect b}_2^2 \le 1 + nA^2$, since for each $j$ only one of $a_j, b_j$ is nonzero and both are bounded by $A$. For $\vect w \in \mathbb R^n$, Lemma~\ref{lem:column_norms} gives $(\mat D)_{jj} = 1/\norm{\mat B\vect e_j}_2 \le 1/(c_B\rho_3\sqrt N) \le 1/(c_B\lambda_3\sqrt N)$, hence
\[
\norm{\Phi\mat D\vect w}_{L^2(\tilde I)}^2 = (\mat D\vect w)^T\mat G\,(\mat D\vect w) \le (1 + nA^2)\,\frac{\norm{\vect w}_2^2}{c_B^2\lambda_3^2N} \le \frac{1 + A^2}{c_B^2\lambda_3^2}\,\norm{\vect w}_2^2 =: C_\Phi^2\norm{\vect w}_2^2 ,
\]
using $n < N$. This is the first claim. For $\vect w \in \mathcal V_t$, $\norm{\Phi\mat D\vect w}_{L^2(I)} \le \norm{\Phi\mat D\vect w}_{L^2(\tilde I)} \le C_\Phi\norm{\vect w}_2$, so $\gamma_t \le C_\Phi$; for $\vect w \in \mathcal V_t^\perp$, $\norm{\mat B_D\vect w}_2^2 = \sum_{\sigma_i > t}\sigma_i^2(\vect v_i\cdot\vect w)^2 \ge t^2\norm{\vect w}_2^2$, so $\Gamma_t \le C_\Phi/t$. Substituting $\Gamma_\eta \le C_\Phi/\eta$ and $\gamma_\eta \le C_\Phi$ into \eqref{eq:lu_bound} gives \eqref{eq:apriori_lu}.
\end{proof}

\section{Technical Lemmas for Theorems~\ref{thm:competitor}--\ref{thm:dual} and the Certified Configurations}
\label{app:tail}

This appendix proves the lemmas of routine character that Section~\ref{sec:tail} uses, the transfer of Lemma~\ref{lem:riemann} to an arbitrary arc, the modulus of continuity of the envelope of a configuration and the estimates for envelopes of bounded densities used in Sections~\ref{sec:equilibrium_proof}, \ref{sec:tail_law_proof} and \ref{sec:initial_slope_proof}, the passage from a function of $\Phi_N$ to its coefficient vector and back, a lower bound through a test function, and the estimate of the Lagrange polynomials of an interpolation configuration, and then lists the configurations of Corollary~\ref{cor:certified_instances} with the certified evaluation of their exponents and intervals. The notation is that of Section~\ref{sec:tail}.

\begin{lemma}[Riemann Sums on an Arc]
\label{lem:riemann_arc}
Let $0 < \xi < \pi/2$, $[\alpha,\beta] \subset [\xi, \pi-\xi]$, $K \ge 2$, $d := (\beta-\alpha)/(K-1) \le \min\{\xi, 1\}/4$, and $\varphi_i := \alpha + (i-1)d$ for $1 \le i \le K$. With $g_\vartheta(\varphi) = \log|\cos\vartheta - \cos\varphi|$ as in Lemma~\ref{lem:riemann}, there is a constant $C_\xi$ depending only on $\xi$ such that for all $\vartheta \in [0,\pi]$:
\begin{enumerate}
    \item[(a)] if $|\vartheta - \varphi_i| \ge d/4$ for every $i$, then $\bigl|d\sum_{i=1}^K g_\vartheta(\varphi_i) - \int_\alpha^\beta g_\vartheta\bigr| \le C_\xi\,d\log(1/d)$;
    \item[(b)] if $|\vartheta - \varphi_j| < d/4$ for some $j$, then the same bound holds for the sum without the term $i = j$.
\end{enumerate}
\end{lemma}
\begin{proof}
The proof of Lemma~\ref{lem:riemann} applies word for word with $\theta_0$, $\Delta$, $N$ and the arc $[\theta_0,\pi-\theta_0]$ replaced by $\xi$, $d$, $K$ and $[\alpha,\beta]$. The cells $I_i = [\varphi_i - d/2, \varphi_i + d/2]$ cover $J = [\alpha - d/2, \beta + d/2] \subset [\xi/2, \pi - \xi/2]$ because $d \le \xi$, so that $\sin\tfrac12(\varphi + \vartheta) \ge \sin(\xi/4)$ for $\varphi \in J$ and $\vartheta \in [0,\pi]$, which is all that the smooth term requires; the bound $N\Delta \le 2\omega$ used there becomes $Kd = \beta - \alpha + d \le \pi$; the treatment of the singular term uses only the spacing $d$ and $d \le 1/4$; and the two end half-cells lie in $J$. All constants then depend only on $\xi$.
\end{proof}

\begin{lemma}[Modulus of Continuity of the Envelope]
\label{lem:E_modulus}
For $0 \le \alpha < \beta \le \pi$ the function $F(\vartheta) := \int_\alpha^\beta\log|\cos\vartheta - \cos\varphi|\,d\varphi$ is absolutely continuous on $[0,\pi]$, with
\begin{equation}\label{eq:arc_derivative}
F'(\vartheta) = \log\frac{\bigl|\sin\tfrac12(\vartheta+\beta)\,\sin\tfrac12(\vartheta-\alpha)\bigr|}{\bigl|\sin\tfrac12(\vartheta+\alpha)\,\sin\tfrac12(\vartheta-\beta)\bigr|} \qquad (\vartheta \notin\{\alpha,\beta\}),
\end{equation}
and $|F(\vartheta) - F(\vartheta')| \le 4s\bigl(1 + \log(2\pi/s)\bigr)$ for $s := |\vartheta - \vartheta'| \le 1/2$. Consequently the envelope of a configuration satisfies
\begin{equation}\label{eq:E_modulus}
|E(\vartheta) - E(\vartheta')| \le M(\mathcal C)\,s\bigl(1 + \log(2\pi/s)\bigr), \qquad M(\mathcal C) := 4\Bigl(\frac2\omega + \sum_{k=1}^Q n_k\rho_k\Bigr),
\end{equation}
where $n_k = 2$ if $\beta_k < \pi/2$ and $n_k = 1$ if $\beta_k = \pi/2$ is the number of arcs in $A_k^\ast$.
\end{lemma}
\begin{proof}
Differentiating \eqref{eq:arc_integral} with $\mathrm{Cl}_2'(x) = -\log|2\sin(x/2)|$ gives \eqref{eq:arc_derivative}; the four logarithms have integrable singularities, so $F$ is absolutely continuous. For $x \in [-\pi,\pi]$ the first elementary fact gives $\bigl|\log|\sin(x/2)|\bigr| = \log\bigl(1/|\sin(x/2)|\bigr) \le \log(\pi/|x|)$, and for $x \in [\pi,2\pi]$ the same bound holds with $2\pi - x$ in place of $x$; the arguments $\vartheta \pm \alpha$ and $\vartheta \pm \beta$ lie in $[-\pi, 2\pi]$. As $\vartheta$ runs over an interval of length $s$, each argument runs over an interval of length $s$. If that interval lies in $[-\pi,\pi]$ or in $[\pi,2\pi]$, the integral of $\log(\pi/|x - x_0|)$ ($x_0 = 0$ or $x_0 = 2\pi$) over it is at most $2\int_0^{s/2}\log(\pi/u)\,du = s\bigl(1 + \log(2\pi/s)\bigr)$, its value for an interval centred at $x_0$, because the integrand decreases with $|x - x_0|$; if it meets both halves, all its points are within $s \le 1/2$ of $\pi$, where both bounds are less than $\log(\pi/(\pi - 1/2)) < 1$, and the integral is less than $s$. Integrating $|F'|$ gives the bound with the factor $4$, and \eqref{eq:E_modulus} follows by adding the arcs of \eqref{eq:envelope} with their densities, the two layer arcs having density $1/\omega$.
\end{proof}

\begin{lemma}[Envelopes of Bounded Densities]
\label{lem:env}
Let $h(t) = \log|2\sin(t/2)|$ as in \eqref{eq:kernel_split} and $\Omega(s) := \sup_{0\le s'\le s}\int_{-\pi}^{\pi}|h(t+s') - h(t)|\,dt$ for $s > 0$. Then $\Omega(s) \le 4s(1 + \log(\pi/s))$ for $0 < s \le 1$, and for $f, g \in L^\infty(0,\pi)$ with $|f| \le M$ and $|g| \le M$:
\begin{enumerate}
    \item[(a)] $|E_f(\vartheta) - E_f(\vartheta')| \le M\,\Omega(|\vartheta - \vartheta'|)$ for $\vartheta, \vartheta' \in [0,\pi]$;
    \item[(b)] $\int_B\bigl|\log|\cos\vartheta - \cos\varphi|\bigr|\,d\varphi \le C_0|B|^{1/2}$ for every $\vartheta \in [0,\pi]$ and every Borel set $B \subset [0,\pi]$, with $C_0 := \sqrt\pi\,(\log 2 + \pi/\sqrt3)$; in particular $|E_f| \le MC_0|B|^{1/2}$ on $[0,\pi]$ if $f$ vanishes outside $B$;
    \item[(c)] if $[0,\pi]$ is the union of arcs $A_1,\dots,A_q$ with pairwise disjoint interiors and lengths in $[\eta/2,\eta]$, $0 < \eta \le 1$, and $\int_{A_j}f = \int_{A_j}g$ for every $j$, then $|E_f - E_g| \le 8M\,\Omega(\eta)$ on $[0,\pi]$;
    \item[(d)] if $f \ge 0$, the potential $U^{f\,d\theta}$ is continuous on $\mathbb C$.
\end{enumerate}
\end{lemma}
\begin{proof}
\emph{The modulus $\Omega$.} Let $0 < s' \le s \le 1$ and $d(x) := \dist(x, 2\pi\mathbb Z)$. For $0 < u < 2\pi$, $|h'(u)| = \tfrac12|\cot(u/2)| \le 1/d(u)$, because $\tan x \ge x$ for $0 \le x < \pi/2$; and $|h(x)| \le \log(\pi/d(x))$ whenever $d(x) \le \pi/2$, because $2d(x)/\pi \le |2\sin(x/2)| \le 2$ by the first elementary fact of Appendix~\ref{app:rates}. By periodicity the integral defining $\Omega$ may be taken over $t \in [-s', 2\pi - s']$. For $t \in (0, 2\pi - s')$ the segment $[t, t+s']$ avoids $2\pi\mathbb Z$ and $|h(t+s') - h(t)| \le \int_t^{t+s'}|h'(u)|\,du$; after exchanging the integrals each $u \in (0,2\pi)$ is counted over a $t$-interval of length at most $\min\{s', d(u)\}$, so this part contributes at most $\int_0^{2\pi}\min\{s',d(u)\}\,d(u)^{-1}\,du = 2s'(1 + \log(\pi/s'))$. For $t \in [-s',0]$, $d(t) = |t|$ and $d(t+s') = t + s'$ are at most $1$, so this part contributes at most $2\int_0^{s'}\log(\pi/u)\,du = 2s'(1 + \log(\pi/s'))$. As $x \mapsto x(1 + \log(\pi/x))$ increases on $(0,\pi]$, $\Omega(s) \le 4s(1 + \log(\pi/s))$.

(a) By \eqref{eq:kernel_split}, $|E_f(\vartheta) - E_f(\vartheta')| \le M\int_0^\pi\bigl(|h(\vartheta+\varphi) - h(\vartheta'+\varphi)| + |h(\vartheta-\varphi) - h(\vartheta'-\varphi)|\bigr)\,d\varphi$, and the variables $\vartheta + \varphi$ and $\vartheta - \varphi$ run over the two halves of the period $[\vartheta - \pi, \vartheta + \pi]$; by the evenness of $h$ the sum is at most $M\Omega(|\vartheta - \vartheta'|)$.

(b) By the Cauchy--Schwarz inequality it suffices to note that, by \eqref{eq:kernel_split} and the same observation, $\norm{\log|\cos\vartheta - \cos\cdot|}_{L^2(0,\pi)} \le \sqrt\pi\log 2 + \sqrt2\,\norm{h}_{L^2(-\pi,\pi)}$, and that $\norm{h}_{L^2(-\pi,\pi)}^2 = \pi\sum_{k\ge1}k^{-2} = \pi^3/6$ by Parseval's identity and the Fourier series of $h$ in the proof of Lemma~\ref{lem:cheb}.

(c) Fix $\vartheta$ and write $k(\varphi) := \log|\cos\vartheta - \cos\varphi|$. Since $\int_{A_j}(f - g) = 0$, subtracting the mean of $k$ over $A_j$ gives $\bigl|\int_{A_j}k\,(f-g)\bigr| \le (2M/|A_j|)\int_{A_j}\int_{A_j}|k(\varphi) - k(\varphi')|\,d\varphi'\,d\varphi$. With $|A_j| \ge \eta/2$, $\varphi' = \varphi + s$ and $|s| \le \eta$, the sum over $j$ is at most $(4M/\eta)\int_{-\eta}^{\eta}\int|k(\varphi) - k(\varphi + s)|\,d\varphi\,ds$, the inner integral taken over the $\varphi$ with $\varphi, \varphi + s \in [0,\pi]$; as in (a) the inner integral is at most $\Omega(|s|) \le \Omega(\eta)$, which gives the bound $8M\Omega(\eta)$.

(d) The potential $U := U^{f\,d\theta}$ is lower semicontinuous on $\mathbb C$ and harmonic off $[-1,1]$ \cite[Thm.~3.1.2]{ransford1995potential}, and $U(\cos\vartheta) = -E_f(\vartheta)$ is continuous on $[-1,1]$ by (a). For $z = x + iy$ let $x_1$ be the point of $[-1,1]$ nearest to $x$; then $|z - s| \ge |x_1 - s|$ for $s \in [-1,1]$, so $U(z) \le U(x_1)$, and $x_1 \to x_0$ as $z \to x_0 \in [-1,1]$. Hence $\limsup_{z\to x_0}U(z) \le U(x_0)$, and $U$ is continuous at the points of $[-1,1]$, hence on $\mathbb C$.
\end{proof}

\begin{lemma}[Coefficient Norm]
\label{lem:coefficient_norm}
Let $\Lop = \Lop_P$ or $\Lop = \Lop_{CD}$ with constant $k \ge 0$, let $v \in \Phi_N$ and $v = \Phi\mat D\vect w$. Then, for $N \ge N_0$ of Lemma~\ref{lem:column_norms},
\[
\norm{\vect w}_2 \le C_B\sqrt N\,\bigl(\norm{v''}_{L^2(\tilde I)} + k\norm{v'}_{L^2(\tilde I)}\bigr),
\]
with the constant $C_B$ of Lemma~\ref{lem:column_norms}; for the Poisson operator the right-hand side is $C_B\sqrt N\norm{\Lop_Pv}_{L^2(\tilde I)}$.
\end{lemma}
\begin{proof}
Write $\vect w = (\varpi_1, \dots, \varpi_n)$ to avoid confusion with the basis functions $w_i$, and let $v = \sum_{J\ge1}c_Jw_J$ be the expansion of $v$ in the orthonormal basis. Since $v = \sum_j\varpi_j(\mat D)_{jj}\varphi_j$ and $\varphi_j = w_{j+2} - a_jw_1 - b_jw_2$, one has $c_J = \varpi_{J-2}(\mat D)_{J-2,J-2}$ for $J \ge 3$. The functions $\sqrt{2/\tilde L}\cos(\mu_J(x+\delta))$, $J \ge 1$, are orthonormal in $L^2(\tilde I)$, so $\norm{v'}^2_{L^2(\tilde I)} = \sum_J\lambda_Jc_J^2$ and $\norm{v''}^2_{L^2(\tilde I)} = \sum_J\lambda_J^2c_J^2$; with $\rho_J^2 = \lambda_J^2 + k^2\lambda_J$ from \eqref{eq:Lw_nodes} this gives $\sum_{J\ge3}\rho_J^2c_J^2 \le \norm{v''}^2_{L^2(\tilde I)} + k^2\norm{v'}^2_{L^2(\tilde I)}$. By \eqref{eq:column_norms}, $(\mat D)_{jj}^{-1} = \norm{\mat B\vect e_j}_2 \le C_B\rho_{j+2}\sqrt N$, hence $\varpi_j^2 = \bigl((\mat D)_{jj}^{-1}c_{j+2}\bigr)^2 \le C_B^2N\rho_{j+2}^2c_{j+2}^2$, and summing over $j$ gives $\norm{\vect w}_2^2 \le C_B^2N\bigl(\norm{v''}^2 + k^2\norm{v'}^2\bigr) \le C_B^2N\bigl(\norm{v''} + k\norm{v'}\bigr)^2$. For the Poisson operator $\Lop_Pv = -v''$.
\end{proof}

The converse estimate bounds the test function by the coefficient vector.

\begin{lemma}[Norm of the Test Function]
\label{lem:test_norm}
Let $\Lop = \Lop_P$ or $\Lop = \Lop_{CD}$ with constant $k \ge 0$. There is $C_\Psi$, depending only on $\theta_0$, $L$ and $k$, such that for $N \ge N_0$ of Lemma~\ref{lem:column_norms} and all $\vect w \in \mathbb R^n$,
\[
\norm{\Lop\Phi\mat D\vect w}_{L^2(\tilde I)} \le C_\Psi\,\frac{\norm{\vect w}_2}{\sqrt N} .
\]
\end{lemma}
\begin{proof}
With the notation of the previous proof, $v := \Phi\mat D\vect w = \sum_{J\ge1}c_Jw_J$ has $c_J = \varpi_{J-2}(\mat D)_{J-2,J-2}$ for $J \ge 3$, $c_1 = -\sum_j\varpi_j(\mat D)_{jj}a_j$ and $c_2 = -\sum_j\varpi_j(\mat D)_{jj}b_j$, and $\norm{v''}^2_{L^2(\tilde I)} = \sum_J\lambda_J^2c_J^2$, $\norm{v'}^2_{L^2(\tilde I)} = \sum_J\lambda_Jc_J^2 \le \lambda_1^{-1}\sum_J\lambda_J^2c_J^2$. By \eqref{eq:column_norms} and $\rho_J \ge \lambda_J$, $(\mat D)_{jj} = 1/\norm{\mat B\vect e_j}_2 \le 1/(c_B\lambda_{j+2}\sqrt N)$, so $\sum_{J\ge3}\lambda_J^2c_J^2 \le \sum_j\varpi_j^2/(c_B^2N) = \norm{\vect w}_2^2/(c_B^2N)$. With $A = \max\{1/\sin\theta_0, 1/\sin2\theta_0\}$ as in the proof of Lemma~\ref{lem:column_norms}, $|a_j|, |b_j| \le A$, and the Cauchy--Schwarz inequality gives
\[
|c_1|, |c_2| \le A\sum_j|\varpi_j|(\mat D)_{jj} \le \frac{A\norm{\vect w}_2}{c_B\sqrt N}\Bigl(\sum_{J\ge3}\lambda_J^{-2}\Bigr)^{1/2} \le \frac{A\norm{\vect w}_2}{c_B\sqrt N}\,\frac{\tilde L^2}{\sqrt{90}},
\]
because $\sum_{J\ge1}\lambda_J^{-2} = (\tilde L/\pi)^4\sum_{J\ge1}J^{-4} = \tilde L^4/90$. Hence
\begin{gather*}
\norm{v''}_{L^2(\tilde I)} \le \Bigl(1 + \frac{(\lambda_1^2 + \lambda_2^2)A^2\tilde L^4}{90}\Bigr)^{1/2}\frac{\norm{\vect w}_2}{c_B\sqrt N}, \\
\norm{\Lop v}_{L^2(\tilde I)} \le \norm{v''}_{L^2(\tilde I)} + k\norm{v'}_{L^2(\tilde I)} \le \bigl(1 + k\lambda_1^{-1/2}\bigr)\norm{v''}_{L^2(\tilde I)} .
\end{gather*}
\end{proof}

\begin{lemma}[A Lower Bound Through a Test Function]
\label{lem:bump}
Let $v \in C^2(\bar I)$, $g := -v''$, and let $[x_a, x_b] \subset I$ have length $h_b := x_b - x_a$. If $g$ has constant sign on $[x_a,x_b]$ and $|g| \ge g_\ast$ there, then $\norm{v}_{L^2(I)} \ge g_\ast\,h_b^{5/2}/(32\pi^2)$.
\end{lemma}
\begin{proof}
Let $\phi(x) := \sin^4\bigl(\pi(x - x_a)/h_b\bigr)$ on $[x_a,x_b]$ and $\phi := 0$ elsewhere. Then $\phi \in C^2(\mathbb R)$, $\phi$ and $\phi'$ vanish at $x_a$ and $x_b$, $\int_{x_a}^{x_b}\phi = 3h_b/8$, and $|\phi''| \le 12\pi^2/h_b^2$, so $\norm{\phi''}_{L^2} \le 12\pi^2h_b^{-3/2}$. Two integrations by parts give $\int_I g\phi = -\int_I v''\phi = -\int_I v\phi''$, hence $g_\ast\cdot 3h_b/8 \le \bigl|\int_Ig\phi\bigr| \le \norm{v}_{L^2(I)}\norm{\phi''}_{L^2}$, which is the claim.
\end{proof}

The last lemma realizes an interpolation configuration by a set of points and bounds the Lagrange polynomials of that set by the envelope, as Lemma~\ref{lem:lagrange_polynomials} does for the node lattice.

\begin{lemma}[Lagrange Polynomials of an Interpolation Configuration]
\label{lem:lattice_lagrange}
Let $\mathcal S$ be an interpolation configuration with envelope $E$, and put $X := 1 + 3\sum_{k\le Q_1}|R_k|$. For $N \ge N_0$ let $S_N \subset [\xi,\pi-\xi]$ consist of
\begin{enumerate}
    \item[(i)] for each node arc $A_k$, the interior nodes $\theta_m \in A_k$, $2 \le m \le N-1$, with $m \equiv r \pmod{q_k}$ for some $r \in R_k$;
    \item[(ii)] for each virtual arc $A_k$, the lattice $\varphi_i^{(k)} := \alpha_k + (i-1)d_k$, $1 \le i \le K_k$, with $K_k := \lceil\rho_k(\beta_k - \alpha_k)N\rceil + X$ and $d_k := (\beta_k - \alpha_k)/(K_k - 1)$.
\end{enumerate}
Then the points of $S_N$ are distinct, their number $N'$ satisfies $N \le N' \le N + K_0$ with $K_0$ depending only on $\mathcal S$, and the Lagrange polynomials $\ell_s(z) := \omega_S(z)/\bigl((z - \cos s)\,\omega_S'(\cos s)\bigr)$ of the set $\{\cos s : s \in S_N\}$, where $\omega_S(z) := \prod_{s\in S_N}(z - \cos s)$, satisfy, for all $s \in S_N$ and all $\vartheta \in [0,\pi]$,
\begin{equation}\label{eq:lattice_lagrange}
N^{-1}\log|\ell_s(\cos\vartheta)| \le E(\vartheta) - E(s) + C\,\frac{\log N}{N},
\end{equation}
with $C$ and $N_0$ depending only on $\mathcal S$.
\end{lemma}
\begin{proof}
\emph{Lattices and counting.} For a node arc $A_k$ and a residue $r \in R_k$, the nodes of (i) with $m \equiv r$ form a lattice of spacing $q_k\Delta$ whose points lie in $A_k$, since the nodes are the points $\theta_0 + \Delta\mathbb Z$ in $[\theta_0,\pi-\theta_0]$; it has between $(\beta_k - \alpha_k)/(q_k\Delta) - 3$ and $(\beta_k - \alpha_k)/(q_k\Delta) + 1$ points, because an interval of length $\ell$ contains between $\ell/(q_k\Delta) - 1$ and $\ell/(q_k\Delta) + 1$ points of a lattice of spacing $q_k\Delta$ and at most two of them, $\theta_1$ and $\theta_N$, are excluded as boundary nodes. With $1/\Delta = (N-1)/\omega$ and $\rho_k = |R_k|/(q_k\omega)$, the node arcs contribute between $(N-1)m_1 - 3\sum_{k\le Q_1}|R_k|$ and $(N-1)m_1 + \sum_{k\le Q_1}|R_k|$ points, where $m_1 := \sum_{k\le Q_1}\rho_k(\beta_k-\alpha_k) \le 1$; the virtual arcs contribute between $Nm_2 + Q_2X$ and $Nm_2 + Q_2(X+1)$ points, where $m_2 := \sum_{k>Q_1}\rho_k(\beta_k-\alpha_k)$ and $Q_2 \ge 1$ is their number. Since $m_1 + m_2 = 1$, $N' \ge N - 1 - 3\sum_{k\le Q_1}|R_k| + Q_2X \ge N$ by the choice of $X$, and $N' \le N + K_0$ with $K_0 := \sum_{k\le Q_1}|R_k| + Q_2(X+1)$. The points are distinct: two residue lattices of the same node arc consist of nodes with different residues, points of different arcs lie in disjoint arcs, and the points of a lattice are distinct. Thus $S_N$ is the union of $P := \sum_{k\le Q_1}|R_k| + Q_2$ lattices $\Lambda_1, \dots, \Lambda_P$; the $j$-th of them has spacing $d_j$, $K_j \ge 2$ points, lies on an arc $[a_j,b_j] \subset [\xi,\pi-\xi]$ with $b_j - a_j = (K_j - 1)d_j$, and has density $\varrho_j$ ($= 1/(q_k\omega)$ for a residue lattice of the node arc $A_k$, $= \rho_k$ for a virtual arc) in the sense that $d_j^{-1} = \varrho_jN + O(1)$: for a residue lattice $d_j^{-1} = (N-1)/(q_k\omega)$, and for a virtual arc $d_j^{-1} = (K_k-1)/(\beta_k - \alpha_k) \in [\rho_kN + (X-1)/(\beta_k-\alpha_k),\ \rho_kN + X/(\beta_k-\alpha_k))$. For $N \ge N_0$ every spacing is at most $\min\{\xi,1\}/4$ and every $K_j \ge 2$. The arc $[a_j,b_j]$ of a virtual lattice is $A_k$; that of a residue lattice lies in $A_k$ and differs from it by two intervals of length less than $(q_k+1)\Delta$ each, which is at most $C/N$ and, for $N \ge N_0$, at most $1/4$.

\emph{Products.} For $\vartheta \in [0,\pi]$, $\log|\omega_S(\cos\vartheta)| = \sum_{j=1}^P\sum_{i=1}^{K_j}g_\vartheta(\varphi_i^{(j)})$ with $g_\vartheta(\varphi) = \log|\cos\vartheta - \cos\varphi|$ as in Lemma~\ref{lem:riemann}. For each lattice, Lemma~\ref{lem:riemann_arc} with $d = d_j$ gives $d_j\sum_ig_\vartheta(\varphi^{(j)}_i) = \int_{a_j}^{b_j}g_\vartheta + O(d_j\log(1/d_j))$, the sum being taken without the point of the lattice within $d_j/4$ of $\vartheta$ if there is one (there is at most one, the points being $d_j$ apart). Dividing by $d_j$ and using $d_j^{-1} = \varrho_jN + O(1)$, $\log(1/d_j) \le C\log N$, $|\int_{a_j}^{b_j}g_\vartheta| \le \int_0^\pi|g_\vartheta| \le C$ and, for the residue lattices, $|\int_{A_k}g_\vartheta - \int_{a_j}^{b_j}g_\vartheta| \le C\log N/N$ (the integral of $|g_\vartheta|$ over an interval of length $\lambda \le 1/4$ inside $[\xi/2,\pi-\xi/2]$ is at most $C\lambda(1 + \log(1/\lambda))$, by the bound on the end half-cells in the proof of Lemma~\ref{lem:riemann}), we obtain for the sum $\Sigma^\circ_\vartheta$ over $S_N$ without the points within a quarter spacing of $\vartheta$
\begin{equation}\label{eq:lattice_sum}
\Sigma^\circ_\vartheta = N\sum_{k=1}^Q\rho_k\int_{A_k}g_\vartheta + O(\log N) = NE(\vartheta) + O(\log N),
\end{equation}
since the residue lattices of $A_k$ together have density $|R_k|/(q_k\omega) = \rho_k$. Each omitted point contributes a factor of modulus at most $2$, so
\begin{equation}\label{eq:lattice_upper}
\log|\omega_S(\cos\vartheta)| \le NE(\vartheta) + C\log N \qquad\text{for all } \vartheta \in [0,\pi].
\end{equation}

\emph{Derivatives.} Let $s \in S_N$. Then $\log|\omega_S'(\cos s)| = \sum_{s' \in S_N\setminus\{s\}}g_s(s')$. In the lattice of $s$ no other point lies within a quarter spacing of $s$, and in every other lattice at most one point $s'$ does; such a point is at distance at least $\Delta$ from $s$ if both are nodes and at least the distance between their arcs otherwise, so $|s - s'| \ge c/N$ with $c > 0$ depending on $\mathcal S$, and since $s, s' \in [\xi,\pi-\xi]$,
\begin{equation}\label{eq:point_distance}
|\cos s - \cos s'| = 2\bigl|\sin\tfrac12(s-s')\bigr|\sin\tfrac12(s+s') \ge \frac{2\sin\xi}{\pi}\,|s - s'| \ge \frac{c'}{N},
\end{equation}
by the first elementary fact of Appendix~\ref{app:rates} and $\tfrac12(s+s') \in [\xi,\pi-\xi]$. Hence $g_s(s') \ge -C\log N$ for the at most $P$ omitted points, and \eqref{eq:lattice_sum} with $\vartheta = s$ gives
\begin{equation}\label{eq:lattice_lower}
\log|\omega_S'(\cos s)| \ge NE(s) - C\log N .
\end{equation}

\emph{Lagrange polynomials.} Let $s \in S_N$, $d_s$ the spacing of its lattice, and $\vartheta \in [0,\pi]$. If $|\vartheta - s| \ge d_s/4$, then $|\cos\vartheta - \cos s| \ge (2\sin(\xi/4)/\pi)|\vartheta - s| \ge c/N$ by the two elementary facts of Appendix~\ref{app:rates} with $\xi$ in place of $\theta_0$ (as $s \in [\xi,\pi-\xi]$), and \eqref{eq:lattice_upper}, \eqref{eq:lattice_lower} give $\log|\ell_s(\cos\vartheta)| = \log|\omega_S(\cos\vartheta)| - \log|\cos\vartheta - \cos s| - \log|\omega_S'(\cos s)| \le N[E(\vartheta) - E(s)] + C\log N$. If $|\vartheta - s| < d_s/4$, write $\log|\ell_s(\cos\vartheta)| = \sum_{s'\ne s}g_\vartheta(s') - \sum_{s'\ne s}g_s(s')$; the second sum is at least $NE(s) - C\log N$ by \eqref{eq:lattice_lower}, and the first is at most $NE(\vartheta) + C\log N$ by the argument of \eqref{eq:lattice_upper}, in which the point of the lattice of $s$ within a quarter spacing of $\vartheta$ is $s$ itself. This proves \eqref{eq:lattice_lagrange}.
\end{proof}

\subsection{The Certified Configuration}

\begin{table}[ht!]
\centering\small
\caption{Configuration for $\delta = L/4$ ($\omega = 2\pi/3$, $\theta_0 = \pi/6$). The layer parameter is $\theta_s = 0.9462$; the eight extension arcs are listed first, then the 16 centre arcs, the last of which ends at $\pi/2$ (read column by column). All arcs are in radians.}
\label{tab:config_quarter}
\begin{tabular}{ccc|ccc|ccc}
$\alpha_k$ & $\beta_k$ & $\rho_k$ & $\alpha_k$ & $\beta_k$ & $\rho_k$ & $\alpha_k$ & $\beta_k$ & $\rho_k$ \\ \hline
$0.0200$ & $0.0633$ & $0.3869$ & $0.9462$ & $0.9494$ & $0.4638$ & $1.1335$ & $1.1882$ & $0.3897$ \\
$0.0633$ & $0.1066$ & $0.1133$ & $0.9494$ & $0.9545$ & $0.4585$ & $1.1882$ & $1.2428$ & $0.3850$ \\
$0.1066$ & $0.1499$ & $0.1573$ & $0.9545$ & $0.9624$ & $0.4509$ & $1.2428$ & $1.2975$ & $0.3804$ \\
$0.1499$ & $0.1933$ & $0.1578$ & $0.9624$ & $0.9745$ & $0.4430$ & $1.2975$ & $1.3522$ & $0.3773$ \\
$0.1933$ & $0.2366$ & $0.1483$ & $0.9745$ & $0.9931$ & $0.4339$ & $1.3522$ & $1.4068$ & $0.3749$ \\
$0.2366$ & $0.2799$ & $0.1259$ & $0.9931$ & $1.0217$ & $0.4237$ & $1.4068$ & $1.4615$ & $0.3732$ \\
$0.2799$ & $0.3232$ & $0.1127$ & $1.0217$ & $1.0658$ & $0.4125$ & $1.4615$ & $1.5161$ & $0.3721$ \\
$0.3232$ & $0.3665$ & $0.0019$ & $1.0658$ & $1.1335$ & $0.4006$ & $1.5161$ & $\pi/2$ & $0.3716$ \\
\end{tabular}
\end{table}

\paragraph{Evaluation of the configuration.} For the configuration of Table~\ref{tab:config_quarter} the values $E(\theta_0)$, $E_{\max}$ and $E_c$ are those of the explicit function \eqref{eq:envelope} and \eqref{eq:arc_integral}, evaluated with the Clausen function of Lemma~\ref{lem:cl2}. The value $E(\theta_0)$ is a single evaluation. Since the configuration is symmetric, $E(\pi - \vartheta) = E(\vartheta)$, so $E_{\max}$ is the maximum over $[0,\pi/2]$ and $E_c$ the maximum over $[\theta_s,\pi/2]$. Each was bounded from above by evaluating $E$ on a grid of spacing $h = 10^{-6}$ that starts at the left endpoint of the respective interval and ends at $\pi/2$, so that every point of the interval is within $h/2$ of a grid point, and adding the modulus of continuity \eqref{eq:E_modulus} at $s = h/2$, which is at most $0.00056$; the largest grid value is a lower bound. The evaluation used double precision, with the argument reduced to $(-\pi,\pi]$ and the Clausen function computed from its series $\mathrm{Cl}_2(x) = x - x\log|x| + \sum_{k\ge1}|B_{2k}|\,x^{2k+1}/\bigl(2k\,(2k+1)!\bigr)$ truncated after $25$ terms, whose remainder is below $10^{-18}$. The computed arguments differ from the exact ones by less than $4\times10^{-15}$, and $|\mathrm{Cl}_2(y) - \mathrm{Cl}_2(y')| \le s\bigl(2 + \log(2\pi/s)\bigr)$ for $s := |y - y'| \le 1$, so each computed value of $\mathrm{Cl}_2$ errs by less than $2\times10^{-13}$, and each computed value of $E$, a combination of Clausen values with coefficients of total modulus $M(\mathcal C)$, by less than $10^{-10}$; this amount, added to the padding, does not change the reported values, which are rounded outwards to five decimals. As a check, the double-precision values of $E$ differ from $30$-digit evaluations of the closed form by less than $3\times10^{-15}$ at $400$ random points of $[0,\pi/2]$, at $\theta_0$ and at the two maximizing grid points.

This gives
\[
E(\theta_0) = -0.58682, \qquad E_{\max} \in [-0.43428, -0.43371], \qquad E_c \in [-0.83815, -0.83759],
\]
hence $\ell \in [-0.15311, -0.15254]$, $\tau \in [0.40331, 0.40444]$ and $\ell + a\tau = E(\theta_0) - (1-a)E_{\max} - aE_c \ge 0.0084$, the worst case over the two intervals; $-\ell/\tau = 0.378$ against $a = 0.4$. The total mass is $0.995$, the extension arcs end before $\theta_0$, and $\ell > -(\kappa_\sigma - \kappa_\Lambda)$ holds with $\kappa_\sigma - \kappa_\Lambda = 0.3231$.

\paragraph{How the configuration was found.} For fixed arcs and fixed $\theta_s$ the densities were chosen by a linear program in $(\rho_1,\dots,\rho_Q, c)$: maximize $E(\theta_0) + c$ subject to $E + c \le 0$ at ten equispaced points of every arc and of the layer and at additional points near $\vartheta = 0$ and near $\theta_0$, $E + c \le -\tau$ at those points that lie in $[\theta_s,\pi/2]$, and $m(\mathcal C) \le 199/200$, for a prescribed $\tau$; the constant $c$ plays the role of the normalization $b$ in \eqref{eq:reduced_problem}, and the constraints are the discrete form of \eqref{eq:reduced_constraints} with the free set $[-\cos\theta_s, \cos\theta_s]$. The arcs were $16$ cells of $[1/50, \theta_0]$ and $16$ cells of $[\theta_s, \pi/2]$, refined geometrically towards $\theta_0$ and towards $\theta_s$; $\theta_s$ was scanned; arcs whose density vanished were discarded, which produced the gap between the extension arcs and $\theta_0$ visible in Table~\ref{tab:config_quarter}. The table lists the arcs rounded to four decimals, and the certified values above were computed from the rounded configuration.

The same linear program with $\omega \le \pi/2$, evaluated a posteriori in the same way, returned negative margins $\ell + a\tau$ at every $\tau$ tried: between $-0.002$ and $-0.005$ for $\delta = L/2$, between $-0.017$ and $-0.026$ for $\delta = L$, and between $-0.038$ and $-0.055$ for $\delta = 2L$, at $\tau/\kappa_\sigma = 0.2$, $0.5$ and $0.8$, and the configurations built from the equilibrium measures of Table~\ref{tab:dual} give margins of the same sign, closer to zero; for $\delta = L/2$ the margins are within the discretization of the density from zero, as the linear tail law of Corollary~\ref{cor:self_dual} requires.

\subsection{The Certified Interpolation Configurations}

Theorem~\ref{thm:dual} bounds the rate of $\max\{\gamma_t, t\Gamma_t\}$ at $t = e^{-N\tau}$ from above by $\Phi(\mathcal S,\tau)$, and interpolation configurations with small $\Phi(\mathcal S,\tau)$ are found by a linear program: minimize $\Phi_\ast$ over piecewise-constant densities on cells of the extension and of the physical interval, at most $1/\omega$ on the physical cells, whose points are nodes, and unrestricted on the extension cells inside $[\xi, \theta_g]$, whose points are virtual, with total mass one, subject to $E(\vartheta) - E(\varphi) - \tau \le \Phi_\ast$ for all $\vartheta$ in the physical interval and all $\varphi$ in a node cell and to $E(\vartheta) - E(\varphi) \le \Phi_\ast$ for all $\varphi$ in a virtual cell, the constraints being imposed at sample points of the cells; cells whose density vanishes are removed from the candidate support and the program is re-solved, and the gap point $\theta_g$ is scanned. Every feasible point is an interpolation configuration up to the rational approximation of its densities.

The interpolation configurations behind Corollary~\ref{cor:certified_instances}(ii) were obtained and certified as follows. For a level $\tau = \bar\tau\kappa_\sigma$ this linear program was solved with $25$ extension cells on $[1/50, \theta_0]$ and $33$ physical cells on $[\theta_0, \pi/2]$, refined towards $\theta_0$ and mirrored about $\pi/2$, the constraints being imposed at two interior points of every candidate cell and at $64$ points of $[\theta_0, \pi/2]$ (the configuration is symmetric about $\pi/2$); the gap point $\theta_g$ was scanned over eight values in $[0.15\theta_0, 0.995\theta_0]$, and the program was solved at most four times, cells whose density vanished being removed from the candidate support between the solves. The minimizer was turned into an interpolation configuration of Definition~\ref{def:interp_configuration} with $\xi = 1/50$: every physical cell of positive node density $\bar\rho$, and its mirror image, became a node arc with modulus $q = 200$ and $r := \lfloor 200\,\omega\bar\rho\rfloor$ residues, of density $r/(200\omega) \le \bar\rho \le 1/\omega$, the cell ending at $\pi/2$ and its mirror image forming a single arc, so that every interpolation configuration has $65$ node arcs; every extension cell of positive virtual density, and its mirror image, became a virtual arc; every arc was shrunk by $10^{-9}$ at both ends, which makes the arcs pairwise disjoint (which residues are used does not enter the envelope); and the virtual densities were then multiplied by a common factor so that the total mass is one. In exact arithmetic the virtual densities are defined as the listed values times $(1 - m_1)/m_2$, with $m_1$ and $m_2$ the exact masses of the node arcs and of the virtual arcs with the listed densities; this factor differs from one by less than $10^{-13}$ and changes $E$ by less than $10^{-13}$, which the padding below covers.

The quantities $E_I$, $E_1$ and $E_2$ of Definition~\ref{def:interp_configuration} were then bounded by means of the closed form \eqref{eq:arc_integral} on grids of spacing $h = 10^{-6}$ on the physical interval and on every arc: the grid maximum was increased and the grid minima decreased by the padding $M(h/2)\bigl(1 + \log(4\pi/h)\bigr) + 10^{-10}$ with $M := 4\sum_k\rho_k$, whose first term bounds $|E(\vartheta) - E(\vartheta')|$ for $|\vartheta - \vartheta'| \le h/2$ by the proof of Lemma~\ref{lem:E_modulus} applied to each arc of \eqref{eq:interp_envelope}, and whose second term bounds the rounding error of the computed values of $E$, as for the configuration above, since $M < 420$. As a check, the double-precision values were compared with $30$-digit evaluations at five points of each of the $116$ interpolation configurations computed (at $\theta_0$, at midpoints of arcs and at random points of the physical interval), with differences below $10^{-14}$. With $\bar E_I$ the padded maximum and $\underline E_1$, $\underline E_2$ the padded minima, the certified values $\bar A_1 := \bar E_I - \underline E_1 \ge A_1(\mathcal S)$ and $\bar A_2 := \bar E_I - \underline E_2 \ge A_2(\mathcal S)$ give the interval $[\bar A_1/(1-a), -\bar A_2/a]$, which is contained in \eqref{eq:certified_interval}; every $\tau$ in it is certified by Corollary~\ref{cor:certified}.

Tables~\ref{tab:certified_L} and \ref{tab:certified_2L} list, for each extension, a subcover of the certified range by the certificates computed. Each row is one interpolation configuration, and its last column is the interval of levels that it certifies, formed from the two preceding columns; the levels $\bar\tau$ were chosen so that the intervals overlap, and their union contains the ranges of Corollary~\ref{cor:certified_instances}(ii). The arcs, densities, residue counts and certified values of all the configurations are given in the supplementary data file \texttt{certified\_\allowbreak interpolation\_\allowbreak configurations.json}.

\begingroup
\footnotesize
\setlength{\tabcolsep}{4pt}
\begin{longtable}{cccccc}
\caption{Certified interpolation configurations for $\delta = L$, a subcover of the certified range by 52 of the 76 certificates computed: level $\bar\tau = \tau/\kappa_\sigma$ at which the linear program was solved, which need not lie in the certified interval, number of virtual arcs, mass carried by the nodes, certified upper bounds $\bar A_1$ and $\bar A_2$ on $A_1(\mathcal S)$ and $A_2(\mathcal S)$, and the certified interval $[\bar A_1/(1-a), -\bar A_2/a]$ in units of $\kappa_\sigma$; the masses are rounded to the nearest $10^{-3}$ and the other entries to the nearest $10^{-4}$, and the certified values are given in the supplementary data file.}
\label{tab:certified_L}\\
$\bar\tau$ & virtual arcs & mass of nodes & $\bar A_1$ & $\bar A_2$ & interval \\ \hline
\endfirsthead
$\bar\tau$ & virtual arcs & mass of nodes & $\bar A_1$ & $\bar A_2$ & interval \\ \hline
\endhead
$0.1005$ & $30$ & $0.478$ & $+0.0626$ & $-0.0947$ & $[0.0969, 0.0977]$ \\
$0.1015$ & $30$ & $0.479$ & $+0.0629$ & $-0.0957$ & $[0.0974, 0.0988]$ \\
$0.1030$ & $30$ & $0.481$ & $+0.0637$ & $-0.0983$ & $[0.0986, 0.1015]$ \\
$0.1055$ & $30$ & $0.483$ & $+0.0647$ & $-0.1009$ & $[0.1001, 0.1041]$ \\
$0.1085$ & $30$ & $0.486$ & $+0.0663$ & $-0.1041$ & $[0.1026, 0.1075]$ \\
$0.1130$ & $30$ & $0.490$ & $+0.0685$ & $-0.1088$ & $[0.1060, 0.1123]$ \\
$0.1165$ & $30$ & $0.494$ & $+0.0704$ & $-0.1135$ & $[0.1090, 0.1171]$ \\
$0.1235$ & $30$ & $0.500$ & $+0.0737$ & $-0.1197$ & $[0.1141, 0.1235]$ \\
$0.1300$ & $30$ & $0.506$ & $+0.0772$ & $-0.1274$ & $[0.1194, 0.1314]$ \\
$0.1441$ & $30$ & $0.520$ & $+0.0848$ & $-0.1428$ & $[0.1312, 0.1474]$ \\
$0.1591$ & $30$ & $0.534$ & $+0.0929$ & $-0.1588$ & $[0.1438, 0.1639]$ \\
$0.1800$ & $30$ & $0.553$ & $+0.1054$ & $-0.1787$ & $[0.1631, 0.1844]$ \\
$0.1960$ & $30$ & $0.568$ & $+0.1154$ & $-0.1923$ & $[0.1786, 0.1984]$ \\
$0.2030$ & $28$ & $0.575$ & $+0.1198$ & $-0.2026$ & $[0.1854, 0.2090]$ \\
$0.2232$ & $28$ & $0.592$ & $+0.1309$ & $-0.2241$ & $[0.2026, 0.2312]$ \\
$0.2382$ & $26$ & $0.610$ & $+0.1455$ & $-0.2350$ & $[0.2251, 0.2425]$ \\
$0.2600$ & $26$ & $0.625$ & $+0.1565$ & $-0.2581$ & $[0.2422, 0.2663]$ \\
$0.2720$ & $26$ & $0.634$ & $+0.1629$ & $-0.2714$ & $[0.2521, 0.2800]$ \\
$0.3000$ & $26$ & $0.655$ & $+0.1784$ & $-0.3024$ & $[0.2760, 0.3120]$ \\
$0.3108$ & $26$ & $0.662$ & $+0.1851$ & $-0.3129$ & $[0.2865, 0.3228]$ \\
$0.3400$ & $26$ & $0.682$ & $+0.2034$ & $-0.3402$ & $[0.3148, 0.3510]$ \\
$0.3499$ & $26$ & $0.690$ & $+0.2099$ & $-0.3490$ & $[0.3248, 0.3601]$ \\
$0.3800$ & $26$ & $0.711$ & $+0.2286$ & $-0.3739$ & $[0.3539, 0.3857]$ \\
$0.4106$ & $26$ & $0.732$ & $+0.2483$ & $-0.3992$ & $[0.3842, 0.4119]$ \\
$0.4256$ & $24$ & $0.743$ & $+0.2566$ & $-0.4254$ & $[0.3972, 0.4389]$ \\
$0.4600$ & $24$ & $0.763$ & $+0.2776$ & $-0.4599$ & $[0.4296, 0.4745]$ \\
$0.5000$ & $24$ & $0.787$ & $+0.3033$ & $-0.4989$ & $[0.4694, 0.5147]$ \\
$0.5400$ & $22$ & $0.810$ & $+0.3290$ & $-0.5377$ & $[0.5092, 0.5548]$ \\
$0.5800$ & $20$ & $0.833$ & $+0.3563$ & $-0.5768$ & $[0.5514, 0.5951]$ \\
$0.6200$ & $20$ & $0.853$ & $+0.3823$ & $-0.6140$ & $[0.5917, 0.6335]$ \\
$0.6312$ & $20$ & $0.859$ & $+0.3899$ & $-0.6242$ & $[0.6034, 0.6440]$ \\
$0.6600$ & $16$ & $0.874$ & $+0.4104$ & $-0.6511$ & $[0.6352, 0.6718]$ \\
$0.6732$ & $16$ & $0.880$ & $+0.4189$ & $-0.6640$ & $[0.6483, 0.6851]$ \\
$0.7000$ & $14$ & $0.893$ & $+0.4376$ & $-0.6876$ & $[0.6772, 0.7094]$ \\
$0.7138$ & $14$ & $0.899$ & $+0.4470$ & $-0.7020$ & $[0.6918, 0.7243]$ \\
$0.7409$ & $14$ & $0.911$ & $+0.4648$ & $-0.7280$ & $[0.7193, 0.7511]$ \\
$0.7559$ & $14$ & $0.918$ & $+0.4751$ & $-0.7426$ & $[0.7354, 0.7662]$ \\
$0.7800$ & $14$ & $0.927$ & $+0.4910$ & $-0.7637$ & $[0.7599, 0.7880]$ \\
$0.8012$ & $12$ & $0.936$ & $+0.5061$ & $-0.7834$ & $[0.7833, 0.8083]$ \\
$0.8200$ & $10$ & $0.944$ & $+0.5203$ & $-0.8003$ & $[0.8052, 0.8258]$ \\
$0.8340$ & $10$ & $0.949$ & $+0.5291$ & $-0.8127$ & $[0.8188, 0.8385]$ \\
$0.8450$ & $10$ & $0.953$ & $+0.5372$ & $-0.8240$ & $[0.8314, 0.8502]$ \\
$0.8600$ & $10$ & $0.958$ & $+0.5478$ & $-0.8380$ & $[0.8479, 0.8646]$ \\
$0.8750$ & $10$ & $0.964$ & $+0.5585$ & $-0.8509$ & $[0.8644, 0.8779]$ \\
$0.8815$ & $10$ & $0.966$ & $+0.5631$ & $-0.8565$ & $[0.8715, 0.8837]$ \\
$0.8894$ & $8$ & $0.969$ & $+0.5687$ & $-0.8644$ & $[0.8801, 0.8919]$ \\
$0.8950$ & $8$ & $0.970$ & $+0.5721$ & $-0.8685$ & $[0.8855, 0.8961]$ \\
$0.9000$ & $6$ & $0.973$ & $+0.5772$ & $-0.8731$ & $[0.8934, 0.9008]$ \\
$0.9031$ & $6$ & $0.974$ & $+0.5793$ & $-0.8758$ & $[0.8966, 0.9036]$ \\
$0.9080$ & $6$ & $0.975$ & $+0.5825$ & $-0.8805$ & $[0.9015, 0.9085]$ \\
$0.9120$ & $6$ & $0.977$ & $+0.5857$ & $-0.8853$ & $[0.9065, 0.9135]$ \\
$0.9160$ & $4$ & $0.978$ & $+0.5899$ & $-0.8857$ & $[0.9130, 0.9139]$ \\
\end{longtable}
\endgroup

\begingroup
\footnotesize
\setlength{\tabcolsep}{4pt}
\begin{longtable}{cccccc}
\caption{Certified interpolation configurations for $\delta = 2L$, a subcover of the certified range by 25 of the 40 certificates computed, with the columns of Table~\ref{tab:certified_L}.}
\label{tab:certified_2L}\\
$\bar\tau$ & virtual arcs & mass of nodes & $\bar A_1$ & $\bar A_2$ & interval \\ \hline
\endfirsthead
$\bar\tau$ & virtual arcs & mass of nodes & $\bar A_1$ & $\bar A_2$ & interval \\ \hline
\endhead
$0.1000$ & $30$ & $0.370$ & $+0.0535$ & $-0.1558$ & $[0.0791, 0.1059]$ \\
$0.1027$ & $30$ & $0.373$ & $+0.0546$ & $-0.1591$ & $[0.0806, 0.1082]$ \\
$0.1400$ & $30$ & $0.416$ & $+0.0722$ & $-0.2226$ & $[0.1067, 0.1514]$ \\
$0.1800$ & $30$ & $0.461$ & $+0.0941$ & $-0.2799$ & $[0.1391, 0.1904]$ \\
$0.2200$ & $26$ & $0.507$ & $+0.1205$ & $-0.3440$ & $[0.1781, 0.2340]$ \\
$0.2600$ & $26$ & $0.543$ & $+0.1405$ & $-0.4110$ & $[0.2077, 0.2796]$ \\
$0.3400$ & $26$ & $0.612$ & $+0.1871$ & $-0.5301$ & $[0.2765, 0.3606]$ \\
$0.4200$ & $24$ & $0.678$ & $+0.2413$ & $-0.6554$ & $[0.3567, 0.4458]$ \\
$0.5000$ & $24$ & $0.737$ & $+0.2914$ & $-0.7732$ & $[0.4307, 0.5259]$ \\
$0.5800$ & $20$ & $0.793$ & $+0.3473$ & $-0.8935$ & $[0.5134, 0.6077]$ \\
$0.6600$ & $16$ & $0.842$ & $+0.4058$ & $-1.0044$ & $[0.5998, 0.6832]$ \\
$0.7000$ & $14$ & $0.866$ & $+0.4368$ & $-1.0611$ & $[0.6457, 0.7217]$ \\
$0.7400$ & $14$ & $0.888$ & $+0.4664$ & $-1.1183$ & $[0.6895, 0.7606]$ \\
$0.7800$ & $14$ & $0.909$ & $+0.4976$ & $-1.1715$ & $[0.7355, 0.7968]$ \\
$0.8200$ & $10$ & $0.929$ & $+0.5296$ & $-1.2256$ & $[0.7827, 0.8336]$ \\
$0.8600$ & $10$ & $0.947$ & $+0.5615$ & $-1.2799$ & $[0.8299, 0.8706]$ \\
$0.8778$ & $10$ & $0.955$ & $+0.5760$ & $-1.3018$ & $[0.8514, 0.8855]$ \\
$0.9000$ & $6$ & $0.966$ & $+0.5966$ & $-1.3337$ & $[0.8818, 0.9072]$ \\
$0.9165$ & $4$ & $0.973$ & $+0.6113$ & $-1.3519$ & $[0.9036, 0.9195]$ \\
$0.9279$ & $4$ & $0.977$ & $+0.6204$ & $-1.3682$ & $[0.9170, 0.9306]$ \\
$0.9375$ & $4$ & $0.980$ & $+0.6280$ & $-1.3810$ & $[0.9283, 0.9393]$ \\
$0.9460$ & $4$ & $0.983$ & $+0.6351$ & $-1.3930$ & $[0.9388, 0.9475]$ \\
$0.9525$ & $4$ & $0.985$ & $+0.6404$ & $-1.4015$ & $[0.9466, 0.9533]$ \\
$0.9562$ & $4$ & $0.987$ & $+0.6437$ & $-1.4066$ & $[0.9515, 0.9567]$ \\
$0.9610$ & $4$ & $0.988$ & $+0.6471$ & $-1.4116$ & $[0.9565, 0.9601]$ \\
\end{longtable}
\endgroup

\section{The Boundary Layer at the Start of the Tail}
\label{app:start}

This appendix contains the two lemmas behind Theorem~\ref{thm:initial_slope}. Throughout, $\omega \in (0,\pi)$ is fixed, $\theta_0 = (\pi - \omega)/2$, $s_0 := \tan\theta_0$, and $g$ and $l$ are the widths \eqref{eq:layer_widths}, so that $\Lambda := g + l = 2\omega/(\pi\sin\omega)$ and $g/l = s_0^2$. For $x \in (-g,l)$ and for $x \in \mathbb R\setminus[-g,l]$, respectively, let
\begin{equation}\label{eq:inner_variables}
t(x) := \Bigl(\frac{x+g}{l-x}\Bigr)^{1/2}, \qquad \eta(x) := \Bigl(\frac{x+g}{x-l}\Bigr)^{1/2},
\end{equation}
both positive, and define $e : \mathbb R \to \mathbb R$ and $m : \mathbb R\setminus\{0\} \to \mathbb R$ by $e := 0$ on $(-\infty,-g]$, $e := -1$ on $[l,\infty)$, $e(0) := -2\theta_0/\pi$,
\begin{equation}\label{eq:inner_solution}
e(x) := -\frac2\pi\arctan t(x) - \frac x\omega\log\Bigl|\frac{t(x) + s_0}{t(x) - s_0}\Bigr| \qquad (x \in (-g,l)\setminus\{0\}),
\end{equation}
and $m := 0$ on $[-g,0)$, $m := 1/\omega$ on $(0,l]$, $m(x) := (2/(\pi\omega))\arctan(\eta(x)/s_0)$ for $x \in \mathbb R\setminus[-g,l]$. In the variable $x = (\vartheta - \theta_0)/\tau$, $\tau e$ and $m$ describe the envelope, up to a constant, and the density of the measures constructed in Section~\ref{sec:initial_slope_proof} near $\theta_0$: the density vanishes on the gap $(-g,0)$ and equals the node density $1/\omega$ on the layer $(0,l)$, the envelope is maximal, $e = 0$, on the extension beyond the gap and smaller by $\tau$, $e = -1$, on the physical interval beyond the layer, and the two are related by the Hilbert transform (Lemma~\ref{lem:inner}(ii)), the limit as $\tau \to 0$ of the relation between envelope and density in Lemma~\ref{lem:shape}.

\begin{lemma}[The Boundary Layer]
\label{lem:inner}
\leavevmode
\begin{enumerate}
    \item[(i)] The function $e$ is absolutely continuous and nonincreasing, with $e(0) = -2\theta_0/\pi$. Its derivative is $e'(x) = -\omega^{-1}\log|(t(x) + s_0)/(t(x) - s_0)|$ on $(-g,l)\setminus\{0\}$ and $e' = 0$ on $\mathbb R\setminus(-g,l)$; it is H\"older continuous with exponent $1/2$ on every compact subset of $\mathbb R\setminus\{0\}$, $e' \in L^p(\mathbb R)$ for $1 \le p < \infty$, and $\int_{\mathbb R}|e'| = 1$.
    \item[(ii)] For every $x \ne 0$ the principal value $\pi^{-1}\,\mathrm{p.v.}\!\int_{\mathbb R}e'(y)\,(x-y)^{-1}\,dy$ exists and equals $1 - \pi m(x)$.
    \item[(iii)] The function $m$ is continuous on $\mathbb R\setminus\{0\}$, $0 \le m \le 1/\omega$, $m > 1/\pi$ on $(l,\infty)$, and $m(x) \to 1/\pi$ as $x \to -\infty$.
\end{enumerate}
\end{lemma}
\begin{proof}
(i) Let $\ell(t) := \log|(t + s_0)/(t - s_0)|$, which is positive for $t > 0$, $t \ne s_0$. By \eqref{eq:inner_variables}, $x = \Lambda\cos^2\theta_0\,(t^2 - s_0^2)/(1+t^2)$ and $dt/dx = (1+t^2)^2/(2\Lambda t)$ for $t = t(x)$. With $\ell'(t) = -2s_0/(t^2 - s_0^2)$ and $2s_0\cos^2\theta_0 = \sin\omega$, differentiating \eqref{eq:inner_solution} gives
\[
e'(x) = -\frac{\ell(t)}{\omega} - \frac{1+t^2}{t}\Bigl(\frac1{\pi\Lambda} - \frac{\sin\omega}{2\omega}\Bigr) = -\frac{\ell(t)}{\omega},
\]
since $\pi\Lambda\sin\omega = 2\omega$. As $x \to -g$, $t(x) \to 0$ and $\ell(t(x)) \to 0$; as $x \to l$, $t(x) \to \infty$, $\arctan t(x) \to \pi/2$ and $\ell(t(x)) \to 0$; as $x \to 0$, $t(x) \to s_0$, $\arctan s_0 = \theta_0$ and $x\,\ell(t(x)) \to 0$, because $t$ is smooth near $0$ with $t'(0) > 0$, so that $|\ell(t(x))| \le C(1 + |\log|x||)$. Hence $e$ is continuous, with $e(-g) = 0$ and $e(l) = -1$; being continuously differentiable off $\{-g,0,l\}$ with $e' \le 0$ and $|e'(x)| \le C(1 + |\log|x||)$, it is absolutely continuous and nonincreasing, $e' \in L^p$ for $p < \infty$, and $\int|e'| = e(-g) - e(l) = 1$. Near $-g$, $e'(x) = -(2/\omega)\operatorname{artanh}(t(x)/s_0)$, and $t(x)$ is a smooth function of $(x+g)^{1/2}$ that vanishes at $-g$; near $l$, $e'(x) = -(2/\omega)\operatorname{artanh}(s_0/t(x))$, and $1/t(x)$ is a smooth function of $(l-x)^{1/2}$ that vanishes at $l$. Hence $|e'(x)| \le C|x + g|^{1/2}$ near $-g$ and $|e'(x)| \le C|l - x|^{1/2}$ near $l$, so that $e$ is differentiable at $-g$ and at $l$ with $e' = 0$ there; as $e' = 0$ outside $[-g,l]$ and $e'$ is smooth on $(-g,l)\setminus\{0\}$, it is H\"older continuous with exponent $1/2$ on compact subsets of $\mathbb R\setminus\{0\}$.

(ii) Let $\mathbb H := \{z : \mathrm{Im}\,z > 0\}$. The M\"obius map $z \mapsto (z+g)/(l-z)$, with real coefficients and determinant $g + l > 0$, maps $\mathbb H$ onto itself, and the principal square root maps $\mathbb H$ onto the open first quadrant $Q$; so $\zeta(z) := ((z+g)/(l-z))^{1/2}$ maps $\mathbb H$ onto $Q$, and it extends continuously to $\bar{\mathbb H}\setminus\{l\}$, with $\zeta(x) = t(x)$ for $x \in (-g,l)$, $\zeta(x) = i\eta(x)$ for $x \in \mathbb R\setminus[-g,l]$, and $\zeta(z) \to \infty$ as $z \to l$. For $\zeta \in \bar Q\setminus\{s_0\}$, $\mathcal T(\zeta) := (\zeta + s_0)/(\zeta - s_0)$ satisfies $\mathcal T \ne 0$ and $\mathrm{Im}\,\mathcal T = -2s_0\,\mathrm{Im}\,\zeta/|\zeta - s_0|^2 \le 0$, and $\mathcal T(\zeta) \to 1$ as $\zeta \to \infty$. Let $\log_\ast$ be the branch of the logarithm on $\mathbb C\setminus\{iy : y \ge 0\}$ whose imaginary part lies in $(-3\pi/2, \pi/2)$, and put
\[
W(z) := -\frac1\omega\Bigl(\log_\ast \mathcal T(\zeta(z)) + 2i\theta_0\Bigr), \qquad z \in \bar{\mathbb H}\setminus\{0\},
\]
where $\mathcal T(\zeta(l)) := 1$. Since $\mathcal T\circ\zeta$ is continuous on $\bar{\mathbb H}\setminus\{0\}$, with values in the closed lower half-plane minus the origin, on which $\log_\ast$ is continuous, $W$ is continuous on $\bar{\mathbb H}\setminus\{0\}$ and analytic in $\mathbb H$; the point $z = 0$, where $\zeta = s_0$, is the pole of $\mathcal T\circ\zeta$. On the real axis: for $x \in (0,l)$, $\mathcal T = (t + s_0)/(t - s_0) > 1$ and $W(x) = -(\ell(t) + 2i\theta_0)/\omega$; for $x \in (-g,0)$, $\mathcal T < -1$, $\log_\ast \mathcal T = \log|\mathcal T| - i\pi$ and $W(x) = -\ell(t)/\omega + i$, as $\pi - 2\theta_0 = \omega$; for $x \in \mathbb R\setminus[-g,l]$, $\zeta = i\eta$ with $\eta = \eta(x)$, and with $\beta := \arctan(\eta/s_0) \in (0,\pi/2)$, $\mathcal T = -(s_0 + i\eta)^2/(s_0^2 + \eta^2) = e^{i(2\beta - \pi)}$ and $W(x) = i(1 - 2\beta/\omega)$. By (i) and the definition of $m$, therefore,
\begin{equation}\label{eq:W_boundary}
\mathrm{Re}\,W(x) = e'(x), \qquad \mathrm{Im}\,W(x) = 1 - \pi m(x) \qquad (x \in \mathbb R\setminus\{0\}).
\end{equation}
Near $z = 0$, $\zeta$ is analytic with $\zeta(0) = s_0$ and $\zeta'(0) > 0$, so $\mathcal T\circ\zeta$ has a simple pole and $|W(z)| \le C(1 + |\log|z||)$ for $z \in \bar{\mathbb H}$ near $0$. For $|z - l| > \Lambda$, $(z+g)/(l-z) = -(1 - \Lambda/(l-z))$, and $\zeta(z) = i(1 - \Lambda/(l-z))^{1/2}$ with the principal root, as the right-hand side lies in $Q$ for $z \in \mathbb H$; so $\zeta$ extends analytically to a neighbourhood of $\infty$ with $\zeta(\infty) = i$, where $\mathcal T(i) = e^{-2i\theta_0}$ lies off the cut of $\log_\ast$. Hence $W$ is analytic at $\infty$ with $W(\infty) = 0$, and $|W(z)| \le C/|z|$ for large $|z|$.

Fix $z \in \mathbb H$. Cauchy's integral formula on $\{w \in \mathbb H : r_1 < |w| < r_2,\ \mathrm{Im}\,w > h_0\}$, followed by $h_0 \to 0$, which is legitimate because $W$ is uniformly continuous on compact subsets of $\bar{\mathbb H}\setminus\{0\}$, and by $r_1 \to 0$ and $r_2 \to \infty$, the integrals over the half-circles tending to zero by the two bounds on $W$, gives
\[
W(z) = \frac1{2\pi i}\int_{\mathbb R}\frac{W(y)}{y - z}\,dy, \qquad 0 = \frac1{2\pi i}\int_{\mathbb R}\frac{W(y)}{y - \bar z}\,dy,
\]
the second because $\bar z$ lies outside the domain. Subtracting the complex conjugate of the second identity from the first and using \eqref{eq:W_boundary} gives $W(z) = (\pi i)^{-1}\int_{\mathbb R}e'(y)(y - z)^{-1}\,dy$, and with $z = x + ih$, $h > 0$,
\[
\mathrm{Im}\,W(x + ih) = \frac1\pi\int_{\mathbb R}e'(y)\,k_h(x - y)\,dy, \qquad k_h(v) := \frac{v}{v^2 + h^2}.
\]
Let $x \ne 0$. As $h \to 0$ the left side tends to $1 - \pi m(x)$ by \eqref{eq:W_boundary}. On the right, choose $\rho_1 \in (0,|x|/2)$ and $K$ with $|e'(y) - e'(x)| \le K|y - x|^{1/2}$ for $|y - x| \le \rho_1$, by (i). Since $k_h$ is odd, $|k_h(v)| \le 1/|v|$ and $k_h(v) \to 1/v$, the right side equals
\[
\frac1\pi\int_{|y-x|>\rho_1}e'(y)\,k_h(x-y)\,dy + \frac1\pi\int_{|y-x|\le\rho_1}\bigl(e'(y) - e'(x)\bigr)\,k_h(x-y)\,dy
\]
and tends by dominated convergence to the same expression with $(x-y)^{-1}$ in place of $k_h(x-y)$, which is the principal value in (ii), because $\int_{\rho'<|y-x|\le\rho_1}(x-y)^{-1}\,dy = 0$ for $0 < \rho' < \rho_1$. This proves (ii).

(iii) On $(-\infty,-g)$, $\eta(x) \in (0,1)$ and $\eta(x) \to 1$ as $x \to -\infty$; on $(l,\infty)$, $\eta(x) > 1$; and $\eta \to 0$ at $-g$, $\eta \to \infty$ at $l$. Since $\arctan(1/s_0) = \pi/2 - \theta_0 = \omega/2$, $m$ takes values in $(0,1/\pi)$ on $(-\infty,-g)$ and tends to $1/\pi$ as $x \to -\infty$, takes values in $(1/\pi, 1/\omega)$ on $(l,\infty)$, and is continuous at $-g$ and at $l$.
\end{proof}

\begin{lemma}[Envelopes of Prescribed Shape]
\label{lem:shape}
Let $q$ be an absolutely continuous function on $[0,\pi]$ with $q' \in L^2(0,\pi)$ and cosine coefficients $\hat q_k := \int_0^\pi q(\varphi)\cos k\varphi\,d\varphi$, and let $\mathcal Nq \in L^2(0,\pi)$ be the function with the cosine series $(2/\pi)\sum_{k\ge1}k\,\hat q_k\cos k\vartheta$.
\begin{enumerate}
    \item[(i)] For $\tau \ge 0$ the function $f := \pi^{-1}(1 - \tau\mathcal Nq)$ satisfies $\int_0^\pi f = 1$ and $E_f = -\log 2 + \tau(q - \pi^{-1}\hat q_0)$ on $[0,\pi]$.
    \item[(ii)] If $q'$ is H\"older continuous on a neighbourhood of $\vartheta \in (0,\pi)$, the principal value
    \[
    P(\vartheta) := \frac1\pi\,\mathrm{p.v.}\!\int_0^\pi q'(\varphi)\,\frac{\sin\varphi}{\cos\varphi - \cos\vartheta}\,d\varphi
    \]
    exists. If this holds for all $\vartheta \in (0,\pi)$ outside a finite set, then $\mathcal Nq = P$ a.e.
\end{enumerate}
\end{lemma}
\begin{proof}
(i) Integration by parts gives $k\hat q_k = -b_k$ with $b_k := \int_0^\pi q'(\varphi)\sin k\varphi\,d\varphi$, $k \ge 1$, and Bessel's inequality gives $\sum_{k\ge1}b_k^2 \le (\pi/2)\norm{q'}_{L^2}^2$. Hence $\mathcal Nq \in L^2(0,\pi)$, and $\sum_{k\ge1}|\hat q_k| \le (\pi^2/6)^{1/2}(\sum_kb_k^2)^{1/2} < \infty$, so the cosine series $\pi^{-1}\hat q_0 + (2/\pi)\sum_{k\ge1}\hat q_k\cos k\vartheta$ converges uniformly; its sum is continuous and has the cosine coefficients of $q$, hence equals $q$. The function $f$ has the cosine coefficients $\hat f_0 = 1$ and $\hat f_k = -\pi^{-1}\tau k\hat q_k$, $k \ge 1$, and Lemma~\ref{lem:cheb} gives $E_f(\vartheta) = -\log 2 + (2\tau/\pi)\sum_{k\ge1}\hat q_k\cos k\vartheta = -\log 2 + \tau(q(\vartheta) - \pi^{-1}\hat q_0)$.

(ii) For $0 < r < 1$ let $\mathcal N_rq := (2/\pi)\sum_{k\ge1}r^kk\hat q_k\cos k\vartheta$. By Parseval's identity, $\norm{\mathcal N_rq - \mathcal Nq}_{L^2}^2 = (2/\pi)\sum_k(1 - r^k)^2k^2\hat q_k^{\,2} \to 0$ as $r \to 1$, so $\mathcal N_{r_n}q \to \mathcal Nq$ a.e.\ for a sequence $r_n \to 1$. On the other hand, since $k\hat q_k = -b_k$ and $\sum_kr^k < \infty$,
\[
\begin{aligned}
\mathcal N_rq(\vartheta) &= -\frac2\pi\int_0^\pi q'(\varphi)\sum_{k\ge1}r^k\sin k\varphi\cos k\vartheta\,d\varphi \\
&= -\frac1\pi\int_0^\pi q'(\varphi)\bigl(S_r(\varphi + \vartheta) + S_r(\varphi - \vartheta)\bigr)\,d\varphi,
\end{aligned}
\]
where $S_r(v) := \sum_{k\ge1}r^k\sin kv = r\sin v/(1 - 2r\cos v + r^2)$. As $1 - 2r\cos v + r^2 = (1-r)^2 + 4r\sin^2(v/2)$, $S_r$ is odd, $|S_r(v)| \le \frac12|\cot(v/2)|$, which is at most $1/|v|$ for $0 < |v| \le \pi$, and $S_r(v) \to \frac12\cot(v/2)$ as $r \to 1$ for $v \notin 2\pi\mathbb Z$. Let $q'$ be H\"older continuous near $\vartheta \in (0,\pi)$. For $\varphi \in [0,\pi]$, $\varphi + \vartheta \in [\vartheta, \pi + \vartheta]$, so the term with $S_r(\varphi + \vartheta)$ converges by dominated convergence; the term with $S_r(\varphi - \vartheta)$ converges to $\frac12\,\mathrm{p.v.}\int_0^\pi q'(\varphi)\cot\frac{\varphi - \vartheta}2\,d\varphi$, which exists, by the argument at the end of the proof of Lemma~\ref{lem:inner}(ii), with the odd functions $S_r(\varphi - \vartheta)$ and $\frac12\cot\frac{\varphi - \vartheta}2$ and the bound $1/|\varphi - \vartheta|$ in place of $k_h(x - y)$, $(x - y)^{-1}$ and $1/|x - y|$, and with $\rho_1 \le \min\{\vartheta, \pi - \vartheta\}$. Since $-\cot\frac{\varphi+\vartheta}2 - \cot\frac{\varphi-\vartheta}2 = 2\sin\varphi/(\cos\varphi - \cos\vartheta)$, the limit is $P(\vartheta)$, and $\mathcal Nq = P$ at every such $\vartheta$ at which $\mathcal N_{r_n}q \to \mathcal Nq$.
\end{proof}


\bibliographystyle{emss}
\bibliography{ref}

\end{document}